\documentclass[10pt, reqno,amsmath,amsthm,amssymb,amscd]{amsart}
\usepackage{amsmath}
\usepackage{mathrsfs,amssymb, amscd,amsmath,amsthm}
\usepackage[enableskew,vcentermath]{youngtab}
\usepackage{multicol}\multicolsep=0pt
\usepackage{tikz}
\newcommand{\END}{\mathcal{E}nd}
\newcommand{\clr}{rgb:black,1;blue,4;red,1}

\newcommand{\ob}[1]{\mathsf{#1}}

\newcommand{\B}{\mathcal{K}}
\newcommand{\CB}{\mathcal{CK}}
\newcommand{\AB}{\mathcal{AK}}

\newcommand{\lcap}{
\begin{tikzpicture}[baseline = 3pt, scale=0.5, color=\clr]
        \draw[-,thick] (1,0) to[out=up, in=right] (0.53,0.5) to[out=left, in=right] (0.47,0.5);
        \draw[-,thick] (0.49,0.5) to[out=left,in=up] (0,0);
\end{tikzpicture}
}
\newcommand{\lcup}{
\begin{tikzpicture}[baseline = 6pt, scale=0.5, color=\clr]
        \draw[-,thick] (1,1) to[out=down, in=right] (0.53,0.5) to[out=left, in=right] (0.47,0.5);
        \draw[-,thick] (0.49,0.5) to[out=left,in=down] (0,1);
\end{tikzpicture}
}

 \providecommand{\og}{``}
\providecommand{\fg}{''} \providecommand{\smfandname}{and}

\usepackage{amssymb}
\usepackage{mathrsfs,amssymb}
\usepackage{multicol}\multicolsep=0pt
\usepackage{pstricks,pst-node}

\usepackage[enableskew,vcentermath]{youngtab}
\usepackage[sort]{cite}
\usepackage{xcolor,graphicx}

\def\crulefill{\leavevmode\leaders\hrule height 1pt\hfill\kern 0pt}
\long\def\QUERY#1{%
\leavevmode\newline%
\noindent$\star\star\star$\thinspace\textsf{Comment/Query}\crulefill\newline%
   \space #1\newline\hbox to 120mm{\crulefill}$\star\star\star$\newline}
\newtheorem{Theorem}{Theorem}[section]
\newtheorem{Lemma}[Theorem]{Lemma}
\newtheorem{Cor}[Theorem]{Corollary}
\newtheorem{Prop}[Theorem]{Proposition}

 \theoremstyle{definition}
\newtheorem{example}[Theorem]{Example}

\newtheorem{Defn}[Theorem]{Definition}

\newtheorem{Assumption}[Theorem]{Assumption}

\numberwithin{equation}{section}
\theoremstyle{definition}

\makeatletter
\def\enumerate{\begingroup\ifnum\@enumdepth>3\@toodeep\else
      \advance\@enumdepth\@ne
      \edef\@enumctr{enum\romannumeral\the\@enumdepth}%
      \topsep\z@\parskip\z@
      \list{\csname label\@enumctr\endcsname}
        {\@nmbrlisttrue\let\@listctr\@enumctr
         \parsep\z@\itemsep\z@\topsep\z@
         \setcounter{\@enumctr}{0}
         \def\makelabel##1{\hss\llap{\rm ##1}}
       }\fi}

\makeatother

\let\bar=\overline
\let\epsilon=\varepsilon
\def\({\big(}
\def\){\big)}

\def\0{\underline{0}}

\DeclareMathOperator{\End}{End}

\def\Hom{\text{Hom}}

\def\U{\mathbf U}

{\catcode`\|=\active
  \gdef\set#1{\mathinner{\lbrace\,{\mathcode`\|"8000%
                                   \let|\midvert #1}\,\rbrace}}
  \gdef\seT#1{\mathinner{\Big\lbrace\,{\mathcode`\|"8000%
                                   \let|\midverT #1}\,\Big\rbrace}}
}
\def\midvert{\egroup\mid\bgroup}
\def\midverT{\egroup\,\Big|\,\bgroup}

\def\Set[#1]#2|#3|{\Big\{\ #2\ \Big| \
           \vcenter{\hsize #1mm\centering #3}\Big\}}

\def\qed{\hfill\mbox{$\Box$}}

\def\Hom{{\rm Hom}}

\def\mfg{{\mathfrak g}}

\def\Set{{\rm Set}}

\newcommand{\C}{\mathcal{C}}

\def\Hom{\text{Hom}}%
\def\U{\mathbf U}%
\def\textsf#1{{\textit{#1}}}%

\definecolor{white}{HTML}{FFFFFF}
\definecolor{darkblue}{HTML}{111199}
\definecolor{darkgreen}{HTML}{336633}
\definecolor{darkred}{HTML}{993333}

\definecolor{darkpurple}{HTML}{995599}

\newcommand{\FT}{\mathcal{FT}}

\begin{document}
\baselineskip15pt
\title[{\tiny  The affine  Kauffmann  category and the   cyclotomic Kauffmann category}]{A basis theorem for the affine  Kauffmann  category and its    cyclotomic quotients}
\author{Mengmeng Gao, Hebing Rui, Linliang Song}
\address{M.G.  School of Mathematical Science, Tongji University,  Shanghai, 200092, China}\email{g19920119@163.com}
\address{H.R.  School of Mathematical Science, Tongji University,  Shanghai, 200092, China}\email{hbrui@tongji.edu.cn}
\address{L.S.  School of Mathematical Science, Tongji University,  Shanghai, 200092, China}\email{llsong@tongji.edu.cn}
\subjclass[2010]{17B10, 18D10, 33D80}
\keywords{Affine   Kauffmann  category,  cyclotomic Kauffmann category, quantum group of type $B/C/D$, basis theorem}
\thanks{ M. Gao and H. Rui is supported  partially by NSFC (grant No.  11971351).  L. Song is supported  partially by NSFC (grant No.  11501368). }
\sloppy
\maketitle

\begin{abstract}The  affine Kauffmann  category is a strict monoidal category and   can be considered as a  $q$-analogue of the affine Brauer category in (Rui et al. in Math. Zeit. 293, 503-550,  2019). In this paper, we prove a basis theorem for the morphism  spaces  in the affine Kauffmann  category. The cyclotomic Kauffmann category is a quotient category of the affine Kauffmann category. We also prove that any morphism space in this category   is free over an integral domain
 $\mathbb K$ with maximal rank if and only if the $\mathbf u$-admissible condition holds in the sense of Definition~\ref{uad}.
\end{abstract}

\maketitle
\section{Introduction}
In the last several decades, one of important developments in representation theory  is a categorification of modules of Lie algebras or quantum groups.
For example,  Ariki's categorification theorem\cite{Ari} establishes relationships between  various  important stories (blocks, irreducible modules, projective indecomposable modules, etc) in the representation theory of cyclotomic Hecke algebras and important invariants (weight spaces, dual canonical basis, canonical basis, etc) of certain irreducible modules of  Lie algebras $\mathfrak {sl}_\infty$ or $\widehat{\mathfrak{sl}}_p$.
In \cite{BCNR}, Brundan et.~al  introduced the affine oriented Brauer category and the cyclotomic oriented Brauer category. A special case of the cyclotomic oriented Brauer category  is the  oriented Brauer category. Associated to it, there is a locally unital and locally finite dimensional algebra.
Reynolds\cite{Re}  constructed a categorical action of the Lie algebra $\mathfrak {sl}_\infty$ or $\widehat{\mathfrak{sl}}_p$ on the  Grothendick group of certain module category   for this locally unital and locally finite dimensional algebra.
Since oriented Brauer category is the category version of the walled Brauer algebra\cite{Kok},  Reynolds' result can be considered as a categorification related to the walled Brauer algebra. See  also  \cite{Br} for the oriented skein category which is the category version of the quantized  walled Brauer algebra\cite{Led}.

The paper is a part  of our project for studying  categorifications related to finite dimensional  algebras arising from Schur-Weyl dualities in types $B,C$ and $D$.
In \cite{RS3}, two of authors introduced the affine Brauer category and the cyclotomic Brauer category. The  affine Nazarov-Wenzl algebras~\cite{Na}   and the  cyclotomic Nazarov-Wenzl algebras~\cite{AMR} appear   as endomorphism algebras of objects in these categories,  respectively. Furthermore, they establish a higher Schur-Weyl duality between cyclotomic Brauer categories and BGG parabolic category $\mathcal O$ in types $B, C$ and $D$. This enables them  to  compute decomposition matrices of the cyclotomic Nazarov-Wenzl algebra over the complex field $\mathbb C$.
A special case of the cyclotomic Brauer category is the Brauer category\cite{LZ}(whose additive Karoubi envelope is the Deligne category $\underline{Rep~ } O_{\omega_0}$\cite{De1, De2}).
 For the Brauer category, further results have been obtained.  More explicitly, using the representation theory of Brauer category, they~\cite{RS4}  gave a categorification related to the Brauer algebra in \cite{B}. In this picture, certain modules of coideal algebras in \cite{Bao} come into picture. Thanks to  certain  exact truncation functors,   representations of  Brauer algebras can be reflected in the representations of the Brauer category.

Throughout this paper,  $\mathbb K$ is an integral domain. We are going to introduce the affine Kauffmann category $\AB$. It is
   a $\mathbb K$-linear  strict monoidal category generated by a single object \begin{tikzpicture}[baseline = 1.5mm]
	\draw[-,thick,darkblue] (0.18,0) to (0.18,.4);
\end{tikzpicture}. The category $\AB$ comes equipped with infinitely many algebraically independent  parameters $\Delta_1,\Delta_2,\ldots$. On evaluating these
parameters at scalars in $\mathbb K$, we obtain specialization $\AB(\omega)$ and
 introduce the cyclotomic Kauffmann category $\CB^f$ associated to a monic polynomial $f(t)\in \mathbb K[t]$.  After we establish
relationships between $\AB$ (resp.,   $\CB^f$) and   the category of endofunctors of the module category (resp.,   parabolic BGG category $\mathcal O$) associated to the quantum symplectic  groups and quantum orthogonal groups, we are able to   prove that any  morphism  space  in $\AB$ is free  over  $\mathbb K$ and  the affine Birman-Murakami-Wenzl algebra~\cite{Good} appears as an endomorphism algebra $\End_{\AB(\omega)}(
\begin{tikzpicture}[baseline = 10pt, scale=0.5, color=\clr]
                \draw[-,thick] (0,0.5)to[out=up,in=down](0,1.2);
                    \end{tikzpicture}^{\otimes r})$ for some positive integer $r$.
We also prove that any morphism space  in  $\CB^f$ is free over $\mathbb K$ with maximal rank if and only if the $\mathbf u$-admissible condition holds in the sense of Definition~\ref{uad}. In this case, the  cyclotomic Birman-Murakami-Wenzl algebra~\cite{BW, Olden} appears  as an  endomorphism algebra $\End_{\CB^f}(
\begin{tikzpicture}[baseline = 10pt, scale=0.5, color=\clr]
                \draw[-,thick] (0,0.5)to[out=up,in=down](0,1.2);
                    \end{tikzpicture}^{\otimes r})$ for some positive integer $r$.

                    In a subsequent work, we will investigate the   representations of the cyclotomic Kauffmann category and the cyclotomic Brauer category.
     Based on previous observations, we believe that representations of the cyclotomic Birman-Murakami-Wenzl algebra (resp., cyclotomic Nazarov-Wenzl algebra) can be reflected in the  representation theory of cyclotomic Kauffmann category (resp., cyclotomic Brauer category).

    In the remaining part of this section, we are going to introduce these categories and formulate our main results precisely.

We start by recalling the definition of the  category $\FT$ of framed tangles. It   is the  $\mathbb K$-linear strict monoidal category
  generated by a single  object
 \begin{tikzpicture}[baseline = 10pt, scale=0.5, color=\clr]
                \draw[-,thick] (0,0.5)to[out=up,in=down](0,1.2);
    \end{tikzpicture} (see e.g.,~\cite{VT}).   Thus,  the set of    objects in  $\FT$ is $\{\ob m
 \mid m\in \mathbb N\}$, where $\ob m$ represents
$ \begin{tikzpicture}[baseline = 10pt, scale=0.5, color=\clr]
                \draw[-,thick] (0,0.5)to[out=up,in=down](0,1.2);
                    \end{tikzpicture}^{\otimes m}$, and    $\ob 0$ represents the unit object.
                        For any  objects $\ob m$ and $\ob s$,     morphisms $f: \ob m \rightarrow \ob s$ are isotropy classes of
    framed tangles in $[0, 1]\times [0, 1]\times \mathbb R$ in $3$-dimensional real space $\mathbb R^3$ with boundary
    $$\left\{ (1-\frac{i}{m+1}, 0, 0)\mid i=1, 2, \ldots, m\right\}
    \cup \left\{(1- \frac{j}{s+1}, 1, 0)\mid j=1, 2, \ldots, s\right\}.$$
    Such tangles will be drawn by projecting them onto the $xy$-plane, and there are neither  triple intersections nor tangencies. Further, any crossing of a tangle will be recorded as either over crossing or under crossing. The resulting diagrams are called  $(m,s)$-tangle diagrams.
Isotropy translates into the equivalence relation on diagrams generated by planar isotropy fixing the boundary together with the Reidemeister Moves of types (RI)-(RIII) as follows:

 $$\begin{tikzpicture}[baseline = -5mm]
 \draw[-,thick,darkblue] (0,0.5) to (0,0.3);
 \draw[-,thick,darkblue] (0.5,0) to [out=90,in=0](.3,0.2);
 \draw[-,thick,darkblue] (0,-0.3) to (0,-0.6);
 \draw[-,thick,darkblue] (0.3,-0.2) to [out=0,in=-90](.5,0);
 \draw[-,thick,darkblue] (0,0.3) to [out=-90,in=180] (.3,-0.2);
 \draw[-,line width=4pt,white] (0.3,.2) to [out=180,in=90](0,-0.3);
 \draw[-,thick,darkblue] (0.3,.2) to [out=180,in=90](0,-0.3);
 \draw[-,thick,darkblue] (0.5,-0.9) to [out=90,in=0](.3,-0.7);
 \draw[-,thick,darkblue] (0,-1.2) to (0,-1.4);
 \draw[-,thick,darkblue] (0.3,-1.1) to [out=0,in=-90](.5,-0.9);
    \draw[-,thick,darkblue] (0.3,-0.7) to [out=180,in=90](0,-1.2);
    \draw[-,line width=4pt,white] (0,-0.6) to [out=-90,in=180](0.3,-1.1);
    \draw[-,thick,darkblue] (0,-0.6) to [out=-90,in=180] (.3,-1.1);
   \node at (0.8,-.4) {$\text {=}$};
 \draw[-,thick,darkblue] (1.2,0.5) to (1.2,-1.2);
 \node at (0.8,-2.) {$\text {(RI)}$};
\end{tikzpicture},\qquad\mathord{
\begin{tikzpicture}[baseline = -1mm]
	\draw[-,thick,darkblue] (0.28,-.6) to[out=90,in=-90] (-0.28,0);
	\draw[-,thick,darkblue] (-0.28,0) to[out=90,in=-90] (0.28,.6);
	\draw[-,line width=4pt,white] (-0.28,-.6) to[out=90,in=-90] (0.28,0);
	\draw[-,thick,darkblue] (-0.28,-.6) to[out=90,in=-90] (0.28,0);
	\draw[-,line width=4pt,white] (0.28,0) to[out=90,in=-90] (-0.28,.6);
	\draw[-,thick,darkblue] (0.28,0) to[out=90,in=-90] (-0.28,.6);
\end{tikzpicture}
}=
\mathord{
\begin{tikzpicture}[baseline = -1mm]
	\draw[-,thick,darkblue] (0.18,-.6) to (0.18,.6);
	\draw[-,thick,darkblue] (-0.18,-.6) to (-0.18,.6);
 \node at (0.,-1.6) {$\text {(RII)}$};
\end{tikzpicture}
}\:=
\mathord{
\begin{tikzpicture}[baseline = -1mm]
	\draw[-,thick,darkblue] (0.28,0) to[out=90,in=-90] (-0.28,.6);
	\draw[-,line width=4pt,white] (-0.28,0) to[out=90,in=-90] (0.28,.6);
	\draw[-,thick,darkblue] (-0.28,0) to[out=90,in=-90] (0.28,.6);
	\draw[-,thick,darkblue] (-0.28,-.6) to[out=90,in=-90] (0.28,0);
	\draw[-,line width=4pt,white] (0.28,-.6) to[out=90,in=-90] (-0.28,0);
	\draw[-,thick,darkblue] (0.28,-.6) to[out=90,in=-90] (-0.28,0);
\end{tikzpicture}},\qquad\mathord{
\begin{tikzpicture}[baseline = -1mm]
	\draw[-,thick,darkblue] (0.45,-.6) to (-0.45,.6);
        \draw[-,thick,darkblue] (0,-.6) to[out=90,in=-90] (-.45,0);
        \draw[-,line width=4pt,white] (-0.45,0) to[out=90,in=-90] (0,0.6);
        \draw[-,thick,darkblue] (-0.45,0) to[out=90,in=-90] (0,0.6);
	\draw[-,line width=4pt,white] (0.45,.6) to (-0.45,-.6);
	\draw[-,thick,darkblue] (0.45,.6) to (-0.45,-.6);
\end{tikzpicture}
}
\begin{tikzpicture}[baseline = -1mm]
	 \node at (0.,0) {$\text {=}$};
 \node at (0.,-1.6) {$\text {(RIII)}$};
\end{tikzpicture}
\mathord{
\begin{tikzpicture}[baseline = -1mm]
	\draw[-,thick,darkblue] (0.45,-.6) to (-0.45,.6);
        \draw[-,line width=4pt,white] (0,-.6) to[out=90,in=-90] (.45,0);
        \draw[-,thick,darkblue] (0,-.6) to[out=90,in=-90] (.45,0);
        \draw[-,thick,darkblue] (0.45,0) to[out=90,in=-90] (0,0.6);
	\draw[-,line width=4pt,white] (0.45,.6) to (-0.45,-.6);
	\draw[-,thick,darkblue] (0.45,.6) to (-0.45,-.6);
\end{tikzpicture}}.$$
Tensor product of morphisms  is given by horizontal  concatenation and  composition of morphisms is given by vertical stacking (in a strict monoidal category).
For example,
$$ \label{com1}
      g\circ h= \begin{tikzpicture}[baseline = 19pt,scale=0.5,color=\clr,inner sep=0pt, minimum width=11pt]
        \draw[-,thick] (0,0) to (0,3);
        \draw (0,2.2) node[circle,draw,thick,fill=white]{$g$};
        \draw (0,0.8) node[circle,draw,thick,fill=white]{$h$};
            \end{tikzpicture}
    ~,~ \ \ \ \ \ g\otimes h=\begin{tikzpicture}[baseline = 19pt,scale=0.5,color=\clr,inner sep=0pt, minimum width=11pt]
        \draw[-,thick] (0,0) to (0,3);
        \draw[-,thick] (2,0) to (2,3);
        \draw (0,1.5) node[circle,draw,thick,fill=white]{$g$};
        \draw (2,1.5) node[circle,draw,thick,fill=white]{$h$};
            \end{tikzpicture}~.
$$

Suppose  $\mathbb K$ contains  $\delta, \delta^{-1}, z, \omega_0$. In this paper, we always assume
\begin{equation}\label{para1}
\delta-\delta^{-1}=z(\omega_0-1).
\end{equation}
 Let $I$ be  the tensor ideal of $\FT$, which is uniquely determined by Kauffmann skein relation (S), twisting relation (T) and free loop relation (L) as follows:

 $$
\mathord{
\begin{tikzpicture}[baseline = 2.5mm]
	\draw[-,thick,darkblue] (0.28,0) to[out=90,in=-90] (-0.28,.6);
	\draw[-,line width=4pt,white] (-0.28,0) to[out=90,in=-90] (0.28,.6);
	\draw[-,thick,darkblue] (-0.28,0) to[out=90,in=-90] (0.28,.6);
\end{tikzpicture}
}-\mathord{
\begin{tikzpicture}[baseline = 2.5mm]
	\draw[-,thick,darkblue] (-0.28,-.0) to[out=90,in=-90] (0.28,0.6);
	\draw[-,line width=4pt,white] (0.28,-.0) to[out=90,in=-90] (-0.28,0.6);
	\draw[-,thick,darkblue] (0.28,-.0) to[out=90,in=-90] (-0.28,0.6);
\end{tikzpicture}
}\begin{tikzpicture}[baseline = -1mm]
	 \node at (0.,0) {$\text {=}$};
 \node at (0.,-1.) {$\text {(S)}$};
\end{tikzpicture}
z(\:\mathord{
\begin{tikzpicture}[baseline = 2.5mm]
	\draw[-,thick,darkblue] (0.18,0) to (0.18,.6);
	\draw[-,thick,darkblue] (-0.18,0) to (-0.18,.6);
\end{tikzpicture}}-\mathord{
\begin{tikzpicture}[baseline = 2.5mm]\draw[-,thick,darkblue] (0,0.6) to[out=down,in=left] (0.28,0.35) to[out=right,in=down] (0.56,0.6);
  \draw[-,thick,darkblue] (0,0) to[out=up,in=left] (0.28,0.25) to[out=right,in=up] (0.56,0);
         \end{tikzpicture}
}),\qquad\quad
\mathord{
\begin{tikzpicture}[baseline = -0.5mm]
	\draw[-,thick,darkblue] (0,0.6) to (0,0.3);
	\draw[-,thick,darkblue] (0.5,0) to [out=90,in=0](.3,0.2);
	\draw[-,thick,darkblue] (0,-0.3) to (0,-0.6);
	\draw[-,thick,darkblue] (0.3,-0.2) to [out=0,in=-90](.5,0);
	\draw[-,thick,darkblue] (0,0.3) to [out=-90,in=180] (.3,-0.2);
	\draw[-,line width=4pt,white] (0.3,.2) to [out=180,in=90](0,-0.3);
	\draw[-,thick,darkblue] (0.3,.2) to [out=180,in=90](0,-0.3);
\end{tikzpicture}
}
\begin{tikzpicture}[baseline = -1mm]
	 \node at (0.,0) {$\text {=}$};
 \node at (0.,-1.) {$\text {(T)}$};
\end{tikzpicture}
\delta\:\mathord{
\begin{tikzpicture}[baseline = -0.5mm]
	\draw[-,thick,darkblue] (0,0.6) to (0,-0.6);
\end{tikzpicture}
}\;,\qquad\quad
\mathord{
\begin{tikzpicture}[baseline = -0.5mm]
\draw[-,thick,darkblue] (0,0) to[out=down,in=left] (0.25,-0.25) to[out=right,in=down] (0.5,0);
  \draw[-,thick,darkblue] (0,0) to[out=up,in=left] (0.25,0.25) to[out=right,in=up] (0.5,0);
  \end{tikzpicture}}\begin{tikzpicture}[baseline = -1mm]
	 \node at (0.,0) {$\text {=}$};
 \node at (0.,-1.) {$\text {(L)}$};
\end{tikzpicture}\omega_01_{\ob 0} \,.
$$

\noindent The Kauffmann category $\B$ is the quotient category $\FT/I$ ~\cite{VT}.
In order to simplify notation,  $\lcup, \lcap, \begin{tikzpicture}[baseline = 2.5mm]
	\draw[-,thick,darkblue] (0.28,0) to[out=90,in=-90] (-0.28,.6);
	\draw[-,line width=4pt,white] (-0.28,0) to[out=90,in=-90] (0.28,.6);
	\draw[-,thick,darkblue] (-0.28,0) to[out=90,in=-90] (0.28,.6);
\end{tikzpicture}$ and $\begin{tikzpicture}[baseline = 2.5mm]
	\draw[-,thick,darkblue] (-0.28,-.0) to[out=90,in=-90] (0.28,0.6);
	\draw[-,line width=4pt,white] (0.28,-.0) to[out=90,in=-90] (-0.28,0.6);
	\draw[-,thick,darkblue] (0.28,-.0) to[out=90,in=-90] (-0.28,0.6);
\end{tikzpicture}$ will be denoted by  $U, A, T$ and $T^{-1}$, respectively.
The following result  gives the presentation of  the Kauffmann category $\B$.

\begin{Theorem}\cite{VT}\label{present1} The Kauffmann category $\B$ is the strict $\mathbb K$-linear monoidal category generated by a single object  \begin{tikzpicture}[baseline = 1.5mm]
	\draw[-,thick,darkblue] (0.18,0) to (0.18,.4);
\end{tikzpicture}  and four elementary morphisms  $U, A, T$ and $T^{-1} $ subject to the relations as follows:
\begin{itemize}\item[(1)]\label{K1}$(1_\ob1\otimes A)\circ (T\otimes 1_\ob1)\circ (1_\ob1\otimes U)\circ(1_\ob1\otimes A)\circ (T^{-1}\otimes 1_\ob1)\circ (1_\ob1\otimes U)= 1_\ob1$,
\item[(2)]\label{K2} $T\circ T^{-1}=T^{-1}\circ T=1_\ob2$,
 \item[(3)] \label{K3}$(T\otimes 1_\ob1)\circ (1_\ob1\otimes T) \circ (1_\ob1\otimes T)=(1_\ob1\otimes T) \circ (1_\ob1\otimes T)\circ (T\otimes 1_\ob1)$,
     \item[(4)]\label{K4} $T-T^{-1}=z(1_\ob2-U\circ A)$,
 \item[(5)]\label{K5} $(1_\ob1\otimes A)\circ (T\otimes 1_\ob1)\circ (1_\ob1\otimes U)=\delta1_\ob1$,
 \item[(6)]\label{K6} $A\circ U=\omega_01_{\ob 0}$,
 \item[(7)]\label{K7}$(1_\ob1\otimes A)\circ (U\otimes 1_\ob1)=1_\ob1=(A\otimes 1_\ob1)\circ (1_\ob1\otimes U)$,
 \item[(8)]\label{K8}$T=(A\otimes 1_\ob2 ) \circ (1_\ob1\otimes T^{-1}\otimes 1_\ob1)\circ (1_\ob2\otimes U)$,
 \item[(9)]\label{K9} $T^{-1} =(A\otimes 1_\ob2 ) \circ (1_\ob1\otimes T\otimes 1_\ob1)\circ (1_\ob2\otimes U)$.
 \end{itemize}

 \end{Theorem}

 Theorem~\ref{present1}(1)-(3) correspond to (RI)-(RIII),   and Theorem~\ref{present1}(4)-(6) correspond to (S)-(L). Further,  Theorem~\ref{present1}(7)-(9)  are depicted as \eqref{relation 1}-\eqref{relation 3} as follows:
\begin{equation}\label{relation 1}
\begin{tikzpicture}[baseline = -0.5mm]
\draw[-,thick,darkblue] (0,0) to (0,0.3);
\draw[-,thick,darkblue] (0,0) to[out=down,in=left] (0.25,-0.25) to[out=right,in=down] (0.5,0);
\draw[-,thick,darkblue] (0.5,0) to[out=up,in=left] (0.75,0.25) to[out=right,in=up] (1,0);
\draw[-,thick,darkblue] (0,0) to (0,0.3);
\draw[-,thick,darkblue] (1,0) to (1,-0.3);
  \end{tikzpicture}
  ~=~
  \begin{tikzpicture}[baseline = -0.5mm]
  \draw[-,thick,darkblue] (0,-0.25) to (0,0.25);
  \end{tikzpicture}
  ~=~\begin{tikzpicture}[baseline = -0.5mm]
\draw[-,thick,darkblue] (0.5,0) to[out=down,in=left] (0.75,-0.25) to[out=right,in=down] (1,0);
\draw[-,thick,darkblue] (0,0) to[out=up,in=left] (0.25,0.25) to[out=right,in=up] (0.5,0);
\draw[-,thick,darkblue] (1,0) to (1,0.3);
\draw[-,thick,darkblue] (0,0) to (0,-0.3);
  \end{tikzpicture},
\end{equation}
\begin{equation}\label{relation 2}
\begin{tikzpicture}[baseline = -0.5mm]
	\draw[-,thick,darkblue] (0.28,-0.3) to[out=90,in=-90] (-0.28,.3);
	\draw[-,line width=4pt,white] (-0.28,-0.3) to[out=90,in=-90] (0.28,.3);
	\draw[-,thick,darkblue] (-0.28,-.3) to[out=90,in=-90] (0.28,.3);
  \end{tikzpicture}
  ~=~\begin{tikzpicture}[baseline = -0.5mm]
\draw[-,thick,darkblue]  (0.75,-0.25) to[out=right,in=down] (1,0);
\draw[-,thick,darkblue] (0,0) to[out=up,in=left] (0.25,0.25) ;
\draw[-,thick,darkblue] (1,0) to (1,0.25);
\draw[-,thick,darkblue] (0,0) to (0,-0.25);
\draw[-,thick,darkblue] (0.25,-0.25) to[out=90,in=-90] (0.75,0.25);
\draw[-,line width=4pt,white] (0.25,0.25) to[out=90,in=-90] (0.75,-0.25);
\draw[-,thick,darkblue] (0.25,0.25) to[out=0,in=180] (0.75,-0.25);
  \end{tikzpicture},
\end{equation}
\begin{equation}\label{relation 3}
\begin{tikzpicture}[baseline = -0.5mm]
\draw[-,thick,darkblue] (-0.28,-.3) to[out=90,in=-90] (0.28,.3);
\draw[-,line width=4pt,white] (0.28,-0.3) to[out=90,in=-90] (-0.28,.3);
\draw[-,thick,darkblue] (0.28,-0.3) to[out=90,in=-90] (-0.28,.3);
  \end{tikzpicture}
  ~=~\begin{tikzpicture}[baseline = -0.5mm]
\draw[-,thick,darkblue]  (0.75,-0.25) to[out=right,in=down] (1,0);
\draw[-,thick,darkblue] (0,0) to[out=up,in=left] (0.25,0.25) ;
\draw[-,thick,darkblue] (1,0) to (1,0.25);
\draw[-,thick,darkblue] (0,0) to (0,-0.25);
\draw[-,thick,darkblue] (0.25,0.25) to[out=0,in=180] (0.75,-0.25);
\draw[-,line width=4pt,white] (0.25,-0.25) to[out=90,in=-90] (0.75,0.25);
\draw[-,thick,darkblue] (0.25,-0.25) to[out=90,in=-90] (0.75,0.25);
  \end{tikzpicture}.
\end{equation}

We are going to state Turaev's result\cite{VT} on bases of morphism spaces in $\B$.
Suppose $ m+s$ is even.
Following \cite{Good},  endpoints at   bottom (resp., top) row  of  an  $(m, s)$-tangle diagram $d$ are labelled by  $1, 2, \ldots, m$
(resp., $\bar 1, \bar 2,   \ldots,  \bar s$) from  left to right.  Then
  $d$ decomposes $\{1, 2, \ldots, m, \bar s, \ldots, \bar 2, \bar 1 \}$ into  $\frac{m+s}{2}$  pairs $\text{conn}(d)=\{(i_k, j_k)\mid 1\le k\le \frac{m+s}{2}\}$, called the $(m, s)$-connector of $d$.
   Let $conn(m, s)$ be the set of all $(m, s)$-connectors. Later on, we always assume that  $i< i+1 < \bar j <\bar {j-1}$ for all admissible $i$ and $j$, and  $i_k<j_k$  and $i_k<i_l$ whenever  $k<l$.
A strand connecting
a pair on different rows (resp., the same row) is called a vertical (resp., horizontal) strand.
Moreover, a horizontal strand connecting a pair on the top (resp., bottom) row is also
called a cup (resp., cap). For example, $U$ is a cup and $A$ is a cap.

 Motivated by \cite[Definition~5.4]{Good}, an $(m, s)$-tangle diagram  is said to be totally descending if it can be traversed successively such that    $(i_k, j_k)$ passes  over  $(i_l, j_l)$ whenever  $k<l$ and $(i_k, j_k)$ crosses $(i_l, j_l)$ and if neither a strand crosses  itself nor there is  a loop.
 It is well-known (e.g. \cite[Proposition~5.7]{Good}) that \begin{equation}\label{bbb} d=e \text{ as morphisms in $\FT$}, \text{if $d, e$ are totally descending and $conn(d)=conn(e)$}.\end{equation} An $(m, s)$-totally descending tangle diagram is said to be reduced, if two strands cross each other at most once. Any $(m, s)$-connector determines many reduced totally descending tangle diagrams. Thanks to \eqref{bbb}, it is reasonable   to  denote one of such reduced totally descending tangle diagrams by $D_c$ for any connector $c$.

  \begin{Theorem}\label{bask} \cite{VT} Suppose  $m, s\in \mathbb N$.
  \begin{itemize}\item[(1)] If $  m+s$ is odd, then $\Hom_{\B}(\ob m, \ob s)=0$.
    \item[(2)]  If $ m+s$ is even, then  $\Hom_{\B}(\ob m, \ob s)$ has $\mathbb K$-basis given by  $\{D_c\mid c\in conn(m, s) \}$. In particular, $\Hom_\B (\ob m, \ob s)$ is of rank  $(m+s-1)!!$.\end{itemize}  \end{Theorem}

Motivated by Theorem~\ref{bask} and our previous work on affine Brauer category in \cite{RS3}, we introduce affine Kauffmann category whose objects are same as those of $\B$. This is one of main objects in this paper.

\begin{Defn}\label{AK defn}
    The affine  Kauffmann category $\AB$ is the $\mathbb K$-linear strict monoidal category generated by a single   object $\ob1$ and six elementary  morphisms $U, A, T, T^{-1}$, and $X^{\pm 1}:\ob1\rightarrow\ob1$ subject to Theorem~\ref{present1}(1)-(9) together with
     the following relations:
 \begin{itemize} \item [(1)]  $X\circ X^{-1}=1_{\ob 1}=X^{-1}\circ X$,
 \item [(2)] $T\circ (X\otimes \text{1}_{\ob 1}) \circ T=\text{1}_{\ob 1}\otimes X$,
        \item [(3)] $A\circ (X\otimes\text{1}_{\ob 1})=A\circ (\text{1}_{\ob 1}\otimes X^{-1})$ and
        $(X\otimes\text{1}_{\ob 1})\circ U=(\text{1}_{\ob 1}\otimes X^{-1})\circ U$.
   \end{itemize}
   \end{Defn}

In this paper,   $X$ and  $X^{-1}$  are drawn  as
\begin{tikzpicture}[baseline=-.5mm]
\draw[-,thick,darkblue] (0,-.3) to (0,.3);
      \node at (0,0) {$\color{darkblue}\scriptstyle\bullet$};
      \end{tikzpicture} and \begin{tikzpicture}[baseline=-.5mm]
\draw[-,thick,darkblue] (0,-.3) to (0,.3);
      \node at (0,0) {$\color{darkblue}\scriptstyle\circ$};
      \end{tikzpicture}, respectively. So,   Definition~\ref{AK defn}(1)-(3) can be depicted as \eqref{relation 4}-\eqref{relation 6}
as follows:
\begin{equation}\label{relation 4}
\begin{tikzpicture}[baseline = -0.5mm]
\draw[-,thick,darkblue] (0,0) to (0,.6);
      \node at (0,0.3) {$\color{darkblue}\scriptstyle\bullet$};
      \draw[-,thick,darkblue] (0,0.2) to (0,-.4);
      \node at (0,-.1) {$\color{darkblue}\scriptstyle\circ$};
  \end{tikzpicture}
  ~=~
  \begin{tikzpicture}[baseline = -0.5mm]
  \draw[-,thick,darkblue] (0,-0.4) to (0,0.6);
  \end{tikzpicture}
  ~=~\begin{tikzpicture}[baseline = -0.5mm]
  \draw[-,thick,darkblue] (0,0) to (0,.6);
      \node at (0,0.3) {$\color{darkblue}\scriptstyle\circ$};
      \draw[-,thick,darkblue] (0,0.2) to (0,-.4);
      \node at (0,-.1) {$\color{darkblue}\scriptstyle\bullet$};
  \end{tikzpicture},
\end{equation}
\begin{equation}\label{relation 5}
\begin{tikzpicture}[baseline = -4mm]
	\draw[-,thick,darkblue] (0.28,-0.3) to[out=90,in=-90] (-0.28,.3);
	\draw[-,line width=4pt,white] (-0.28,-0.3) to[out=90,in=-90] (0.28,.3);
	\draw[-,thick,darkblue] (-0.28,-.3) to[out=90,in=-90] (0.28,.3);
\node at (-0.28,-.3) {$\color{darkblue}\scriptstyle\bullet$};
\draw[-,thick,darkblue] (0.28,-0.9) to[out=90,in=-90] (-0.28,-.3);
	\draw[-,line width=4pt,white] (-0.28,-0.9) to[out=90,in=-90] (0.28,-.3);
	\draw[-,thick,darkblue] (-0.28,-.9) to[out=90,in=-90] (0.28,-.3);
  \end{tikzpicture}
  ~=~\begin{tikzpicture}[baseline = -4mm]
 \draw[-,thick,darkblue] (0,0.3) to (0,-.9);
 \draw[-,thick,darkblue] (0.56,0.3) to (0.56,-.9);
 \node at (0.56,-.3) {$\color{darkblue}\scriptstyle\bullet$};
  \end{tikzpicture},
\end{equation}
\begin{equation}\label{relation 6}
\begin{tikzpicture}[baseline = -0.5mm]
  \draw[-,thick,darkblue] (0,0) to[out=up,in=left] (0.28,0.28) to[out=right,in=up] (0.56,0);
  \draw[-,thick,darkblue] (0,0) to (0,-.2);
 \node at (0,0) {$\color{darkblue}\scriptstyle\bullet$};
 \draw[-,thick,darkblue] (0.56,0) to (0.56,-.2);
  \end{tikzpicture}
  ~=~\:\begin{tikzpicture}[baseline = -0.5mm]
\draw[-,thick,darkblue] (0,0) to[out=up,in=left] (0.28,0.28) to[out=right,in=up] (0.56,0);
 \draw[-,thick,darkblue] (0.56,0) to (0.56,-.2);
 \draw[-,thick,darkblue] (0,0) to (0,-.2);
 \node at (0.56,0) {$\color{darkblue}\scriptstyle\circ$};
  \end{tikzpicture},\ \
 \begin{tikzpicture}[baseline = -0.5mm]
 \draw[-,thick,darkblue] (0,0) to[out=down,in=left] (0.28,-0.28) to[out=right,in=down] (0.56,0);
 \node at (0,0) {$\color{darkblue}\scriptstyle\bullet$};
  \draw[-,thick,darkblue] (0,0) to (0,.2);
   \draw[-,thick,darkblue] (0.56,0) to (0.56,.2);
 \end{tikzpicture}
 ~=~\:
 \begin{tikzpicture}[baseline = -0.5mm]
 \draw[-,thick,darkblue] (0,0) to[out=down,in=left] (0.28,-0.28) to[out=right,in=down] (0.56,0);
 \node at (0.56,0) {$\color{darkblue}\scriptstyle\circ$};
  \draw[-,thick,darkblue] (0,0) to (0,.2);
   \draw[-,thick,darkblue] (0.56,0) to (0.56,.2);
 \end{tikzpicture}.
\end{equation}

\begin{Lemma}\label{ktau}Suppose $\AB$ is the affine Kauffmann category. \begin{itemize}\item [(1)] There is a $\mathbb K$-linear  monoidal contravariant functor $\sigma: \AB\rightarrow \AB$
switching  A and U and fixing $T, T^{-1}, X, X^{-1}$.
\item [(2)] There is a monoidal functor from $\B$ to $ \AB$
sending the generators of $\B$ to the generators of $\AB$ with the same names.
\end{itemize}
\end{Lemma}
\begin{proof}Easy exercise. \end{proof}
Thanks to Lemma~\ref{ktau}(2), any
tangle diagram can be interpreted as a morphism in $\AB$.
\begin{Lemma}\label{selfcrossing} As morphisms in $\AB$, we have:
\item[(1)]
~\begin{tikzpicture}[baseline = -3mm]
	\draw[-,thick,darkblue] (0.28,0) to[out=90, in=0] (0,0.2);
	\draw[-,thick,darkblue] (0,0.2) to[out = 180, in = 90] (-0.28,0);
	\draw[-,thick,darkblue] (-0.28,-.6) to[out=60,in=-90] (0.28,0);
	\draw[-,line width=4pt,white] (0.28,-.6) to[out=120,in=-90] (-0.28,0);
	\draw[-,thick,darkblue] (0.28,-.6) to[out=120,in=-90] (-0.28,0);
   \end{tikzpicture}~=~$\delta$
\begin{tikzpicture}[baseline = 5pt, scale=0.5, color=\clr]
    	\draw[-,thick,darkblue] (1,-0.1) to[out=90, in=0] (0.5,0.55);
	\draw[-,thick,darkblue] (0.5,0.55) to[out = 180, in = 90] (0,-0.1);
\end{tikzpicture},
\begin{tikzpicture}[baseline = -3mm]
	\draw[-,thick,darkblue] (0.28,-.6) to[out=120,in=-90] (-0.28,0);
	\draw[-,thick,darkblue] (0.28,0) to[out=90, in=0] (0,0.2);
	\draw[-,thick,darkblue] (0,0.2) to[out = 180, in = 90] (-0.28,0);
	\draw[-,line width=4pt,white] (-0.28,-.6) to[out=60,in=-90] (0.28,0);
	\draw[-,thick,darkblue] (-0.28,-.6) to[out=60,in=-90] (0.28,0);
     \end{tikzpicture}~=~$\delta^{-1}$\begin{tikzpicture}[baseline = 5pt, scale=0.5, color=\clr]
    	\draw[-,thick,darkblue] (1,-0.1) to[out=90, in=0] (0.5,0.55);
	\draw[-,thick,darkblue] (0.5,0.55) to[out = 180, in = 90] (0,-0.1);
\end{tikzpicture},
\item[(2)]
~\begin{tikzpicture}[baseline = 5pt, scale=0.5, color=\clr]
    \draw[-,thick] (0,0) to[out=up,in=down] (1,1);
        \draw[-,thick]  (1,0) to[out=down,in=right] (0.5,-0.5)to[out=left,in=down] (0,0);

        \draw[-,thick] (0,1) to[out=down,in=up] (1,0);
        \draw[-,line width=4pt,white] (0,1) to[out=down,in=up] (1,0);
        \draw[-,thick] (0,1) to[out=down,in=up] (1,0);
            \end{tikzpicture}~=~$\delta$
\begin{tikzpicture}[baseline = 5pt, scale=0.5, color=\clr]
        \draw[-,thick] (0,1) to[out=down,in=left] (0.5,0.35) to[out=right,in=down] (1,1);
    \end{tikzpicture},
\begin{tikzpicture}[baseline = 5pt, scale=0.5, color=\clr]
    \draw[-,thick] (0,1) to[out=down,in=up] (1,0);
        \draw[-,thick]  (1,0) to[out=down,in=right] (0.5,-0.5)to[out=left,in=down] (0,0);

        \draw[-,thick] (0,0) to[out=up,in=down] (1,1);
        \draw[-,line width=4pt,white] (0,0) to[out=up,in=down] (1,1);
        \draw[-,thick] (0,0) to[out=up,in=down] (1,1);
           \end{tikzpicture}~=~$\delta^{-1}$
\begin{tikzpicture}[baseline = 5pt, scale=0.5, color=\clr]
        \draw[-,thick] (0,1) to[out=down,in=left] (0.5,0.35) to[out=right,in=down] (1,1);
    \end{tikzpicture},

\item[(3)]
\begin{tikzpicture}[baseline = -0.5mm]
 \draw[-,thick,darkblue] (0,0) to[out=down,in=left] (0.28,-0.28) to[out=right,in=down] (0.56,0);
 \node at (0,0) {$\color{darkblue}\scriptstyle\circ$};
  \draw[-,thick,darkblue] (0,0) to (0,.2);
   \draw[-,thick,darkblue] (0.56,0) to (0.56,.2);
 \end{tikzpicture}
 ~=~
 \begin{tikzpicture}[baseline = -0.5mm]
 \draw[-,thick,darkblue] (0,0) to[out=down,in=left] (0.28,-0.28) to[out=right,in=down] (0.56,0);
 \node at (0.56,0) {$\color{darkblue}\scriptstyle\bullet$};
  \draw[-,thick,darkblue] (0,0) to (0,.2);
   \draw[-,thick,darkblue] (0.56,0) to (0.56,.2);
 \end{tikzpicture},
\begin{tikzpicture}[baseline = -0.5mm]
  \draw[-,thick,darkblue] (0,0) to[out=up,in=left] (0.28,0.28) to[out=right,in=up] (0.56,0);
  \draw[-,thick,darkblue] (0,0) to (0,-.2);
 \node at (0,0) {$\color{darkblue}\scriptstyle\circ$};
 \draw[-,thick,darkblue] (0.56,0) to (0.56,-.2);
  \end{tikzpicture}
  ~=~\begin{tikzpicture}[baseline = -0.5mm]
\draw[-,thick,darkblue] (0,0) to[out=up,in=left] (0.28,0.28) to[out=right,in=up] (0.56,0);
 \draw[-,thick,darkblue] (0.56,0) to (0.56,-.2);
 \draw[-,thick,darkblue] (0,0) to (0,-.2);
 \node at (0.56,0) {$\color{darkblue}\scriptstyle\bullet$};
  \end{tikzpicture}.
\end{Lemma}
\begin{proof} We have
 $$\begin{aligned} & \text{ \begin{tikzpicture}[baseline = -3mm]
	\draw[-,thick,darkblue] (0.28,-.6) to[out=120,in=-90] (-0.28,0);
	\draw[-,thick,darkblue] (0.28,0) to[out=90, in=0] (0,0.2);
	\draw[-,thick,darkblue] (0,0.2) to[out = 180, in = 90] (-0.28,0);
	\draw[-,line width=4pt,white] (-0.28,-.6) to[out=60,in=-90] (0.28,0);
	\draw[-,thick,darkblue] (-0.28,-.6) to[out=60,in=-90] (0.28,0);
\end{tikzpicture}~=~\begin{tikzpicture}[baseline = 1mm]
	\draw[-,thick,darkblue] (0.4,-0.1) to[out=90, in=0] (0.1,0.3);
	\draw[-,thick,darkblue] (0.1,0.3) to[out = 180, in = 90] (-0.2,-0.1);
\end{tikzpicture}$\circ$
 \begin{tikzpicture}[baseline = -3mm]
	\draw[-,thick,darkblue] (0.28,-.6) to[out=90,in=-90] (-0.28,0);
\draw[-,thick,darkblue] (-0.28,-.6) to[out=90,in=-90] (0.28,0);
\draw[-,line width=4pt,white] (-0.28,-.65) to[out=60,in=-90] (0.28,0.3);
	\draw[-,thick,darkblue] (-0.28,-.6) to[out=90,in=-90] (0.28,0);
\end{tikzpicture}~=~\begin{tikzpicture}[baseline = 1mm]
	\draw[-,thick,darkblue] (0.4,-0.1) to[out=90, in=0] (0.1,0.3);
	\draw[-,thick,darkblue] (0.1,0.3) to[out = 180, in = 90] (-0.2,-0.1);
\end{tikzpicture}$\circ$
\begin{tikzpicture}[baseline = -0.5mm]
\draw[-,thick,darkblue]  (0.75,-0.25) to[out=right,in=down] (1,0);
\draw[-,thick,darkblue] (0,0) to[out=up,in=left] (0.25,0.25) ;
\draw[-,thick,darkblue] (1,0) to (1,0.25);
\draw[-,thick,darkblue] (0,0) to (0,-0.25);
\draw[-,thick,darkblue] (0.25,-0.25) to[out=90,in=-90] (0.75,0.25);
\draw[-,line width=4pt,white] (0.25,0.25) to[out=90,in=-90] (0.75,-0.25);
\draw[-,thick,darkblue] (0.25,0.25) to[out=0,in=180] (0.75,-0.25);
  \end{tikzpicture}
~=~\begin{tikzpicture}[baseline = -2.5mm]
  \draw[-,thick,darkblue] (0,0) to[out=up,in=left] (0.28,0.28) to[out=right,in=up] (0.56,0);
  \draw[-,thick,darkblue] (0,0) to (0,-.6);
	\draw[-,thick,darkblue] (0.56,0.07)   to (0.56,-0.);
	\draw[-,thick,darkblue] (1.06,-0.3) to [out=90,in=0](0.86,-0.1);
	\draw[-,thick,darkblue] (0.86,-0.5) to [out=0,in=-90](1.06,-0.3);
    \draw[-,thick,darkblue] (0.86,-.1) to [out=180,in=90](0.56,-0.66);
    \draw[-,line width=4pt,white] (0.56,-0.) to [out=-90,in=180](0.86,-0.5);
	\draw[-,thick,darkblue] (0.56,-0.) to [out=-90,in=180] (.86,-0.5);
\end{tikzpicture}
~=~$\delta^{-1}$\begin{tikzpicture}[baseline = 1mm]
	\draw[-,thick,darkblue] (0.4,-0.1) to[out=90, in=0] (0.1,0.3);
	\draw[-,thick,darkblue] (0.1,0.3) to[out = 180, in = 90] (-0.2,-0.1);
\end{tikzpicture}}. \cr
  \end{aligned}$$
  One can prove the first equation in (1) similarly.
 Thanks to \eqref{relation 1} and  \eqref{relation 6}, we have
 $$        \begin{tikzpicture}[baseline = 7.5pt, scale=0.5, color=\clr]
            \draw[-,thick] (0,2) to (0,0.5)
                            to[out=down, in=left] (0.5,0)
                            to[out=right, in=down] (1,0.5) to (1,2);
                            \node at (1,0.75) {$\color{darkblue}\scriptstyle\bullet$};

        \end{tikzpicture}
        ~{=}~
        \begin{tikzpicture}[baseline = 7.5pt, scale=0.5, color=\clr]
            \draw[-,thick] (0,2) to (0,0.5)
                            to[out=down, in=left] (0.25,0)
                            to[out=right, in=down] (0.5,0.5) to (0.5,1)
                            to[out=up, in=left] (0.75,1.5)
                            to[out=right, in=up] (1,1)
                            to[out=down, in=left] (1.25,0.5)
                            to[out=right, in=down] (1.5,1) to (1.5,2);
            \node at (1,1) {$\color{darkblue}\scriptstyle\circ$};
        \end{tikzpicture}
        ~{=}~
        \begin{tikzpicture}[baseline = 7.5pt, scale=0.5, color=\clr]
            \draw[-,thick] (0,2) to (0,0.5)
                            to[out=down, in=left] (0.5,0)
                            to[out=right, in=down] (1,0.5) to (1,2);
           \node at (0,0.75) {$\color{darkblue}\scriptstyle\circ$};
        \end{tikzpicture}.$$
        So far, we have  proved (1) and the first equation in (3).
    Applying   the contravariant functor $\sigma$ in Lemma~\ref{ktau}
        on
these equations yields (2) and  the second equation in (3).
 \end{proof}

A point on a strand of a tangle diagram is called a critical point if it is either an endpoint
or a point such that the tangent line at it is  horizontal.
A segment of a tangle diagram is defined to be a connected
component of the diagram obtained when all crossings and critical points are deleted.
A \textsf{dotted (m,s)-tangle diagram} $d$ is an $(m,s)$-tangle diagram  such that
there are finitely many $\bullet$'s or $\circ$'s  on each segment of $d$. Both  $\bullet$ and $\circ$ are called dots later on.
If there are $ h $ $\bullet$'s (resp., $\circ$'s) on a segment, then such $\bullet$'s (resp., $\circ$'s) can be viewed as
\begin{tikzpicture}[baseline=-.5mm]
 \draw[-,thick,darkblue] (0,-.3) to (0,.3);
      \node at (0,0) {$\color{darkblue}\scriptstyle\bullet$};
\end{tikzpicture} $\circ\ldots\circ$
\begin{tikzpicture}[baseline=-.5mm]
 \draw[-,thick,darkblue] (0,-.3) to (0,.3);
      \node at (0,0) {$\color{darkblue}\scriptstyle\bullet$};
\end{tikzpicture}(resp,.\begin{tikzpicture}[baseline=-.5mm]
 \draw[-,thick,darkblue] (0,-.3) to (0,.3);
      \node at (0,0) {$\color{darkblue}\scriptstyle\circ$};
\end{tikzpicture} $\circ\ldots\circ$
\begin{tikzpicture}[baseline=-.5mm]
 \draw[-,thick,darkblue] (0,-.3) to (0,.3);
      \node at (0,0) {$\color{darkblue}\scriptstyle\circ$};
\end{tikzpicture}), and will be denoted by $\bullet h$ (resp., $\circ h$). In order to simplify notation, we say that $h$ ``$\bullet$" is the same as $-h$ ``$\circ$" if $h<0$. In other words,
$$  \begin{tikzpicture}[baseline = 23pt, scale=0.35, color=\clr]
\draw[-,thick] (1,3) to[out=up,in=down] (1,2);
 \node at (1,2.5) {$\color{darkblue}\scriptstyle\bullet$};
\node at (1.4,2.7) {$\color{darkblue}\scriptstyle h$};
   \end{tikzpicture}=\begin{tikzpicture}[baseline = 23pt, scale=0.35, color=\clr]
\draw[-,thick] (1,3) to[out=up,in=down] (1,2);
 \node at (1,2.5) {$\color{darkblue}\scriptstyle\circ$};
\node at (1.7,2.7) {$\color{darkblue}\scriptstyle -h$};
   \end{tikzpicture} \text{ and }  \begin{tikzpicture}[baseline = 23pt, scale=0.35, color=\clr]
\draw[-,thick] (1,3) to[out=up,in=down] (1,2);
 \node at (1,2.5) {$\color{darkblue}\scriptstyle\circ$};
\node at (1.4,2.7) {$\color{darkblue}\scriptstyle h$};
   \end{tikzpicture}=\begin{tikzpicture}[baseline = 23pt, scale=0.35, color=\clr]
\draw[-,thick] (1,3) to[out=up,in=down] (1,2);
 \node at (1,2.5) {$\color{darkblue}\scriptstyle\bullet$};
\node at (1.7,2.7) {$\color{darkblue}\scriptstyle -h$};
   \end{tikzpicture}
 \text{ if $h<0$.}$$
For example,
$$
\begin{tikzpicture}[baseline = 1mm]
	\draw[-,thick,darkblue] (0.4,-0.1) to[out=90, in=0] (0.1,0.3);
	\draw[-,thick,darkblue] (0.1,0.3) to[out = 180, in = 90] (-0.2,-0.1);
      \node at (-0.14,0.15) {$\color{darkblue}\scriptstyle\bullet$};
      \node at (-0.2,0.) {$\color{darkblue}\scriptstyle\bullet$};
\end{tikzpicture}
~=~\begin{tikzpicture}[baseline = 1mm]
	\draw[-,thick,darkblue] (0.4,-0.1) to[out=90, in=0] (0.1,0.3);
	\draw[-,thick,darkblue] (0.1,0.3) to[out = 180, in = 90] (-0.2,-0.1);
\end{tikzpicture}\circ
\begin{tikzpicture}[baseline=-.5mm]
 \draw[-,thick,darkblue] (0.18,-0.3) to (0.18,.3);
	\draw[-,thick,darkblue] (-0.18,-.3) to (-0.18,.3);
     \node at (-0.18,0) {$\color{darkblue}\scriptstyle\bullet$};
\end{tikzpicture}
\circ
\begin{tikzpicture}[baseline=-.5mm]
 \draw[-,thick,darkblue] (0.18,-0.3) to (0.18,.3);
	\draw[-,thick,darkblue] (-0.18,-.3) to (-0.18,.3);
     \node at (-0.18,0) {$\color{darkblue}\scriptstyle\bullet$};
\end{tikzpicture}
\begin{tikzpicture}[baseline = 1mm]
\end{tikzpicture}
~=~
 \begin{tikzpicture}[baseline = 1mm]
	\draw[-,thick,darkblue] (0.4,-0.1) to[out=90, in=0] (0.1,0.3);
	\draw[-,thick,darkblue] (0.1,0.3) to[out = 180, in = 90] (-0.2,-0.1);
      \node at (-0.2,0.) {$\color{darkblue}\scriptstyle\bullet$};
      \node at (-0.4,0.) {$\color{darkblue}\scriptstyle 2$};
\end{tikzpicture}~=~
\begin{tikzpicture}[baseline = 1mm]
	\draw[-,thick,darkblue] (0.4,-0.1) to[out=90, in=0] (0.1,0.3);
	\draw[-,thick,darkblue] (0.1,0.3) to[out = 180, in = 90] (-0.2,-0.1);
      \node at (-0.2,0.) {$\color{darkblue}\scriptstyle\circ$};
      \node at (-0.45,0.) {$\color{darkblue}\scriptstyle -2$};
\end{tikzpicture}
.$$

Let $\mathbb{T}_{m,s}$ be the set of all dotted $(m,s)$-tangle diagram. By arguments similar to those in \cite{RS3}, any $d$ in $\mathbb{T}_{m,s}$  can be interpreted as a morphism in  $\Hom_{\AB}(\ob m,\ob s)$. Conversely, if a diagram $d$ is obtained by tensor product and composition of $U, A, T, T^{-1}, X, X^{-1}$ together with \begin{tikzpicture}[baseline=-.5mm]
 \draw[-,thick,darkblue] (0,-0.3) to (0,.3);
\end{tikzpicture}  such that there are $m$ (resp., $s$) endpoints on the bottom (resp., top) row, then $d\in \mathbb{T}_{m,s}$. So,  $\Hom_{\AB}(\ob m,\ob s)$ is spanned by $\mathbb{T}_{m,s}$.
For any $i\in \mathbb Z$, let
\begin{equation} \label{bubbs} \Delta_i=  \begin{tikzpicture}[baseline = 5pt, scale=0.5, color=\clr]
        \draw[-,thick] (0.6,1) to (0.5,1) to[out=left,in=up] (0,0.5)
                        to[out=down,in=left] (0.5,0)
                        to[out=right,in=down] (1,0.5)
                        to[out=up,in=right] (0.5,1);
        \node at (0,0.5) {$\color{darkblue}\scriptstyle\bullet$};
        \draw (-0.4,0.5) node{\footnotesize{$i$}};
    \end{tikzpicture}.\end{equation}
Then   $\Delta_i\in  \End_{\AB}(\ob 0)$ and $\Delta_0=\omega_01_{\ob 0}$.
By \cite[Proposition 2.2.10]{EGNO}, $\End_{\AB}(\ob 0)$ is commutative and $\Delta_i\circ\Delta_j=\Delta_i\otimes\Delta_j$ for all admissible $i,j$.

\begin{Lemma} \label{addmi} For any positive integer  $j$,  $\Delta_{-j}=\delta^{-2}\Delta_j+\delta^{-1} z \sum_{i=1}^{j-1} ( \Delta_{2i-j}-\Delta_i\Delta_{i-j})$.
\end{Lemma}

 \begin{proof} Thanks to  Lemma~\ref{selfcrossing}(1), we have
$$\begin{aligned}\begin{tikzpicture}[baseline = 5pt, scale=0.5, color=\clr]
        \draw[-,thick] (0.6,1) to (0.5,1) to[out=left,in=up] (0,0.5)
                        to[out=down,in=left] (0.5,0)
                        to[out=right,in=down] (1,0.5)
                        to[out=up,in=right] (0.5,1);
        \node at (0,0.5) {$\color{darkblue}\scriptstyle\circ$};
        \draw (-0.4,0.5) node{\footnotesize{$j$}};
    \end{tikzpicture}&=
\delta^{-1}\begin{tikzpicture}[baseline = -3mm, color=\clr]
	\draw[-,thick,darkblue] (0.28,0) to[out=90, in=0] (0,0.2);
	\draw[-,thick,darkblue] (0,0.2) to[out = 180, in = 90] (-0.28,0);
	\draw[-,thick,darkblue] (-0.28,-.6) to[out=60,in=-90] (0.28,0);
	\draw[-,line width=4pt,white] (0.28,-.6) to[out=120,in=-90] (-0.28,0);
	\draw[-,thick,darkblue] (0.28,-.6) to[out=120,in=-90] (-0.28,0);
    \node at (-0.2,-0.5) {$\color{darkblue}\scriptstyle\circ$};
    \draw (-0.3,-0.3) node{\footnotesize{$j$}};

    \draw[-,thick,darkblue]  (0.28,-0.6) to[out=down,in=right] (-0.,-0.8)to[out=left,in=down] (-0.28,-.6);
\end{tikzpicture}\overset{ (\ref{relation 6}) }=
\delta^{-1}\begin{tikzpicture}[baseline = -3mm, color=\clr]
	\draw[-,thick,darkblue] (0.28,0) to[out=90, in=0] (0,0.2);
	\draw[-,thick,darkblue] (0,0.2) to[out = 180, in = 90] (-0.28,0);
	\draw[-,thick,darkblue] (-0.28,-.6) to[out=60,in=-90] (0.28,0);
	\draw[-,line width=4pt,white] (0.28,-.6) to[out=120,in=-90] (-0.28,0);
	\draw[-,thick,darkblue] (0.28,-.6) to[out=120,in=-90] (-0.28,0);
    \node at (0.2,-0.5) {$\color{darkblue}\scriptstyle\bullet$};
    \draw (0.3,-0.3) node{\footnotesize{$j$}};

    \draw[-,thick,darkblue]  (0.28,-0.6) to[out=down,in=right] (-0.,-0.8)to[out=left,in=down] (-0.28,-.6);

\end{tikzpicture}
=
\delta^{-1}\begin{tikzpicture}[baseline = -3mm, color=\clr]
	\draw[-,thick,darkblue] (0.28,0) to[out=90, in=0] (0,0.2);
	\draw[-,thick,darkblue] (0,0.2) to[out = 180, in = 90] (-0.28,0);
	\draw[-,thick,darkblue] (-0.28,-.6) to[out=60,in=-90] (0.28,0);
	\draw[-,line width=4pt,white] (0.28,-.6) to[out=120,in=-90] (-0.28,0);
	\draw[-,thick,darkblue] (0.28,-.6) to[out=120,in=-90] (-0.28,0);
    \node at (0.2,-0.5) {$\color{darkblue}\scriptstyle\bullet$};
    \node at (0.28,-0.6) {$\color{darkblue}\scriptstyle\bullet$};
    \draw (0.7,-0.7) node{\footnotesize{$j-1$}};

    \draw[-,thick,darkblue]  (0.28,-0.6) to[out=down,in=right] (-0.,-0.8)to[out=left,in=down] (-0.28,-.6);

\end{tikzpicture}\overset{ (\ref{relation 5}) }=
\delta^{-1}\begin{tikzpicture}[baseline = -3mm, color=\clr]
	\draw[-,thick,darkblue] (0.28,0) to[out=90, in=0] (0,0.2);
	\draw[-,thick,darkblue] (0,0.2) to[out = 180, in = 90] (-0.28,0);
\draw[-,thick,darkblue] (0.28,-.6) to[out=120,in=-90] (-0.28,0);
	
	\draw[-,line width=4pt,white] (-0.28,-.6) to[out=60,in=-90] (0.28,0);
	\draw[-,thick,darkblue] (-0.28,-.6) to[out=60,in=-90] (0.28,0);
    \node at (-0.25,-0.1) {$\color{darkblue}\scriptstyle\bullet$};
    \node at (0.28,-0.6) {$\color{darkblue}\scriptstyle\bullet$};
    \draw (0.7,-0.7) node{\footnotesize{$j-1$}};

    \draw[-,thick,darkblue]  (0.28,-0.6) to[out=down,in=right] (-0.,-0.8)to[out=left,in=down] (-0.28,-.6);

\end{tikzpicture}\\
&\overset{ (S) }=
\delta^{-1}\begin{tikzpicture}[baseline = -3mm, color=\clr]
	\draw[-,thick,darkblue] (0.28,0) to[out=90, in=0] (0,0.2);
	\draw[-,thick,darkblue] (0,0.2) to[out = 180, in = 90] (-0.28,0);
	\draw[-,thick,darkblue] (-0.28,-.6) to[out=60,in=-90] (0.28,0);
	\draw[-,line width=4pt,white] (0.28,-.6) to[out=120,in=-90] (-0.28,0);
	\draw[-,thick,darkblue] (0.28,-.6) to[out=120,in=-90] (-0.28,0);
    \node at (-0.25,-0.1) {$\color{darkblue}\scriptstyle\bullet$};
    \node at (0.28,-0.6) {$\color{darkblue}\scriptstyle\bullet$};
    \draw (0.7,-0.7) node{\footnotesize{$j-1$}};

    \draw[-,thick,darkblue]  (0.28,-0.6) to[out=down,in=right] (-0.,-0.8)to[out=left,in=down] (-0.28,-.6);

\end{tikzpicture}+\delta^{-1}z\begin{tikzpicture}[baseline = 5pt, scale=0.5, color=\clr]
        \draw[-,thick] (0.6,1) to (0.5,1) to[out=left,in=up] (0,0.5)
                        to[out=down,in=left] (0.5,0)
                        to[out=right,in=down] (1,0.5)
                        to[out=up,in=right] (0.5,1);
                         \node at (1,0.5) {$\color{darkblue}\scriptstyle\bullet$};
        \node at (0,0.5) {$\color{darkblue}\scriptstyle\bullet$};
        \draw (1.8,0.6) node{\footnotesize{$j-1$}};
    \end{tikzpicture}-\delta^{-1}z\begin{tikzpicture}[baseline = 5pt, scale=0.5, color=\clr]
     \draw[-,thick] (0.6,1) to (0.5,1) to[out=left,in=up] (0,0.5)
                        to[out=down,in=left] (0.5,0)
                        to[out=right,in=down] (1,0.5)
                        to[out=up,in=right] (0.5,1);
                        \node at (0,0.5) {$\color{darkblue}\scriptstyle\bullet$};
        \draw[-,thick] (1.8,1) to (1.7,1) to[out=left,in=up] (1.2,0.5)
                        to[out=down,in=left] (1.7,0)
                        to[out=right,in=down] (2.2,0.5)
                        to[out=up,in=right] (1.7,1);
                         \node at (2.2,0.5) {$\color{darkblue}\scriptstyle\bullet$};
        \draw (3,0.6) node{\footnotesize{$j-1$}};

    \end{tikzpicture}\\
    &\overset{(\ref{relation 6})}=
\delta^{-1}\begin{tikzpicture}[baseline = -3mm, color=\clr]
	\draw[-,thick,darkblue] (0.28,0) to[out=90, in=0] (0,0.2);
	\draw[-,thick,darkblue] (0,0.2) to[out = 180, in = 90] (-0.28,0);
	\draw[-,thick,darkblue] (-0.28,-.6) to[out=60,in=-90] (0.28,0);
	\draw[-,line width=4pt,white] (0.28,-.6) to[out=120,in=-90] (-0.28,0);
	\draw[-,thick,darkblue] (0.28,-.6) to[out=120,in=-90] (-0.28,0);
    \node at (-0.25,-0.1) {$\color{darkblue}\scriptstyle\bullet$};
    \node at (0.28,-0.6) {$\color{darkblue}\scriptstyle\bullet$};
    \draw (0.7,-0.7) node{\footnotesize{$j-1$}};

    \draw[-,thick,darkblue]  (0.28,-0.6) to[out=down,in=right] (-0.,-0.8)to[out=left,in=down] (-0.28,-.6);

\end{tikzpicture}+\delta^{-1}z\begin{tikzpicture}[baseline = 5pt, scale=0.5, color=\clr]
        \draw[-,thick] (0.6,1) to (0.5,1) to[out=left,in=up] (0,0.5)
                        to[out=down,in=left] (0.5,0)
                        to[out=right,in=down] (1,0.5)
                        to[out=up,in=right] (0.5,1);
        \node at (0,0.5) {$\color{darkblue}\scriptstyle\circ$};
        \draw (-0.6,1) node{\footnotesize{$j-2$}};
    \end{tikzpicture}-\delta^{-1}z\begin{tikzpicture}[baseline = 5pt, scale=0.5, color=\clr]
     \draw[-,thick] (0.6,1) to (0.5,1) to[out=left,in=up] (0,0.5)
                        to[out=down,in=left] (0.5,0)
                        to[out=right,in=down] (1,0.5)
                        to[out=up,in=right] (0.5,1);
                        \node at (0,0.5) {$\color{darkblue}\scriptstyle\bullet$};
        \draw[-,thick] (1.8,1) to (1.7,1) to[out=left,in=up] (1.2,0.5)
                        to[out=down,in=left] (1.7,0)
                        to[out=right,in=down] (2.2,0.5)
                        to[out=up,in=right] (1.7,1);
                         \node at (1.2,0.5) {$\color{darkblue}\scriptstyle\circ$};
        \draw (1.5,1.2) node{\footnotesize{$j-1$}};

    \end{tikzpicture}\\
    &=
\delta^{-1}\begin{tikzpicture}[baseline = -3mm, color=\clr]
	\draw[-,thick,darkblue] (0.28,0) to[out=90, in=0] (0,0.2);
	\draw[-,thick,darkblue] (0,0.2) to[out = 180, in = 90] (-0.28,0);
	\draw[-,thick,darkblue] (-0.28,-.6) to[out=60,in=-90] (0.28,0);
	\draw[-,line width=4pt,white] (0.28,-.6) to[out=120,in=-90] (-0.28,0);
	\draw[-,thick,darkblue] (0.28,-.6) to[out=120,in=-90] (-0.28,0);
    \node at (0.2,-0.5) {$\color{darkblue}\scriptstyle\bullet$};
    \node at (0.28,-0.6) {$\color{darkblue}\scriptstyle\bullet$};
    \node at (-0.25,-0.1) {$\color{darkblue}\scriptstyle\bullet$};
    \draw (0.7,-0.7) node{\footnotesize{$j-2$}};

    \draw[-,thick,darkblue]  (0.28,-0.6) to[out=down,in=right] (-0.,-0.8)to[out=left,in=down] (-0.28,-.6);

\end{tikzpicture}
+\delta^{-1}z\begin{tikzpicture}[baseline = 5pt, scale=0.5, color=\clr]
        \draw[-,thick] (0.6,1) to (0.5,1) to[out=left,in=up] (0,0.5)
                        to[out=down,in=left] (0.5,0)
                        to[out=right,in=down] (1,0.5)
                        to[out=up,in=right] (0.5,1);
        \node at (0,0.5) {$\color{darkblue}\scriptstyle\circ$};
        \draw (-0.6,1) node{\footnotesize{$j-2$}};
    \end{tikzpicture}-\delta^{-1}z\begin{tikzpicture}[baseline = 5pt, scale=0.5, color=\clr]
     \draw[-,thick] (0.6,1) to (0.5,1) to[out=left,in=up] (0,0.5)
                        to[out=down,in=left] (0.5,0)
                        to[out=right,in=down] (1,0.5)
                        to[out=up,in=right] (0.5,1);
                        \node at (0,0.5) {$\color{darkblue}\scriptstyle\bullet$};
        \draw[-,thick] (1.8,1) to (1.7,1) to[out=left,in=up] (1.2,0.5)
                        to[out=down,in=left] (1.7,0)
                        to[out=right,in=down] (2.2,0.5)
                        to[out=up,in=right] (1.7,1);
                         \node at (1.2,0.5) {$\color{darkblue}\scriptstyle\circ$};
        \draw (1.5,1.2) node{\footnotesize{$j-1$}};

    \end{tikzpicture}\\
    &=\cdots\\
    &=\delta^{-1}\begin{tikzpicture}[baseline = -3mm, color=\clr]
	\draw[-,thick,darkblue] (0.28,0) to[out=90, in=0] (0,0.2);
	\draw[-,thick,darkblue] (0,0.2) to[out = 180, in = 90] (-0.28,0);
\draw[-,thick,darkblue] (0.28,-.6) to[out=120,in=-90] (-0.28,0);
	
	\draw[-,line width=4pt,white] (-0.28,-.6) to[out=60,in=-90] (0.28,0);
	\draw[-,thick,darkblue] (-0.28,-.6) to[out=60,in=-90] (0.28,0);
    \node at (-0.25,-0.1) {$\color{darkblue}\scriptstyle\bullet$};
    \draw (-0.6,0.1) node{\footnotesize{$j$}};

    \draw[-,thick,darkblue]  (0.28,-0.6) to[out=down,in=right] (-0.,-0.8)to[out=left,in=down] (-0.28,-.6);

\end{tikzpicture}
+\delta^{-1}z \sum_{i=1}^{j-1}\begin{tikzpicture}[baseline = 5pt, scale=0.5, color=\clr]
        \draw[-,thick] (0.6,1) to (0.5,1) to[out=left,in=up] (0,0.5)
                        to[out=down,in=left] (0.5,0)
                        to[out=right,in=down] (1,0.5)
                        to[out=up,in=right] (0.5,1);
        \node at (0,0.5) {$\color{darkblue}\scriptstyle\circ$};
        \draw (-0.6,1) node{\footnotesize{$j-2i$}};
    \end{tikzpicture}-\delta^{-1}z \sum_{i=1}^{j-1}\begin{tikzpicture}[baseline = 5pt, scale=0.5, color=\clr]
     \draw[-,thick] (0.6,1) to (0.5,1) to[out=left,in=up] (0,0.5)
                        to[out=down,in=left] (0.5,0)
                        to[out=right,in=down] (1,0.5)
                        to[out=up,in=right] (0.5,1);
                        \node at (0,0.5) {$\color{darkblue}\scriptstyle\bullet$};
        \draw[-,thick] (1.8,1) to (1.7,1) to[out=left,in=up] (1.2,0.5)
                        to[out=down,in=left] (1.7,0)
                        to[out=right,in=down] (2.2,0.5)
                        to[out=up,in=right] (1.7,1);
                         \node at (1.2,0.5) {$\color{darkblue}\scriptstyle\circ$};
         \draw (-0.1,0.8) node{\footnotesize{$i$}};
        \draw (1.5,1.2) node{\footnotesize{$j-i$}};

    \end{tikzpicture}\\
    &=\delta^{-2}\begin{tikzpicture}[baseline = 5pt, scale=0.5, color=\clr]
        \draw[-,thick] (0.6,1) to (0.5,1) to[out=left,in=up] (0,0.5)
                        to[out=down,in=left] (0.5,0)
                        to[out=right,in=down] (1,0.5)
                        to[out=up,in=right] (0.5,1);
        \node at (0,0.5) {$\color{darkblue}\scriptstyle\bullet$};
        \draw (-0.4,0.5) node{\footnotesize{$j$}};
    \end{tikzpicture}+\delta^{-1}z \sum_{i=1}^{j-1}\begin{tikzpicture}[baseline = 5pt, scale=0.5, color=\clr]
        \draw[-,thick] (0.6,1) to (0.5,1) to[out=left,in=up] (0,0.5)
                        to[out=down,in=left] (0.5,0)
                        to[out=right,in=down] (1,0.5)
                        to[out=up,in=right] (0.5,1);
        \node at (0,0.5) {$\color{darkblue}\scriptstyle\circ$};
        \draw (-0.6,1) node{\footnotesize{$j-2i$}};
    \end{tikzpicture}-\delta^{-1}z \sum_{i=1}^{j-1}\begin{tikzpicture}[baseline = 5pt, scale=0.5, color=\clr]
     \draw[-,thick] (0.6,1) to (0.5,1) to[out=left,in=up] (0,0.5)
                        to[out=down,in=left] (0.5,0)
                        to[out=right,in=down] (1,0.5)
                        to[out=up,in=right] (0.5,1);
                        \node at (0,0.5) {$\color{darkblue}\scriptstyle\bullet$};
        \draw[-,thick] (1.8,1) to (1.7,1) to[out=left,in=up] (1.2,0.5)
                        to[out=down,in=left] (1.7,0)
                        to[out=right,in=down] (2.2,0.5)
                        to[out=up,in=right] (1.7,1);
                         \node at (1.2,0.5) {$\color{darkblue}\scriptstyle\circ$};
         \draw (-0.1,0.8) node{\footnotesize{$i$}};
        \draw (1.5,1.2) node{\footnotesize{$j-i$}};,\end{tikzpicture}, \text{ by Lemma~\ref{selfcrossing}(2).}
    \end{aligned}$$
\end{proof}

Suppose $d\in\mathbb{T}_{m,s}$. Throughout this paper,    $\hat{d}$ is always  the tangle diagram  obtained from $d$ by removing all loops and all dots on it.


\begin{Defn}\label{D:N.O. dotted  OBC tangle diagram} Suppose $d\in\mathbb{T}_{m,s}$.
  $d$ is said to be  normally ordered if
\begin{enumerate}

\item[(a)] all of its loops are crossing-free, and  there are no other strands shielding any of them from the leftmost edge of the picture,
  \item[(b)] there is no  \begin{tikzpicture}[baseline = 5pt, scale=0.5, color=\clr]
        \draw[-,thick] (0.6,1) to (0.5,1) to[out=left,in=up] (0,0.5)
                        to[out=down,in=left] (0.5,0)
                        to[out=right,in=down] (1,0.5)
                        to[out=up,in=right] (0.5,1);
    \end{tikzpicture} on $d$,
  \item[(c)] there is no $\circ$ on each loop.  All bullets on each loop will be  at the leftmost boundary of it.  For example,   \begin{tikzpicture}[baseline = 5pt, scale=0.5, color=\clr]
        \draw[-,thick] (0.6,1) to (0.5,1) to[out=left,in=up] (0,0.5)
                        to[out=down,in=left] (0.5,0)
                        to[out=right,in=down] (1,0.5)
                        to[out=up,in=right] (0.5,1);
                        \node at (0,0.5) {$\color{darkblue}\scriptstyle\bullet$};
        \draw (0.4,0.5) node{\footnotesize{$2$}};
    \end{tikzpicture} is allowed but  non of  \begin{tikzpicture}[baseline = 5pt, scale=0.5, color=\clr]
        \draw[-,thick] (0.6,1) to (0.5,1) to[out=left,in=up] (0,0.5)
                        to[out=down,in=left] (0.5,0)
                        to[out=right,in=down] (1,0.5)
                        to[out=up,in=right] (0.5,1);
                        \node at (1,.5) {$\color{darkblue}\scriptstyle\bullet$};
        \draw (1.4,0.5) node{\footnotesize{$2$}};
    \end{tikzpicture},\begin{tikzpicture}[baseline = 5pt, scale=0.5, color=\clr]
        \draw[-,thick] (0.6,1) to (0.5,1) to[out=left,in=up] (0,0.5)
                        to[out=down,in=left] (0.5,0)
                        to[out=right,in=down] (1,0.5)
                        to[out=up,in=right] (0.5,1);
                        \node at (0,0.5) {$\color{darkblue}\scriptstyle\circ$};
        \draw (0.4,0.5) node{\footnotesize{$2$}};
    \end{tikzpicture} and   \begin{tikzpicture}[baseline = 5pt, scale=0.5, color=\clr]
        \draw[-,thick] (0.6,1) to (0.5,1) to[out=left,in=up] (0,0.5)
                        to[out=down,in=left] (0.5,0)
                        to[out=right,in=down] (1,0.5)
                        to[out=up,in=right] (0.5,1);
        \node at (0,0.5) {$\color{darkblue}\scriptstyle\bullet$};
       \node at (1,0.5) {$\color{darkblue}\scriptstyle\bullet$};\end{tikzpicture}
     are allowed,
    \item[(d)] $ \hat d$ is  a  reduced totally descending  tangle diagram,

    \item[(e)] whenever a dot ($\bullet$ or $\circ$) appears on a vertical strand, it is on the boundary of the bottom row,

    \item[(f)] whenever a dot ($\bullet$ or $\circ$) appears on  a cap (resp.,  cup),  it is on   the leftmost boundary of the cap (resp.,   the rightmost boundary  of a cup),
    \item[(g)]  $\bullet$ and  $\circ$ can not occur on the same strand.
\end{enumerate}
\end{Defn}
For any $m, s\in\mathbb{N}$, let
$\mathbb {NT}_{m, s}=\{d\in \mathbb T_{m, s}\mid \text{$d$ is normally ordered}\}$.
The following tangle diagrams   represent two morphisms in $\Hom_{\AB}(\ob 5,\ob 5)$. The  left one is in $\mathbb {NT}_{5, 5}$ whereas  the right one is not.

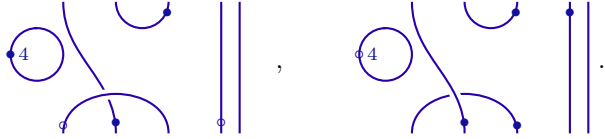
\begin{center}
    \begin{tikzpicture}[baseline = 25pt, scale=0.35, color=\clr]

        \draw[-,thick] (2,0) to[out=up,in=down] (0,5);
       \draw[-,line width=4pt,white] (0,0) to[out=up,in=left] (2,1.5) to[out=right,in=up] (4,0);
       \draw[-,thick] (0,0) to[out=up,in=left] (2,1.5) to[out=right,in=up] (4,0);
               \draw[-,thick] (6,0) to[out=up,in=down] (6,5);
         \draw[-,thick] (6.7,0) to[out=up,in=down] (6.7,5);
        \draw[-,thick] (2,5) to[out=down,in=left] (3,4) to[out=right,in=down] (4,5);
        \draw[-,thick] (-1.1,4) to (-1,4) to[out=left,in=up] (-2,3) to[out=down,in=left] (-1,2)to[out=right,in=down] (0,3) to[out=up,in=right] (-1,4);
        \node at (0,0.3) {$\color{darkblue}\scriptstyle\circ$};
        \node at (-2,3) {$\color{darkblue}\scriptstyle\bullet$};
        \node at (6,0.4) {$\color{darkblue}\scriptstyle\circ$};
        \node at (3.95,4.6) {$\color{darkblue}\scriptstyle\bullet$};
       \node at (2,0.4) {$\color{darkblue}\scriptstyle\bullet$};
    \node at (-1.5,3) {$\color{darkblue}\scriptstyle 4$};
           \end{tikzpicture}
    \quad, \qquad
    \begin{tikzpicture}[baseline = 25pt, scale=0.35, color=\clr]
    \draw[-,thick] (0,0) to[out=up,in=left] (2,1.5) to[out=right,in=up] (4,0);

       \draw[-,line width=4pt,white] (2,0) to[out=up,in=down] (0,5);

       \draw[-,thick] (2,0) to[out=up,in=down] (0,5);
              %
               \draw[-,thick] (6,0) to[out=up,in=down] (6,5);
         \draw[-,thick] (6.7,0) to[out=up,in=down] (6.7,5);
        \draw[-,thick] (2,5) to[out=down,in=left] (3,4) to[out=right,in=down] (4,5);
        \draw[-,thick] (-1.1,4) to (-1,4) to[out=left,in=up] (-2,3) to[out=down,in=left] (-1,2)to[out=right,in=down] (0,3) to[out=up,in=right] (-1,4);
        \node at (4,0.3) {$\color{darkblue}\scriptstyle\bullet$};
        \node at (-2,3) {$\color{darkblue}\scriptstyle\circ$};
        \node at (6,4.6) {$\color{darkblue}\scriptstyle\bullet$};
        \node at (3.95,4.6) {$\color{darkblue}\scriptstyle\bullet$};
       \node at (2,0.4) {$\color{darkblue}\scriptstyle\bullet$};
    \node at (-1.5,3) {$\color{darkblue}\scriptstyle 4$};
           \end{tikzpicture}~~.
\end{center}

\begin{Defn}\label{equiv1} For any $d, d'\in \mathbb {NT}_{m, s}$, write  $d\sim d'$ if \begin{itemize}\item  they have the same number of $\begin{tikzpicture}[baseline = 5pt, scale=0.5, color=\clr]
        \draw[-,thick] (0.6,1) to (0.5,1) to[out=left,in=up] (0,0.5)
                        to[out=down,in=left] (0.5,0)
                        to[out=right,in=down] (1,0.5)
                        to[out=up,in=right] (0.5,1);
        \node at (0,0.5) {$\color{darkblue}\scriptstyle\bullet$};
        \draw (-0.4,0.5) node{\footnotesize{$i$}};
    \end{tikzpicture}$, for any $i\in \mathbb N\setminus 0$,  \item  $\text{conn}(\hat d)=\text{conn}(\hat{d'})$, \item  there are same number of $\bullet$ or $\circ $   on their corresponding strands. \end{itemize}\end{Defn}

We have $\hat d= \hat d'$ as morphisms in $\B$ and hence in $\AB$ if  $d, d'\in \mathbb {NT}_{m, s}$ and   $d\sim d'$.  So  $ d=  d'$ as morphisms in $\AB$.
We will identify each equivalence class of $\mathbb{NT}_{m,s}$ with any element in
it.  The following is the first main result of this paper.

\begin{Theorem}\label{affbasis} Suppose  $m, s\in \mathbb N$.
\begin{itemize}\item [(1)] If $ m+s$ is odd, then  $\Hom_{\AB}(\ob m, \ob s)=0$.
\item[(2)] If $ m+s$ is even, then $\Hom_{\AB}(\ob m, \ob s)$ has $\mathbb K$-basis given by
$\mathbb {NT}_{m, s}/\sim$. In particular,  $\Hom_{\AB}(\ob m, \ob s)$ is of infinite rank.
\end{itemize}
\end{Theorem}

\begin{Defn}\label{admi}\cite[Definition~2.19]{RX} Suppose  $u_1, \ldots, u_a$ are units in $\mathbb K$ and $\omega_i\in \mathbb K$, $ i\in \mathbb Z$. Let $\omega=\{\omega_i\mid i\in \mathbb Z\}$ and $\mathbf u=\{u_1,u_2,\ldots,u_a\}$. Then $\omega$ is called admissible if
\begin{itemize} \item [(1)]
$\omega_{-i}  =\delta^{-2} \omega_i+\delta^{-1} z \sum_{l=1}^{i-1} (\omega_{2l-i}-\omega_l\omega_{l-i})$, for any positive integer $i$,
\item [(2)]
    $\omega_i =-\sum_{j=1}^a b_{a-j} \omega_{i-j}$
  if $i\ge a$,  where
 $b_j=(-1)^{a-j} e_{a-j}(\mathbf u)$ and $e_i(\mathbf u)$ is the $i$th elementary symmetric function on $u_1, u_2, \ldots, u_a$.\end{itemize}
\end{Defn}

\begin{Defn}\label{scbmw}Suppose
 $f(t)=\prod_{i=1}^a (t-u_i)$, where $u_1, \ldots, u_a$ are units in $\mathbb K$. For any admissible $\omega$, define $\CB^f=\AB/I$ where $I$ is the right tensor ideal generated by $f(\begin{tikzpicture}[baseline=-.5mm]
 \draw[-,thick,darkblue] (0,-.3) to (0,.3);
      \node at (0,0) {$\color{darkblue}\scriptstyle\bullet$};
\end{tikzpicture})$ together with  $\Delta_i-\omega_i1_{\ob 0}$ for all  $i\in \mathbb Z\setminus\{0\} $.
 \end{Defn}
  We call $\CB^f$ the  cyclotomic Kauffmann category. It is available only if $\omega$ is admissible in the sense of Definition~\ref{admi}. In fact, Definition~\ref{admi}(1) follows from Lemma~\ref{addmi} and
Definition~\ref{admi}(2) follows from the equation  $$A\circ  (X^{i-a}\otimes 1)\circ ( f(X)\otimes 1) \circ U=0 $$  in $\CB^f$ for any $i\ge a$.

\begin{Assumption}\label{asump} Let $\mathbb K$ be an integral domain   containing  units $q, q-q^{-1}, u_1, \ldots, u_a$. We always assume
$$
\delta=\alpha \prod_{i=1}^a u_i, \text{ and } \omega_0=1+z^{-1} (\delta-\delta^{-1} ), $$ where $ z=q-q^{-1}$ and
    $\alpha\in \{1, -1\}$ if $a$ is odd and $\alpha\in \{-q, q^{-1}\}$ if $a$ is even.
\end{Assumption}

We always keep the  Assumption~\ref{asump} when we talk about  $\CB^f$ over $\mathbb K$. Later on,   we denote $q-q^{-1}$ by $z_q$.

\begin{Defn}\cite[Lemma~2.28]{RX}\label{uad} Suppose $u$ is an indeterminate. Then $\mathbf \omega$ is called $\mathbf u$-admissible if
$$\begin{aligned} & \sum_{i\ge 0} \frac{\omega_i}{u^i}=\frac{u^2}{u^2-1}-(z_q\delta)^{-1}
 +\left ( (z_q \delta)^{-1}\prod_{i=1}^a u_i+\frac{ug_{a}(u)}{u^2-1}\right )\prod_{i=1}^a u_i \prod_{i=1}^a \frac{u-u_i^{-1}}{u-u_i},\\
& \sum_{i\ge 1} \frac{\omega_{-i}}{u^i}=\frac{1}{u^2-1}+(z_q\delta)^{-1}
 -
 \left( (z_q \delta)^{-1}\prod_{i=1}^a u_i-\frac{u}{g_{a}(u) (u^2-1)}\right )\prod_{i=1}^a u_i^{-1}\prod_{i=1}^a \frac{u-u_i}{u-u_i^{-1}},
 \end{aligned}$$
where $g_a(u)=1$ (resp., $-u$) if $a$ is odd (resp.,  even).
\end{Defn}

 Recall that $\lfloor \ell \rfloor$ is the maximal
integer such that  $\lfloor \ell \rfloor\le \ell$, where $\ell$ is a real number.
\begin{Defn} For any   $m, s\in \mathbb N$, let $\bar{\mathbb {NT}}_{m,s}$  be the set of all $d\in \mathbb {NT}_{m, s}$ such that there is no $\begin{tikzpicture}[baseline = 5pt, scale=0.5, color=\clr]
        \draw[-,thick] (0.6,1) to (0.5,1) to[out=left,in=up] (0,0.5)
                        to[out=down,in=left] (0.5,0)
                        to[out=right,in=down] (1,0.5)
                        to[out=up,in=right] (0.5,1);
        \node at (0,0.5) {$\color{darkblue}\scriptstyle\bullet$};
        \draw (-0.4,0.5) node{\footnotesize{$i$}};
    \end{tikzpicture}$ on $d$, $\forall i\in \mathbb N$. For any positive integer $\ell$, let   $\bar{\mathbb {NT}}_{m,s}^\ell$ be the subset of all  $d\in \bar{\mathbb {NT}}_{m,s}$ such that $$-i, j\in \left\{\lfloor \frac{\ell-1}{2}\rfloor, \lfloor\frac{\ell-1}{2}\rfloor-1,\ldots,-\lfloor \frac{\ell}{2}\rfloor\right\}$$
     if there are $i$ (resp., $j$) $\bullet$'s near each endpoint at top (resp., bottom) row of $d$.
\end{Defn}

The    equivalence relation $\sim$  in  Definition~\ref{equiv1} induces  equivalence relations $\sim$ on both $\bar{\mathbb {NT}}_{m,s}$ and $\bar{\mathbb {NT}}_{m,s}^a$.
We have $ d=  d'$ as morphisms in $\AB$ and hence in $\CB^f$ if  $d, d'\in \bar{\mathbb {NT}}_{m, s}^a$ and   $d\sim d'$. So,  each equivalence class in $\bar{\mathbb{NT}}_{m,s}^a$ can be identified with any element in it. Theorem~\ref{cycb}  is the second main result of this paper.

\begin{Theorem}\label{cycb}  Keep the Assumption~\ref{asump}.  Suppose  $m, s\in \mathbb N$.\begin{itemize}\item[(1)]  If $m+s$ is  odd,
  then $\Hom_{\CB^f}(\ob m, \ob s)=0$.
  \item[(2)]If $ m+s$ is even, then $\Hom_{\CB^f}(\ob m, \ob s)$ has  $\mathbb K$-basis given by  $\bar{\mathbb {NT}}_{m, s}^a/\sim$  if and only if $\mathbf \omega$ is $\mathbf u$-admissible. In this case, $\Hom_{\CB^f}(\ob m, \ob s)$ is of rank $a^{\frac{m+s}{2}}(m+s-1)!!$.
      \end{itemize} \end{Theorem}

\section{Connections to  quantum  symplectic and orthogonal groups}
We follow the conventions in \cite{Bou}. In this paper,
$\mathfrak g$ is always  either a complex symplectic  Lie algebra  $\mathfrak{sp}_{2n}$ or orthogonal Lie algebras $\mathfrak{so}_{2n}, \mathfrak{so}_{2n+1}$. Let  $\mathcal R$ be the root system associated to $\mathfrak g$ with fixed simple roots $\Pi=\{\alpha_1, \ldots, \alpha_n\}$ such that $\alpha_i=\epsilon_i-\epsilon_{i+1}$, $1\le i\le n-1$  and
$\alpha_n= \epsilon_n$ (resp., $2\epsilon_n, \epsilon_{n-1}+\epsilon_n$)
if $\mathfrak g=\mathfrak{so}_{2n+1}$ (resp., $\mathfrak{sp}_{2n}, \mathfrak{so}_{2n}$).
Then the set of positive roots

\begin{equation}\label{typeb}
\mathcal{R}^+=
        \{\epsilon_i\pm \epsilon_j\mid 1\le i<j\le n\}\cup  \mathcal Z_{\mfg} \end{equation}
        where $\mathcal Z_{\mathfrak{so}_{2n+1}}=\{\epsilon_i\mid 1\le i\le n\}$,
       $\mathcal Z_{ \mathfrak{sp}_{2n}}= \{2\epsilon_i\mid 1\le i\le n\}$, and
     $\mathcal Z_{\mathfrak{so}_{2n}}=\emptyset$.
     From here onwards, we always assume  \begin{equation}\label{NN}\mathrm N= 2n+1, \text{(resp., $2n$)  if $\mathfrak g=\mathfrak{so}_{2n+1}$ (resp., $\mfg \in \{\mathfrak{so}_{2n}, \mathfrak {sp}_{2n}\}$),}\end{equation}
  and    $i'=\mathrm N +1-i, 1\le i\le \mathrm N$.
     Then the half sum of positive roots \begin{equation} \label{posr} \varrho=  \sum_{i=1}^n \varrho_i\epsilon_i,\end{equation}
 where $\varrho_i= \frac{1}{2} \mathrm N -i+b_\mfg$ and $b_{\mathfrak{sp}_{\mathrm N}}=1$ and  $b_{\mathfrak{so}_{\mathrm N}}=0$.
Later on, we assume  $\varrho_{i'}=-\varrho_i$ $1\le i\le n$ and $\varrho_{n+1}=0$ if    $\mfg=\mathfrak{so}_{2n+1}$.

 On the real vector space spanned by $\mathcal R$, there is an inner product
$(\ \mid \ )$ such that
  $(\epsilon_i\mid \epsilon_j)= 2\delta_{i,j}$ (resp., $\delta_{i,j}$)  if $\mathfrak g=\mathfrak{so}_{2n+1}$ (resp., otherwise). So,
  \begin{equation}\label{varrho}(2\varrho\mid \alpha_i)=(\alpha_i\mid\alpha_i) \text{ for any  simple root $\alpha_i$.}\end{equation}
The fundamental weights  $\varpi_1, \ldots, \varpi_n$ satisfy  $(\varpi_i\mid \alpha_j)=\delta_{ij} d_i$, where
   $d_i=\frac{1}{2} (\alpha_i\mid \alpha_i)$. The integral weight lattice $\mathcal{P}=\oplus_{i=1}^n \mathbb Z \varpi_i$ and the set of dominant integral weights    $\mathcal{P}^+=\oplus_{i=1}^n \mathbb N \varpi_i$.

 If $\mfg=\mathfrak{so}_{2n}$, let $\mathbb F=\mathbb C(q^{1/4})$  where   $q^{1/4}$ is an indeterminate.
Otherwise,   let $\mathbb F=\mathbb C(q^{1/2}) $.
   Later on,
 \begin{equation}\label{vv} \text{ $v=q^{1/2}$  if $\mfg=\mathfrak{so}_{2n+1}$, and $v=q$ if  $\mfg =\mathfrak{sp}_{2n}, \mathfrak {so}_{2n}$.}\end{equation}
The quantum group
$\U_v(\mathfrak g)$ of simply connected type  is the $\mathbb F$-algebra generated by $\{x_i^\pm,k_\lambda\mid 1\le i\le n, \lambda\in \mathcal{P}\}$ subject to the relations  in~\cite[\S 3.1.1]{Lu}. Let  $\U_v(\mathfrak h)$ be the subalgebra generated by $\{k_\lambda\mid\lambda\in\mathcal{P}\}$. The quantum group  $\bar{\U}_v(\mathfrak g)$ of adjoint type is the $\mathbb F$-algebra   generated by $\{x_i^\pm,k_\lambda\mid\lambda\in \mathcal{Q}\}$, where $\mathcal Q$ is the root lattice associated to $\mathfrak g$. Later on, we denote by $\mathcal Q^+$ the positive root lattice.  It is known that $\U_v(\mathfrak g)$ is a Hopf algebra and the comultiplication $\Delta$, counit $\epsilon$ and antipode $S$ satisfy

\begin{equation} \label{rell} \begin{array} {ccc} \Delta(x_i^+) =x_i^+\otimes 1+k_{\alpha_i}\otimes x_i^+, & \epsilon(x_i^+)=0, & S(x_i^+)=-k_{-\alpha_i} x_i^+,\\
  \Delta(x_i^-)  =x_i^-\otimes k_{-\alpha_i}+1\otimes x_i^-, &  \epsilon(x_i^-)=0, &  S(x_i^-)=-x_i^- k_{\alpha_i},\\
 \Delta(k_\mu)=k_\mu \otimes k_\mu, &   \epsilon(k_\mu)=1, & S(k_\mu)=k_{-\mu}.\\
\end{array}\end{equation}
For all $u\in \U_v(\mathfrak g)$, let
\begin{equation}\label{bar123} \bar\Delta(u)=(-\otimes -)\circ \Delta(\bar u), \end{equation}  where $\bar{ \  }$ is the $\mathbb C$-linear automorphism of $\U_v(\mathfrak g)$ such that  $\bar x_i^\pm= {x_i}^\pm$,  $\bar k_\mu=k_{-\mu}$ and $\bar q^{1/4}=q^{-1/4}$ (resp., $\bar q^{1/2}=q^{-1/2}$) if $\mfg=\mathfrak{so}_{2n}$ (resp., otherwise). Then
 \begin{equation} \label{rell1} \bar\Delta(x_i^+) =x_i^+\otimes 1+k_{-\alpha_i}\otimes x_i^+, \ \bar\Delta(x_i^-) =x_i^-\otimes k_{\alpha_i}+1\otimes x_i^-, \ \bar\Delta(k_\mu)=k_\mu \otimes k_\mu,\end{equation}
for all admissible $i$ and $\mu$.

In this paper,  a $  \U_v(\mathfrak g)$-module $M$ is  always a left module. It is called a weight module if $$M=\oplus_{\lambda\in \mathcal{P}} M_\lambda,$$ where $ M_\lambda=\{m\in M\mid k_\mu m=v^{(\lambda\mid\mu)} m, \text{ for any $\mu\in \mathcal{P}$}\}$, called the $\lambda$-weight space of $M$.
Write $wt(m)=\lambda \text{ if } m\in M_\lambda$.
It is easy to check   that  a weight $\U_v(\mathfrak g)$-module  is always a  weight $\bar {\U}_v(\mathfrak g)$-module and vice versa. Further, a $\U_v(\mathfrak g)$-homomorphism between two weight modules is always a $\bar {\U}_v(\mathfrak g)$-homomorphism and vice versa. This enables us to use  previous results on  $\bar {\U}_v(\mathfrak g)$ in \cite{RS2}, directly.

Let $W$ be  the Weyl group corresponding to $\mathcal R$. It is a Coxeter group
 generated by $\{s_i\mid 1\le i\le n\}$, where $s_i$ is the simple reflection $s_{\alpha_i}$.
 For any $ 1\le i\le n$,  let $$T_i=T_{i,+1}^{''},$$   the braid group generator defined by Lusztig  in \cite[\S 37.1.3]{Lu}. In general, $T_w=T_{i_1}\cdots T_{i_k}$ if $s_{i_1}\cdots s_{i_k}$ is a reduced expression of $w$.

  Throughout this paper, $w_0$ is always   the longest element in $W$. Let $\ell(\ )$ be the length function on $W$.
Fix a reduced expression $\vec{\prod}_{j=1}^{\ell(w_0)} s_{i_j}$ of $w_0$  and   define   $ \beta_1=\alpha_{i_1}$,  $\beta_j=s_{i_1}s_{i_2}\cdots s_{i_{j-1}}(\alpha_{i_j})$   and $d_{\beta_j}= \frac{1} {2}(\beta_j|\beta_j)$, $1\le j\le \ell(w_0)$.
 Then  $$\mathcal R^+=\{{\beta_j}\mid1\le j\le \ell(w_0)\}$$ such that
there is a  convex ordering on $\mathcal{R}^{+}$\cite{Pa} satisfying
 $\beta_j<\beta_k$ whenever $j<k$ and
 $ \beta_i+\beta_j=\beta_k$
for some $k, i<k<j$   if  $\beta_i+\beta_j\in \mathcal{R}^{+}$.
Motivated by \cite{Lu1}, we consider    root elements $$x_{\beta_j}^\pm =T_{i_1}T_{i_2}\cdots T_{i_{j-1}} (x_{i_j}^\pm)$$ for all admissible $j$. For any  $\mathbf r\in \mathbb N^{\ell(w_0)}$,  let
\begin{equation}\label{xplus}x_{\mathbf r}^+=\vec{\prod}_{i=1}^{\ell(w_0)} (x_{\beta_i}^+)^{r_i} \text{ and  }
 x^-_{\mathbf r}=\vec{\prod}_{i=\ell(w_0)}^1 (x_{\beta_i}^-)^{r_i}. \   \end{equation}
Then $\{x^-_{\mathbf r}k_\mu x^+_{\mathbf s}\mid  \mathbf r, \mathbf s\in \mathbb N^{\ell(w_0)}, \mu\in \mathcal P\}$ forms a PBW-like basis of $\U_v(\mathfrak g)$. Let $\U_v^+(\mathfrak g)$ (resp., $\U_v^-(\mathfrak g)$) be the subalgebra of $\U_v(\mathfrak g)$ with basis $ \{ x^+_{\mathbf r}\mid  \mathbf r\in \mathbb N^{\ell(w_0)}\}$ (resp., $ \{ x^-_{\mathbf r}\mid  \mathbf r\in \mathbb N^{\ell(w_0)}\}$). The Borel subalgebras $\U_v(\mathfrak b)=\U_v(\mathfrak h)\otimes_{\mathbb F}\U_v^+(\mathfrak g)$ (resp., $\U_v(\mathfrak b^-)=\U_v^-(\mathfrak g)\otimes_{\mathbb F}\U_v(\mathfrak h)$).

In \cite{CP, LS, Lu1} etc,   root elements are defined via braid generators
 $$T_i=T_{i,-1}^{''}$$ in  \cite[\S 37.1.3]{Lu}. Let  $\hat{x}_{\beta_j}^{\pm}$'s be  the corresponding root elements. Checking their  actions on generators yields  \begin{equation}\label{t1}T_{i,+1}^{''}=\tau' T_{i,-1}^{''} \tau', 1\le i\le n,\end{equation}
where  $\tau'$ is the $\mathbb F$-linear automorphism of $\U_v(\mathfrak g)$ such that  $\tau'(x_i^{\pm})=-x_i^{\pm}$ and    $\tau'(k_\mu)=k_\mu$  for all admissible $\mu$ and $i$. So,   \begin{equation}\label{ccc123} {x}_{\beta_j}^{\pm}=-\tau'(\hat{x}_{\beta_j}^{\pm}).\end{equation}  Thanks to \cite[Theorem 9.5]{CP}, $\hat x_{\alpha_i}^{\pm}=x_i^{\pm}$ for any simple root $\alpha_i$. So $$ x_{\alpha_i}^{\pm}=-\tau'(\hat x_{\alpha_i}^{\pm})=x_i^{\pm}.$$
Let $\tau$ be the $\mathbb F$-linear  anti-automorphism of $\U_v(\mathfrak g)$ such that
 $$\tau (x_i^+)=x_i^-,  \text{ $\tau(x_i^-)=x_i^+$ and
 $\tau(k_\mu)=k_\mu$}$$ for all admissible $i$ and $\mu$.
  Let $\bar{\tau}:=-\circ \tau$ where $\bar{\ }$ is the $\mathbb C$-linear automorphism of $\U_v(\mathfrak g)$ in \eqref{bar123}.
   Then $\bar \tau$  is a $\mathbb C$-linear anti-automorphism of $\U_v(\mathfrak g)$.  Checking the actions on generators yields $\bar{\tau }T_i=T_i\bar{\tau}$. So,  \begin{equation}\label{switch}\bar{\tau}(x_{\beta_i}^{\pm})=x_{\beta_i}^{\mp}, \text{ for any $\beta_i\in \mathcal{R}^{+}$.} \end{equation}

\begin{Defn} \label{codeg} For any $i\in \mathbb Z$, define  $\mathcal A_i=\{x\in \mathbb C(q^{1/2})\mid
ev(x)\geq i\}$ where $ev: \mathbb C(q^{1/2})\rightarrow \mathbb Z$ such that $ev(0)=\infty$ and
$ev(x)=j$ if $0\neq x=(q^{1/2}-1)^j g/h$ and  $g, h\in \mathbb C[q^{1/2}]$ such that $(g, q^{1/2}-1)=1$, $(h, q^{1/2}-1)=1$. \end{Defn}

  Obviously,  $\mathcal A_i\subset \mathcal A_{i-1}$ for any
$i\in\mathbb Z$ and  $\mathcal A_i$ is a subring of $\mathbb C(q^{1/2})$ if $i\in \mathbb N$.

 \begin{Lemma}\label{commut} Suppose  $1\leq i<j\leq \ell(w_0)$. Then \begin{itemize} \item[(1)]${x_{\beta_i}^-}{x_{\beta_j}^-}-v^{(\beta_i\mid\beta_j)}{x_{\beta_j}^-}{x_{\beta_i}^-}=\sum_{\mathbf r\in\mathbb{N}^{\ell(w_0)} } c_{\mathbf r}^{-}x^-_{\mathbf r}$,
\item[(2)]${x_{\beta_j}^+}{x_{\beta_i}^+}-v^{-(\beta_i\mid\beta_j)}{x_{\beta_i}^+}{x_{\beta_j}^+}=\sum_{\mathbf r\in\mathbb{N}^{\ell(w_0)} } c_\mathbf{r}^{+} x^+_{\mathbf r}$,
\end{itemize}
    such that   $c_\mathbf{r}^{\pm}\in \mathcal A_0$   and $\sum_{l=1}^{\ell(w_0)} r_l\beta_l=\beta_i+\beta_j$.
Moreover,  $c_\mathbf{r}^{\pm}= 0$ unless  $r_l=0$ for all $l$ such that either  $l\leq i$ or  $l\geq j$.
 \end{Lemma}
 \begin{proof}By \cite[Proposition~5.5.2]{LS},  \begin{equation}\label{bbb123} {\hat {x}_{\beta_j}^+}{\hat {x}_{\beta_i}^+}-v^{-(\beta_i\mid\beta_j)}{\hat{x}_{\beta_i}^+}{\hat{x}_{\beta_j}^+}=\sum_{{\mathbf r\in \mathbb{N}^{\ell(w_0)}}}\hat{c}_{\mathbf r} \vec{\prod}_{i=1}^{\ell(w_0)} (\hat x_{\beta_i}^+)^{r_i},\end{equation} where $\hat{c}_{\mathbf r}=0$ unless   $r_l=0$ for all $l$ such that  $l\le  i$ or   $l\geq j$,  and $\sum_{l=1}^{\ell(w_0)} r_l\beta_l=\beta_i+\beta_j$. Recall $\tau'$ in \eqref{t1}.
 Acting $\tau'$ on both sides of \eqref{bbb123} and using \eqref{ccc123}  yield  the commuting relation about $x_{\beta_j}^{+} x_{\beta_i}^{+}$.  In this case, $c_{\mathbf r}^+=\hat c_{\mathbf r}$ up to a sign.
In order to prove (2), we need to verify $c_{\mathbf r}^+\in\mathcal A_0$.

Thank to \cite[Theorem~6.7(ii)]{Lu1}, $\hat c_{\mathbf r}=g(v)\prod_{j=1}^{\ell(w_0)} ([r_j]^!_{d_{\beta_j}})^{-1} $ for some $g(v)\in\mathbb{Z}[v,v^{-1}]$,  where $$[i]_d=\frac{v^{di}-v^{-di}}{v^d-v^{-d}} \text{  and
 $ [k]^!_d=\prod_{i=1}^k [i]_d $.}$$
  So,
 $\hat c_{\mathbf r}\in\mathcal A_0$, and hence  $ c_{\mathbf r}^+\in\mathcal A_0$.
Note that $\bar{\tau}( \mathcal A_0)\subseteq \mathcal A_0$.
  Acting $\bar{\tau}$ on both sides of  the equation in (2)
  and using \eqref{switch} yield (1).
 \end{proof}

  For any  ordered sequence of positive roots $I$, we say that $I$ is of length (resp., weight)  $j$ (resp.,   $\sum_{l=1}^j \beta_{i_l}$) and write $\ell(I)=j$ (resp.,  $wt(I)=\sum_{l=1}^j \beta_{i_l}$)
   if $I =(\beta_{i_1}, \beta_{i_2}, \ldots, \beta_{i_j})$.
   In this case, define
 $x_I^\pm=\vec{\prod}_{l=1}^j x_{\beta_{i_l}}^\pm $. If $j=0$, write  $x_I^\pm=1$.

\begin{Cor}\label{monoial} Suppose  $I=(\beta_{i_1}, \beta_{i_2}, \ldots, \beta_{i_j})$.
Then $x_I^-$ can be written  as an $\mathcal A_0 $-linear combination of $x_{\mathbf r}^-$, $\mathbf r\in \mathbb N^{\ell(w_0)}$.
\end{Cor}
\begin{proof}  The  result follows from   Lemma~\ref{commut}, immediately.\end{proof}

Suppose  $\nu,\gamma\in \mathcal R^+$ such that $\nu >\gamma$ with respect to a fixed  convex order $<$.
 Following
\cite{BKM}, $(\nu,\gamma)$  is called
a minimal pair of $\beta$ if $\beta=\nu +\gamma\in\mathcal{R}^{+}$ and  if   there is  no  $\nu ',\gamma'\in \mathcal R^+$
such that  $\nu '+\gamma'=\beta$ and $\nu >\nu '>\beta>\gamma'>\gamma$.
 Let $ \Upsilon_\beta$ be the set of minimal pairs of $\beta$,  $\beta\in \mathcal R^+$.

Suppose  $(\nu,\gamma)$ is a minimal pair of $\beta$.
Following \cite{BKM}, define \begin{equation}\label{p1} p_{\nu,\gamma}=  1, \text{(resp., $0$) if $(d_\nu, d_\beta, d_\gamma) =(1, 2, 1)$  (resp., otherwise).}
  \end{equation}
 Later on, we simply denote $[i]_1$ by $[i]$.

\begin{Lemma}\label{min} (\cite[(4.4)]{BKM}) If  $(\nu,\gamma)\in \Upsilon_\beta$,  $\beta\in\mathcal{R}^{+}$, then  \begin{itemize} \item[(1)]$x_{\gamma}^{+}x_{\nu }^{+}-v^{(\nu \mid\gamma)}x_{\nu }^{+}x_{\gamma}^{+}=[p_{\nu,\gamma}+1]x_{\beta}^{+}$,
\item[(2)]$x_{\nu }^{-}x_{\gamma}^{-}-v^{-(\nu \mid\gamma)}x_{\gamma}^{-}x_{\nu }^{-}=[p_{\nu,\gamma}+1]x_{\beta}^{-}$.
\end{itemize}
\end{Lemma}
\begin{proof}  (1) follows from \cite[(4.4)]{BKM} and the positive embedding in \cite[p 1365]{BKM}.  Acting  $\bar{\tau}$  on (1) and using  \eqref{switch} yield  (2). Finally, we remark that the current $v$ is $q^{-1}$ in \cite{BKM}.
\end{proof}

So far, results in this section are available for any reduced expression of $w_0$. From here to the end of this section, we choose a reduced expression of $w_0$ as follows. If $\mathfrak g\in\{ \mathfrak {sp}_{2n}, \mathfrak{so}_{2n+1}\}$,
\begin{equation}
w_0=\vec{\prod}_{i=1}^n s_{i, n} s_n s_{n, i} \end{equation}
 where $s_{i, j}=s_i s_{i+1, j}$ if $i<j$ and $s_{i, i}=1$ and $s_{i, j}=s_{i, j+1} s_{j}$ if $i>j$.
 If  $\mathfrak g=\mathfrak{so}_{2n}$,
  \begin{equation} w_0=\vec{\prod}_{i=1}^{n-1} s_{i, n} s_n s_{n-1, i}, \text{ (resp., $ \vec{\prod}_{i=1}^{n-2} s_{i, n} s_n s_{n-1, i} \cdot  s_n s_{n-1}$) if $ 2\mid n$ (resp., $2\nmid n$).}\end{equation}
   Then the corresponding convex orders on $\mathcal R^+$ are given as follows:

\begin{equation} \label{ord}
  \begin{array}{ll}
    \epsilon_i-\epsilon_j<\epsilon_i-\epsilon_{j+1}<\epsilon_i+\epsilon_k<\epsilon_i+\epsilon_{k-1}
<\epsilon_{i+1}-\epsilon_{i+2}, & \hbox{if $\mathfrak g=\mathfrak{so}_{2n}$,} \\
    \epsilon_i-\epsilon_j<\epsilon_i-\epsilon_{j+1}<\epsilon_i <\epsilon_i+\epsilon_k<\epsilon_i+\epsilon_{k-1}
<\epsilon_{i+1}-\epsilon_{i+2}, & \hbox{if $\mathfrak g=\mathfrak{so}_{2n+1}$,} \\
  \epsilon_i-\epsilon_j<\epsilon_i-\epsilon_{j+1}<2\epsilon_i <\epsilon_i+\epsilon_k<\epsilon_i+\epsilon_{k-1}
<\epsilon_{i+1}-\epsilon_{i+2} , & \hbox{if  $\mathfrak g=\mathfrak{sp}_{2n}$,}
  \end{array}
 \end{equation}
for all admissible positive integers $i, j, k$ such that $i<j$ if $\epsilon_i+\epsilon_j$ appears in \eqref{ord}.   In this case,
 we have root elements  $\{x_{\beta_j}^{\pm}\mid 1\le j\le \ell(w_0)\}$  with respect to braid group generators $T_{i, +1}''$.

 From here to the end of this paper, we keep using these  $\beta_j$'s and $x_{\beta_j}^{\pm}$'s.  Unless otherwise specified,  the convex order $<$  is always the one in \eqref{ord}.

\begin{Cor}\label{equa}Suppose $i<j$ and $1\le l\le n-1$. Then
\begin{enumerate} \item  [(1)] $\Upsilon_{2\epsilon_l}=\{(\epsilon_l+\epsilon_n,  \epsilon_l-\epsilon_{n})\} $    if $\mfg=\mathfrak{sp}_{2n}$,
\item [(2)]  $\Upsilon_{\epsilon_l}=  \{(\epsilon_k,\epsilon_l-\epsilon_k)\mid l<k\leq n\}$  if $\mfg=\mathfrak{so}_{2n+1}$,
    \item [(3)] $ \Upsilon_{\epsilon_i-\epsilon_j}= \{(\epsilon_m-\epsilon_j,\epsilon_i-\epsilon_m)\mid    i<m<j\}$,
\item [(4)] $\Upsilon_{\epsilon_i+\epsilon_{j}}= \{(\epsilon_{j}-\epsilon_{j+1},\epsilon_i+\epsilon_{j+1})\}\cup \{(\epsilon_l+\epsilon_{j}, \epsilon_i-\epsilon_l)\mid  i<l<j\}$,
\item  [(5)]  $\Upsilon_{\epsilon_l+\epsilon_n}= \{(\epsilon_k+\epsilon_n,\epsilon_l-\epsilon_k)\mid 1\leq l<k<n\}\cup B$, where $B= \{(\epsilon_n,\epsilon_l)\} $ (resp., $\{(2\epsilon_n, \epsilon_l-\epsilon_n)\}$, $\emptyset$) if $\mfg=\mathfrak{so}_{2n+1}$ (resp.,  $\mathfrak{sp}_{2n}$, $\mathfrak{so}_{2n}$).
        \end{enumerate}\end{Cor}

\begin{proof} (1)-(5) can be verified  directly via  \eqref{ord}.\end{proof}

\begin{Prop}\label{commut1}   For any $\mathbf r=(r_i)\in \mathbb N^{\ell(w_0)}$, let
  $c_{\mathbf r}^{\pm}$ be   given in Lemma~ \ref{commut}. Then
$c_{\mathbf r}^\pm \in \mathcal A_{|\mathbf r|-1}$, where $|\mathbf r|=\sum_i r_i$.
\end{Prop}
 Since we verify Proposition~\ref{commut1} case by case and could  not find a conceptual   proof, we will give details in Appendix A.
We also give a proof of Proposition~\ref{roo}  in Appendix B. It involves tedious computation.

\begin{Prop}\label{roo} Suppose  $ \beta\in \mathcal R^+$. We have
\begin{itemize}\item[(1)]$\Delta(x_{\beta}^{-})=x_{\beta}^{-}\otimes k_{-\beta}+1\otimes x_{\beta}^{-} + \sum_{K,H} h_{K,H}x_K^{-}\otimes k_{-wt (K)}x_H^{-}$,
\item[(2)]$\overline{\Delta}(x_{\beta}^{-})=x_{\beta}^{-}\otimes k_{\beta}+1\otimes x_{\beta}^{-} + \sum_{K,H} g_{K,H}x_K^{-}\otimes k_{wt (K)}x_H^{-}$,
\end{itemize}
 where $K, H$ range over non-empty sequences of  positive roots such that $wt(K)+wt(H)=\beta$, and
  $h_{K,H}, g_{K,H}\in \mathcal A_{max \{\ell(K), \ell(H)\}}$.
\end{Prop}

Throughout this paper, $V$ is always the  natural $\U_v(\mathfrak g)$-module and $\{v_i\mid i\in\underline{ \mathrm N}\}$ is always a basis of $V$ such that $wt(v_i)=-wt(v_{i'})=\epsilon_i$, $1\leq i\leq n$ and $wt(v_{n+1})=0$ if  $\mfg=\mathfrak{so}_{2n+1}$, where $\underline{ \mathrm N}=\{1,2,\ldots,\mathrm N\}$. The following result is well-known.

\begin{Lemma}\label{natr1}
 As  homomorphisms in $\End(V)$, we have
\begin{itemize}\item[(1)]  $x_i^{+}=E_{i,i+1}-E_{(i+1)', i'}$, and $x_i^{-}= E_{i+1,i}-E_{i', (i+1)'} $ if $i\neq n$, where $E_{i,j}$'s are matrix units with respect to $\{v_i\mid i\in\underline{ \mathrm N} \}$,
\item[(2)]  $k_\mu=\delta_{\mfg, \mathfrak{so}_{2n+1}} E_{n+1, n+1}+\sum_{1\le j\le n} (v^{(\mu\mid \epsilon_j)} E_{j,j}+v^{-(\mu\mid \epsilon_j)} E_{j',j'})$ for all admissible $\mu$,
\item[(3)] $x_n^+=E_{n, n+1}-v^{-1} E_{n+1, n'}$,  $x_n^-=[2]E_{n+1, n}-v^{-1}[2] E_{n', n+1}$, if  $\mfg=\mathfrak{so}_{2n+1}$,
\item[(4)] $x_n^+=E_{n, n'}$, $x_n^-=E_{n', n}$,  if   $\mfg=\mathfrak{sp}_{2n}$,
\item[(5)] $x_n^+=E_{n-1, n'}-E_{n, (n-1)'}$, $x_n^-=E_{n', n-1}-E_{(n-1)', n}$, if  $\mfg=\mathfrak{so}_{2n}$.
\end{itemize}
\end{Lemma}

We are going to give formulae on $x_\beta^{\pm}$ and $\tau(x_\beta^{\pm})$ as endomorphisms of $V$    in   Lemmas~\ref{natural0}--\ref{natural3}.
 The idea is to use  minimal pairs  in Corollary~\ref{equa} and  Lemma~\ref{min} so as to use the corresponding formulae for simple roots
 in  Lemma~\ref{natr1}. The computations are straightforward, hence we omit  details.

\begin{Lemma}\label{natural0}  For all admissible $i<j$,    \begin{itemize}\item[(1)]
$x_{\epsilon_i-\epsilon_j}^{+}=E_{i,j}-q^{i-j+1}E_{j',i'}$, $\tau (x_{\epsilon_i-\epsilon_j}^{+})=E_{j,i}-q^{i-j+1}E_{i',j'}$, \item[(2)]$ x_{\epsilon_i-\epsilon_j}^{-}=E_{j,i}-q^{j-i-1}E_{i',j'}$,  $\tau(x_{\epsilon_i-\epsilon_j}^{-})=E_{i,j}-q^{j-i-1}E_{j',i'}$.
\end{itemize}
\end{Lemma}

\begin{Lemma}\label{natural1}
Suppose  $\mfg=\mathfrak{so}_{2n+1}$. For all admissible $i<j$ and admissible $k$, we have  \begin{itemize}\item[(1)]
$x_{\epsilon_k}^{+}=-q^{k-n-\frac{1}{2}}E_{n+1,k'}+E_{k,n+1}$,  $\tau(x _{\epsilon_k}^{+})=-q^{k-n+\frac{1}{2}}[2]E_{k',n+1}+[2]E_{n+1,k}$,
\item [(2)]  $x_{\epsilon_k}^{-}=-q^{n-k+\frac{1}{2}}[2]E_{k',n+1}+[2]E_{n+1,k}$,  $\tau (x_{\epsilon_k}^{-})=-q^{n-k-\frac{1}{2}}E_{n+1,k'}+E_{k,n+1}$,
\item [(3)] $x_{\epsilon_i+\epsilon_j}^{+}=(-1)^{n-j-1}q^{-\frac{1}{2}}[2]^{-1}E_{i,j'}+(-1)^{n-j}q^{i+j-2n
    -\frac{1}{2}}[2]^{-1}E_{j,i'} $,
\item[(4)] $ \tau (x_{\epsilon_i+\epsilon_j}^{+})=(-1)^{n-j-1}[2]q^{\frac{1}{2}}E_{j',i}+(-1)^{n-j}[2]
    q^{i+j-2n+\frac{1}{2}}E_{i',j}$,
\item [(5)] $x_{\epsilon_i+\epsilon_j}^{-}=(-1)^{n-j-1}[2]q^{\frac{1}{2}}E_{j',i}
    +(-1)^{n-j}[2]q^{2n-i-j+\frac{1}{2}}E_{i',j}$,
\item[(6)] $ \tau (x_{\epsilon_i+\epsilon_j}^{-})=
    (-1)^{n-j-1}q^{-\frac{1}{2}}[2]^{-1}E_{i,j'}+(-1)^{n-j}q^{2n-i-j-\frac{1}{2}}[2]^{-1}E_{j,i'}$.\end{itemize}
\end{Lemma}
\begin{Lemma}\label{natural2}
 Suppose  $\mfg=\mathfrak{sp}_{2n}$. For all admissible $i<j$ and admissible $k$, we have
\begin{itemize}\item [(1)] $x_{2\epsilon_k}^{+}=[2]^{-1}(q^{k-n-1}+q^{k-n+1})E_{k,k'}$, $\tau (x_{2\epsilon_k}^{+})=[2]^{-1}(q^{k-n-1}+q^{k-n+1})E_{k',k}$,
\item [(2)]  $x_{2\epsilon_k}^{-}=[2]^{-1}(q^{n-k+1}+q^{n-k-1})E_{k',k}$,
    $\tau (x_{2\epsilon_k}^{-})=[2]^{-1}(q^{n-k+1}+q^{n-k-1})E_{k,k'}$,
\item [(3)] $x_{\epsilon_i+\epsilon_j}^{+}=(-1)^{n-j}(E_{i,j'}+q^{i-j'}E_{j,i'})$,
$\tau (x_{\epsilon_i+\epsilon_j}^{+})=(-1)^{n-j}(E_{j',i}+
q^{i-j'}E_{i',j})$,
\item [(4)] $x_{\epsilon_i+\epsilon_j}^{-}=(-1)^{n-j}(E_{j',i}+q^{j'-i}E_{i',j})$,
 $\tau (x_{\epsilon_i+\epsilon_j}^{-})=(-1)^{n-j}(E_{i,j'}+q^{j'-i}E_{j,i'})$.\end{itemize}\end{Lemma}

\begin{Lemma}\label{natural3}
Suppose  $\mfg=\mathfrak{so}_{2n}$. For all admissible $i<j$, we have  \begin{itemize} \item [(1)]
$x_{\epsilon_i+\epsilon_j}^{+}=(-1)^{n-j}E_{i,j'}+(-1)^{n-j-1}q^{i-j'+2}E_{j,i'}$, \item [(2)]
$\tau (x_{\epsilon_i+\epsilon_j}^{+})=(-1)^{n-j}E_{j',i}+(-1)^{n-j+1}q^{i-j'+2}E_{i',j}$, \item [(3)]
$x_{\epsilon_i+\epsilon_j}^{-}=(-1)^{n-j}E_{j',i}+(-1)^{n-j+1}q^{j'-i-2}E_{i',j}$, \item[(4)]
$\tau(x_{\epsilon_i+\epsilon_j}^{-})=(-1)^{n-j}E_{i,j'}+(-1)^{n-j-1}q^{j'-i-2}E_{j,i'}$.\end{itemize}
\end{Lemma}

 The  quasi-R-matrix $\Theta$ of $\U_v(\mathfrak g)$  \cite[\S 4.1]{Lu} and its inverse
$\bar \Theta$ are
\begin{equation}\label{ttheta11}\Theta=\vec{\prod}_{j=\ell(w_0)}^1 \text{exp}_{v_{\beta_j}^{-1}}[(1-v_{\beta_j}^{2})x_{\beta_j}^{-}\otimes x_{\beta_j}^{+}], \ \  \bar\Theta=\vec{\prod}_{i=1}^{\ell(w_0)} \text{exp}_{v_{\beta_i}} [(1-v_{\beta_i}^{-2}) x_{\beta_i}^-\otimes x_{\beta_i}^+], \end{equation}
where  $v_{\beta_j}=v^{d_{\beta_j}}$ and $\text{exp}_{v_{\beta_j}^{\pm1}}(x)=\sum_{k=0}^\infty v_{\beta_j}^{\pm \frac{1}{2}k(k+1)} \frac {x^k}{[k]^!_{d_{\beta_j}}}$.  Thanks to \cite[Theorem 4.1.2(a)]{Lu},  $\Theta$ is uniquely determined by \eqref{qrm} as follows:
\begin{equation}\label{qrm} \Theta_0=1\otimes 1, \text{ and
  $\Delta(u) \Theta=\Theta \bar\Delta(u)$}\end{equation}   for all $u\in \U_v(\mathfrak g)$.
Write
$$\Theta=\sum_{\beta\in \mathcal{Q}^+}\Theta_\beta\ \ \text{  and  } \ \ \bar{\Theta}=\sum_{\beta\in \mathcal{Q}^+}\bar{\Theta}_\beta,$$ where
both $\Theta_\beta$ and  $\bar\Theta_\beta$ are in $\U^-_v(\mathfrak g)_{-\beta} \otimes \U^+_v(\mathfrak g)_\beta$ and $ \U^\pm_v(\mathfrak g)_{\pm\beta}$ is the subspace of $\U^\pm_v(\mathfrak g)$ spanned by $\{x^\pm_{I}|wt(I)=\beta\}$.
Thanks to \cite[7.1(2)--(3)]{J},
$$(\iota\otimes \iota)(\Theta_\beta)=\Theta_\beta, \ (\dag\otimes \dag)(\Theta_\beta)=P( \Theta_\beta)$$  for any $\beta\in\mathcal{Q}^+$,
where $P$ is the tensor flip and, $\iota$ (resp., $\dag$) is the anti-automorphism (resp., automorphism) of $\U_v(\mathfrak g)$ fixing $x_i^{\pm}$ (resp., switching $x_i^+$ and $x^-_i$) and sending $k_\mu$ to $k_{-\mu}$. Since $\tau=\dag\circ \iota$,  $$ (\tau\otimes \tau)(\Theta_\beta)=P ( \Theta_\beta).$$
Noting that $\Theta \bar \Theta =\bar \Theta  \Theta =1\otimes 1$, we have $ (\tau\otimes \tau)(\bar\Theta_\beta)=P ( \bar\Theta_\beta)$. So, \begin{equation}\label{reversed}\bar{\Theta}_\beta=P\circ (\tau \otimes \tau) (\bar{\Theta}_\beta).\end{equation}

\begin{Prop} \label {matrix} Let  $M$ be a  $\U_v(\mathfrak g)$-module.  For any $0\neq m\in M$,
$${\bar\Theta} (m\otimes v_j)=\begin{cases} [1\otimes 1+z_q\sum_{i<j}  \tau(x_{\epsilon_i-\epsilon_j}^+)\otimes \tau(x_{\epsilon_i-\epsilon_j}^-)]( m\otimes v_j), &\text{if $1\le j\le n$,}\\
[1\otimes 1+z_{q^{1/2}}\sum_{1\le i\le n}  \tau(x_{\epsilon_i}^+)\otimes \tau(x_{\epsilon_i}^-)] (m\otimes v_{n+1}),
&\text{if $(\mathfrak g, j)=(\mathfrak{so}_{2n+1}, n+1)$,}\\
[1\otimes 1+c_j] ( m\otimes v_{j}), &\text{if $n'\le j\le 1'$,} \\
\end{cases}
$$
  where  $$
\begin{aligned} c_j  =& z_q  \left \{\sum_{i<j'} \tau (x_{\epsilon_i+\epsilon_{j'}}^+)\otimes  \tau (x_{\epsilon_i+\epsilon_{j'}}^-) +  \sum_{i>j'} (x_{\epsilon_{j'}+\epsilon_i}^-\otimes  x_{\epsilon_{j'}+\epsilon_i}^+ +
    x_{\epsilon_{j'}-\epsilon_i}^-\otimes x_{\epsilon_{j'}-\epsilon_i}^+)\right\} \\ &+z_q^2\sum_{i>j'}   x_{\epsilon_{j'}-\epsilon_i}^-x_{\epsilon_{j'}+\epsilon_i}^-\otimes x_{\epsilon_{j'}-\epsilon_i}^+x_{\epsilon_{j'}+\epsilon_i}^+
  + \delta_{\mathfrak g, \mathfrak {sp}_{2n} }z_{q^2}x_{2\epsilon_{j'}}^-\otimes x_{2\epsilon_{j'}}^+ \\ &+ \delta_{\mathfrak g, \mathfrak {so}_{2n+1} }[q^{\frac{1}{2}} z_{q^{1/2}}^2 {[2]}^{-1}  (x_{\epsilon_{j'}}^-)^2\otimes (x_{\epsilon_{j'}}^+)^2  + z_{q^{1/2}}  x_{\epsilon_{j'}}^-\otimes x_{\epsilon_{j'}}^+].
  \end{aligned}$$
In any case,  the formulae on $\Bar\Theta (v_{j'}\otimes m)$ are obtained by using $v_{j'}\otimes m$  instead of  $m\otimes v_j$ in the previous formulae.
\end{Prop}

\begin{proof}
If $1\le j\le n$, then $wt(v_j)=\epsilon_j$. So
 $\bar\Theta_\beta$ acts on $m\otimes v_j$ non-trivially only if either $\beta =0$ or  $\beta =\epsilon_i-\epsilon_j$ for some $i$, $1\le i<j$.
Thanks to \eqref{ttheta11}, $\bar\Theta_0=1\otimes 1$ and  \begin{equation}\label{pro1}{\bar\Theta}_{\epsilon_k-\epsilon_l}-z_qx_{\epsilon_k-\epsilon_l}^-\otimes x_{\epsilon_k-\epsilon_l}^+ =\sum_{r\ge 2}\sum_I z_q^r x_{I}^-\otimes x_{I}^+, \text{for all  $1\leq k<l\leq n$},\end{equation} where $I$  ranges over all sequences of positive roots  $(\epsilon_{i_0}-\epsilon_{i_1}, \ldots,  \epsilon_{i_{r-1}}-\epsilon_{i_r})$ such that $ i_0=k$ and $i_r=l$.
In particular,  we have \eqref{pro1} for ${\bar\Theta}_{\epsilon_i-\epsilon_j}$.
 Obviously, $\tau(x_{\epsilon_{i_{0}}-\epsilon_{i_1}}^-)v_j=0$. Otherwise $wt(\tau(x_{\epsilon_{i_{0}}-\epsilon_{i_1}}^-)v_j)=\epsilon_{i_0}-\epsilon_{i_1}+\epsilon_j$, and  $i_0<i_1<j$, a contradiction.
So, $P\circ (\tau\otimes \tau) (\text{RHS of \eqref{pro1}})$  acts on $m\otimes v_j$ as zero.
By \eqref{reversed}, we have the required formula on $\bar\Theta(m\otimes v_j)$ under the assumption
$1\le j\le n$.

  Now, we assume $\mathfrak g=\mathfrak{so}_{2n+1}$ and $ j=n+1$. So,  $wt(v_{j})=0$. In this case,
   $\bar\Theta_\beta $ acts on $m\otimes v_{n+1}$ non-trivially only if
   either $\beta =0$ or $\beta =\epsilon_l$, $1\le l\le n$.
   Thanks to \eqref{ttheta11}, \begin{equation}\label{pro2}\bar{\Theta}_{\epsilon_l}-z_{q^{1/2}}x_{\epsilon_l}^-\otimes x_{\epsilon_l}^+ =\sum_{s>l}\sum_{wt(J)=\epsilon_s}a_{s,J}x_{\epsilon_l-\epsilon_s}^-x_{J}^-\otimes x_{\epsilon_l-\epsilon_s}^+x_{J}^+, \text{ for some $a_{s,J}\in\mathbb{F}$}. \end{equation}
Note that  $\tau(x_{\epsilon_l-\epsilon_s}^-)v_{n+1}=0$.  So,  $P\circ (\tau\otimes \tau) (\text{RHS of \eqref{pro2}})$  acts on $m\otimes v_{n+1}$ as zero. Using  \eqref{reversed} again, one immediately has
the required formula on
 $\bar\Theta(m\otimes v_{n+1})$.

Suppose $n'\leq j\leq 1'$.  Since $wt(v_j)=-\epsilon_{j'}$,  $\bar\Theta_{\beta }$ acts  on  $m\otimes v_j$ non-trivially  only if $\beta$ and $\mfg$ satisfy one of conditions as follows:   \begin{itemize} \item  $\beta \in \{0, \epsilon_d+\epsilon_{j'}, \epsilon_{j'}-\epsilon_t\mid 1\leq d\leq n, j'< t\leq n\}$,  and $\mfg\neq \mathfrak{so}_{2n+1}$, \item
 $\beta \in \{0, \epsilon_d+\epsilon_{j'}, \epsilon_{j'}-\epsilon_t\mid 1\leq d\leq n, j'< t\leq n\}
 \cup \{\epsilon_{j'}\}$, and  $\mfg=\mathfrak{so}_{2n+1}$.\end{itemize}
We have \eqref{pro1} for ${\bar\Theta}_{\epsilon_{j'}-\epsilon_{t}}$.
Since $x_{\epsilon_{i_{r-1}}-\epsilon_{i_r}}^+ v_j=0$,
the RHS of \eqref{pro1} acts on $m\otimes v_j$ as zero. So, $$\bar\Theta_{\epsilon_{j'}-\epsilon_t}(m\otimes v_j)=z_qx_{\epsilon_{j'}-\epsilon_t}^-\otimes x_{\epsilon_{j'}-\epsilon_t}^+(m\otimes v_j).$$ Thanks to \eqref{ttheta11}, we have
\begin{equation}\label{pro3}\begin{aligned}  {\bar\Theta}_{2\epsilon_c}&=  z_q^2 \sum_{k>c}  x_{\epsilon_c-\epsilon_k}^-x_{\epsilon_c+\epsilon_k}^- \otimes  x_{\epsilon_c-\epsilon_k}^+x_{\epsilon_c+\epsilon_k}^+
+ \delta_{\mathfrak g, \mathfrak{so}_{2n+1} } q^{\frac{1}{2}} {[2]}^{-1} z_{q^{1/2}}^2 (x_{\epsilon_c}^-)^2\otimes (x_{\epsilon_c}^+)^2 \\
& +\sum_{c<l<k} \sum_{wt(J)=\epsilon_l-\epsilon_k}a_{l,k,J}x_{\epsilon_c-\epsilon_l}^-x_{\epsilon_c+\epsilon_k}^-x_{J}^- \otimes x_{\epsilon_c-\epsilon_l}^+x_{\epsilon_c+\epsilon_k}^+x_{J}^++\delta_{\mathfrak g, \mathfrak{sp}_{2n} } z_{q^2}  x_{2\epsilon_c}^-\otimes x_{2\epsilon_c}^+\\&+\sum_{c<l<k}\sum_{wt(J)=\epsilon_l+\epsilon_k}b_{l,k,J} x_{\epsilon_c-\epsilon_l}^-x_{\epsilon_c-\epsilon_k}^- x_{J}^- \otimes x_{\epsilon_c-\epsilon_l}^+x_{\epsilon_c-\epsilon_k}^+x_{J}^+,
 \\ \end{aligned}\end{equation}  for any  $c, 1\leq c\leq n$, where $a_{l,k,J}$'s and $b_{l,k,J}$'s are in $\mathbb F$.
 Using  \eqref{ttheta11} again, we have
  \begin{equation}\begin{aligned}\label{pro4}  &{\bar\Theta}_{\epsilon_r+\epsilon_p}  = z_q x_{\epsilon_{r}+\epsilon_p}^-\otimes x_{\epsilon_{r}+\epsilon_{p}}^+  +\sum_{k>p}\sum_{wt(I)=\epsilon_p-\epsilon_k} a_{k,I}x_{\epsilon_r+\epsilon_k}^-x_{I}^-
    \otimes  x_{\epsilon_r+\epsilon_k}^+x_{I}^+
     \\ &   +
  \sum_{k>r} \sum_{wt(I)=\epsilon_k+\epsilon_p}b_{k,I} x_{\epsilon_r-\epsilon_k}^-x_{I}^- \otimes  x_{\epsilon_r-\epsilon_k}^+x_{I}^+ + \delta_{\mathfrak g, \mathfrak {so}_{2n+1} }  \sum_{wt(I)=\epsilon_p}c_{I}x_{\epsilon_r}^-x_{I}^-\otimes x_{\epsilon_r}^+ x_{I}^+ , \\ \end{aligned}  \end{equation} for all admissible $1\leq r<p\leq n$, where $a_{k,I}$'s, $b_{k,I}$'s and $c_I$'s are in $\mathbb F$. Note that \eqref{pro3}-\eqref{pro4} give formulae on $\bar \Theta_{\epsilon_d+\epsilon_{j'}}$ in any case.
  Acting $P\circ (\tau\otimes \tau)$ on both sides of \eqref{pro2}-\eqref{pro4} and
using  \eqref{reversed} yields another  two formulae on $\bar\Theta_{\epsilon_{j'}}$ and  $\bar\Theta_{\epsilon_d+\epsilon_{j'}}$. The orders in the ordered products of these formulae are reversed with each other.
This makes us to choose a suitable forms of $\bar\Theta_{\epsilon_d+\epsilon_{j'}}$ and $\bar\Theta_{\epsilon_{j'}}$ so as to prove that the terms with respect to $I$ and $J$  act on $m\otimes v_j$ as zero. Therefore, the formula for $n'\leq j\leq 1'$ can be checked directly.
Finally, one can check the formulae on $\bar\Theta (v_{j}\otimes m)$  similarly.
\end{proof}

Given    two finite-dimensional weight modules $M$ and $N$,
Lusztig  defines \begin{equation} \label{rinv} R_{M, N} = \Theta \circ \pi  \circ P: M\otimes N\rightarrow N\otimes M,\end{equation} where
$\pi : M\otimes N\rightarrow M\otimes N$ such that
$$\pi  (m\otimes n)=v^{-(\lambda|\mu)} m\otimes n,$$  for all $(m, n)\in M_\lambda\times N_\mu$.
He proves that $R_{M, N}$  is a $\U_v(\mathfrak g)$-isomorphism  and \begin{equation}\label{invr} R_{M, N}^{-1} = P\circ\pi ^{-1}\circ\bar\Theta: N\otimes M\rightarrow M\otimes N.\end{equation}
The current $R_{M, N}$ is Lusztig's  ${}_f\mathcal R_{N, M}$ in \cite[Theorem 32.1.5]{Lu}, where $f$ is the function such that $f(\lambda, \mu)=-(\lambda|\mu)$~\cite[\S 31.1.3]{Lu}. When one of $M, N$ is a finite-dimensional weight module,
 we still have the $\U_v(\mathfrak g)$-homomorphism $\Theta \circ \pi  \circ P$ such that $\pi  : M\otimes N\rightarrow M\otimes N$
with
\begin{equation}\label{pif}\pi  (m\otimes n)=\begin{cases} k_{-\lambda}m \otimes n, & \text{if $n\in N_\lambda$,}\\
                                    m\otimes k_{-\lambda} n, &  \text{if $m\in M_\lambda$. }\\
                                    \end{cases}
                                   \end{equation}
Moreover, by arguments similar to those in \cite[\S5]{BSW}, $\Theta \circ \pi  \circ P$ is also an isomorphism and  \eqref{invr} is still available in this case.  In fact, the proof of this result  reduces to the case that both $M$ and $N$ are finite-dimensional weight modules by using the fact that  the intersection of  annihilators of all finite-dimensional
weight modules is zero (see e.g. \cite[Proposition~5.11]{J}).

For  finite-dimensional  weight modules $M_1, M_2, M_3$,
\begin{equation}\label{Hexagon}R_{M_1, M_2\otimes M_3}=(1\otimes R_{M_1, M_3})\circ (R_{M_1, M_2}\otimes 1),  \ R_{M_1\otimes M_2, M_3}=(R_{M_1, M_3}\otimes 1) \circ (1\otimes R_{M_3, M_2}),\end{equation}
 which is called  the  Hexagon property in \cite[Proposition 32.2.2]{Lu}. So,
 \begin{equation}\label{bian}(R_{M_2, M_3}\otimes1)\circ (1\otimes R_{M_1, M_3})\circ(R_{M_1, M_2}\otimes1)=(1\otimes R_{M_1, M_2})\circ (R_{M_1, M_3}\otimes 1)\circ(1\otimes R_{M_2, M_3}).\end{equation}
Similarly,
 \eqref{Hexagon}-\eqref{bian} are  still  available if two of $M_1, M_2, M_3$ are finite-dimensional weight modules.

By  Proposition~\ref{matrix} and \eqref{pif}, one can check that:
\begin{equation} \label{bar}\begin{aligned} & \pi ^{-1}  \circ\bar\Theta\mid_{ V^{\otimes 2} }=
\sum_{i\neq i'}(q E_{i,i}\otimes E_{i,i}+q^{-1} E_{i,i}\otimes E_{i',i'})  +\sum_{i\neq j, j'} E_{j,j}\otimes E_{i,i}\\ & +
 z_q\sum_{i<j} (E_{j,i}\otimes E_{i,j} -  q^{\varrho_j-\varrho_i} \varsigma_i \varsigma_j E_{i',j'} \otimes E_{i,j}) +\delta_{\mathfrak g, \mathfrak {so}_{2n+1}} E_{n+1, n+1} \otimes E_{n+1, n+1},\cr \end{aligned}
 \end{equation}
where $\varsigma_i=1$ unless $\mathfrak g=\mathfrak{sp}_{2n}$ and $n+1\le  i\le 2n$. In the later case, $\varsigma_i=-1$. The following result follows from \eqref{bar} directly.

\begin{Lemma}\label{act123} Let  $E: V^{\otimes 2}\rightarrow V^{\otimes 2} $ be  such that $E(v_k\otimes v_l)=\delta_{k,l'} \sum_{i=1}^\mathrm N q^{\varrho_{i'}-\varrho_k} \varsigma_{i'} \varsigma_k v_{i}\otimes v_{i'}$ for all admissible $k$ and $l$.
\begin{itemize} \item [(1)] If either $\mathfrak g\neq \mathfrak{so}_{2n+1}$ or  $\mathfrak g= \mathfrak{so}_{2n+1}$ and $(k, l)\neq (n+1, n+1)$, then
$${R_{V, V}^{-1}} (v_k\otimes v_l)=\begin{cases} q v_k\otimes v_k, & \text{if $k=l$,}\\
v_l\otimes v_k,   & \text{if $k>l$ and $k\neq l'$,}\\
q^{-1} v_l\otimes v_k-z_q \sum_{i>k} q^{\varrho_i-\varrho_k} \varsigma_i\varsigma_k v_{i'}\otimes v_i,   & \text{if $k>l$ and $k= l'$,}\\
v_l\otimes v_k+z_q v_k\otimes v_l, & \text{if $k<l$ and $k\neq l'$,}\\
q^{-1} v_l\otimes v_k +z_q (v_k\otimes v_l -\sum_{i>k} q^{\varrho_i-\varrho_k} \varsigma_i\varsigma_k v_{i'}\otimes v_i), & \text{if $k<l$ and $k= l'$.}\\
\end{cases} $$
\item [(2)] If $\mathfrak g=\mathfrak{so}_{2n+1}$, then
${R_{V, V}^{-1}} (v_{n+1}\otimes v_{n+1}) = v_{n+1}\otimes v_{n+1}-z_q \sum_{i>n+1} q^{\varrho_i}v_{i'}\otimes v_i$.\end{itemize}
In any case, ${R_{V, V}^{-1}}- {R_{V, V}}=z_q (1-E)$.
    \end{Lemma}

In fact, the left action of ${R_{V, V}^{-1}} $  on $V^{\otimes 2}$ is the same as the right action of $\check{R}$ on $V^{\otimes 2}$ in \cite[Lemma~2.4]{RS2}.  This makes us to    use results in \cite{RS2}, freely.

\begin{Lemma}\label{AU}  Let
 $\underline\alpha: V^{\otimes 2}\rightarrow \mathbb F$ and $\underline\beta: \mathbb F\rightarrow V^{\otimes 2}$ be two linear maps  such that
$\underline\alpha (v_k\otimes v_l)=\delta_{k,l'} q^{-\varrho_k} \varsigma_k$ for all admissible $k$ and $l$, and $\underline\beta (1)=\sum_{i=1}^\mathrm N q^{\varrho_{i'}} \varsigma_{i'} v_i\otimes v_{i'}$. Then both $\underline\alpha$ and $\underline\beta$ are $\U_v(\mathfrak g)$-homomorphisms, where $\mathbb F$ is considered as  the trivial  $\U_v(\mathfrak g)$-module given by the counit.
\end{Lemma}
\begin{proof} We have explained that a weight $\U_v(\mfg)$-module is a weight $\bar \U_v(\mathfrak g)$-module and vice versa.
Since both $V^{\otimes 2}$ and $\mathbb F$ are weight $\U_v(\mfg)$-module, we can use corresponding results in \cite{RS2}.

Thanks to  \cite[Coro.~2.3]{RS2}, $\underline \beta$ is a $\U_v(\mathfrak g)$-homomorphism and   $\text{Im}\underline \beta\cong \mathbb F $ as $\U_v(\mfg)$-modules.  The corresponding isomorphism is denoted by $\underline \beta'$. Let $\jmath$ be the embedding $\text{Im}\underline \beta \hookrightarrow V^{\otimes 2}$. Then $\underline \beta=\jmath\circ \underline \beta'$.

 Recall $E$ in Lemma~\ref{act123}. It is proved in  \cite{RS2} that $E$ is a $\bar\U_v(\mathfrak g)$-homomorphism and hence a  $\U_v(\mathfrak g)$-homomorphism. Further, $E=\underline \beta\circ \underline \alpha=\jmath\circ\underline \beta'\circ \underline \alpha$. So $\underline \beta'\circ \underline \alpha$ is a $\U_v(\mathfrak g)$-homomorphism, forcing  $\underline \alpha$ to be a  $\U_v(\mathfrak g)$-homomorphism, too.
\end{proof}

     In Propositions~\ref{bmw1}-\ref{abmw1}, we assume   $\B$  and $\AB$  are defined over   $\mathbb F$ with  defining parameters $\delta, z$ and  $\omega_0$ such that
      $z=z_q$ and \begin{equation} \label{lamb} \delta=\epsilon_\mfg q^{\mathrm N -\epsilon_\mfg}\end{equation}
      where $\epsilon_\mfg=-1 $ if $\mfg=\mathfrak{sp}_{2n}$ and $1$, otherwise.
      Since $z_q$ is invertible in $\mathbb F$,    $\omega_0$ is uniquely determined by  \eqref{para1}.

\begin{Prop}\label{bmw1}Let $\U_v(\mathfrak g)\text{-mod}$  be the category of all
$\U_v(\mathfrak g)$-modules.
There is a  monoidal functor $\Phi: \B\rightarrow \U_v(\mathfrak g)\text{-mod}$ such
that \begin{itemize} \item [(1)] $\Phi(\ob 0)=\mathbb F$ and   $\Phi(\ob 1)=V$, \item [(2)]
$\Phi(A)=\underline \alpha$, $\Phi(U)=\underline \beta$, $\Phi(T)={R_{V, V}^{-1}}$ and  $\Phi(T^{-1})=R_{V, V}$, \end{itemize}
where $\underline \alpha$ and $\underline \beta$ are given in Lemma~\ref{AU}.
 \end{Prop}

\begin{proof} When $\mathfrak g=\mathfrak{sp}_{2n}$, this result was given in \cite{xyz}. We need to  deal with $\mathfrak{so}_{2n}$ and $\mathfrak{so}_{2n+1}$  so as to study the dualities between cyclotomic Birman-Murakami-Wenzl algebras and  parabolic BGG category $\mathcal O$ in types $B, C$ and $D$ in the future. For this reason, we give a sketch of proof.

Thanks to results on Schur-Weyl duality between $\bar\U_v(\mfg)$ and  Birman-Murakami-Wenzl algebras in \cite{RS2}, we need only to  verify relations given in Theorem~\ref{present1}(1),(5)-(9). In \cite{RS2}, two of us   verified \begin{equation}\label{t2}(Id_V\otimes E)\circ (R_{V,V}^{\mp1}\otimes Id_V)\circ (Id_V\otimes E)= \delta^{\pm1}(Id_V\otimes E) \text{ and $(Id_V\otimes E)^2=\omega_0 (Id_V\otimes E)$}\end{equation}
 as $\bar\U_v(\mathfrak g)$-homomorphism.
So, for any $ 1\leq j\leq \mathrm N $, we have  $$(Id_V\otimes E)\circ (R_{V, V}^{\mp1}\otimes Id_V)\circ (Id_V\otimes E)( v_j\otimes v_1\otimes v_{1'})=\delta^{\pm1}(Id_V\otimes E)( v_j\otimes v_1\otimes v_{1'}).$$ Since  $E=\underline \beta\circ \underline \alpha$ and $\underline \beta$ is injective, we have $$(Id_V\otimes \underline \alpha)\circ(R_{V,V}^{\mp1}\otimes Id_V)\circ (Id_V\otimes \underline \beta)( v_j\otimes 1)=\delta^{\pm1}( v_j\otimes 1), \text{ $ 1\leq j\leq \mathrm N $.}$$
This verified  Theorem~\ref{present1}(1),(5). Similarly, Theorem~\ref{present1}(6)  can be verified via   the second equation in \eqref{t2}. It is not difficult to verify
$$(Id_V\otimes \underline \alpha)(\underline \beta\otimes Id_V) (v_j)=v_j=(\underline \alpha\otimes Id_V) (Id_V \otimes \underline \beta)(v_j), \text{ $ 1\leq j\leq \mathrm N $.}$$  This implies Theorem~\ref{present1}(7).
Thanks to Lemma~\ref{act123},  we can check that \begin{equation}\label{key111} (\underline \alpha\otimes Id_V) \circ (Id_V\otimes R_{V,V}) =(Id_V \otimes \underline \alpha) \circ (R_{V,V}^{-1} \otimes Id_V)\end{equation}
as operators acting  on  $V^{\otimes 3}$.  For example, when  $k=l'=j$ and $k\neq l$, we have
 $$(\underline \alpha\otimes Id_V) \circ (Id_V\otimes R_{V, V}) (v_j\otimes v_k\otimes v_l)=(Id_V \otimes \underline \alpha) \circ (R_{V,V}^{-1} \otimes Id_V) (v_j\otimes v_k\otimes v_l)=q^{1-\varrho_j}\varsigma_j v_k.$$
Now, Theorem~\ref{present1}(8) follows from Theorem~\ref{present1}(7) and \eqref{key111}.
Finally, one can verify Theorem~\ref{present1}(9) in a similar way.
\end{proof}

\begin{Prop}\label{abmw1}Let $ \END(\U_v(\mathfrak g)\text{-mod})$ be the category of endofunctors of  $\U_v(\mathfrak g)\text{-mod}$. Then
there is a strict monoidal functor $\Psi: \AB\rightarrow \END(\U_v(\mathfrak g)\text{-mod})$ such
that  \begin{itemize}\item[(1)] $\Psi(\ob 0)=Id$ and   $\Psi(\ob 1)=-\otimes V$,
\item[(2)]
$\Psi(Y)_M=Id_M\otimes \Phi(Y)$,
 $\Psi(X)_M =\delta R_{M, V}^{-1}R_{V, M}^{-1}  $,   $\Psi(X^{-1})_M =\delta^{-1}R_{ V, M}R_{M, V} $ for all   $Y\in \{A, U, T, T^{-1}\}$ and any  $\U_v(\mathfrak g)$-module $M$, where $\Phi$ is given in Proposition~\ref{bmw1} and $\delta$ is given in \eqref{lamb}.
\end{itemize}  \end{Prop}

 \begin{proof}Thanks to Proposition~\ref{bmw1}, it is enough to  verify Definition~\ref{AK defn}(1)-(3). First of all,  Definition~\ref{AK defn}(1) follows since  $$\Psi(X)_M \Psi(X^{-1})_M=\Psi(X^{-1})_M\Psi(X)_M =Id_M.$$ By  \eqref{Hexagon}, and $R_{M,N}R_{M,N}^{-1}=Id_{N\otimes M}$,  we have $$R_{M\otimes V,V}^{-1}R_{V,M\otimes V}^{-1}=(Id_M\otimes R_{V,V}^{-1})\circ(R_{M,V}^{-1}\otimes Id_V)\circ(R_{V,M}^{-1}\otimes Id_V)\circ(Id_M\otimes R_{V,V}^{-1}).$$ Therefore,
 \begin{equation}\label{commu}\Psi(\text{1}_{\ob 1}\otimes X)_{M} =\Psi(T\circ (X\otimes \text{1}_{\ob 1}) \circ T)_{M},
 \end{equation} proving   Definition~\ref{AK defn}(2).

 Let $L(\mu)$ be the irreducible highest weight $\U_v(\mathfrak g)$-module with highest weight $\mu$.  Then  $V\cong L(\epsilon_1)$, $\mathbb F\cong L(0)$ and $V\otimes V\cong L(0)\oplus L(2\epsilon_1)\oplus L(\epsilon_1+\epsilon_2)$. Note  that $\underline \alpha: V^{\otimes 2}\twoheadrightarrow L(0)$ is the projection. So,
   $$
\begin{aligned} & \Psi (A\circ  (X\otimes \text{1}_{\ob 1})\circ (\text{1}_{\ob 1}\otimes X))_{M} = (\Psi (A)\circ\Psi(X\otimes \text{1}_{\ob 1}) \circ \Psi(\text{1}_{\ob 1}\otimes X))_{M}
\\ = & (\Psi (A)\circ \Psi(X\otimes \text{1}_{\ob 1}) \circ\Psi( T)\circ \Psi(X\otimes \text{1}_{\ob 1}) \circ \Psi(T))_{M}, \text{ by \eqref{commu}} \\
=& (\Psi (A)\circ \Psi( T)\circ \Psi(X\otimes \text{1}_{\ob 1}) \circ \Psi(T) \circ\Psi(X\otimes \text{1}_{\ob 1}))_{M}, \text{ by \eqref{Hexagon}--\eqref{bian}}\\
=& \delta^{-1}(\Psi (A)\circ\Psi(X\otimes \text{1}_{\ob 1}) \circ \Psi(T) \circ \Psi(X\otimes \text{1}_{\ob 1}))_{M}, \text{ by Lemma~\ref{selfcrossing}(1) and Proposition~\ref{bmw1}}\\
=& \delta(Id_M\otimes \underline \alpha)\circ(R_{M,V}^{-1}R_{V,M}^{-1}\otimes Id_V)\circ(Id_M\otimes R_{V,V}^{-1})\circ(R_{M,V}^{-1}R_{V,M}^{-1}\otimes Id_V)_{M\otimes V^{\otimes 2}}
\\
=& \delta(Id_M\otimes \underline \alpha)\circ(R_{V,V}^{-1}\otimes Id_M)\circ(R_{M,V}^{-1}\otimes Id_V)\circ(Id_V\otimes R_{M,V}^{-1})\circ(Id_V\otimes R_{V,M}^{-1})\\
  & \circ(R_{V,M}^{-1}\otimes Id_V)_{M\otimes V^{\otimes 2}}, \text{ by \eqref{bian}}
\\
=&  \delta((Id_M\otimes \underline \alpha)\circ(R_{V,V}^{-1}\otimes Id_M)\circ R_{M,V^{\otimes 2}}^{-1} R_{V^{\otimes 2},M}^{-1})_{M\otimes V^{\otimes 2}}, \text{ by \eqref{Hexagon}}
\\=& ((Id_M\otimes \underline \alpha)\circ R_{M,V^{\otimes 2}}^{-1}R_{V^{\otimes 2},M}^{-1})_{M\otimes V^{\otimes 2}}, \text{ by Lemma~\ref{selfcrossing}(1) and Proposition~\ref{bmw1}}\\
=&(Id_M\otimes \underline \alpha)_{M\otimes V^{\otimes 2}},\text{ since $\phi$ stabilizes  $M\otimes L(\eta)$ and acts on  ${M\otimes L(0)}$ as scalar  $1$,}\\=&\Psi(A)_{M},
\end{aligned} $$ where $\phi=R_{M,V^{\otimes 2}}^{-1} R_{V^{\otimes 2},M}^{-1}$ and $\eta\in \{0, 2\epsilon_1, \epsilon_1+\epsilon_2\}$.
 This proves  the first equation in  Definition~\ref{AK defn}(3).
Finally, one can check the second equation in  Definition~\ref{AK defn}(3) in a similar way. \end{proof}

\section{Connections to  category $\mathcal O$}

We consider $\mathcal O$,  the subcategory of $\U_v(\mathfrak g)$-mod whose objects $M$ satisfy the following conditions:
\begin{itemize}\item [(1)] $M=\oplus_{\mu\in\mathcal P} M_\mu$ and $\dim M_\mu<\infty$,
\item [(2)]
there are finitely many  weights $\lambda_1,\lambda_2,\ldots,\lambda_t\in\mathcal P$ such that  $\mu\in \lambda_i-\mathcal{Q}^{+}$ for some $ i$ if  $\mu$ is a weight of $M$.
\end{itemize}
Obviously, $ \mathcal O$ is closed under tensor product.

Recall that
 $\Phi=\sum_{\beta \in \mathcal{Q}^+}\Phi_\beta $, where $\Phi\in\{\Theta, \bar\Theta\}$ and
$\Phi_\beta \in \U^-_v(\mathfrak g)_{- \beta } \otimes \U^+_v(\mathfrak g)_ \beta $.
Thanks to $\bar\Theta \Theta=\Theta\bar\Theta=1\otimes1$ and \eqref{qrm}, \begin{equation}\label{qrm11}\bar\Theta\Delta(u)=\bar\Delta(u)\bar\Theta\end{equation} for all $u\in \U_v(\mathfrak g)$. By  \eqref{ttheta11},
 \begin{equation}\label{theta123}\bar\Theta=\sum_{\beta \in \mathcal{Q}^+}\sum_{wt(J)= \beta }z_v^{\ell(J)}g_Jx_{J}^{-}\otimes x_{J}^{+}, ~\Theta=\sum_{\beta \in \mathcal{Q}^+}\sum_{wt(J')= \beta }z_v^{\ell(J')}h_{J'}x_{J'}^{-}\otimes x_{J'}^{+},\end{equation} where
 $J$(resp., $J'$) ranges over weakly increasing (resp., decreasing)  sequences of positive roots and $ g_J, h_{J'}\in \mathcal A_0\setminus \mathcal A_1$ (see Definition \ref{codeg}). In particular, $g_{J}=h_{J'}=1$ when $wt(J)=wt(J')=0$.

Let  $\mathbf m$  be the multiplication map of $\U_v(\mfg)$. For any $J$ in \eqref{theta123}, let
  $b_J=z_v^{\ell(J)} g_{J}$.
We keep using  $b_J$'s in  Lemmas~\ref{homm}-\ref{chu} as follows.

\begin{Lemma}\label{homm} Suppose   $M\in \mathcal O$.
\begin{itemize}\item[(1)] For any $m\in M_\lambda$, define  $\sigma_M(m)=v^{-(\lambda\mid \lambda+2\varrho)}  \mathbf m (Id\otimes S)(\bar \Theta) m$, where $S$ is the antipode  of $\U_v(\mathfrak g)$ in \eqref{rell} and $\varrho$ is given in \eqref{posr}. Then $\sigma_M: M\rightarrow M$ is a $\U_v(\mathfrak g)$-homomorphism.
\item [(2)]$\sigma_{M}=v^{-(\lambda\mid \lambda+2\varrho)}Id_M$ if  $M$ is a highest weight $\U_v(\mathfrak g)$-module with the highest weight $\lambda$. \end{itemize}
\end{Lemma}
\begin{proof} Since  $ {\bar\Theta}\Delta(x_i^\pm )=\overline{\Delta}(x_i^\pm ){\bar\Theta}$, we  have
 $\mathbf m (Id\otimes S)({\bar\Theta}\Delta(x_i^\pm))=\mathbf m (Id\otimes S)(\overline{\Delta}(x_i^\pm ){\bar\Theta})$.
   Thanks to \eqref{rell} and \eqref{theta123}, it is routine to check $\mathbf m (Id\otimes S )({\bar\Theta}\Delta(x_i^\pm ))=0$. Using this fact and \eqref{rell1} yields
  \begin{equation}\label{s}\begin{aligned} \sum_{\nu\in \mathcal Q^+}\sum_{wt(J)=\nu}b_Jx_i^+x_J^-S(x_J^+)& =\sum_{\nu\in \mathcal Q^+}\sum_{wt(J)=\nu}b_Jx_J^-S(x_J^+)k_{-2\alpha_i}x_i^+,\\ \sum_{\nu\in \mathcal Q^+}\sum_{wt(J)=\nu}b_Jx_i^-x_J^-S(x_J^+)&=\sum_{\nu\in \mathcal Q^+}\sum_{wt(J)=\nu}b_Jx_J^-S(x_J^+)k_{2\alpha_i}x_i^-.\\
\end{aligned}  \end{equation}
Suppose $m\in M_\lambda$. By \eqref{varrho} and \eqref{s}, we have
$$\begin{aligned} x_{i}^{+}\sigma_{M}(m)&=v^{-(\lambda\mid\lambda+2\varrho)}x_i^+ \mathbf m (Id\otimes S)(\bar \Theta)m=v^{-(\lambda\mid\lambda+2\varrho)}\sum_{\nu\in \mathcal Q^+}\sum_{wt(J)=\nu}b_Jx_J^-S(x_J^+)k_{-2\alpha_i}x_i^+m\\
&=v^{-(\lambda\mid\lambda+2\varrho)}v^{-(2\alpha_i\mid\lambda+\alpha_i)}\mathbf m (Id\otimes S)(\bar \Theta) x_i^+m\\ & =v^{-(\lambda+\alpha_i\mid\lambda+\alpha_i+2\varrho)} \mathbf m (Id\otimes S)(\bar \Theta) x_{i}^{+}m=\sigma_{M}(x_{i}^{+}m).\end{aligned}$$
Similarly, one can verify  $x_{i}^{-}\sigma_{M}(m)=\sigma_{M}(x_{i}^{-}m)$.  Finally,  $k_{\eta}\sigma_{M}(m)=\sigma_{M}(k_\eta m)$ for any $\eta\in \mathcal P$  since $k_{\eta}\mathbf m (Id\otimes S)(\bar \Theta)   =\mathbf m (Id\otimes S)(\bar \Theta)k_{\eta}$. This completes the proof of (1).
Under the assumption in (2),  $\sigma_M$ is a scalar map. Since  $$\sigma_{M}(m_\lambda)=v^{-(\lambda\mid \lambda+2\varrho)}\sum_{\nu\in \mathcal Q^+}\sum_{wt(J)=\nu}b_Jx_{J}^{-}S (x_{J}^{+}) m_\lambda=v^{-(\lambda\mid \lambda+2\varrho)}m_\lambda,$$
 where
 $m_\lambda$ is the maximal vector of $M$, we have (2).

\end{proof}

Suppose $u=\sum x\otimes y\in \U_v(\mathfrak g)^{\otimes 2}$. For all $1\le i<j\le r$, define $u_r^{j, i}=P(u)_r^{i , j}$ where
 $u_r^{i, j}=\sum 1^{\otimes i-1}\otimes x\otimes 1^{j-i-1}\otimes y\otimes 1^{r-j}$.
 Later on, we use $u^{i, j}$(resp., $u^{j,i}$) to replace $u_r^{i, j}$(resp., $u_r^{j, i}$) if we know $r$ from the context.

\begin{Lemma}\label{eqn2} Suppose $\nu\in \mathcal Q^+$.
  \begin{itemize}\item[(1)]  $ (\bar {\Delta}\otimes Id)({{\bar\Theta}}_\nu)=\sum_{\nu'+\nu''=\nu}{{\bar\Theta}}_{\nu'}^{2,3}(1\otimes k_{\nu''}\otimes 1){{\bar\Theta}}_{\nu''}^{1,3}$, where $\nu', \nu''\in \mathcal Q^+$. \item[(2)] $ (Id\otimes \bar {\Delta})({{\bar\Theta}}_\nu)=\sum_{\nu'+\nu''=\nu}{{\bar\Theta}}_{\nu'}^{1,2}(1\otimes k_{-\nu''}\otimes 1){{\bar\Theta}}_{\nu''}^{1,3}$, where $\nu', \nu''\in \mathcal Q^+$.\item [(3)] $ (Id\otimes \Delta)({{\bar\Theta}}_\nu)=\sum_{\nu'+\nu''=\nu}{{\bar\Theta}}_{\nu'}^{1,3}(1\otimes k_{\nu'}\otimes 1){{\bar\Theta}}_{\nu''}^{1,2}$, where  $\nu', \nu''\in \mathcal Q^+$. \item [(4)]
   $\sum_{\nu'+\nu''=\nu}(\overline{\Delta}\otimes Id)({\bar\Theta}_{\nu'}){{\bar\Theta}}_{\nu''}^{1,2}=\sum_{\nu'+\nu''=\nu}(Id\otimes \bar {\Delta})({{\bar\Theta}}_{\nu'}){{\bar\Theta}}_{\nu''}^{2,3}$,  where $\nu', \nu''\in \mathcal Q^+$.
   \item [(5)]  $\sum_{\mu,\eta\in \mathcal Q^+} \sum_{J, H} b_Jb_{H}x_J^{-}x_{H}^{-}\otimes k_{\mu}x_{H}^{+}S(x_{J}^{+})
   =1\otimes 1$  for all possible $J,H$ such that $wt(J)=\mu$ and  ${wt(H) =
\eta}$.
\item[(6)]  $ \Delta(u)=\sum_{\mu\in \mathcal Q^+}\sum_{wt(J)=\mu} b_J \Delta(x_{J}^{-}) (S\otimes S ) \Delta^{op} (x_{J}^{+})
$, where   $u=\mathbf m (Id\otimes S)(\bar \Theta)$ and  $\Delta^{op}(x)=(\Delta(x))^{2,1}$ for all $ x\in\U_v(\mfg)$.
\end{itemize} \end{Lemma}
 \begin{proof}
 Acting $(-\otimes-\otimes-)$ on both sides of equations in \cite[Proposition~4.2.2-4.2.4]{Lu} yields (1)-(4).
  By (3), $$\text{LHS of (5)}=(Id\otimes \mathbf m )(Id^{\otimes 2}\otimes S)(Id\otimes \Delta)(\bar\Theta)=(Id\otimes \epsilon)(\bar\Theta)=1\otimes 1,$$ where  $\epsilon$ is the counit of $\U_v(\mathfrak g)$. Finally,
  $$\text{RHS of (6)}=\sum_{\mu\in \mathcal Q^+}\sum_{wt(J)=\mu} b_J \Delta(x_{J}^{-} )\Delta (S(x_{J}^{+}))= \Delta(u),$$ where
  the first equality follows from   the well-known equality
$\Delta\circ  S =(S\otimes S)\circ \Delta^{op} $.
 \end{proof}

The following is the counterpart of \cite[(3.4)]{Drin}. The  proof is motivated by Drinfeld's arguments.

\begin{Lemma}\label{chu} Suppose $M, N$ are two objects in $\mathcal O$.    As morphisms in    $\End(M\otimes N)$,  $$P\pi \bar{\Theta}\pi ^{-1}P\bar \Theta\Delta(\mathbf m (Id\otimes S)(\bar \Theta))=\mathbf m(Id\otimes S)(\bar \Theta)\otimes \mathbf m (Id\otimes S)(\bar \Theta).$$
\end{Lemma}
 \begin{proof}  Suppose  $ y_1,\ldots,y_6\in \U_v(\mathfrak g)$. Motivated by Drinfeld's formula in \cite[line -4, Page 34]{Drin}, we define
 $$(y_1\otimes y_2\otimes y_3\otimes y_4)\times(y_5\otimes y_6)=y_1y_5S(y_3)\otimes y_2y_6S(y_4).$$
 It is routine to check that \begin{equation} \label{varphi1} (z_1\otimes z_2\otimes z_3\otimes z_4)\times[(y_1\otimes y_2\otimes y_3\otimes y_4)\times (x_1\otimes x_2)]=(z_1y_1\otimes z_2y_2\otimes z_3y_3\otimes z_4y_4)\times (x_1\otimes x_2),
  \end{equation}
for any $x_i, y_j, z_k\in \U_v(\mfg)$ for all admissible $i, j, k$.

 Suppose $u=\mathbf m (Id\otimes S)(\bar \Theta)$ and $\bar{\Theta}_1=\sum_{\nu\in \mathcal Q^+}(1\otimes k_\nu)\bar{\Theta}_\nu(k_{-\nu}\otimes 1)$. It is routine to check  $P\bar\Theta_1 P=\bar\Theta_1^{2,1}$ in  $\End(M\otimes N)$. By \eqref{pif},
  $\bar\Theta_1=\pi \bar\Theta {\pi }^{-1}$  in  $\End(M\otimes N)$.
So, $$\begin{aligned} P\pi \bar\Theta {\pi }^{-1}  P\bar \Theta\Delta(u)=&
\bar\Theta_1^{2,1}\sum_{\nu \in \mathcal Q^+}\sum_{wt(J)=\nu} b_J (\bar \Theta\Delta(x_{J}^{-})) (S\otimes S ) \Delta^{op} (x_{J}^{+}),\text{by Lemma~\ref{eqn2}(6)}\\=& \bar\Theta_1^{2,1}\sum_{\nu\in \mathcal Q^+}\sum_{wt(J)=\nu} b_J \bar\Delta(x_{J}^{-}) \bar \Theta(S\otimes S ) \Delta^{op} (x_{J}^{+}), \text{by \eqref{qrm11}}\\=& [{{\bar\Theta}}_1^{2,1}(\bar{\Delta}\otimes \Delta^{op}) ({\bar\Theta})]\times {\bar\Theta}.\end{aligned}$$
 So, we need to verify $[{{\bar\Theta}}_1^{2,1}(\bar{\Delta}\otimes \Delta^{op}) ({\bar\Theta})]\times {\bar\Theta}=u\otimes u$ in $\End(M\otimes N)$.
  This is the case since
$$\begin{aligned}
  & [{{\bar\Theta}}_1^{2,1}(\bar{\Delta}\otimes \Delta^{op}) ({\bar\Theta})]\times {\bar\Theta}\\
    =& [{{\bar\Theta}}_1^{2,1}\sum_{\lambda,\lambda',\eta,\eta'}{{\bar\Theta}}_\eta^{2,3}(1\otimes k_{\eta'}\otimes 1^{\otimes 2} ){{\bar\Theta}}_{\eta'}^{1,3}(1^{\otimes 3} \otimes k_{\eta+\eta'}){{\bar\Theta}}_\lambda^{2,4}(1\otimes k_{\lambda'}\otimes1^{\otimes 2}){{\bar\Theta}}_{\lambda'}^{1,4}]\times {\bar\Theta}\\
 = & [{{\bar\Theta}}_1^{2,1}\sum_{\lambda,\eta,\eta'}{{\bar\Theta}}_\eta^{2,3}(1\otimes k_{\eta'}\otimes 1^{\otimes 2} ){{\bar\Theta}}_{\eta'}^{1,3}(1^{\otimes 3} \otimes k_{\eta+\eta'}){{\bar\Theta}}_\lambda^{2,4}]\times (1\otimes 1), \text{by \eqref{varphi1},  Lemma~\ref{eqn2}(5)}\\
= & [{{\bar\Theta}}_1^{2,1}\sum_{\eta,\eta'}{{\bar\Theta}}_\eta^{2,3}(1\otimes k_{\eta'}\otimes 1^{\otimes 2}){{\bar\Theta}}_{\eta'}^{1,3}(1^{\otimes 3} \otimes k_{\eta+\eta'})]\times (1\otimes u),\text{ by \eqref{varphi1}}   \\
=&\sum_{\beta,\eta,\eta'}\sum_{wt(J_1)= \beta } \sum_{wt(J_2)=\eta} \sum_{wt(J_3)=\eta'}b_{J_1}b_{J_2}b_{J_3}k_\beta x_{J_1}^{+}x_{J_3}^{-}S(x_{J_2}^{+}x_{J_3}^{+})\otimes x_{J_1}^{-}k_{-\beta }x_{J_2}^{-}k_{\eta'}uS(k_{\eta+\eta'})\\
= & \sum_{\beta,\eta,\eta'}\sum_{wt(J_1)= \beta } \sum_{wt(J_2)=\eta} \sum_{wt(J_3)=\eta'}       b_{J_1}b_{J_2}b_{J_3}k_\beta x_{J_1}^{+}x_{J_3}^{-}S(x_{J_3}^{+})S(x_{J_2}^{+})\otimes x_{J_1}^{-}k_{-\beta }x_{J_2}^{-}k_{-\eta}u\\
=&\pi  \left(\sum_{\beta,\eta,\eta'}\sum_{wt(J_1)= \beta } \sum_{wt(J_2)=\eta} \sum_{wt(J_3)=\eta'}       b_{J_1}b_{J_2}b_{J_3} x_{J_1}^{+}k_{-\eta} x_{J_3}^{-}S(x_{J_3}^{+})S(x_{J_2}^{+})\otimes x_{J_1}^{-}x_{J_2}^{-}u\right)\pi ^{-1} \\
= & \pi  \left([\sum_{\beta,\eta,\eta'}{{\bar\Theta}}_\beta ^{2,1}(k_{-\eta}\otimes 1\otimes 1){{\bar\Theta}}_\eta^{2,3}{{\bar\Theta}}_{\eta'}^{1,3}]\times (1\otimes u)\right)\pi ^{-1}\\
=& \pi  \left([\sum_{\beta,\eta,\eta'}{{\bar\Theta}}_{\eta'}^{1,3}(k_{\beta }\otimes 1\otimes 1){{\bar\Theta}}_\beta ^{2,3}{{\bar\Theta}}_\eta^{2,1}]\times (1\otimes u)\right)\pi ^{-1}, \text{ by Lemma~\ref{eqn2}(1)--(2),(4)}\\
= & \pi  \left(\sum_{\eta'}{{\bar\Theta}}_{\eta'}^{1,3}\times (\sum_{\beta,\eta}\sum_{wt(J_1)= \beta } \sum_{wt(J_2)=\eta}b_{J_1}b_{J_2}k_\beta x_{J_2}^{+}S(x_{J_1}^{+})\otimes x_{J_1}^{-}x_{J_2}^{-}u)\right)\pi ^{-1}
\\=&\pi  \left(\sum_{\eta'}{{\bar\Theta}}_{\eta'}^{1,3}\times (1\otimes u)\right)\pi ^{-1}, \text{ by Lemma~\ref{eqn2}(5)}\\
= & \pi  \left(u\otimes u\right)\pi ^{-1}=u\otimes u, \text{ since $u$ stabilizes any weight space}.
\end{aligned}$$
We remark that the first equality follows from  Lemma~\ref{eqn2}(1),(3), and  the fifth equality follows from the equation $k_\nu u=uk_\nu$ for any $\nu\in\mathcal P$,  and the sixth equality can be checked by direct calculation.
\end{proof}

\begin{Prop}\label{ptrace}For any two objects  $M$ and $N$ in $\mathcal O$,  $\sigma_{M\otimes N}=(\Theta\pi P)^2\sigma_{M}\otimes\sigma_{N}$.
\end{Prop}
\begin{proof} Suppose  $(m, n)\in(M_\lambda, N_\mu)$ and    $u=\mathbf m (Id\otimes S)(\bar \Theta)$. Note that $P\pi =\pi P$ in $\End(M\otimes N)$.
Using Lemma~\ref{homm} and \eqref{pif}, we have
$$\begin{aligned}& (\Theta\pi P)^2\sigma_{M}(m)\otimes \sigma_{N}(n)
=\Theta\pi P\Theta Pv^{-(\lambda\mid\mu)}v^{-(\lambda\mid \lambda+2\varrho)}v^{-(\mu\mid \mu+2\varrho)}um\otimes un\\
=&v^{-(\lambda+\mu\mid \lambda+\mu+2\varrho)}\Theta\pi P\Theta P\pi ^{-1}um\otimes un
= v^{-(\lambda+\mu\mid \lambda+\mu+2\varrho)}\Delta(u)
(m\otimes n)\\= &\sigma_{M\otimes N}(m\otimes n), \\ \end{aligned}$$
where
the third equality follows from Lemma~\ref{chu} and $ \Theta\bar \Theta=1\otimes 1$.
\end{proof}

 Recall that $\Pi$ is the set of simple roots and $\mathcal{R}$ is the
root system.
Fix positive integers $q_1,\ldots,q_k$ such that $\sum_{i=1}^{k}q_i=n$.
Define \begin{equation} \label{defofpiee}
 I_1=\Pi\setminus \{\alpha_{p_1}, \alpha_{p_2}, \ldots, \alpha_{p_{k}}\}, \text{ and  $I_2= I_1\cup \{ \alpha_{n}\}$} \end{equation} where $p_i=\sum_{j=1}^i q_j$ for all admissible $i$.

For each $I\in \{ I_1, I_2\}$, there is a subroot system $\mathcal{R}_{I}= \mathcal{R}\cap \mathbb Z I$. The corresponding  positive roots $\mathcal{R}_{I}^+=\mathcal{R}^+\cap \mathbb Z I$. Let $w_{0,I}$ be the longest element in   the standard parabolic subgroup $W_{I}$ of $W$ with respect to $I$. Then $w_0=w_{0,I} y$ where $y$ is  a distinguished right coset representative in $W_{I}\backslash W$.

As in \cite{GU}, we fix  reduced expressions of $w_{0, I}$ and $y$ so as to get a corresponding   reduced expression of $w_0$.
We have the corresponding convex order of $\mathcal{R}^{+}$. To distinguish from the convex order in \eqref{ord}, we denote it by $\check \beta_1< \check \beta_2<...< \check \beta_{\ell(w_0)}$. In this case,
we also denote the root elements which are defined via braid generators
 $T_i=T_{i,-1}^{''}$ in  \cite[\S 37.1.3]{Lu} by $\check x_{\check \beta_j}$'s. The parabolic subalgebra $\U_v(\mathfrak p_{I})$ is generated by $\{x_i^{-}\mid \alpha_i\in I\}\cup\{x_i^{+}, k_\mu\mid\alpha_i\in\Pi,\mu\in\mathcal{P}\}$, and the corresponding Levi subalgebra  $ \U_v(\mathfrak l_{I})$ is generated by $\{x_i^{\pm}\mid \alpha_i \in I\}\cup\{k_\mu\mid  \mu\in\mathcal{P}\}$. Let $\U_v(\mathfrak u_{I}^{+})$ be the subalgebra generated by
 $\{\check x_{\check{\beta}_j}^{+}\mid j>\ell(w_{0, I})\}$. Similarly, let
 $\U_v(\mathfrak u_{I}^-)$ be  the subalgebra generated by $\{\check{x}_{\check{\beta}_j}^{-}\mid j> \ell(w_{0, I})\}$. It is known that $\U_v(\mathfrak g)=\U_v(\mathfrak u_{I}^{-})\otimes_{\mathbb{F}} \U_v(\mathfrak p_{I})$  and $ \U_v(\mathfrak p_{I})=\U_v(\mathfrak l_{I})\otimes_{\mathbb{F}} \U_v(\mathfrak u^+_{I})$.

Let
 $  \Lambda^{\mathfrak p_{I}}=\{\lambda\in \mathcal{P}\mid  \frac{2 (\lambda\mid \alpha_j)}{(\alpha_j\mid \alpha_j) }  \in \mathbb N, \forall \alpha_j\in I\}$. Then  $  \Lambda^{\mathfrak p_{I}}$  is  the set of  all $\mathfrak l_{I}$-dominant integral weights.
  We  have
 $$v^{(\lambda\mid\mu)}=q^{(\lambda,\mu)} , $$ where $(\ , )$ is the symmetric bilinear form such that $(\epsilon_i,\epsilon_j)=\delta_{i,j}$.
 We will  keep using  $q^{(\lambda,\mu)}$ later on.
It is known that the irreducible $\U_v(\mathfrak l_{I})$-module $L_{I}(\lambda)$ with highest weight $\lambda$ is of finite dimensional if and only if  $\lambda\in \Lambda^{\mathfrak p_{I}}$. By inflation, $L_{I}(\lambda)$ can be considered as a $\U_v(\mathfrak p_{I})$-module.
For any $\lambda\in \Lambda^{\mathfrak p_{I}}$, the  parabolic
Verma module with highest weight $\lambda$ is
$$M^{\mathfrak p_{I}}(\lambda):=\U_v(\mfg) \otimes_{\U_v(\mathfrak p_{I})} L_{I}(\lambda),  \text{ $\lambda\in \Lambda^{\mathfrak p_{I}}$.} $$  As vector spaces, $M^{\mathfrak p_{I}}(\lambda)\cong \U_v(\mathfrak u_{I}^{-})\otimes L_{I}(\lambda)$.
Given a $\mathbf c=(c_1, c_2, \ldots, c_k)\in  \mathbb Z^k$ such that
  $c_k=0$ if $I=I_2$, we define \begin{equation} \label{deltac} \lambda_{I, \mathbf c}=\sum_{j=1}^{k} c_j(\epsilon_{p_{j-1}+1}+\epsilon_{p_{j-1}+2}+\ldots+\epsilon_{p_j}),\end{equation}  where $p_0=0$, and  $p_j$'s are given in ~\eqref{defofpiee}.
One can check
$\lambda_{I, \mathbf c}\in \Lambda^{\mathfrak p_{I}}$ and   $\text{dim}_{\mathbb F }L_{I}(\lambda_{I, \mathbf c})=1$.

In the remaining part of this paper, we denote by $m_I$ the highest weight vector of  $L_{I}(\lambda_{I, \mathbf c})$ for any $I\in \{I_1, I_2\}$. Then
$L_{I}(\lambda_{I, \mathbf c})=\mathbb F m_I$.
Suppose \begin{equation}\label{pararoot} \mathcal{R}^+\setminus\mathcal{R}_{I}^{+}=\{{\beta}_{t_j}\mid 1\le j\le  \ell(w_{0, I}^{-1}w_0) \}\end{equation} and $\beta_{t_j}< \beta_{t_l}$ in the sense of \eqref{ord}  whenever $j<l$. So,  $t_j<t_l$ if and only if
$j<l$.

 \begin{Lemma}\label{tsmodule1} Suppose $r\in \mathbb N$. Let  \begin{equation}\label{sir}  \mathcal S_{I,r}=\left\{\vec{\prod}_{j=\ell(w_{0, I}^{-1}w_0)  }^1 (x^-_{\beta_{t_j}})^{i_j}  m_I\otimes v_{\mathbf j}\mid  i_j\in \mathbb N,  \mathbf j\in \underline {\mathrm N} ^r\right\},\end{equation} where  $v_\mathbf j=v_{j_1}\otimes v_{j_2}\otimes\ldots\otimes v_{j_r}$. Then  $\mathcal S_{I,r} $ is a basis of $M^{\mathfrak p_{I}}(\lambda_{I, \mathbf c})\otimes V^{\otimes r}$.
\end{Lemma}
\begin{proof}  It is enough to prove that $\mathcal S_{I,0}$ is a basis of  $M^{\mathfrak p_{I}}(\lambda_{I, \mathbf c}) $.
Suppose $M=M^{\mathfrak p_{I}}(\lambda_{I, \mathbf c})$.
  We have $M\cong \U_v(\mathfrak u_{I}^{-})\otimes_{\mathbb{F}}L_{I}(\lambda_{I, \mathbf c})$, and  $S$  is a basis of $M$, where  $${S}=\left \{\vec{\prod}_{j=\ell(w_0)}^{\ell(w_{0, I})+1} (\check x_{\check{\beta}_{j}}^-)^{r_{j}}\otimes m_I\mid  r_j \in\mathbb{N}, \ell(w_{0,I} )+1\le j\le \ell(w_0) \right\}.$$  Note that $M=\oplus_{\eta} M_\eta$ where $\eta\in \lambda_{I, \mathbf c}-\mathcal{Q}^+$. For all admissible $\eta$, define $({\mathcal S_{I, 0}})_\eta={\mathcal S_{I, 0}}\cap M_\eta$ and ${S}_\eta={S}\cap M_\eta$. Then
  $$\sharp{S}_\eta=\sharp\left\{\mathbf r\in\mathbb{N}^{\ell(w_{0,I}^{-1} w_0)}| \sum_{j=\ell(w_{0, I})+1}^{\ell(w_0)}r_{j-\ell(w_{0,I})}\check{\beta}_{j}=\lambda_{I, \mathbf c}-\eta\right\}$$
  and $\text{dim} M_\eta=\sharp{S}_\eta<\infty$.
   Since $\{\check{\beta}_{j}| \ell(w_{0, I})+1\le j\le \ell(w_0)\}=\mathcal{R}^+\setminus\mathcal R_{I}^+$ and
    $$\sharp({\mathcal S_{I, 0}})_\eta=\sharp\left\{\mathbf r\in\mathbb{N}^{\ell(w_{0,I}^{-1} w_0)}| \sum_{j=1}^{\ell(w_{0,I}^{-1} w_0)  }r_{j}{\beta}_{t_j}=\lambda_{I, \mathbf c}-\eta\right\}=\sharp{S}_\eta, $$
    it is enough to  verity that $\mathcal S_{I, 0}$ spans $M$.

If the result were false, we can find a non-zero element $x_{\mathbf r}^-\otimes m_I$ which is not a linear combination of elements in $\mathcal S_{I, 0}$.  Take the maximal integer $j$ such that  $r_j\neq  0$ and $\beta_j\in\mathcal{R}_{I}^{+}$. If $j=1$, then
$x_{\mathbf r}^-\otimes m_I=\vec{\prod}_{j=\ell(w_0)}^2 (x_{\beta_{j}}^-)^{r_{j}}\otimes (x_{\beta_{1}}^-)^{r_{1}}m_I=0$, a contradiction. The  general case follows  from Lemma~\ref{commut}(1) together with arguments for induction on $j$.
\end{proof}

In the remaining part of this section,  $\mathbf c$ is always the one in \eqref{deltac}.
If  $\mfg=\mathfrak{so}_{2n+1}$,  we define
\begin{equation} \label{polf123}
b_j=\begin{cases} 2(c_j-p_{j-1})+\mathrm N -\epsilon_\mfg, & \text{if  $1\leq j\leq k$,}\\
  0, &\text{if $j=k+1$,}\\
    -2c_{2k-j+2}+2p_{2k-j+2}-\mathrm N +\epsilon_\mfg, &\text{if $k+2\leq j\leq 2k+1$.}\end{cases}
    \end{equation}
 Otherwise,  define
\begin{equation}\label{polf12} b_j=\begin{cases} 2(c_j-p_{j-1})+\mathrm N -\epsilon_\mfg, &\text{if $1\le j\le k$,}\\
-2c_{2k-j+1}+2p_{2k-j+1}-\mathrm N +\epsilon_\mfg, &\text{if $k+1\leq j\leq 2k$,}\\ \end{cases}\end{equation}where $\epsilon_\mfg$ is given in \eqref{lamb}.

 \begin{Defn} \label{fi} Define $f_{I}(t)=\prod_{j\in J_I} (t- u_j)$ and  $ \bar f_{I}(t)=\prod_{j\in J_I} (t- u_j^{-1})$, where \begin{itemize}\item  $J_{I_1}=\{1, 2, \ldots, 2k\}$ and $J_{I_2}= J_{I_1}\setminus \{k+1\}$ if $\mfg\neq \mathfrak{so}_{2n+1}$,
  \item $J_{I_1}=\{1, 2, \ldots, 2k+1\}$ and  $J_{I_2}=J_{I_1}\setminus \{k+1, k+2\}$ if  $\mfg= \mathfrak{so}_{2n+1}$,
  \item  $u_j=\epsilon_\mfg q^{b_j}$ for any $j\in J_I$ and  $b_j$'s are in  \eqref{polf123}-\eqref{polf12}.\end{itemize}
\end{Defn}

 \begin{Defn}\label{V}Define  $V_{I,j}=\{v_{l}\mid  p_{j-1}+1\le l\le  p_j \}$ if either $I=I_1$ and $1\leq j\leq k$  or $I=I_2$ and $1\leq j\leq k-1$, and
$$
 V_{I,j}=\begin{cases}
       \{v_{l'}\mid   p_{2k-j+1}\le l\le  p_{2k-j}+1  \}, & \text{$k+1\leq j\leq 2k$ and $I=I_1$;} \\
  \{v_{l'}\mid p_{2k-j-1}+1\le l\le  p_{2k-j}    \}, & \text{$k+1\leq j\leq 2k-1$ and $I=I_2$;} \\
 \{v_{l}, v_{l'} \mid p_{k-1}+1\le l \le  p_{k}\}\cup \{\delta_{\mfg, \mathfrak{so}_{2n+1}}v_{n+1}\}, & \text{$j=k$ and $I=I_2$.} \\
      \end{cases}
$$
 \end{Defn}
Recall the functor $\Psi$ in Proposition~\ref{abmw1}.
Let $\Psi_M:\AB\to \U_v(\mfg)\text{-mod}$ be the functor obtained by  the composition of $\Psi$ followed by evaluation at $M\in \U_v(\mfg)\text{-mod}$.
Thanks to Proposition~\ref{abmw1}, for any $d_1\in \Hom_{\AB}(\ob m, \ob s)$ and $d_2\in \Hom_{\AB}(\ob h, \ob t)$, $m, s, h, t \in
\mathbb N $, we have
\begin{equation}\label{nnn}\Psi_M(d_1\otimes d_2)=\Psi(d_2)_{M\otimes V^{\otimes s}}\circ (\Psi(d_1)_M\otimes Id_{V^{\otimes h}}).
\end{equation}

Obviously both $M^{\mathfrak p_{I}}(\lambda_{I, \mathbf c})$ and finite dimensional weight $\U_v(\mfg)$-modules are in $\mathcal{O}$. Since $\mathcal{O}$ is closed under tensor product,
 the image of $\Psi_{M^{\mathfrak p_{I}}(\lambda_{I, \mathbf c})}$ is in $\mathcal{O}$. Thus,  we  have a functor
$ \Psi_{M^{\mathfrak p_{I}}(\lambda_{I, \mathbf c})}:\AB\to \mathcal O$. In the following, we denote $ \Psi_{M^{\mathfrak p_{I}}(\lambda_{I, \mathbf c})}(d)$ by $d$ for any morphism $d$ in $\AB$.

\begin{Lemma}\label{polyofx} If  $\mathfrak g\in \{\mathfrak {so}_{2n}, \mathfrak {sp}_{2n}\}$,
then    $M^{\mathfrak p_{I}}(\lambda_{I, \mathbf c})\otimes V$ has    a parabolic Verma flag
$0=N_0\subseteq N_1\subseteq  \ldots\subseteq N_{2k}=M^{\mathfrak p_{I}}(\lambda_{I, \mathbf c})\otimes V$
   such that
 \begin{equation} \label{pbvf1} N_j/N_{j-1}\cong \begin{cases}  M^{\mathfrak p_{I}}(\lambda_{I, \mathbf c}+\varepsilon_{p_{j-1}+1}), & \text{if $1\leq j\leq k$,} \\
 M^{\mathfrak p_{I}}(\lambda_{I, \mathbf c}-\varepsilon_{p_{k}}), & \text{if $I=I_1$ and $j=k+1$,} \\
 0, & \text{if $I=I_2$ and $j=k+1$,} \\
                               M^{\mathfrak p_{I}}(\lambda_{I, \mathbf c}-\varepsilon_{p_{2k+1-j}}) , & \text{if $k+2\leq j\leq 2k$.}
                               \end{cases}\end{equation}
 Moreover, $X$ preserves previous flag and  $f_{I}(X)$
 acts  on  $M^{\mathfrak p_{I}}(\lambda_{I, \mathbf c})\otimes V$  trivially,   where $ f_{I}(t)$ is given in Definition~\ref{fi}.
   \end{Lemma}

   \begin{proof} Keep the notation in Definition \ref{V}. By arguments similar to those  in \cite[Lemma~4.11]{RS3}, we have the required flag where
  $N_j$ is   the left $\U_v(\mathfrak g)$-module  generated by $m_{I}\otimes u$'s and  \begin{itemize}
\item [(1)] $u\in \cup _{l=1}^j V_{I,l}$ if $I=I_1$ or $I=I_2$ and $1\leq j\leq k$,
\item [(2)] $ u\in \cup _{l=1}^k V_{I, l}  $  if $I=I_2$ and  $j=k+1$,
\item [(3)] $u\in \cup _{l=1}^{j-1} V_{I, l}$  if  $I=I_2$ and $k+2\leq j\leq 2k$.
\end{itemize}

 Note that  $\Theta=\sum_{\beta\in\mathcal{Q}^+}\Theta_\beta$ and $\Theta_\beta\in \U_v^{-}(\mfg)_{-\beta}\otimes \U_v^{+}(\mfg)_{\beta}$. Since  $\U_v^{+}(\mfg)_{\beta}m_{I}=0$ for any  $\beta\neq 0$, we  have
$$X^{-1}(m_{I}\otimes v_j)=\delta^{-1}\Theta\pi P\Theta\pi P(m_{I}\otimes v_j)
=\delta^{-1}q^{-2(\lambda_{I, \mathbf c}, wt(v_j))}\Theta(m_{I}\otimes v_j),$$ where $\delta$ is given in \eqref{lamb}. So,  $X^{-1}$ preserves the required flag.
Let $M_{I}=M^{\mathfrak p_{I}}(\lambda_{I, \mathbf c})$. Thanks to Lemma~\ref{homm}(2),
$$\sigma_{M_{I}}\otimes  \sigma_{V}=q^{-[(\lambda_{I, \mathbf c}, \lambda_{I, \mathbf c}+2\varrho)+(\epsilon_1, \epsilon_1+2\varrho)]} {Id}_{M_{I}\otimes  {V}}, \ \
 \sigma_{N_j/N_{j-1}}=q^{-(\nu, \nu+2\varrho)} {Id}_{N_j/N_{j-1}}$$ if $N_j/N_{j-1}\cong M^{\mathfrak p_{I}}(\nu)$.
By Proposition~\ref{ptrace},  $\sigma_{M_{I}\otimes V}=(\Theta\pi  P)^2 \sigma_{M_{I}}\otimes \sigma_V$.
Since   $X^{-1}$ acts on $M_{I}\otimes V$ as $\delta^{-1} (\Theta \pi  P)^2$,
$$X^{-1}|_{N_j/N_{j-1}}= q^{-(\nu, \nu+2\varrho)}q^{(\lambda_{I, \mathbf c}, \lambda_{I, \mathbf c}+2\varrho)}q^{(\epsilon_1, \epsilon_1+2\varrho)}\epsilon_{\mathfrak g} q^{-\mathrm N +\epsilon_\mfg}{Id}_{N_j/N_{j-1}}=\epsilon_\mfg q^{-b_j}{Id}_{N_j/N_{j-1}},$$  where $b_j$'s are given in ~\eqref{polf12}. So,    $\bar {f}_{I}(X^{-1})$  acts on $M_{I}\otimes V$  trivially, where  $ \bar f_{I}(t)$ is given in Definition \ref{fi}.
Now, the corresponding result on $X$ follows immediately.
\end{proof}

\begin{Lemma}\label{polyofxodd}  If $\mathfrak g=\mathfrak {so}_{2n+1}$, then
  $M^{\mathfrak p_{I}}(\lambda_{I, \mathbf c})\otimes V$   has a   parabolic Verma flag
$0=N_0\subseteq N_1\subseteq \ldots\subseteq N_{2k+1}=M^{\mathfrak p_{I}}(\lambda_{I, \mathbf c})\otimes V$ such that
$$N_j/N_{j-1}\cong \begin{cases}
                                M^{\mathfrak p_{I}}(\lambda_{I, \mathbf c}+\varepsilon_{p_{j-1}+1}), & \text{if $1\leq j\leq k$,} \\
                                M^{\mathfrak p_{I}}(\lambda_{I, \mathbf c}), & \text{if $I=I_1$ and $j= k+1$,}\\
                                 M^{\mathfrak p_{I}}(\lambda_{I, \mathbf c}-\varepsilon_{p_{k}}),  & \text{if $I=I_1$ and $j= k+2$,}\\
                                 0,  & \text{if $I=I_2$ and $j=k+1, k+2 $,} \\
                               M^{\mathfrak p_{I}}(\lambda_{I, \mathbf c}-\varepsilon_{p_{2k+1-j}}), & \text{if $k+3\leq j\leq 2k+1$.}
                               \\
                               \end{cases}$$      Moreover, $X$ preserves previous flag and  $ f_{I}(X)$
 acts  on  $M^{\mathfrak p_{I}}(\lambda_{I, \mathbf c})\otimes V$  trivially,   where $ f_{I}(t)$ is given in Definition~\ref{fi}.
\end{Lemma}

\begin{proof} The result can be verified by arguments similar to those in the proof of Lemma~\ref{polyofx}. We give the explicit construction of $N_j$'s and leave others to the reader.    Recall  $V_{I,l}$'s   in Definition~\ref{V}. Then    $N_j$ is     the left $\U_v(\mfg)$-module  generated by $m_{I}\otimes u$, where
\begin{itemize}
\item [(1)] $ u\in \cup_{l=1}^j V_{I, l}$ if  $1\leq j\leq k$,
\item [(2)] $u\in \cup_{l=1}^k V_{I, l}$ if  $I=I_2$ and $k+1\leq j\leq k+2$,
\item [(3)] $u\in \cup_{l=1}^{j-2} V_{I, l}$ if $I=I_2$ and  $k+3\leq j\leq 2k+1$,
\item [(4)] $ u\in \cup_{l=1}^k V_{I, l}\cup\{v_{n+1}\}$ if $I=I_1$ and $j=k+1$,
\item [(5)] $  u\in \cup_{l=1}^{j-1} V_{I, l}\cup\{v_{n+1}\}  $  if $I=I_1$ and $k+2\leq j\leq 2k+1$.
\end{itemize}
\end{proof}

The  following definition is motivated by \cite[(2.24)]{DRV}.
\begin{Defn}\label{qtr} Suppose $M$ is a  $\U_v(\mathfrak g)$-module and $N$ is a finite dimensional weight  $\U_v(\mathfrak g)$-module. For any  $\psi\in End_{\U_v(\mathfrak g)}(M\otimes N)$, define $Id \otimes \text{qtr}_N(\psi): M\rightarrow M$ such that $$Id \otimes \text{qtr}_N(\psi)=(Id\otimes \text{tr}_N)((1\otimes \widetilde{K})\circ\psi),$$ where $1\otimes \widetilde{K}(m \otimes n_\lambda)=q^{-(\lambda, 2\varrho)}m \otimes n_\lambda$ for any $(m,n_\lambda)\in(M,N_\lambda)$.
\end{Defn}
 It is well known that  the natural transformations between the identity functor and itself in $\U_v(\mfg)\text{-mod}$ form  an algebra which can be identified with   the center $Z(\U_v(\mathfrak g) )$ of $\U_v(\mathfrak g )$.
 Let \begin{equation}\label{cen} z_j=\Psi(\Delta_j)_{\U_v(\mathfrak g)}(1),\end{equation} where $\Delta_j$ is given in \eqref{bubbs}. Then
 $z_j\in Z(\U_v(\mathfrak g) )$, for any $j\in \mathbb Z$.
Let $\gamma=P\pi ^{-1}{\bar\Theta}P\pi ^{-1}{\bar\Theta}$.
Thanks to Proposition~\ref{abmw1},
 $$
 z_j
=(Id_{\U_v(\mfg)}\otimes \underline{\alpha})\circ( (\delta \gamma)^{j}\otimes Id_{V})(Id_{\U_v(\mfg)}\otimes \sum_{i\in \underline{\mathrm N }}q^{\varrho_{i'}} \varsigma_{i'} (v_i\otimes v_{i'}))
 =\epsilon_\mfg({Id}\otimes \text{qtr}_V)((\delta \gamma)^{j})
.$$  Write $z^+(u)=\sum_{l\geq0}z_{l}u^{-l}$ and $z^-(u)=\sum_{l\geq1}z_{-l}u^{-l}$, where $u$ is an indeterminate.
By    arguments similar to those in the proof of \cite[Theorem~3.5]{DRV},  we have   \begin{equation}\label{zkco}\begin{aligned}
\pi_{0}(z^+(u)+\frac{\delta^{-1}}{q-q^{-1}}-\frac{u^2}{u^2-1})& =\sigma_{\varrho}(\frac{\delta (u^2-q^{2\epsilon_{\mfg}})}{(q-q^{-1})(u^2-1)}\prod_{j\in\underline{\mathrm N }}\frac{1-k_{wt(v_j)}^2q^{-1}(\epsilon_\mfg u)^{-1}}{1-k_{wt(v_j)}^2q(\epsilon_\mfg u)^{-1}}),\\
\pi_{0}(z^-(u)-\frac{\delta^{-1}}{q-q^{-1}}-\frac{1}{u^2-1})& =\sigma_{\varrho}(-\frac{\delta^{-1} ( u^2-q^{-2\epsilon_{\mfg}})} {(q-q^{-1})(u^2-1)}\prod_{j\in\underline{\mathrm N }}\frac{1-k_{wt(v_j)}^2q(\epsilon_\mfg u)^{-1}}{1-k_{wt(v_j)}^2q^{-1}(\epsilon_\mfg u)^{-1}}),
\\ \end{aligned}\end{equation}
where $\pi_0$ is Harish-Chandra homomorphism and $\sigma_\varrho(k_\mu)=q^{(\varrho,\mu) }k_\mu$, and  $\varrho$ is given in \eqref{posr}. Our $z^+(u)$(resp., $z^-(u))$ is $Z_V^+(u)$(resp., $Z_V^-(u)-z_0)$ in \cite[Theorem~3.5]{DRV}. The  difference is that we
deal with the quantum group $\U_v(\mathfrak g)$ whereas they deal with $\U_h(\mathfrak g)$. So,  we have to use $P\pi ^{-1}{\bar\Theta}$ to  replace   their   $\mathcal R$.

\begin{Lemma}\label{ghom123}Suppose $j\in \mathbb Z$. Then  $\Psi(\Delta_j)_{M^{\mathfrak p_{I}} (\lambda_{I, \mathbf c})}=\omega_jId_{{M^{\mathfrak p_{I}} (\lambda_{I, \mathbf c})}}$
for some  $\omega_j\in \mathbb F$, where $\Delta_j$ is given in \eqref{bubbs}.
Further, $\omega$ is $\mathbf u$-admissible in the sense of Definition~\ref{uad}, where $\mathbf u =\{u_j\mid j\in J_I\}$, and $u_j$'s  and $J_I$ are given in Definition~\ref{fi}.
\end{Lemma}
\begin{proof}
Recall $z_j$ in \eqref{cen}. Suppose  $\omega_j=z_j|_{M^{\mathfrak p_{I}} (\lambda_{I, \mathbf c})},  \forall j\in \mathbb Z$. Obviously, there is a  $\U_v(\mathfrak g)$-homomorphism $\phi: \U_v(\mathfrak g)\rightarrow M^{\mathfrak p_{I}} (\lambda_{I, \mathbf c})$ such that $\phi(1)=m_I$. So,
$$\begin{aligned}\Psi(\Delta_j)_{M^{\mathfrak p_{I}} (\lambda_{I, \mathbf c})} m_I&=\Psi(\Delta_j)_{M^{\mathfrak p_{I}} (\lambda_{I, \mathbf c})}\phi(1)\\&=\phi(z_j)=z_j\phi(1)\\&=\omega_jm_I,\end{aligned}$$ and $\Psi(\Delta_j)_{M^{\mathfrak p_{I}} (\lambda_{I, \mathbf c})}=\omega_jId_{{M^{\mathfrak p_{I}} (\lambda_{I, \mathbf c})}}$.
This proves  the first assertion.

Let $a=\sharp J_I$.
It is routine to check that $\delta, u_1, \ldots, u_a$ satisfy Assumption~\ref{asump} where  $\delta$ is   in \eqref{lamb} and  $u_{k+1}=1$ (resp., $q^{\epsilon_\mathfrak g}$)  if $\mathfrak g=\mathfrak {so}_{2n+1}$ (resp.,
 $I=I_2$ and  $\mathfrak g\in \{\mathfrak {so}_{2n}, \mathfrak {sp}_{2n}\}$)
 and $u_{k+2}=1 $  if  $I=I_2 $ and $\mathfrak g=\mathfrak {so}_{2n+1}$.
 In fact, when  $a$ is odd,
 $\delta=\alpha\prod_{j\in J_I} u_j$, where $\alpha=1$ (resp., $-1$) if  $\mathfrak g\in\{\mathfrak {so}_{2n}, \mathfrak {so}_{2n+1}\}$ (resp., $\mathfrak g=\mathfrak {sp}_{2n}$).
    When $ a$ is even, we still have $\delta=\alpha\prod_{j\in J_I} u_j$, where
    $\alpha=-q$ (resp., $q^{-1}$) if  $\mathfrak g=\mathfrak {sp}_{2n}$  (resp., $\mathfrak {so}_{2n}$).
    We omit details since one can verify them by straightforward computation.

    Define  $u_\mfg=\frac{u-q^{-1}}{u-q}$ (resp., $1$) if $\mfg=\mathfrak {so}_{2n+1}$ (resp.,   otherwise). Let
$z=z_q$. By (\ref{zkco}), we have the following equalities in   $ \End(M^{\mathfrak p_{I}} (\lambda_{I, \mathbf c}))$:
$$\begin{aligned}&\sum_{j\geq0}\omega_ju^{-j}+\frac{\delta^{-1}}{ q-q^{-1}}-\frac{u^2}{u^2-1}\\=&\sigma_{\varrho}\left(\frac{\delta}{z}\frac{u^2-q^{2\epsilon_{\mfg}}}{u^2-1}\prod_{j\in\underline{\mathrm N }}\frac{1-\epsilon_\mfg q^{-1} u^{-1}k_{wt(v_j)}^2}{1-\epsilon_\mfg  q u^{-1} k_{wt(v_j)}^2}\right)\\
 =&\frac{\delta}{z}\frac{u^2-q^{2\epsilon_{\mfg}}}{u^2-1}\prod_{j=1}^k \frac  {(1-\epsilon_\mfg u^{-1} q^{(2\varrho,\epsilon_{p_j})-1}k_{wt(v_{p_j})}^2)(1-\epsilon_\mfg u^{-1}q^{-(2\varrho,\epsilon_{p_{j-1}+1})-1}k_{-wt(v_{p_{j-1}+1})}^2)}{
 (1-\epsilon_\mfg u^{-1} q^{-(2\varrho,\epsilon_{p_j})+1}k_{-wt(v_{p_j})}^2 )  (1-\epsilon_\mfg u^{-1} q^{(2\varrho,\epsilon_{p_{j-1}+1})+1}k_{wt(v_{p_{j-1}+1})}^2 )}u_\mfg\\
 =&\frac{\delta}{z}\frac{u^2-q^{2\epsilon_{\mfg}}}{u^2-1}\prod_{j=1}^k\frac{(1-\epsilon_\mfg u^{-1}q^{\mathrm N -\epsilon_\mfg-2p_j+2c_j})(1-\epsilon_\mfg u^{-1} q^{\epsilon_\mfg-\mathrm N +2p_{j-1}-2c_j})}{(1-\epsilon_\mfg u^{-1}q^{-(\mathrm N -\epsilon_\mfg-2p_j+2c_j)})(1- \epsilon_\mfg u^{-1}q^{-(\epsilon_\mfg-\mathrm N +2p_{j-1}-2c_j)})}u_{\mfg}\\
 =&\frac{\delta}{z}\frac{u^2-q^{2\epsilon_{\mfg}}}{u^2-1}\prod_{j\in J_{I_1}}\frac{1- u_j^{-1}u^{-1}}{1-u'_j u^{-1}}u_\mfg
 =\frac{\delta}{z}\frac{u^2-q^{2\epsilon_{\mfg}}}{u^2-1}\prod_{j\in J_{I_1}}\frac{u-{u_j}^{-1}}{u-u_j}u_{\mfg}\\
 =& \left ({z}^{-1} {\delta}^{-1}\prod_{j\in J_I} u_j+\frac{u g_{a}(u)}{u^2-1}\right )\prod_{j\in J_I} u_j \prod_{j\in J_I} \frac{{u-u_j}^{-1}}{u-u_j}, \end{aligned}
$$
where $g_a(u)$ is defined in Definition~\ref{uad},
and
$$\begin{aligned}&\sum_{j\geq1}\omega_{-j}u^{-j}-\frac{\delta^{-1}}{ q-q^{-1}}-\frac{1}{u^2-1}\\=&\sigma_{\varrho}\left(-\frac{\delta^{-1}}{z}\frac{u^2-q^{-2\epsilon_{\mfg}}}{u^2-1}\prod_{j\in\underline{\mathrm N }}\frac{1-\epsilon_\mfg q u^{-1}k_{wt(v_j)}^2}{1-\epsilon_\mfg  q^{-1} u^{-1} k_{wt(v_j)}^2}\right)\\
 =&-\frac{\delta^{-1}}{zu_\mfg}\frac{u^2-q^{-2\epsilon_{\mfg}}}{u^2-1}\prod_{j=1}^k \frac  {(1-\epsilon_\mfg u^{-1} q^{-(2\varrho,\epsilon_{p_j})+1}k_{-wt(v_{p_j})}^2)(1-\epsilon_\mfg u^{-1}q^{(2\varrho,\epsilon_{p_{j-1}+1})+1}k_{wt(v_{p_{j-1}+1})}^2)}{
 (1-\epsilon_\mfg u^{-1} q^{(2\varrho,\epsilon_{p_j})-1}k_{wt(v_{p_j})}^2 )  (1-\epsilon_\mfg u^{-1} q^{-(2\varrho,\epsilon_{p_{j-1}+1})-1}k_{-wt(v_{p_{j-1}+1})}^2 )}\\
 =&-\frac{\delta^{-1}}{zu_\mfg}\frac{u^2-q^{-2\epsilon_{\mfg}}}{u^2-1}\prod_{j=1}^k\frac{(1-\epsilon_\mfg u^{-1}q^{-(\mathrm N -\epsilon_\mfg-2p_j+2c_j)})(1-\epsilon_\mfg u^{-1} q^{-(\epsilon_\mfg-\mathrm N +2p_{j-1}-2c_j)})}{(1-\epsilon_\mfg u^{-1}q^{\mathrm N -\epsilon_\mfg-2p_j+2c_j})(1- \epsilon_\mfg u^{-1}q^{\epsilon_\mfg-\mathrm N +2p_{j-1}-2c_j})}\\
 =&-\frac{\delta^{-1}}{zu_\mfg}\frac{u^2-q^{-2\epsilon_{\mfg}}}{u^2-1}\prod_{j\in J_{I_1}}\frac{1-u_j u^{-1}}{1- u_j^{-1}u^{-1}}
 =-\frac{\delta^{-1}}{zu_\mfg}\frac{u^2-q^{-2\epsilon_{\mfg}}}{u^2-1}\prod_{j\in J_{I_1}}\frac{u-u_j}{u-u_j^{-1}}\\
 =& -\left ({z}^{-1} {\delta}^{-1}\prod_{j\in J_I} u_j-\frac{u}{g_{a}(u)(u^2-1)}\right )\prod_{j\in J_I} u_j^{-1} \prod_{j\in J_I} \frac{u-u_j}{u-u_j^{-1}}.\end{aligned}
$$
So,   $\omega$ is $\mathbf u$-admissible in the sense of Definition~\ref{uad}.
\end{proof}

It is proved in \cite[Corollary~2.29]{RX} that $\omega$ is  admissible if it is $\mathbf u$-admissible. So, we have $\CB^{f_I}$ (see Definition~\ref{scbmw}), where  $f_{I}(t)$ is given in Definition~\ref{fi}.
\begin{Theorem}\label{th}
$\Psi_{M^{\mathfrak p_{I}}(\lambda_{I, \mathbf c}) } $ factors through $\CB^{f_I}$, where $f_{I}(t)$ is given  in Definition~\ref{fi}.
\end{Theorem}
\begin{proof}
Thanks to  Lemmas~\ref{polyofx}-\ref{polyofxodd} and \ref{ghom123},  $$\Psi_{M^{\mathfrak p_{I}}(\lambda_{I, \mathbf c}) } (f_I(X))=0~\text{ and } \Psi_{M^{\mathfrak p_{I}}(\lambda_{I, \mathbf c}) } (\Delta_j-\omega_j1_{\ob 0})=0,~\forall j\in \mathbb Z.$$
 Recall that  a  right tensor ideal $B$ of a $\mathbb K$-linear strict monoidal category $\mathcal C$ is the collection of  $\mathbb K$-submodules $ \{B(\ob a,\ob b)\mid B(\ob a, \ob b)\subset  \Hom_{\mathcal C}(\ob a,\ob b),  \forall \ob a,\ob b \in \mathcal C\}$, such that
\begin{equation}\label{idea} h_3\circ h_2\circ h_1\in B(\ob a ,\ob d), \text{ and $ h_2\otimes 1_{\ob d}\in B(\ob b\otimes \ob d ,\ob c\otimes \ob d)$}\end{equation}
 whenever $(h_1, h_2, h_3)
 \in  \Hom_{\mathcal C}(\ob a,\ob b)\times B(\ob b,\ob c)\times \Hom_{\mathcal C}(\ob c,\ob d)$ for all  objects $\ob a,\ob b, \ob c,\ob d$ in $\C$. By~\eqref{nnn}, $\Psi_{M^{\mathfrak p_{I}}(\lambda_{I, \mathbf c})}(B(\ob m, \ob s))=0$, for all $m, s\in \mathbb N$, where $B$  is the right ideal of $\AB$ generated by $f_I(X)$ and $\Delta_j-\omega_j1_{\ob 0}$, $j\in \mathbb Z$. So
$\Psi_{M^{\mathfrak p_{I}}(\lambda_{I, \mathbf c})}$   factors through  $\CB^{f_I}$.
\end{proof}

\section{A basis theorem for affine Kauffmann category}

The aim of this section is to prove Theorem~\ref{affbasis}, which says that any morphism space  in $\AB$ is free over $\mathbb K$ with infinite rank.

 \begin{Lemma}\label{tangle equa} For any positive integer $m$, define
  \begin{equation}\label{varep}
 \eta_{\ob m}=\begin{tikzpicture}[baseline = 25pt, scale=0.35, color=\clr]
        \draw[-,thick] (0,5) to[out=down,in=left] (2.5,2) to[out=right,in=down] (5,5);
        \draw[-,thick] (0.5,5) to[out=down,in=left] (2.5,2.5) to[out=right,in=down] (4.5,5);
        \draw(1.5,4.5)node{$\cdots$};\draw(3.5,4.5)node{$\cdots$};
        \draw[-,thick] (2,5) to[out=down,in=left] (2.5,4) to[out=right,in=down] (3,5);
        \draw (1,5.5)node {$\ob m$}; \draw (4,5.5)node {$\ob m$};
           \end{tikzpicture} \quad \text{ and }\quad
         \gamma_{\ob m}= \begin{tikzpicture}[baseline = 5pt, scale=0.35, color=\clr]
           \draw[-,thick] (0,0) to[out=up,in=left] (2.5,4) to[out=right,in=up] (5,0);
           \draw[-,thick] (0.2,0) to[out=up,in=left] (2.3,3.5) to[out=right,in=up] (4.8,0);
           \draw[-,thick] (1.8,0) to[out=up,in=left] (2.5,1.5) to[out=right,in=up] (3.2,0);
           \draw(1,0.5)node{$\cdots$};\draw(4,0.5)node{$\cdots$};
          \draw (1,-0.5)node {$\ob m$}; \draw (4,-0.5)node {$\ob m$};
           \end{tikzpicture}.
           \end{equation}
 Then $(1_{\ob m}\otimes \gamma_{\ob m})\circ (\eta_{\ob m}\otimes 1_{\ob m})=(\gamma_{\ob m}\otimes 1_{\ob m})\circ ( 1_{\ob m}\otimes\eta_{\ob m})=1_{\ob m}$ in  $\AB$ and $ \CB^f$.\end{Lemma}
\begin{proof}We have explained that $d=d'$ as morphisms in $\AB$ if $d, d'\in\mathbb{NT}_{  m,  s}$ such that  $d\sim d'$.
So, $d=d'$ as morphisms in $ \CB^f$. The required equalities follow from this observation.
\end{proof}

\begin{Lemma}\label{etannxs}Suppose    $m, s\in \mathbb N$ such that $m\neq 0$ and  $2\mid m+s$. Let $$
\bar\eta_{\ob m}: \Hom_{\AB}(\ob m,\ob {s})\rightarrow  \Hom_{\AB}(\ob 0,\ob {m+s} ), \ \  \bar\gamma_{\ob m}: \Hom_{\AB}(\ob 0,\ob {m+s})\rightarrow  \Hom_{\AB}(\ob m,\ob s)$$
be two  linear maps such that $\bar\eta_{\ob m}(d)=(d\otimes 1_{\ob m}) \circ \eta_{\ob m}$ and $\bar\gamma _{\ob m}(d)=\gamma_{\ob m}  \circ(d\otimes 1_{\ob m})$. Then
\begin{itemize}\item[(1)]$\Hom_{\C}(\ob m,\ob s)\cong \Hom_{\C}(\ob 0,\ob{m+s})$ where  $\mathcal C \in \{\AB, \CB^f\}$,
\item[(2)]$\bar\eta_{\ob m}$  gives a bijection between $\mathbb{NT}_{m,s}/\sim$ and $\mathbb{NT}_{0,m+s}/\sim$, and its inverse is $\bar\gamma_{\ob m}$,
\item[(3)]$\bar\eta_{\ob m}$ induces a bijection between $\mathbb{NT}^a_{m,s}/\sim$ and $ \mathbb{NT}^a_{0,m+s}/\sim$ and its inverse is $\bar\gamma_{\ob m}$.
     \end{itemize}
    \end{Lemma}
\begin{proof}
 Thanks to Lemma~\ref{tangle equa},   $ \bar\eta_{\ob m}$  and $ \bar\gamma_{\ob m}$ are mutually inverse to each other.  So, we have the required isomorphism  when   $\mathcal C =\AB$.

 By \eqref{idea}, it is easy to see that  $\bar\eta_{\ob m}(I(\ob m,\ob s))\subset I(\ob 0,\ob m+\ob s)$ and $\bar\gamma_{\ob m}(I(\ob 0,\ob m+\ob s))\subset I(\ob m,\ob s)$ for any right tensor
ideal $I$ of $\AB$. So, both $\bar \eta_{\ob m}$ and $\bar \gamma_{\ob m}$ induce required $\mathbb K$-isomorphisms in $\CB^f$.

Thanks to \eqref{relation 6} and  Lemma~\ref{selfcrossing}(3), for  any $b\in \mathbb {NT}_{m,s}/\sim$, there is a unique  $b'\in \mathbb {NT}_{0,2r}/\sim$ such that
 $ \bar\eta_{\ob m}(b)=b' \text{ as morphisms in $\AB$}.$
 So,  $\bar\eta_{\ob m}(\mathbb{NT}_{m,s}/\sim)\subseteq \mathbb{NT}_{0,2r}/\sim$. Similarly, we have  $\bar\gamma_{\ob m}(\mathbb{NT}_{0,2r}/\sim)\subseteq \mathbb{NT}_{m,s}/\sim$. Since
$\bar\gamma_{\ob m}$ is the inverse of $ \bar\eta_{\ob m}$, (2) holds. (3) can be proved similarly.
\end{proof}

\begin{Lemma}\label{spa}In $\AB$, we have
 \begin{itemize}
 \item[(1)]\begin{tikzpicture}[baseline = 2.5mm]
	\draw[-,thick,darkblue] (0.28,0) to[out=90,in=-90] (-0.28,.6);
	\draw[-,line width=4pt,white] (-0.28,0) to[out=90,in=-90] (0.28,.6);
	\draw[-,thick,darkblue] (-0.28,0) to[out=90,in=-90] (0.28,.6);
   \node at (-0.26,0.5) {$\color{darkblue}\scriptstyle\bullet$};
\end{tikzpicture}
$=$
\begin{tikzpicture}[baseline = 2.5mm]
	\draw[-,thick,darkblue] (0.28,0) to[out=90,in=-90] (-0.28,.6);
	\draw[-,line width=4pt,white] (-0.28,0) to[out=90,in=-90] (0.28,.6);
	\draw[-,thick,darkblue] (-0.28,0) to[out=90,in=-90] (0.28,.6);
 \node at (0.26,0.1) {$\color{darkblue}\scriptstyle\bullet$};
\end{tikzpicture}
$-z$
\begin{tikzpicture}[baseline = 2.5mm]
	\draw[-,thick,darkblue] (0.18,0) to (0.18,.6);
	\draw[-,thick,darkblue] (-0.18,0) to (-0.18,.6);
\node at (0.18,0.3) {$\color{darkblue}\scriptstyle\bullet$};
\end{tikzpicture}
$+z$
\begin{tikzpicture}[baseline = 2.5mm]\draw[-,thick,darkblue] (0,0.6) to[out=down,in=left] (0.28,0.35) to[out=right,in=down] (0.56,0.6);
  \draw[-,thick,darkblue] (0,0) to[out=up,in=left] (0.28,0.25) to[out=right,in=up] (0.56,0);
  \node at (0.03,0.1) {$\color{darkblue}\scriptstyle\circ$};
         \end{tikzpicture},\quad
         \begin{tikzpicture}[baseline = 2.5mm]
	\draw[-,thick,darkblue] (-0.28,0) to[out=90,in=-90] (0.28,.6);
	\draw[-,line width=4pt,white] (0.28,0) to[out=90,in=-90] (-0.28,.6);
	\draw[-,thick,darkblue] (0.28,0) to[out=90,in=-90] (-0.28,.6);
   \node at (-0.26,0.5) {$\color{darkblue}\scriptstyle\bullet$};
\end{tikzpicture}
$=$
\begin{tikzpicture}[baseline = 2.5mm]
	\draw[-,thick,darkblue] (-0.28,0) to[out=90,in=-90] (0.28,.6);
	\draw[-,line width=4pt,white] (0.28,0) to[out=90,in=-90] (-0.28,.6);
	\draw[-,thick,darkblue] (0.28,0) to[out=90,in=-90] (-0.28,.6);
 \node at (0.26,0.1) {$\color{darkblue}\scriptstyle\bullet$};
\end{tikzpicture}
$-z$
\begin{tikzpicture}[baseline = 2.5mm]
	\draw[-,thick,darkblue] (0.18,0) to (0.18,.6);
	\draw[-,thick,darkblue] (-0.18,0) to (-0.18,.6);
\node at (-0.18,0.3) {$\color{darkblue}\scriptstyle\bullet$};
\end{tikzpicture}
$+z$
\begin{tikzpicture}[baseline = 2.5mm]\draw[-,thick,darkblue] (0,0.6) to[out=down,in=left] (0.28,0.35) to[out=right,in=down] (0.56,0.6);
  \draw[-,thick,darkblue] (0,0) to[out=up,in=left] (0.28,0.25) to[out=right,in=up] (0.56,0);
  \node at (0.53,0.5) {$\color{darkblue}\scriptstyle\circ$};
         \end{tikzpicture}~,
\item[(2)]
\begin{tikzpicture}[baseline = 2.5mm]
	\draw[-,thick,darkblue] (0.28,0) to[out=90,in=-90] (-0.28,.6);
	\draw[-,line width=4pt,white] (-0.28,0) to[out=90,in=-90] (0.28,.6);
	\draw[-,thick,darkblue] (-0.28,0) to[out=90,in=-90] (0.28,.6);
   \node at (-0.26,0.5) {$\color{darkblue}\scriptstyle\circ$};
\end{tikzpicture}
$=$\begin{tikzpicture}[baseline = 2.5mm]
	\draw[-,thick,darkblue] (0.28,0) to[out=90,in=-90] (-0.28,.6);
	\draw[-,line width=4pt,white] (-0.28,0) to[out=90,in=-90] (0.28,.6);
	\draw[-,thick,darkblue] (-0.28,0) to[out=90,in=-90] (0.28,.6);
 \node at (0.26,0.1) {$\color{darkblue}\scriptstyle\circ$};
\end{tikzpicture}
$+z$
\begin{tikzpicture}[baseline = 2.5mm]
	\draw[-,thick,darkblue] (0.18,0) to (0.18,.6);
	\draw[-,thick,darkblue] (-0.18,0) to (-0.18,.6);
\node at (-0.18,0.3) {$\color{darkblue}\scriptstyle\circ$};
\end{tikzpicture}
$-z$
\begin{tikzpicture}[baseline = 2.5mm]\draw[-,thick,darkblue] (0,0.6) to[out=down,in=left] (0.28,0.35) to[out=right,in=down] (0.56,0.6);
  \draw[-,thick,darkblue] (0,0) to[out=up,in=left] (0.28,0.25) to[out=right,in=up] (0.56,0);
  \node at (0.53,0.5) {$\color{darkblue}\scriptstyle\bullet$};
         \end{tikzpicture}, \quad
         \begin{tikzpicture}[baseline = 2.5mm]
	\draw[-,thick,darkblue] (-0.28,0) to[out=90,in=-90] (0.28,.6);
	\draw[-,line width=4pt,white] (0.28,0) to[out=90,in=-90] (-0.28,.6);
	\draw[-,thick,darkblue] (0.28,0) to[out=90,in=-90] (-0.28,.6);
   \node at (-0.26,0.5) {$\color{darkblue}\scriptstyle\circ$};
\end{tikzpicture}
$=$
\begin{tikzpicture}[baseline = 2.5mm]
	\draw[-,thick,darkblue] (-0.28,0) to[out=90,in=-90] (0.28,.6);
	\draw[-,line width=4pt,white] (0.28,0) to[out=90,in=-90] (-0.28,.6);
	\draw[-,thick,darkblue] (0.28,0) to[out=90,in=-90] (-0.28,.6);
 \node at (0.26,0.1) {$\color{darkblue}\scriptstyle\circ$};
\end{tikzpicture}
$+z$
\begin{tikzpicture}[baseline = 2.5mm]
	\draw[-,thick,darkblue] (0.18,0) to (0.18,.6);
	\draw[-,thick,darkblue] (-0.18,0) to (-0.18,.6);
\node at (0.18,0.3) {$\color{darkblue}\scriptstyle\circ$};
\end{tikzpicture}
$-z$
\begin{tikzpicture}[baseline = 2.5mm]\draw[-,thick,darkblue] (0,0.6) to[out=down,in=left] (0.28,0.35) to[out=right,in=down] (0.56,0.6);
  \draw[-,thick,darkblue] (0,0) to[out=up,in=left] (0.28,0.25) to[out=right,in=up] (0.56,0);
  \node at (0.03,0.1) {$\color{darkblue}\scriptstyle\bullet$};
         \end{tikzpicture}~,
\item[(3)]
\begin{tikzpicture}[baseline = 2.5mm]
	\draw[-,thick,darkblue] (0.28,0) to[out=90,in=-90] (-0.28,.6);
	\draw[-,line width=4pt,white] (-0.28,0) to[out=90,in=-90] (0.28,.6);
	\draw[-,thick,darkblue] (-0.28,0) to[out=90,in=-90] (0.28,.6);
   \node at (0.26,0.5) {$\color{darkblue}\scriptstyle\bullet$};
\end{tikzpicture}
$=$
\begin{tikzpicture}[baseline = 2.5mm]
	\draw[-,thick,darkblue] (0.28,0) to[out=90,in=-90] (-0.28,.6);
	\draw[-,line width=4pt,white] (-0.28,0) to[out=90,in=-90] (0.28,.6);
	\draw[-,thick,darkblue] (-0.28,0) to[out=90,in=-90] (0.28,.6);
 \node at (-0.26,0.1) {$\color{darkblue}\scriptstyle\bullet$};
\end{tikzpicture}
$+z$
\begin{tikzpicture}[baseline = 2.5mm]
	\draw[-,thick,darkblue] (0.18,0) to (0.18,.6);
	\draw[-,thick,darkblue] (-0.18,0) to (-0.18,.6);
\node at (0.18,0.3) {$\color{darkblue}\scriptstyle\bullet$};
\end{tikzpicture}
$-z$
\begin{tikzpicture}[baseline = 2.5mm]\draw[-,thick,darkblue] (0,0.6) to[out=down,in=left] (0.28,0.35) to[out=right,in=down] (0.56,0.6);
  \draw[-,thick,darkblue] (0,0) to[out=up,in=left] (0.28,0.25) to[out=right,in=up] (0.56,0);
  \node at (0.53,0.5) {$\color{darkblue}\scriptstyle\bullet$};
         \end{tikzpicture},\quad
         \begin{tikzpicture}[baseline = 2.5mm]
	\draw[-,thick,darkblue] (-0.28,0) to[out=90,in=-90] (0.28,.6);
	\draw[-,line width=4pt,white] (0.28,0) to[out=90,in=-90] (-0.28,.6);
	\draw[-,thick,darkblue] (0.28,0) to[out=90,in=-90] (-0.28,.6);
   \node at (0.26,0.5) {$\color{darkblue}\scriptstyle\bullet$};
\end{tikzpicture}
$=$\begin{tikzpicture}[baseline = 2.5mm]
	\draw[-,thick,darkblue] (-0.28,0) to[out=90,in=-90] (0.28,.6);
	\draw[-,line width=4pt,white] (0.28,0) to[out=90,in=-90] (-0.28,.6);
	\draw[-,thick,darkblue] (0.28,0) to[out=90,in=-90] (-0.28,.6);
 \node at (-0.26,0.1) {$\color{darkblue}\scriptstyle\bullet$};
\end{tikzpicture}
$+z$
\begin{tikzpicture}[baseline = 2.5mm]
	\draw[-,thick,darkblue] (0.18,0) to (0.18,.6);
	\draw[-,thick,darkblue] (-0.18,0) to (-0.18,.6);
\node at (-0.18,0.3) {$\color{darkblue}\scriptstyle\bullet$};
\end{tikzpicture}
$-z$
\begin{tikzpicture}[baseline = 2.5mm]\draw[-,thick,darkblue] (0,0.6) to[out=down,in=left] (0.28,0.35) to[out=right,in=down] (0.56,0.6);
  \draw[-,thick,darkblue] (0,0) to[out=up,in=left] (0.28,0.25) to[out=right,in=up] (0.56,0);
  \node at (0.03,0.1) {$\color{darkblue}\scriptstyle\bullet$};
         \end{tikzpicture}~,
  \item[(4)]
\begin{tikzpicture}[baseline = 2.5mm]
	\draw[-,thick,darkblue] (0.28,0) to[out=90,in=-90] (-0.28,.6);
	\draw[-,line width=4pt,white] (-0.28,0) to[out=90,in=-90] (0.28,.6);
	\draw[-,thick,darkblue] (-0.28,0) to[out=90,in=-90] (0.28,.6);
   \node at (0.26,0.5) {$\color{darkblue}\scriptstyle\circ$};
\end{tikzpicture}
$=$
\begin{tikzpicture}[baseline = 2.5mm]
	\draw[-,thick,darkblue] (0.28,0) to[out=90,in=-90] (-0.28,.6);
	\draw[-,line width=4pt,white] (-0.28,0) to[out=90,in=-90] (0.28,.6);
	\draw[-,thick,darkblue] (-0.28,0) to[out=90,in=-90] (0.28,.6);
 \node at (-0.26,0.1) {$\color{darkblue}\scriptstyle\circ$};
\end{tikzpicture}
$-z$
\begin{tikzpicture}[baseline = 2.5mm]
	\draw[-,thick,darkblue] (0.18,0) to (0.18,.6);
	\draw[-,thick,darkblue] (-0.18,0) to (-0.18,.6);
\node at (-0.18,0.3) {$\color{darkblue}\scriptstyle\circ$};
\end{tikzpicture}
$+z$
\begin{tikzpicture}[baseline = 2.5mm]\draw[-,thick,darkblue] (0,0.6) to[out=down,in=left] (0.28,0.35) to[out=right,in=down] (0.56,0.6);
  \draw[-,thick,darkblue] (0,0) to[out=up,in=left] (0.28,0.25) to[out=right,in=up] (0.56,0);
  \node at (0.03,0.1) {$\color{darkblue}\scriptstyle\circ$};
         \end{tikzpicture},\quad
         \begin{tikzpicture}[baseline = 2.5mm]
	\draw[-,thick,darkblue] (-0.28,0) to[out=90,in=-90] (0.28,.6);
	\draw[-,line width=4pt,white] (0.28,0) to[out=90,in=-90] (-0.28,.6);
	\draw[-,thick,darkblue] (0.28,0) to[out=90,in=-90] (-0.28,.6);
   \node at (0.26,0.5) {$\color{darkblue}\scriptstyle\circ$};
\end{tikzpicture}
$=$
\begin{tikzpicture}[baseline = 2.5mm]
	\draw[-,thick,darkblue] (-0.28,0) to[out=90,in=-90] (0.28,.6);
	\draw[-,line width=4pt,white] (0.28,0) to[out=90,in=-90] (-0.28,.6);
	\draw[-,thick,darkblue] (0.28,0) to[out=90,in=-90] (-0.28,.6);
 \node at (-0.26,0.1) {$\color{darkblue}\scriptstyle\circ$};
\end{tikzpicture}
$-z$
\begin{tikzpicture}[baseline = 2.5mm]
	\draw[-,thick,darkblue] (0.18,0) to (0.18,.6);
	\draw[-,thick,darkblue] (-0.18,0) to (-0.18,.6);
\node at (0.18,0.3) {$\color{darkblue}\scriptstyle\circ$};
\end{tikzpicture}
$+z$
\begin{tikzpicture}[baseline = 2.5mm]\draw[-,thick,darkblue] (0,0.6) to[out=down,in=left] (0.28,0.35) to[out=right,in=down] (0.56,0.6);
  \draw[-,thick,darkblue] (0,0) to[out=up,in=left] (0.28,0.25) to[out=right,in=up] (0.56,0);
  \node at (0.53,0.5) {$\color{darkblue}\scriptstyle\circ$};
         \end{tikzpicture}~.
  \end{itemize}
 \end{Lemma}
\begin{proof} (1)-(4) can be checked via  \eqref{relation 4}-\eqref{relation 6}, Lemma~\ref{selfcrossing}(3),   (S) and (RII), directly.
\end{proof}
\begin{Defn}\label{d1}Let $\hat{\mathbb{T}}_{m,s}$ be the subset of $\mathbb{T}_{m,s}$ such that  each dotted tangle diagram in  $\hat{\mathbb{T}}_{m,s}$ satisfies the following conditions:  \begin{itemize}
 \item  [(1)]neither a strand crosses  itself nor two
strands cross each other more than once,
 \item [(2)] Definition~\ref{D:N.O. dotted  OBC tangle diagram}(a)-(c),
 \item [(3)] whenever a dot appears on a stand, it is near  the endpoints of the strand,
 \item [(4)] Definition~\ref{D:N.O. dotted  OBC tangle diagram}(d).
 \end{itemize}
 \end{Defn}
\begin{Lemma}\label{sim} Any $d$ in $\mathbb{T}_{m,s}$ can be written as a linear combination of elements in $\hat{\mathbb{T}}_{m,s}$.
 \end{Lemma}
 \begin{proof}

 For  any  $d\in \mathbb{T}_{m,s}$ which does not satisfies Definition~\ref{d1}(1),  either there is a
  strand crossing  itself or there are two
strands crossing  each other more than once. Thanks to  Lemma~\ref{spa}(1)-(4), one can slide dots ($\bullet$'s or $\circ$'s) along each crossing in $d$ modulo
diagrams with fewer crossings. By \eqref{relation 6} and  Lemma~\ref{selfcrossing}(3), one can also  slide dots ($\bullet$'s or $\circ$'s) along each cup and each cap. By induction on the number of crossings, we can assume that there is no dots in the local area on which either there is a self-crossing strand or  there are  two
strands crossing  each other more than once. Using Theorem \ref{bask} and the monoidal  functor  in Lemma~\ref{ktau}(2), we see that $d$ can be written as a linear combination of dotted tangle diagrams  which satisfy Definition~\ref{d1}(1).

Now, we assume that $d\in \mathbb{T}_{m,s}$ satisfying Definition~\ref{d1}(1). Then  all loops of $d$ are crossing free. Otherwise, there is  a crossing which appears on a loop such that  either there is a strand crossing  itself or there are two
strands crossing each other more than once.
 Thanks to  \eqref{relation 6}, (L) and Lemma~\ref{addmi},  we can  assume all loops of $d$ satisfy Definition \ref{D:N.O. dotted  OBC tangle diagram}(b)-(c).  In order to write $d$ as a linear combination of dotted tangle
 diagrams satisfying  Definition~\ref{d1}(1)-(2), it is enough to move a loop containing  $k$ $\bullet$ to the left of a vertical line, where $k$ is a positive integer.  In fact, applying Lemma~\ref{spa} repeatedly together with  \eqref{relation 6} and Lemma~\ref{selfcrossing}(3) yields the following equations:
\begin{equation}\begin{aligned}\label{free}\begin{tikzpicture}[baseline = 25pt, scale=0.35, color=\clr]
\draw[-,thick] (-2.5,4) to[out=up,in=down] (-2.5,2);
\draw[-,thick]  (-1,4) to[out=left,in=up] (-2,3) to[out=down,in=left] (-1,2)to[out=right,in=down] (0,3) to[out=up,in=right] (-1,4);
 \node at (-2,3) {$\color{darkblue}\scriptstyle\bullet$};
\node at (-2.1,4) {$\color{darkblue}\scriptstyle k$};
   \end{tikzpicture}&=\begin{tikzpicture}[baseline = -5mm]
	\draw[-,thick,darkblue] (0.28+0.56+0.56,0) to[out=90, in=0] (0+0.56+0.56,0.3);
	\draw[-,thick,darkblue] (0+0.56+0.56,0.3) to[out = 180, in = 90] (-0.28+0.56+0.56,0);
	
\draw[-,thick,darkblue] (0.28,-0.45) to[out=60,in=-90] (0.28+0.56,0);
	\draw[-,line width=4pt,white](0.28+0.56,-0.45) to[out=120,in=-90] (0.28,0);
	\draw[-,thick,darkblue](0.28+0.56,-0.45) to[out=120,in=-90] (0.28,0);
  \draw[-,thick,darkblue] (0.28+0.56,-0.9) to[out=120,in=-90] (0.28,-.45);
	\draw[-,line width=4pt,white] (0.28,-0.9) to[out=60,in=-90] (0.28+0.56,-.45);
	\draw[-,thick,darkblue] (0.28,-0.9) to[out=60,in=-90] (0.28+0.56,-.45);
    \draw[-,thick,darkblue] (-0.28+0.56+0.56,-0.9) to[out=down,in=left] (0+0.56+0.56,-1.2) to[out=right,in=down] (0.28+0.56+0.56,-0.9);
    \draw[-,thick,darkblue] (-0.28+0.56+0.56+0.56,0) to (-0.28+0.56+0.56+0.56,-0.9);
    \node at(-0.28+0.56+0.56,0) {$\color{darkblue}\scriptstyle\bullet$};\node at (-0.48+0.56+0.56,0) {$\color{darkblue}\scriptstyle k$};
\end{tikzpicture}\\&=\begin{tikzpicture}[baseline = -5mm]
	\draw[-,thick,darkblue] (0.28+0.56+0.56,0) to[out=90, in=0] (0+0.56+0.56,0.3);
	\draw[-,thick,darkblue] (0+0.56+0.56,0.3) to[out = 180, in = 90] (-0.28+0.56+0.56,0);
	
\draw[-,thick,darkblue] (0.28,-0.45) to[out=60,in=-90] (0.28+0.56,0);
	\draw[-,line width=4pt,white](0.28+0.56,-0.45) to[out=120,in=-90] (0.28,0);
	\draw[-,thick,darkblue](0.28+0.56,-0.45) to[out=120,in=-90] (0.28,0);
  \draw[-,thick,darkblue] (0.28+0.56,-0.9) to[out=120,in=-90] (0.28,-.45);
	\draw[-,line width=4pt,white] (0.28,-0.9) to[out=60,in=-90] (0.28+0.56,-.45);
	\draw[-,thick,darkblue] (0.28,-0.9) to[out=60,in=-90] (0.28+0.56,-.45);
    \draw[-,thick,darkblue] (-0.28+0.56+0.56,-0.9) to[out=down,in=left] (0+0.56+0.56,-1.2) to[out=right,in=down] (0.28+0.56+0.56,-0.9);
    \draw[-,thick,darkblue] (-0.28+0.56+0.56+0.56,0) to (-0.28+0.56+0.56+0.56,-0.9);
    \node at(0.28,-0.45) {$\color{darkblue}\scriptstyle\bullet$};\node at (0.28,-0.25) {$\color{darkblue}\scriptstyle k$};
\end{tikzpicture}~+~z\sum_{i=1}^{k}\left(\begin{tikzpicture}[baseline = -0.5mm]
	\draw[-,thick,darkblue] (0,0.6) to (0,0.3);
	\draw[-,thick,darkblue] (0.5,0) to [out=90,in=0](.3,0.2);
	\draw[-,thick,darkblue] (0,-0.3) to (0,-0.6);
	\draw[-,thick,darkblue] (0.3,-0.2) to [out=0,in=-90](.5,0);
	\draw[-,thick,darkblue] (0,0.3) to [out=-90,in=180] (.3,-0.2);
	\draw[-,line width=4pt,white] (0.3,.2) to [out=180,in=90](0,-0.3);
	\draw[-,thick,darkblue] (0.3,.2) to [out=180,in=90](0,-0.3);
\node at(0,0.4) {$\color{darkblue}\scriptstyle\bullet$};\node at (-0.2,0.4) {$\color{darkblue}\scriptstyle i$};\node at(0.15,0.1) {$\color{darkblue}\scriptstyle\bullet$};\node at (0.3,0.3) {$\color{darkblue}\scriptstyle k-i$};
\end{tikzpicture}~
-~\begin{tikzpicture}[baseline = -3mm]
	\draw[-,thick,darkblue] (0.28,-.6) to[out=120,in=-90] (-0.28,0);
	\draw[-,thick,darkblue] (0.28,0) to[out=90, in=0] (0,0.2);
	\draw[-,thick,darkblue] (0,0.2) to[out = 180, in = 90] (-0.28,0);
	\draw[-,line width=4pt,white] (-0.28,-.6) to[out=60,in=-90] (0.28,0);
	\draw[-,thick,darkblue] (-0.28,-.6) to[out=60,in=-90] (0.28,0);
\draw[-,thick,darkblue] (0.28,-0.6) to[out=down,in=left] (0.56,-.9) to[out=right,in=down] (0.28+0.56,-.6);
 \draw[-,thick,darkblue] (0.56+0.28,-0.6) to(0.56+0.28,0.2);
     \node at (0.56+0.28,-0.6) {$\color{darkblue}\scriptstyle\circ$};
      \node at (0.56,-0.4) {$\color{darkblue}\scriptstyle k-i$};
      \node at (-0.28,0) {$\color{darkblue}\scriptstyle\bullet$};
      \node at (-0.18,0) {$\color{darkblue}\scriptstyle i$};
\end{tikzpicture}\right)\\
&=\begin{tikzpicture}[baseline = -5mm]
	\draw[-,thick,darkblue] (0.28+0.56+0.56,0) to[out=90, in=0] (0+0.56+0.56,0.3);
	\draw[-,thick,darkblue] (0+0.56+0.56,0.3) to[out = 180, in = 90] (-0.28+0.56+0.56,0);
	
\draw[-,thick,darkblue] (0.28,-0.45) to[out=60,in=-90] (0.28+0.56,0);
	\draw[-,line width=4pt,white](0.28+0.56,-0.45) to[out=120,in=-90] (0.28,0);
	\draw[-,thick,darkblue](0.28+0.56,-0.45) to[out=120,in=-90] (0.28,0);
  \draw[-,thick,darkblue] (0.28+0.56,-0.9) to[out=120,in=-90] (0.28,-.45);
	\draw[-,line width=4pt,white] (0.28,-0.9) to[out=60,in=-90] (0.28+0.56,-.45);
	\draw[-,thick,darkblue] (0.28,-0.9) to[out=60,in=-90] (0.28+0.56,-.45);
    \draw[-,thick,darkblue] (-0.28+0.56+0.56,-0.9) to[out=down,in=left] (0+0.56+0.56,-1.2) to[out=right,in=down] (0.28+0.56+0.56,-0.9);
    \draw[-,thick,darkblue] (-0.28+0.56+0.56+0.56,0) to (-0.28+0.56+0.56+0.56,-0.9);
    \node at(0.28,-0.45) {$\color{darkblue}\scriptstyle\bullet$};\node at (0.28,-0.25) {$\color{darkblue}\scriptstyle k$};
\end{tikzpicture}+z\sum_{i=1}^{k}\left(\begin{tikzpicture}[baseline = -0.5mm]
	\draw[-,thick,darkblue] (0,0.6) to (0,0.3);
	\draw[-,thick,darkblue] (0.5,0) to [out=90,in=0](.3,0.2);
	\draw[-,thick,darkblue] (0,-0.3) to (0,-0.6);
	\draw[-,thick,darkblue] (0.3,-0.2) to [out=0,in=-90](.5,0);
	\draw[-,thick,darkblue] (0,0.3) to [out=-90,in=180] (.3,-0.2);
	\draw[-,line width=4pt,white] (0.3,.2) to [out=180,in=90](0,-0.3);
	\draw[-,thick,darkblue] (0.3,.2) to [out=180,in=90](0,-0.3);
\node at(0,0.4) {$\color{darkblue}\scriptstyle\bullet$};\node at (-0.2,0.4) {$\color{darkblue}\scriptstyle i$};\node at(0,-0.3) {$\color{darkblue}\scriptstyle\bullet$};\node at (-0.3,-0.3) {$\color{darkblue}\scriptstyle k-i$};
\end{tikzpicture}+z\sum_{j=1}^{k-i}\left(\begin{tikzpicture}[baseline = 25pt, scale=0.35, color=\clr]
\draw[-,thick] (-3,4) to[out=up,in=down] (-3,2);
\draw[-,thick]  (-1,4) to[out=left,in=up] (-2,3) to[out=down,in=left] (-1,2)to[out=right,in=down] (0,3) to[out=up,in=right] (-1,4);
 \node at (-2,3) {$\color{darkblue}\scriptstyle\bullet$};
\node at (-1.5,3.4) {$\color{darkblue}\scriptstyle j$};
 \node at (-3,3) {$\color{darkblue}\scriptstyle\bullet$};
\node at (-4.2,3) {$\color{darkblue}\scriptstyle k-j$};
   \end{tikzpicture}-\begin{tikzpicture}[baseline = 25pt, scale=0.35, color=\clr]
\draw[-,thick] (1,4) to[out=up,in=down] (1,2);
 \node at (1,3) {$\color{darkblue}\scriptstyle\bullet$};
\node at (2,3) {$\color{darkblue}\scriptstyle k-2j$};
   \end{tikzpicture}\right)\right)
\\&\ \ \ \  -z\left(\begin{tikzpicture}[baseline = -3mm]
	\draw[-,thick,darkblue] (0.28,-.6) to[out=120,in=-90] (-0.28,0);
	\draw[-,thick,darkblue] (0.28,0) to[out=90, in=0] (0,0.2);
	\draw[-,thick,darkblue] (0,0.2) to[out = 180, in = 90] (-0.28,0);
	\draw[-,line width=4pt,white] (-0.28,-.6) to[out=60,in=-90] (0.28,0);
	\draw[-,thick,darkblue] (-0.28,-.6) to[out=60,in=-90] (0.28,0);
\draw[-,thick,darkblue] (0.28,-0.6) to[out=down,in=left] (0.56,-.9) to[out=right,in=down] (0.28+0.56,-.6);
 \draw[-,thick,darkblue] (0.56+0.28,-0.6) to(0.56+0.28,0.2);
     \node at (0.56+0.28,-0.6) {$\color{darkblue}\scriptstyle\circ$};
      \node at (0.56,-0.4) {$\color{darkblue}\scriptstyle k-i$};
      \node at (0.28,-0.6) {$\color{darkblue}\scriptstyle\bullet$};
      \node at (0.12,-0.6) {$\color{darkblue}\scriptstyle i$};
\end{tikzpicture}+z\sum_{j=1}^i\left(\begin{tikzpicture}[baseline = 25pt, scale=0.35, color=\clr]
\draw[-,thick] (1,4) to[out=up,in=down] (1,2);
\draw[-,thick]  (-1,4) to[out=left,in=up] (-2,3) to[out=down,in=left] (-1,2)to[out=right,in=down] (0,3) to[out=up,in=right] (-1,4);
 \node at (-2,3) {$\color{darkblue}\scriptstyle\bullet$};
\node at (-0.8,3) {$\color{darkblue}\scriptstyle i-j$};
 \node at (1,3) {$\color{darkblue}\scriptstyle\circ$};
\node at (2.3,2.4) {$\color{darkblue}\scriptstyle k+j-i$};
   \end{tikzpicture}-\begin{tikzpicture}[baseline = 25pt, scale=0.35, color=\clr]
\draw[-,thick] (1,4) to[out=up,in=down] (1,2);
 \node at (1,3) {$\color{darkblue}\scriptstyle\bullet$};
\node at (2.8,3) {$\color{darkblue}\scriptstyle 2i-2j-k$};
   \end{tikzpicture}  \right)\right)
   \\&=\begin{tikzpicture}[baseline = 25pt, scale=0.35, color=\clr]
\draw[-,thick] (1,4) to[out=up,in=down] (1,2);
\draw[-,thick]  (-1,4) to[out=left,in=up] (-2,3) to[out=down,in=left] (-1,2)to[out=right,in=down] (0,3) to[out=up,in=right] (-1,4);
 \node at (-2,3) {$\color{darkblue}\scriptstyle\bullet$};
\node at (-1.5,3) {$\color{darkblue}\scriptstyle k$};
   \end{tikzpicture}~+~z\sum_{i=1}^{k}\left(\delta \begin{tikzpicture}[baseline = -0.5mm]
	\draw[-,thick,darkblue] (0,-0.3) to (0,0.3);
	
\node at(0,0) {$\color{darkblue}\scriptstyle\bullet$};\node at (-0.2,0) {$\color{darkblue}\scriptstyle k$};
\end{tikzpicture}+z\sum_{j=1}^{k-i}\left( \begin{tikzpicture}[baseline = 25pt, scale=0.35, color=\clr]
\draw[-,thick] (-3,4) to[out=up,in=down] (-3,2);
\draw[-,thick]  (-1,4) to[out=left,in=up] (-2,3) to[out=down,in=left] (-1,2)to[out=right,in=down] (0,3) to[out=up,in=right] (-1,4);
 \node at (-2,3) {$\color{darkblue}\scriptstyle\bullet$};
\node at (-1.5,3.4) {$\color{darkblue}\scriptstyle j$};
 \node at (-3,3) {$\color{darkblue}\scriptstyle\bullet$};
\node at (-4.2,3) {$\color{darkblue}\scriptstyle k-j$};
   \end{tikzpicture}-\begin{tikzpicture}[baseline = 25pt, scale=0.35, color=\clr]
\draw[-,thick] (1,4) to[out=up,in=down] (1,2);
 \node at (1,3) {$\color{darkblue}\scriptstyle\bullet$};
\node at (2.1,3) {$\color{darkblue}\scriptstyle k-2j$};
   \end{tikzpicture}\right)\right)\\& -z\sum_{i=1}^{k}\left(\delta^{-1}~\begin{tikzpicture}[baseline = 25pt, scale=0.35, color=\clr]
\draw[-,thick] (1,4) to[out=up,in=down] (1,2);
 \node at (1,3) {$\color{darkblue}\scriptstyle\bullet$};
\node at (2,3) {$\color{darkblue}\scriptstyle 2i-k$};
   \end{tikzpicture}+z\sum_{j=1}^{i}\left(\begin{tikzpicture}[baseline = 25pt, scale=0.35, color=\clr]
\draw[-,thick] (1,4) to[out=up,in=down] (1,2);
\draw[-,thick]  (-1,4) to[out=left,in=up] (-2,3) to[out=down,in=left] (-1,2)to[out=right,in=down] (0,3) to[out=up,in=right] (-1,4);
 \node at (-2,3) {$\color{darkblue}\scriptstyle\bullet$};
\node at (-0.8,3) {$\color{darkblue}\scriptstyle i-j$};
 \node at (1,3) {$\color{darkblue}\scriptstyle\circ$};
\node at (2.3,2.4) {$\color{darkblue}\scriptstyle k+j-i$};
   \end{tikzpicture}-\begin{tikzpicture}[baseline = 25pt, scale=0.35, color=\clr]
\draw[-,thick] (1,4) to[out=up,in=down] (1,2);
 \node at (1,3) {$\color{darkblue}\scriptstyle\bullet$};
\node at (2.8,3) {$\color{darkblue}\scriptstyle 2i-2j-k$};
   \end{tikzpicture}\right)\right),
   \end{aligned}\end{equation}
   where the last equation follows from Lemma~\ref{selfcrossing}(1), (T), (RII) and
   \begin{equation} \label{kkkk} \begin{tikzpicture}[baseline = -1mm]
	\draw[-,thick,darkblue] (0.28+0.56+0.56,0) to[out=90, in=0] (0+0.56+0.56,0.3);
	\draw[-,thick,darkblue] (0+0.56+0.56,0.3) to[out = 180, in = 90] (-0.28+0.56+0.56,0);
	
\draw[-,thick,darkblue] (0.28,-0.45) to[out=60,in=-90] (0.28+0.56,0);
	\draw[-,line width=4pt,white](0.28+0.56,-0.45) to[out=120,in=-90] (0.28,0);
	\draw[-,thick,darkblue](0.28+0.56,-0.45) to[out=120,in=-90] (0.28,0);
 \draw[-,thick,darkblue] (0.28+0.56+0.56,-0.45) to (0.28+0.56+0.56,0);\end{tikzpicture}
	=\begin{tikzpicture}[baseline = -1mm]
	\draw[-,thick,darkblue] (0.28+0.56+0.56,0) to[out=90, in=0] (0+0.56+0.56,0.3);
	\draw[-,thick,darkblue] (0+0.56+0.56,0.3) to[out = 180, in = 90] (-0.28+0.56+0.56,0);
	
\draw[-,thick,darkblue] (0.28+0.56+0.56+0.56,-0.45) to[out=120,in=-90] (0.28+0.56+0.56,0);
	\draw[-,line width=4pt,white](0.28+0.56+0.56,-0.45) to[out=60,in=-90] (0.28+0.56+0.56+0.56,0);
	\draw[-,thick,darkblue](0.28+0.56+0.56,-0.45) to[out=60,in=-90] (0.28+0.56+0.56+0.56,0);
    \draw[-,thick,darkblue] (0.28+0.56,-0.45) to (0.28+0.56,-0);
\end{tikzpicture}, \begin{tikzpicture}[baseline = -9mm]
	
	
  \draw[-,thick,darkblue] (0.28+0.56,-0.9) to[out=120,in=-90] (0.28,-.45);
	\draw[-,line width=4pt,white] (0.28,-0.9) to[out=60,in=-90] (0.28+0.56,-.45);
	\draw[-,thick,darkblue] (0.28,-0.9) to[out=60,in=-90] (0.28+0.56,-.45);
    \draw[-,thick,darkblue] (-0.28+0.56+0.56,-0.9) to[out=down,in=left] (0+0.56+0.56,-1.2) to[out=right,in=down] (0.28+0.56+0.56,-0.9);
    \draw[-,thick,darkblue] (-0.28+0.56+0.56+0.56,-0.45) to (-0.28+0.56+0.56+0.56,-0.9);
\end{tikzpicture}= \begin{tikzpicture}[baseline = -9mm]
	
	
  \draw[-,thick,darkblue] (0.28+0.56+0.56,-0.9) to[out=60,in=-90] (0.28+0.56+0.56+0.56,-.45);
	\draw[-,line width=4pt,white] (0.28+0.56+0.56+0.56,-0.9) to[out=120,in=-90] (0.28+0.56+0.56,-.45);
	\draw[-,thick,darkblue] (0.28+0.56+0.56+0.56,-0.9) to[out=120,in=-90] (0.28+0.56+0.56,-.45);
    \draw[-,thick,darkblue] (-0.28+0.56+0.56,-0.9) to[out=down,in=left] (0+0.56+0.56,-1.2) to[out=right,in=down] (0.28+0.56+0.56,-0.9);
    \draw[-,thick,darkblue] (-0.28+0.56+0.56,-0.45) to (-0.28+0.56+0.56,-0.9);
\end{tikzpicture}. \end{equation}
 We remark that \eqref{kkkk} follows from \eqref{relation 1}-\eqref{relation 3}. Now, any $d$ can be written as a linear combination of dotted tangle diagrams satisfying Definition~\ref{d1}(1)--(2) since no new crossing appears when we  move a loop with $k$ $\bullet$ to the left of a vertical
line.

Now, we assume that  $d\in \mathbb{T}_{m,s}$ satisfying Definition~\ref{d1}(1)-(2). Using Lemmas~\ref{spa},\ref{selfcrossing}(3), \eqref{relation 6} together with induction on the number of crossings, we see that $d$ can be written as a linear combination of dotted tangle diagrams satisfying  Definition~\ref{d1}(1)-(3) since neither new loops nor crossing occurs when we apply the previous results to move dots along a strand.

Finally, we can assume that $d$ satisfies    Definition \ref{d1}(1)-(3).
 Thanks to   Theorem~\ref{bask} and the functor in Lemma~\ref{ktau}(2), $d$ can be written as a linear combination of dotted tangle diagrams in $\hat{\mathbb{T}}_{m,s}$.
  \end{proof}

\begin{Prop}\label{span1}The $\mathbb K$-module ${\Hom}_{\AB}(\ob m, \ob s)$ is spanned by $\mathbb {NT}_{m, s}/\sim$.
\end{Prop}
 
 \begin{proof} Thanks to  Lemma \ref{sim},  it is enough to verify that any $d\in\hat{\mathbb{T}}_{m,s}$  is a linear combination of elements in $\mathbb{NT}_{m,s}$.
 If $d$ has no crossing, then the result follows from \eqref{relation 4},\eqref{relation 6} and  Lemma~\ref{selfcrossing}(3).
 Suppose that there are some crossings on $d$. Thanks to   Lemmas~\ref{spa}, \ref{selfcrossing}(3),
 \eqref{relation 6}, and \eqref{relation 4}, there is a ${d_1}\in {\mathbb{NT}}_{m,s}$ such that $\hat d=\hat{d_1}$ and  $d=d_1$  up to some dotted tangle diagrams in $\hat{\mathbb{T}}_{m,s}$ with fewer crossings than that of  $d$. Now the result follows from induction on the number of crossings.
 \end{proof}
 \begin{Lemma}\label{XX} Suppose $r$ is a positive integer. For  all admissible $i, j$,  define
$$X_i1_{\ob r} =1_{i-1}\otimes X\otimes 1_{r-i}, ~ 1_{\ob r} U_j=1_{j-1}\otimes U\otimes 1_{r-j-1} \text{ and  $Z_j1_{\ob r} =1_{j-1}\otimes Z\otimes 1_{r-j-1}$,}$$  where $Z\in\{A,T,T^{-1}\}$. Then
$X_i1_{\ob r}  =T_{i-1}1_{\ob r} X_{i-1}1_{\ob r} T_{i-1}1_{\ob r}$, $2\le i\le r$. Further, $X_i1_{\ob r} X_j1_{\ob r}=X_j1_{\ob r}X_i1_{\ob r}$ for all admissible $i$ and $j$.  \end{Lemma}
\begin{proof}The first assertion follows from \eqref{relation 5} and the last follows from the interchange law.
\end{proof}

Suppose $d\in \mathbb{NT}_{m, s}$ and $m+s$ is even. In section~1, we have labelled the endpoints of $\hat d$  at  the bottom (resp., top) row  as   $1, 2, \ldots, m$
(resp., $\bar 1, \bar 2,   \ldots,  \bar s$) from the left to the right and $i< i+1<\bar j<\overline{j-1}$ for all admissible $i$ and $j$.  We also have $conn(\hat d)=\{(i_l, j_l)| 1\leq l\leq \frac{m+s}{2}\}$ such that  $i_l<j_l$ and  $  i_k<i_{k+1}$ and  for all admissible $k, l $.

\begin{Defn}\label{dots}For any  $d\in \mathbb{NT}_{m, s}$ such that $conn(\hat d)=\{(i_l, j_l)| 1\leq l\leq \frac{m+s}{2}\}$ and any  $i\in\{1, 2, \ldots, m, \bar 1, \bar 2,   \ldots,  \bar s\}$, define $b_{d, i}=j$ if there are $j$ ``$\bullet$" near the $i$th endpoint.\end{Defn}
Note that $b_{d, i}\in\mathbb Z$.  Thanks to Definition~\ref{D:N.O. dotted  OBC tangle diagram}, $b_{d, j_l}=0$ for all admissible $l$. Later on, we simply denote $X_i1_{\ob r}$ by $X_i$ etc if we know $r$ from the context.
Then any $d\in \mathbb{NT}_{m,s}$ is of form

\begin{equation} \label {form}
d=\prod_{j=1}^\infty {\Delta_j}^{i_j} \cdot \prod_{l=s}^1 {Y_l}^{b_{d, \bar l}} \cdot  \hat{d}\cdot
\prod_{j=m}^1 {X_j}^{b_{d, j}}  \end{equation}
where $Y_l$ is  $X_l\in\End_{\AB}(\ob s, \ob s)$ and $i_j$'s $\in \mathbb N$ such that only finite number of $i_j$'s are non-zero.

For any positive integer $b\geq 2$,
let $\mathfrak S_b$ be the symmetric group on $b$ letters $1, 2, \ldots, b$. Then $\mathfrak S_b$ is generated by basic transpositions $r_{i}=(i, i+1)$, $1\le i\le b-1$. Now,
$\mathfrak S_b$ acts on the  right of  $\mathbb N^b$ via place permutation. More explicitly,
$ \mathbf i\sigma =(i_{\sigma(1)}, i_{\sigma(2)}, \ldots, i_{\sigma(b)})$ for all $\sigma\in \mathfrak S_b$ and
  $\mathbf{i}=(i_1,i_2,\ldots, i_b)\in \mathbb{N}^b$.
 Define  $r_{k, l}=r_k r_{k+1}\cdots r_{l-1}$ if $k<l$ and $r_{k, k}=1$ and $r_{k, l}=r_{k-1} r_{k-2}\ldots r_{l}$ if $k>l$. If $r_{i_1}r_{i_2}\cdots r_{i_k}$ is a reduced expression of $w\in\mathfrak S_b$, define \begin{equation}\label{c5}T_w=T_{i_1}\cdots T_{i_k},\quad \text{ and }\quad   T^{inv}_w=T^{-1}_{i_1}\cdots T^{-1}_{i_k}\end{equation} in $\End_{\AB}(\ob b)$,  where $T_j$ is given in Lemma \ref{XX}. It is well-known that  both $T_w$ and $T^{inv}_w$ are independent of a reduced expression of $w$. We denote $T_{r_{k,l}}$(resp., $T_{r_{k,l}}^{inv}$ ) by $T_{k,l}$ (resp., $T_{k,l}^{inv}$).
 Let $\mathfrak{B}_b$ be the subgroup of $\mathfrak{S}_{2b}$ generated by $\{r_{2b-2i+2}r_{2b-2i+1}r_{2b-2i+3}r_{2b-2i+2}\mid 2<i<b\}$ and $\{r_{2b-1}\}$. Define \begin{equation}\label{c13}\mathcal {D}_{b,2b}=\left\{
r_{1,i_1}r_{2,j_1}\cdots r_{2b-1,i_b}r_{2b,j_b}\bigg|\begin{array}{lll}1\leq i_1<\ldots<i_r\leq 2b\\1\leq i_k<j_k\leq 2k;1\leq k\leq b\\\end{array}\right\}.
\end{equation} Thanks to \cite[Lemma~4.3]{RX}, $\mathcal{D}_{b,2b}$ is a complete set of right coset representatives for $\mathfrak{B}_{b}$ in $\mathfrak{S}_{2b}$ and $\sharp\mathcal{D}_{b,2b}=(2b-1)!!$.

\begin{Prop}\label{cspan3}    Suppose $d\in\bar{\mathbb{NT}}^1_{2s, 0}/\sim$  and   $conn(d)=\{(i_l, j_l)| 1\leq l\leq s\}$.
As morphisms in $\AB$,  $d=\lcap^{\otimes s} T^{inv}_w$,    where   $w= \vec{ \prod}_{t=1}^sr_{2t-1,i_t}r_{2t,k_t}\in \mathcal{D}_{s,2s}$  for some admissible $k_t$'s.
 \end{Prop}

 \begin{proof} Suppose $d_0=d$. Define   $d_1=d_0T_{j_s,2s}T_{i_s,2s-1}$ and write  $conn(d_{1})=\{(i^{1}_l, j^{1}_l)| 1\leq l\leq s\}$. Obviously, $i_t^{1}=i_t$ if $t<s$ and $(i^1_s, j^1_s)=(2s-1,2s)$.
In general, write $conn(d_{t-1})=\{(i^{t-1}_l, j^{t-1}_l)| 1\leq l\leq s\}$ and
define $$d_{t}=d_{t-1}T_{j^{t-1}_{s-t+1},2(s-t+1)}T_{i^{t-1}_{s-t+1},2(s-t)+1},$$ $1\leq t\leq s$.  Then $conn(d_t)=\{(i_k, j^{t}_k),(2l+1, 2l+2)| 1\leq k\leq s-t, s-t\leq l<s\}$. In particular $i_t<j^{s-t}_t\leq 2t$ for any $1\leq t\leq s$.
 So, both $d_s$ and $\lcap^{\otimes s}$ are totally descending tangle diagrams with the same connector. By  \eqref{bbb},
 $d_s=\lcap^{\otimes s}$ as morphisms in $\FT$ and hence in $\B$. Using the functor in Lemma~\ref{ktau}(2), we see that $d_s=\lcap^{\otimes s}$ as morphisms in $\AB$.
 For example: $$
d_0=d=\begin{tikzpicture}[baseline = 3mm]\draw[-,thick,darkblue] (-.2,0) to[out=up,in=left] (.5,.6) to[out=right,in=up] (1.1,0);
         \draw[-,thick,darkblue] (.3,0) to[out=up,in=left] (.6,.4) to[out=right,in=up] (.9,0);
       \draw[-,line width=4pt,white] (0,0) to[out=up,in=left] (.3,.4) to[out=right,in=up] (.6,0);
       \draw[-,thick,darkblue]  (0,0) to[out=up,in=left] (.3,.4) to[out=right,in=up] (.6,0);\end{tikzpicture},\quad
       d_1=\begin{tikzpicture}[baseline = -5mm]
\draw[-,thick,darkblue] (-.56,0) to[out=up,in=left] (.84,.6) to[out=right,in=up] (2.24,0);
         \draw[-,thick,darkblue] (.56,0) to[out=up,in=left] (1.12,.4) to[out=right,in=up] (1.68,0);
       \draw[-,line width=4pt,white] (0,0) to[out=up,in=left] (.56,.4) to[out=right,in=up] (1.12,0);
       \draw[-,thick,darkblue]  (0,0) to[out=up,in=left] (.56,.4) to[out=right,in=up] (1.12,0);\draw[-,thick,darkblue] (2.24,-.5) to[out=90,in=-90] (1.68,0);
\draw[-,thick,darkblue] (1.68,-.5) to[out=90,in=-90] (2.24,0);
\draw[-,line width=4pt,white]  (1.68,-.5) to[out=90,in=-90] (2.24,0);
	\draw[-,thick,darkblue] (1.68,-.5) to[out=90,in=-90] (2.24,0);
\draw[-,thick,darkblue] (2.24-1.12,-.5) to[out=90,in=-90] (1.68-1.12,0);
\draw[-,thick,darkblue] (1.68-1.12,-.5) to[out=90,in=-90] (2.24-1.12,0);
\draw[-,line width=4pt,white]  (1.68-1.12,-.5) to[out=90,in=-90] (2.24-1.12,0);
	\draw[-,thick,darkblue] (1.68-1.12,-.5) to[out=90,in=-90] (2.24-1.12,0);
\draw[-,thick,darkblue] (2.24-0.56,-.5-.5) to[out=90,in=-90] (1.68-0.56,0-.5);
\draw[-,thick,darkblue] (1.68-0.56,-.5-.5) to[out=90,in=-90] (2.24-0.56,0-.5);
\draw[-,line width=4pt,white]  (1.68-0.56,-.5-.5) to[out=90,in=-90] (2.24-0.56,0-.5);
	\draw[-,thick,darkblue] (1.68-0.56,-.5-.5) to[out=90,in=-90] (2.24-0.56,0-.5);
    \draw[-,thick,darkblue] (-.56,-1) to[out=up,in=down] (-.56,-0);
     \draw[-,thick,darkblue] (-.56+.56,-1) to[out=up,in=down] (0,-0);
        \draw[-,thick,darkblue] (-.56+.56+0.56,-1) to[out=up,in=down] (0.56,-0.5);
         \draw[-,thick,darkblue] (2.24,-1) to[out=up,in=down] (2.24,-0.5);
    \end{tikzpicture},
     \quad
     d_2=
       \begin{tikzpicture}[baseline = -5mm]
\draw[-,thick,darkblue] (-.56,0) to[out=up,in=left] (.84,.6) to[out=right,in=up] (2.24,0);
         \draw[-,thick,darkblue] (.56,0) to[out=up,in=left] (1.12,.4) to[out=right,in=up] (1.68,0);
       \draw[-,line width=4pt,white] (0,0) to[out=up,in=left] (.56,.4) to[out=right,in=up] (1.12,0);
       \draw[-,thick,darkblue]  (0,0) to[out=up,in=left] (.56,.4) to[out=right,in=up] (1.12,0);\draw[-,thick,darkblue] (2.24,-.5) to[out=90,in=-90] (1.68,0);
\draw[-,thick,darkblue] (1.68,-.5) to[out=90,in=-90] (2.24,0);
\draw[-,line width=4pt,white]  (1.68,-.5) to[out=90,in=-90] (2.24,0);
	\draw[-,thick,darkblue] (1.68,-.5) to[out=90,in=-90] (2.24,0);
\draw[-,thick,darkblue] (2.24-1.12,-.5) to[out=90,in=-90] (1.68-1.12,0);
\draw[-,thick,darkblue] (1.68-1.12,-.5) to[out=90,in=-90] (2.24-1.12,0);
\draw[-,line width=4pt,white]  (1.68-1.12,-.5) to[out=90,in=-90] (2.24-1.12,0);
	\draw[-,thick,darkblue] (1.68-1.12,-.5) to[out=90,in=-90] (2.24-1.12,0);
\draw[-,thick,darkblue] (2.24-0.56,-.5-.5) to[out=90,in=-90] (1.68-0.56,0-.5);
\draw[-,thick,darkblue] (1.68-0.56,-.5-.5) to[out=90,in=-90] (2.24-0.56,0-.5);
\draw[-,line width=4pt,white]  (1.68-0.56,-.5-.5) to[out=90,in=-90] (2.24-0.56,0-.5);
	\draw[-,thick,darkblue] (1.68-0.56,-.5-.5) to[out=90,in=-90] (2.24-0.56,0-.5);
     \draw[-,thick,darkblue] (2.24-0.56-0.56,-.5-.5-.5) to[out=90,in=-90] (1.68-0.56-0.56,0-.5-.5);
\draw[-,thick,darkblue] (1.68-0.56-0.56,-.5-.5-.5) to[out=90,in=-90] (2.24-0.56-0.56,0-.5-.5);
\draw[-,line width=4pt,white]  (1.68-0.56-0.56,-.5-.5-.5) to[out=90,in=-90] (2.24-0.56-0.56,0-.5-.5);
	\draw[-,thick,darkblue] (1.68-0.56-0.56,-.5-.5-.5) to[out=90,in=-90] (2.24-0.56-0.56,0-.5-.5);

 \draw[-,thick,darkblue] (2.24-0.56-0.56-0.56,-.5-.5-.5-.5) to[out=90,in=-90] (1.68-0.56-0.56-0.56,0-.5-.5-.5);
\draw[-,thick,darkblue] (1.68-0.56-0.56-0.56,-.5-.5-.5-.5) to[out=90,in=-90] (2.24-0.56-0.56-0.56,0-.5-.5-.5);
\draw[-,line width=4pt,white]  (1.68-0.56-0.56-0.56,-.5-.5-.5-.5) to[out=90,in=-90] (2.24-0.56-0.56-0.56,0-.5-.5-.5);
	\draw[-,thick,darkblue] (1.68-0.56-0.56-0.56,-.5-.5-.5-.5) to[out=90,in=-90] (2.24-0.56-0.56-0.56,0-.5-.5-.5);
\draw[-,thick,darkblue] (-.56,-2) to[out=up,in=down] (-.56,-0);
\draw[-,thick,darkblue] (0,-2+0.5) to[out=up,in=down] (0,-0);
\draw[-,thick,darkblue] (.56,-2+0.5+0.5) to[out=up,in=down] (.56,-0.5);
\draw[-,thick,darkblue] (.56+.56,-2) to[out=up,in=down] (.56+0.56,-2+0.5);
\draw[-,thick,darkblue] (.56+.56+.56,-2) to[out=up,in=down] (.56+0.56+.56,-2+0.5+.5);
\draw[-,thick,darkblue] (.56+.56+.56+.56,-2) to[out=up,in=down] (.56+0.56+.56+.56,-2+0.5+1);
\end{tikzpicture}\overset{ \text{ \eqref{bbb}} }=\lcap^{\otimes 3}.$$Let $w=\vec{ \prod}_{t=1}^sr_{2t-1,i_t}r_{2t,j^{s-t}_t}$. Then  $w\in\mathcal{D}_{s,2s}$ and
 $\lcap^{\otimes s} T_w^{inv}=d_s\vec{ \prod}_{t=1}^s T^{inv}_{2t-1,i_t} T^{inv}_{2t,j^{s-t}_{t}}=d$ as morphisms in $\AB$.\end{proof}

Consider  $\U_v(\mathfrak h)$ as a $ \U_v(\mathfrak b)$-module such that $x_i^+$ acts as zero for all admissible $i$, and define
\begin{equation}{M^{gen}:=\U_v(\mathfrak g)\otimes_{\U_v(\mathfrak b)}\U_v(\mathfrak h).}\end{equation} Later on,  $M^{gen}$
 is called  the \emph{generic Verma module}. Then $M^{gen}$ is a  right $\U_v(\mathfrak h)$-module with basis $\{x^-_{\mathbf s}\otimes 1\mid \mathbf{s}\in \mathbb N^{\ell(w_0)}\}$, where $x^-_\mathbf{s}$ is defined in \eqref{xplus}.
Recall that $V$ is the natural $\U_v(\mathfrak g)$-module and $\dim V=\mathrm N $. Then
$M^{gen}\otimes V^{\otimes r}$ has  $\mathbb F$-basis $\{ x^-_{\mathbf s}\otimes k_\mu\otimes v_{\mathbf i}\mid {\mathbf s}\in \mathbb N^{\ell(w_0)},  \mathbf{i}\in \underline{\mathrm N }^{r}, \mu \in \mathcal{P}\}$,
where $v_{\mathbf{i}}=v_{i_1}\otimes \ldots \otimes v_{i_r}$.

From here to the end of this section,   we assume $\mfg=\mathfrak{so}_{2n}$. So, $\mathrm N =2n$,  and $v=q$.

\begin{Defn}For any  $r\in \mathbb N$ and $j\in\{0, 1\}$, let ${M}_r^j$  be the free  $\mathcal A_0$-module with basis  $\{z_v^j x^-_{\mathbf s}\otimes k_\mu\otimes v_{\mathbf i}\mid \mathbf s\in \mathbb N^{\ell(w_0)},  \mathbf{i}\in \underline{\mathrm N }^{r}, \mu \in \mathcal{P}\}$, where $\mathcal{A}_0$ is given in Definition~\ref{codeg}.
\end{Defn}

Obviously, $ {M}_1^1\subset {M}_1^0$ and $z_v {M}_1^0={M}_1^1$.

\begin{Lemma}\label{q0}Suppose
$H, K$ are sequences of positive roots and $j\in\{0, 1\}$. Then \begin{itemize}\item[(1)]$x^-_K\otimes x^-_H$ stabilizes ${M}_1^j$,
\item[(2)]$x^-_K\otimes x^+_H$ stabilizes ${M}_1^j$.\end{itemize}
\end{Lemma}

\begin{proof}For any basis  element $v_l\in V$,
by Lemmas~\ref{natural0}, \ref{natural3}, $x^+_Hv_l=g_1v_{k_1}$ (resp., $x^-_Hv_l=g_2v_{k_2}$) for some $g_1, g_2\in \mathcal A_0\setminus \mathcal A_1$ and some $v_{k_1}, v_{k_2}\in V$ if   $x^+_Hv_l\neq 0$ (resp., $x^-_Hv_l\neq 0$).    Then (1)-(2) follow from Corollary~\ref{monoial}.
\end{proof}

\begin{Lemma}\label{theta} Suppose  $\mathbf s\in \mathbb N^{\ell(w_0)}$ and $\mu\in \mathcal P$. For any $i\in \underline {\mathrm N } $, and any  $\Psi\in \{\Theta, \bar\Theta\}$,
\begin{itemize}
\item [(1)]$\Psi((x^-_{\mathbf s} \otimes k_\mu) \otimes v_i)\equiv x^-_{\mathbf s} \otimes k_\mu  \otimes v_i \pmod {{M}_1^1}$, \item [(2)]   $ P\Psi P((x^-_{\mathbf s} \otimes k_\mu )\otimes v_i  )\in {M}_1^0$, \item [(3)] $\Psi ({M}_1^1)\subseteq {M}_1^1$ and $P\Psi P({M}_1^1)\subseteq {M}_1^1$.
\end{itemize}
\end{Lemma}
\begin{proof}Recall the terms  $z_v^{\ell(J)}g_Jx^-_J\otimes x^+_J$'s in \eqref{theta123}.
  If $\ell (J)=0$, then $z_v^{\ell(J)} g_Jx^-_J\otimes x^+_J=1\otimes 1$ and hence
 $z_v^{\ell(J)} g_Jx^-_J\otimes x^+_J (( x^-_{\mathbf s} \otimes k_\mu)  \otimes v_i)=x^-_{\mathbf s} \otimes k_\mu\otimes v_i$.
 If $\ell (J)>0$,
 by Lemma~\ref{q0}(2),
 $z_v^{\ell(J)} g_Jx^-_J\otimes x^+_J ( ( x^-_{\mathbf s} \otimes k_\mu)  \otimes v_i)\in {M}_1^1$. Thanks to \eqref{theta123}, $$\bar{\Theta} ((x^-_{\mathbf s} \otimes k_\mu) \otimes v_i)\equiv x^-_{\mathbf s} \otimes k_\mu \otimes v_i \pmod {{M}_1^1}.$$ The corresponding result for  $\Theta$ can be proved similarly. This proves (1).
Note that $x_\beta^+ (1\otimes k_\mu)=1\otimes x_\beta^+k_\mu=0$ for any $\beta\in \mathcal R^+$. By \eqref{theta123}, \begin{equation}\label{t3}{P\bar \Theta P}((1\otimes k_\mu)\otimes v_i)=1\otimes k_\mu\otimes v_i.\end{equation}
   This proves (2) when $\sum_i s_i\beta_i=0$ and $\Psi=\bar\Theta$.
In general,  let $j=max\{t\mid s_t\neq 0\}$.  Suppose $$y_\Psi=P\Psi P(  ( x^-_{\mathbf s} \otimes k_\mu)  \otimes v_i)$$ and $\mathbf c\in \mathbb{N}^{\ell(w_0)}$ such that $c_k=s_k$ for any $k\neq j$ and $c_j=s_j-1$. Then

 \begin{equation}\label{grs1}\begin{aligned} & y_{\bar \Theta}- P {\bar\Theta}\Delta(x_{\beta_j}^-)(v_i\otimes (x^-_{\mathbf c} \otimes k_\mu))\\=&-P{\bar\Theta}(\sum_{K,H}h_{K,H}x^-_Kv_i\otimes k_{-wt (K)}x^-_H(x^-_{\mathbf c}\otimes k_\mu)), \text{ by Proposition~\ref{roo}(1)}\\=&-P{\bar\Theta}P(\sum_{K,H}h_{K,H}k_{-wt (K)}x^-_H\otimes x^-_K((x^-_{\mathbf c}\otimes k_\mu)\otimes v_i ))\\ \in& M_1^0, ~\text{by Lemma~\ref{q0}(1) and induction assumption on  $\sum_{i} s_i \beta_i-wt(K)$},\end{aligned}\end{equation}  where  $K, H$ are sequences of positive roots such that  $wt (K)+wt (H)=\beta_{j}$,  $K\neq \emptyset$ and $h_{K,H}\in \mathcal A_0$. In particular, $h_{K, H}=1$ when $(K, H)=(\beta_j, \emptyset)$.
  For any  $\U_v(\mathfrak g)$-modules $M, N$ and any   $a, b\in \U_v(\mathfrak g) $,
 it is easy to see $P(a\otimes b) P=b\otimes a$ in $\End(M\otimes N)$. So, \begin{equation} \label{grs2} \begin{aligned}&P {\bar\Theta}\Delta(x_{\beta_j}^-)(v_i\otimes (x^-_{\mathbf c} \otimes k_\mu))=P\bar\Delta(x_{\beta_j}^-){\bar\Theta}(v_i\otimes (x^-_{\mathbf c} \otimes k_\mu)),~\text{by \eqref{qrm11}}\\=&
 P\bar\Delta(x_{\beta_j}^-)PP{\bar\Theta}P((x^-_{\mathbf c} \otimes k_\mu)\otimes v_i)\\ \in& P\bar\Delta(x_{\beta_j}^-)P (M_1^0), ~\text{by induction assumption on  $\sum_{i} s_i \beta_i-\beta_j$}\\ \in &
 M_1^0,~\text{by Proposition~\ref{roo}(2) and Lemma~\ref{q0}(1)}.\end{aligned}\end{equation}
  Combining \eqref{grs1}--\eqref{grs2} yields $y_{\bar\Theta}\in M_1^0$.
 The corresponding result for
 $y_\Theta$  can be checked   by arguments similar to those above. The only difference is that one has to replace \eqref{qrm11} by \eqref{qrm}. This proves (2).  Finally,  (3) follows from (1)-(2).
\end{proof}

To simplify the notation, we use $d$ to replace  $\Psi_{M^{gen}}(d)$. Let  $\hat{\epsilon}_i=wt(v_i)$, $1\le i\le \mathrm N $.  So, $\hat\epsilon_i= \epsilon_i$ if  $1\leq i\leq n$, and  $ \hat\epsilon_i=-\epsilon_{i'}$ if $n'\leq i\leq 1'$.

\begin{Lemma}\label{mu} For any  $1\le i\le \mathrm N $ and any $\mu\in \mathcal P$,     $$X^\pm ((1\otimes k_\mu)\otimes v_i)\equiv v^{\pm(2n-1)} (1\otimes k_{\mu\pm 2\hat\epsilon_i}\otimes v_i) \pmod  {{M}_1^1}.$$
Further, $X^\pm{M}_1^1\subseteq {M}_1^1$.
\end{Lemma}
\begin{proof}Thanks to \eqref{pif}, $P \pi =\pi P$ and $\pi^{-1}(M_1^1)\subseteq M_1^1$.
Then
$$ \begin{aligned} X((1\otimes k_\mu)\otimes v_i) =& v^{2n-1}P\pi ^{-1}\bar{\Theta}P\pi ^{-1}\bar{\Theta}((1\otimes k_\mu)\otimes v_i),~\text{by Proposition~\ref{abmw1}}\\
 \equiv &   v^{2n-1} P\pi ^{-1}\bar \Theta P( (1\otimes k_{\hat\epsilon_i+\mu})\otimes v_i) \pmod {{M}_1^1}, \text{by Lemma \ref{theta} and \eqref{pif}}\\ \equiv &  v^{2n-1} \pi ^{-1}(1\otimes k_{\hat\epsilon_i+\mu})\otimes v_i) \pmod {{M}_1^1},~\text{by \eqref{t3}}\\
\equiv &  v^{2n-1}(1\otimes k_{2\hat\epsilon_i+\mu}\otimes v_i) \pmod {{M}_1^1},~\text{by \eqref{pif}}.\end{aligned}$$
  So,  $X{M}_1^1\subseteq {M}_1^1$.
 Finally, one can verify the result for $X^{-1}$  by arguments similar to those above. The only difference is that one has to use the result on $\bar \Theta$ to replace that for $\Theta$
in Lemma~\ref{theta}.
\end{proof}

From here to the end of this section, we assume   $n> r$.

\begin{Defn}\label{eta}For any  $d\in \mathbb{NT}_{2r, 0}$, define $\eta(d)\in \underline{\mathrm N }^{2r}$ such that   $\eta(d)_{i_l}=l$ and  $\eta(d)_{j_l}=l'$, $1\leq l\leq r$, where $conn(\hat d)=\{(i_1, j_1), \ldots, (i_r, j_r)\}$. 
\end{Defn}
 Since we are assuming that $n> r$, $\eta(d_1)=\eta(d_2)$ for any $d_1, d_2\in \mathbb{NT}_{2r, 0}$ if and only if  $conn(\hat d_1)=conn(\hat d_2)$.

\begin{Lemma}\label{actingd}Suppose $d, e\in \bar{\mathbb{NT}}_{2r, 0}$  such that  $conn(\hat d)=\{(i_1, j_1), \ldots, (i_r, j_r)\}$. For any
 $\mu\in \mathcal P$, there is a $c\in \mathbb Z$ such that
 $$d((1\otimes k_\mu) \otimes v_{\eta{(e)}})\equiv  \delta_{\eta(d),  {\eta({e})}} v^c (1\otimes k_\beta) \pmod{ {M}_0^1},$$
 where $\beta={\sum_{t=1}^r 2b_{d, i_t}\epsilon_t+\mu}$ and $b_{d, i_t}$'s are given in Definition~\ref{dots}.
 \end{Lemma}
\begin{proof} By Proposition~\ref{cspan3}, $\hat d=d_rT_w^{inv}$ as morphisms, where $w=\vec{ \prod}_{t=1}^rr_{2t-1,i_t}r_{2t,c_t}\in \mathfrak{S}_{2r}$ for some $c_t$'s and $d_r=\lcap^{\otimes r}$. Then $\eta(\hat d)=\eta(d_r)w$. For any $\mathbf{k}\in \underline{\mathrm N }^{2r}$, define $\delta_\mathbf k=1$ if $k_{2l-1}=k_{2l}'$ for all $1\leq l\leq r$ and $\delta_\mathbf k=0$ otherwise. Since $A$ acts on $V^{\otimes 2}$ via $\underline\alpha$ (see  Lemma~\ref{AU}), we have
\begin{equation}\label{add1}d_r(M_{2r}^1)\subseteq M_0^1 \text { and } d_r((1\otimes k_\eta)\otimes v_\mathbf k)=v^b\delta_{\mathbf k}(1\otimes k_\eta), \text{for $ \forall \eta\in \mathcal P$, }
\end{equation}
where $b\in \mathbb Z$.
Note that $T_j^{-1}$ acts on $M^{gen}\otimes V^{\otimes 2r}$  via $Id_{M^{gen}}\otimes Id_{V}^{\otimes {j-1}}\otimes R_{V,V}\otimes Id_{V}^{\otimes {2r-j-1}}$. By   Lemma~\ref{act123}, $T_j^{-1}$  stabilizes $M_{2r}^1$ and \begin{equation}\label{add2}T_j^{-1}((1\otimes k_\eta)\otimes v_{\mathbf k})\equiv v^{(wt(v_{k_j})\mid wt(v_{k_{j+1}}))}(1\otimes k_\eta)\otimes v_{\mathbf k r_j}\pmod {M_{2r}^1}
\end{equation} for any $\eta\in \mathcal P$.

Suppose $\mathbf s=\eta(e)$ and $\beta_{\mathbf s}=\mu+2\sum_{j=1}^{2r}b_{d, j}\hat\epsilon_{s_j}$.
 Then there exist $c, a_1, a_2\in \mathbb Z$ and $x, x_1\in M_{2r}^1$ such that
 \begin{equation}\label{sss123} \begin{aligned}d((1\otimes k_\mu)\otimes v_{\mathbf{s}})&=d_rT^{inv}_wX_{2r}^{b_{d, 2r}}\circ\ldots\circ X_1^{b_{d, 1} } ((1\otimes k_\mu)\otimes v_{\mathbf{s}}),~\text{by \eqref{form},}\\&=v^{a_1}d_rT^{inv}_w((1\otimes k_{\beta_{\mathbf s}})\otimes v_{\mathbf{s}}+x), \text {  by Lemma~\ref{mu} and \eqref{add2} }\\&=v^{a_2}d_r((1\otimes k_{\beta_{\mathbf s}})\otimes v_{\mathbf{s}w^{-1}}+x_1), \text {  by \eqref{add2} }\\&\equiv v^{c}\delta_{\mathbf{s}w^{-1}}(1\otimes k_{\beta_{\mathbf s}}) \pmod{M_0^1}, \text{  by \eqref{add1}}.\end{aligned}\end{equation} Let $\mathbf c=\mathbf{s}w^{-1}
=(s_{w^{-1}(1)}, \ldots,  s_{w^{-1}(2r)})$. Then $\mathbf c$ is an arrangement of $1, 1',\ldots, r, r'$.
  Note that $\delta_{\mathbf c}=1$ if and only if $c_{2l-1}=c_{2l}'$  if and only if $s_{w^{-1}(2l-1)}=s_{w^{-1}(2l)}'$ for all $1\leq l\leq r$. So, $\delta_{\mathbf c}=1$ if and only if
    $\{\{w^{-1}(2l-1), w^{-1}(2l)\}\mid 1\le l\le r\}$ is the set of all caps in $e$.
On the other hand,  $\eta( d)=\eta(d_r)w$. So $\{\{w^{-1}(2l-1), w^{-1}(2l)\}\mid 1\le l\le r\}$ is the set of all caps in $d$.
Thus, $\delta_{\mathbf{c}}=1$ if and only if $conn(\hat e)=conn(\hat d)$, proving  $\delta_{\mathbf{c}}=\delta_{\eta(d),  {\eta(e)}}$.
 \end{proof}

\begin{Lemma}\label{mmm} For any positive integer  $i$ and any  $\mu\in \mathcal P$, we have
 \begin{itemize}\item[(1)] $\Delta_i {M}_0^1\subseteq {M}_0^1$,
 \item [(2)] $\Delta_i (1\otimes {k}_\mu)\equiv \sum_{j=1}^{2n} v^{2\varrho_{j'}} (1\otimes k_{2i\hat\epsilon_j+\mu} )\pmod{{M}_0^1}$.\end{itemize}
\end{Lemma}
\begin{proof} Thanks to   Lemma~\ref{AU},    $U({M}_0^1)\subseteq {M}_2^1$ and $A({M}_2^1)\subseteq {M}_0^1$.  By Lemma~\ref{mu}, we have  $\Delta_i {M}_0^1=A\circ (X^{i}\otimes Id_V)\circ U({M}_0^1)\subseteq {M}_0^1$  and
$$
\begin{aligned}
 \Delta_i (1\otimes {k}_\mu) &=(Id_{M^{gen}}\otimes \underline\alpha)\circ(X^{i}\otimes Id_V)((1\otimes {k}_\mu)\otimes \sum_{j=1}^{2n} v^{\varrho_{j'}}v_j\otimes v_{j'}) \\ &\equiv (Id_{M^{gen}}\otimes \underline\alpha)(\sum_{j=1}^{2n} v^{\varrho_{j'}}( (1\otimes k_{2i\hat\epsilon_j+\mu})\otimes v_j\otimes v_{j'})) \pmod { M_0^1}
\\&\equiv\sum_{j=1}^{2n} v^{2\varrho_{j'}} (1\otimes k_{2i\hat\epsilon_j+\mu}) \pmod{M_0^1}.
\end{aligned}
 $$
 \end{proof}

 The proof of Theorem~\ref{aff} depends on quantum group $\U_v(\mathfrak{so}_{2n})$ over $\mathbb{C}(v^{1/4})$.   Since we consider generic Verma module and natural module of  $\U_v(\mathfrak{so}_{2n})$, it is enough use $\mathbb{C}(v^{1/2})$ instead of  $\mathbb{C}(v^{1/4})$.

\begin{Theorem}\label{aff} Suppose $\mathbb K=\mathbb C(v^{1/2})$. If  $(\delta,  z)=(v^{2n-1}, z_v)$ and $\omega_0$ is uniquely determined by    \eqref{para1}, then   $\Hom_{\AB}(\ob 2r,\ob 0)$ has $\mathbb K$-basis given by $\mathbb{NT}_{2r, 0}/\sim$.
\end{Theorem}

\begin{proof}
Recall that any $e\in \mathbb{NT}_{2r, 0}/\sim$ is of form in \eqref{form}. By  Proposition \ref{span1}, it is enough to prove $p_d=0$  for all $d$ if \begin{equation}\label{rrr123} \sum_{d\in B}p_{d}(\Delta_1,\Delta_2,\ldots)d=0,\end{equation}
where $B$ is a finite subset of $\bar{\mathbb{NT}}_{2r, 0}$ and $p_d\in\mathbb K[t_1,t_2,\ldots]$. If it were false, we assume all previous $p_d$'s are non-zero and choose a $j$ such that all previous $p_d\in \mathbb K[t_1,t_2,\ldots,t_j]$. In this case, write \begin{equation}\label{rrr1234}  p_{d}(\Delta_1,\Delta_2,\ldots)=\sum_{\mathbf s\in \mathbb N^{j}}f_{\mathbf s}^d(v^{1/2})\Delta^{\mathbf s},\end{equation}  where $f_{\mathbf s}^{d}(v^{1/2})\in\mathbb K $ and  $\Delta^{\mathbf s} =\Delta_1^{s_1}\Delta_2^{s_2}\cdots\Delta_j^{s_j}$. Without loss of any generality, we can assume that
$f_{\mathbf s}^{d}(v^{1/2})\in \mathcal A_0$ for all  previous pairs $(d, \mathbf s)$ and moreover, there is a pair $(d_1, \tilde{\mathbf s})$ such that
 $f_{\tilde{\mathbf s}}^{d_1}(v^{1/2})\in \mathcal A_0\setminus \mathcal A_1$. Suppose $conn(\hat d_1)=\{(i_l, j_l)| 1\leq l\leq r\}$ and  $B_1=\{d\in B\mid \eta(d)=\eta(d_1)\}$.
 Since we are assuming that $n> r$,  $\eta(d)=\eta(d_1)$ if and only if $conn(\hat d)=conn(\hat d_1)$. Recall   $b_{d, i}$'s  in Definition~\ref{dots}.
 Thanks to
Lemmas~\ref{actingd}--\ref{mmm}, there are some  $c(d)\in \mathbb Z$ depending on $d$ such that
$$\begin{aligned} & \sum_{d\in B}p_{d}  (\Delta_1,\Delta_2,\ldots)d ((1\otimes 1)\otimes v_{\eta(d_1)})\equiv \sum_{d\in B_1}p_{d}(\Delta_1,\Delta_2,\ldots)d ((1\otimes 1)\otimes v_{\eta(d_1)})\pmod { M_0^1}\\
  &\equiv\sum_{d\in B_1} v^{c(d)} p_{d}(\sum_{l=1}^{2n} v^{2\varrho_{l'}}(1\otimes k_{2\hat{\epsilon}_l}),\sum_{l=1}^{2n} v^{2\varrho_{l'}}(1\otimes k_{4\hat{\epsilon}_l}),\ldots) (1\otimes k_{\sum_{t=1}^r2b_{d, i_t}\epsilon_t})
 \pmod{ M_0^1}
.\end{aligned} $$  Thanks to \eqref{rrr123}-\eqref{rrr1234},
\begin{equation}\label{ddd6}\sum_{d\in B_1}v^{c(d)}\sum_{\mathbf s\in \mathbb N^{j}}f_{\mathbf s}^d(v^{1/2})
\prod_{i=1}^j ( \sum_{l=1}^{2n} 1\otimes k_{2i\hat{\epsilon}_l})^{s_i}
(1\otimes k_{\sum_{t=1}^r2b_{d, i_t}\epsilon_t})\equiv 0\pmod { M_0^1}. \end{equation}
Now, define $\text{deg}(k_{\pm \epsilon_i})=\pm 1$ for all admissible $i$.  Considering the terms in the LHS of \eqref{ddd6} with the highest degree  yields
\begin{equation}\label{kk}
\sum_{d\in B_1}\sum_{\mathbf s\in \mathbb N^{j}}v^{c(d)}f_{\mathbf s}^d(v^{1/2})
\prod_{i=1}^j ( \sum_{l=1}^{n}1\otimes k_{ 2i\epsilon_l})^{s_i}
(1\otimes k_{\sum_{t=1}^r2b_{d, i_t}\epsilon_t})\equiv 0\pmod { M_0^1}
\end{equation}
Note that \eqref{kk} is something like \cite[(3.27)]{RS3}. Using arguments on the leading monomials at the end of the proof
\underline{} of \cite[Proposition~3.12]{RS3}, we have   $v^{c(d)}f_{\mathbf s}^d(v^{1/2})\in \mathcal{A}_1$ for all pairs $(d, \mathbf s)$ such that $d\in B_1$. In particular, $v^{c(d_1)}f_{\tilde{\mathbf s}}^{d_1}(v^{1/2})\in \mathcal A_1$, a   contradiction since $f_{\tilde{\mathbf s}}^{d_1}(v^{1/2})\in \mathcal A_0\setminus \mathcal A_1$.
\end{proof}

\begin{proof}[\textbf{Proof of Theorem~\ref{affbasis}}]
If $m+s$ is odd, then $\Hom_{\AB}(\ob m, \ob s)=0$. So, we assume   $m+s=2r$ for some $r\in \mathbb{N}$.
 Thanks to Lemma~\ref{etannxs}(2), both  $\bar\eta_{\ob m}$ and $\bar\gamma_{\ob m}$ give bijections between  $\mathbb{NT}_{m,s}/\sim$ and $\mathbb{NT}_{0,2r}/\sim$ (as morphisms in $\AB$). Thanks to  Proposition \ref{span1}, it is enough to verify that $\mathbb {NT}_{2r, 0}/\sim$ is linear independent over $\mathbb K$.

We consider $\AB$ over the quotient ring $\mathbb{C}[\delta,\delta^{-1},z,\omega_0]/I$ and $I $ is the ideal generated by $ \delta-\delta^{-1}-z(\omega_0-1)$. Since $\delta-\delta^{-1}-z(\omega_0-1)$ is  irreducible, $\mathbb{C}[\delta,\delta^{-1},z,\omega_0]/I$ is a domain. Suppose $$\sum_{d\in B}f_dd=0,$$  where $B$ is a finite set of $\mathbb {NT}_{2r, 0}/\sim$ and $f_d\in \mathbb{C}[\delta, \delta^{-1},z, \omega_0]/I$.

We claim that $f_d=0$ for all $d\in B$. Since $\mathbb{C}[\delta,\delta^{-1},z,\omega_0]/I$  is a domain, $f_d=0$ if and only if $z^a\delta^bf_d=0$ for any $a\in\mathbb N,b\in\mathbb{Z}$. So,  we can assume each monomial of  $f_d$  has neither $\delta^{-1}$ nor $\omega_0$ as its factor  when we prove the claim.
Thanks to  Theorem \ref{aff}, we have $f_d=0$ for all  $\delta=v^{2n-1}, z=v-v^{-1}$ whenever $n>r$.
 Now the claim follows from the fundamental theorem of algebra.
  So, $\mathbb {NT}_{2r, 0}/\sim$ is linear independent over $\mathbb{C}[\delta,\delta^{-1},z,\omega_0]/I$ and hence over $\mathcal{Z}=\mathbb{Z}[\delta,\delta^{-1},z,\omega_0]/J$, where $J$  is the ideal generated by $ \delta-\delta^{-1}-z(\omega_0-1)$.

   Consider the affine  Kauffmann category  $\AB_\mathcal{Z}$ (resp., $\AB_\mathcal{\mathbb K}$)  over $\mathcal{Z}$ (resp., $\mathbb K$). By base change property, $\Hom_{\mathbb K\otimes_\mathcal{Z}\AB_{\mathcal{Z}}}(\ob 2r,\ob 0)$ has basis given by $\mathbb {NT}_{2r, 0}/\sim$. Then Theorem \ref{affbasis} follows from the fact that there is an obvious functor from $\AB_\mathcal{\mathbb K}$ to $\mathbb K\otimes_\mathcal{Z}\AB_{\mathcal{Z}}$, which sends the required basis element of $\Hom_{\AB_\mathcal{\mathbb K}}(\ob 2r, \ob 0)$ to the corresponding basis element in $\Hom_{\mathbb K\otimes_\mathcal{Z}\AB_{\mathcal{Z}}}(\ob 2r,\ob 0)$.
\end{proof}

\section{Affine Birman-Murakami-Wenzl algebras}

Thanks to  Theorem~\ref{affbasis},  $\AB$ can be considered as $\mathbb K[\Delta_1, \Delta_2, \ldots]$-linear category.
Suppose $\omega$ satisfies Definition~\ref{admi}(1).  Consider $\mathbb K$ as the right $\mathbb K[\Delta_1, \Delta_2, \ldots]$-module on which $\Delta_i$ acts on $\mathbb K$ via $\omega_i$ for any positive integer $i$. Define
 $$\AB( \mathbf \omega)=\mathbb K\otimes_{\mathbb K[\Delta_1, \Delta_2, \ldots]} \AB.$$
Let $\B'$ be the subcategory of $\AB(\mathbf \omega)$ generated by $\ob 1$ and four elementary morphisms   $A,U,T$ and $T^{-1}$.
Thanks to Theorem \ref{affbasis},
we have the following result, immediately. This shows that $\B$ is a subcategory of $\AB( \mathbf \omega)$.
\begin{Cor}\label{c4}  Suppose $m, s\in \mathbb N$. If  $\omega$ satisfies Definition~\ref{admi}(1), then
\begin{itemize}\item [(1)]
 $\Hom_{\AB( \mathbf \omega)}(\ob m, \ob s)$ has  $\mathbb K$-basis  given by $\bar {\mathbb {NT}}_{m, s}/\sim$,

\item [(2)]  $\B'\simeq \B$.
\end{itemize} \end{Cor}

\begin{Defn}\label{BMW}\cite{Good} The affine Birman-Murakami-Wenzl algebra $W_{r,\mathbb K}^{\text{aff}}$ is the $\mathbb K$-algebra generated by $x_1^{\pm 1}$, $g_i^{\pm1}$ and $e_i (1\leq i\leq r-1)$ subject to  the following  relations:
\begin{multicols}{2}\item[(1)] $g_ig_i^{-1}=g_i^{-1}g_i=1$, \item[(2)]  $x_1x_1^{-1}=x_1^{-1}x_1=1$,
\item[(3)]$e_i^2=\omega_0 e_i$,
\item[(4)]$g_ig_{i+1}g_i=g_{i+1}g_ig_{i+1}$,
\item[(5)]$g_i-g_i^{-1}=z(1-e_i)$,
\item[(6)]$x_1g_1x_1g_1=g_1x_1g_1x_1$,
    \item[(7)]$y_iz_j=z_jy_i$, if $\vert i-j\vert \geq 2$,
    \item[(8)]$x_1y_j=y_jx_1$ if $j\geq 2$,
    \item[(9)]$e_ie_{i\pm1}e_i=e_i$,
    \item[(10)]$g_ig_{i\pm1}e_i=e_{i\pm1}e_i$ and $e_ig_{i\pm1}g_i=e_ie_{i\pm1}$,
    \item[(11)]$e_1x_1^se_1=\omega_se_1$, for any $s\ge 1$,
        \item[(12)]$g_ie_i=e_ig_i=\delta^{-1} e_i$,
            \item[(13)]$e_1x_1g_1x_1= \delta e_1=x_1g_1x_1e_1$,
\end{multicols}
\noindent where $y_i, z_i\in \{e_i, g_i\}$, $1\le i\le r-1$.
\end{Defn}
Note that the current  $z$ in Definition~\ref{BMW}  is $-z$ in \cite{Good}.

\begin{Theorem}\label{c3} As  $\mathbb K$-algebras, $\End_{\AB(\omega)}(\ob r)\cong {W}_{r,\mathbb K}^{\text{aff}}$.
\end{Theorem}
\begin{proof}The required  algebra homomorphism $\gamma: {W}_{r,\mathbb K}^{\text{aff}}\rightarrow \End_{\AB(\omega)}(\ob r)$ satisfies $\gamma(e_i)=E_i$, $\gamma(x_1^{\pm1})=X_1^{\pm1}$ and $\gamma(g_i^{\pm1})=T_i^{\pm1}$ for all admissible $i$, where
\begin{equation}\label{t6}E_i=U_i\circ A_i.\end{equation}
 In order to verify that $\gamma$ is an algebra homomorphism, we need to check that the images of $x_1^\pm, g_i^\pm$ and $e_i$ satisfy Definition~\ref{BMW}(1)-(12). Later on, we say Definition~\ref{BMW}(1) holds if  the images of $x_1^\pm, g_i^\pm$ and $e_i$ satisfy Definition~\ref{BMW}(1).

  By (RII)-(RIII),\eqref{relation 4},  (L) and (S), we see that  Definition~\ref{BMW}(1)--(5) hold.
Thanks to   the definition of $\AB(\omega)$,   Definition~\ref{BMW}(11) holds. For $1\leq i\leq r$, let $x_i=g_{i-1}...g_1x_1g_1...g_{i-1}$. Then $\gamma(x_i^{\pm1})=X_i^{\pm1}$. Now, Definition~\ref{BMW}(6) follows from Lemma~\ref{XX}. Definition~\ref{BMW}(7)-(8) follows from the interchange law.
Note that both sides of each equation in Definition~\ref{BMW}(9)-(10) are totally descending tangle diagrams with the same connector. Thanks to  \eqref{bbb} and the monoidal functor in  Lemma~\ref{ktau}(2), Definition~\ref{BMW}(9)-(10) hold.
 Definition~\ref{BMW}
 (12) follows from Lemma~\ref{selfcrossing}(1)-(2). Thanks to Lemma~\ref{selfcrossing}(1)(3), we have $$E_1X_1T_1X_1=\delta E_1T_1X_1T_1X_1=\delta E_1X_2X_1=\delta E_1X_1^{-1}X_1=\delta E_1.$$ One can verify $X_1T_1X_1E_1=\delta E_1$ similarly. So,  Definition~\ref{BMW}(13) holds. This verifies that $\gamma$ is an algebra homomorphism.

  Let ${W}_{r,\mathbb K}$ be the subalgebra of ${W}_{r,\mathbb K}^{\text{aff}}$ generated by $g_i^{\pm1}$ and $e_i ,1\leq i\leq r-1$. Thanks to \cite[Theorem 5.41]{Good} and Corollary~\ref{c4}(2), $\gamma\mid_{{W}_{r,\mathbb K}}: {W}_{r,\mathbb K}\rightarrow \End_{\B'}(\ob r)$ is a $\mathbb K$-algebra isomorphism. Let $\gamma_0=\gamma\mid_{{W}_{r,\mathbb K}}$.  Thanks to \cite[Propotion 6.12(1)]{Good}, ${W}_{r,\mathbb K}^{\text{aff}}$ have  basis $\{x^{\mathbf t}\gamma_0^{-1}(d)x^{\mathbf s}\}$ where $d\in\bar{\mathbb {NT}}^1_{r,r}/\sim$, $\mathbf t, \mathbf s\in \mathbb Z^r$, and  $x^{\mathbf t}=x_1^{t_1}\cdots x_{r}^{t_r}$ and $x^{\mathbf s}=x_1^{s_1}\cdots x_{r}^{s_r}$ such that $\gamma(x^{\mathbf t}\gamma_0^{-1}(d)x^{\mathbf s})\in\bar{\mathbb {NT}}_{r,r}/\sim$. By Corollary~\ref{c4}(1), $\bar{\mathbb {NT}}_{r,r}/\sim$
   is a basis of $\End_{\AB(\omega)}(\ob r)$.
   Via \eqref{form}, one can check that $\gamma$ gives a bijective map between $\{x^{\mathbf t}\gamma_0^{-1}(d)x^{\mathbf s}\}$ and $\bar{\mathbb {NT}}_{r,r}/\sim$.  So, $\gamma$ has to be a $\mathbb K$-linear isomorphism, and hence an algebra  isomorphism.
\end{proof}

\section{A basis theorem  of cyclotomic Kauffmann category}

The aim of this section is to prove Theorem~\ref{cycb} under the Assumption~\ref{asump}. We always assume that $\omega$ is admissible.  So, both $\AB(\omega)$ and  $\CB^f$ are available.
Thanks to Definition~\ref{scbmw}, there is a functor from $\AB(\omega)$ to $\CB^f$. It   results in    an epimorphism
$$\xi:\End_{\AB(\omega)}(\ob r)\twoheadrightarrow \End_{\CB^f}(\ob r).$$
Let $\tilde\gamma=\xi\circ\gamma$, where
$\gamma$ is the $\mathbb K$-algebra isomorphism in Theorem~\ref{c3}.
Then   $$\tilde\gamma:  {W}_{r,\mathbb K}^{\text{aff}}\twoheadrightarrow
\End_{\CB^f}(\ob r)$$ is an epimorphism such that
 $\tilde\gamma(f(x_1))=f(X)\otimes 1_{\ob r-\ob 1}=0$. So, $\tilde\gamma$ factors through the cyclotomic Birman-Wenzl-Murakami algebra  $W_{a, r}=W_{r,\mathbb K}^{\text{aff}}/J$, where $J$ is the two sided ideal of ${W}_{r,\mathbb K}^{\text{aff}}$ generated by $f(x_1)=\prod_{i=1}^a (x_1-u_i)$. The  induced epimorphism is denoted by
 $\bar \gamma: W_{a, r}\twoheadrightarrow \End_{\CB^f}(\ob r)$.
 The current   $W_{a, r}$ is the same as  $\mathcal{B}_{r}^{a}(q,\delta^{-1},\omega_i,-b_i)$ in \cite[Definition 1.1]{Yu}, where $b_i$ is the coefficient of $x_1^{i}$ in $f(x_1)$.

 For all $p\in \mathbb Z$ and all  admissible positive  integers $i, j, l$ such that $i\le j\le l$,  Yu~\cite{Yu} defined $$\alpha_{i,j,l}^{p}=x_i^p g_{i, j}e_{j, l+1}\in W_{a, r}$$
 where $g_{i, j}=g_ig_{i+1}\cdots g_{j-1}$ and $e_{j, l+1}=e_{j} e_{j+1}\cdots e_{l}$.
 Note that $g_{i, j}=1$ if $i=j$.  For any $r\in \mathbb N\setminus 0$, let
 \begin{equation}\label{brr} B_{2r,2r}=\mathbb K\text{-span} \left\{\vec{\prod}_{l=r}^1 \alpha_{i_l,j_l,{2l-1}}^{s_l}\mid i_1<i_2<...<i_r,-\lfloor \frac{a-1}{2}\rfloor\le  s_l \le \lfloor \frac{a}{2}\rfloor\right\}.\end{equation}
Then $B_{2r,2r}\subseteq W_{a, 2r}$. Recall
$E_i$ in \eqref{t6}.
 Mimicking Yu's construction, we define $$\gamma_{i,j,\ell}^p=X_i^{p}T_{i,j}^{inv}E_{j, \ell}U_{\ell}\in \Hom_{\CB^f}(\ob {\ell-1}, \ob {\ell+1})$$ for any $p\in \mathbb{Z}$ and   all   positive integers  $i, j, \ell$ such that $i\leq j\leq \ell$, where $T^{inv}_{i, j}=T^{-1}_i T^{-1}_{i+1}\cdots T^{-1}_{j-1}$ and $E_{j, \ell}=E_{j} E_{j+1}\cdots E_{\ell-1}$.
 For any positive integer $r$, let
 \begin{equation} \label{c9} B_{0,2r}=\mathbb K\text{-span} \left\{\vec{\prod}_{l=r}^1 \gamma_{i_l ,j_l, 2l-1}^{s_l}\mid  i_1<i_2<...<i_r,  -\lfloor \frac{a}{2}\rfloor\le  s_l \le \lfloor \frac{a-1}{2}\rfloor \right\}.\end{equation} Then
 $B_{0,2r}\in \Hom_{\CB^f}(\ob 0, \ob 2r)$.
Recall the  contravariant functor $\sigma$   in Lemma~\ref{ktau}(1). Since
$\sigma$ stabilizes the right tensor ideal of $\AB$ generated by $f(\begin{tikzpicture}[baseline=-.5mm]
 \draw[-,thick,darkblue] (0,-.3) to (0,.3);
      \node at (0,0) {$\color{darkblue}\scriptstyle\bullet$};
 \end{tikzpicture})$ together with  $\Delta_k-\omega_k$ for all  $k\in \mathbb Z$, it  induces a contravariant functor $\bar\sigma:\CB^f \rightarrow\CB^f
$.
\begin{Lemma}\label{add} For any positive integer $r$, $\Hom_{\CB^f}(\ob {2r}, \ob 0)=\bar\sigma(B_{0, 2r})$.
\end{Lemma}
\begin{proof} Obviously,   $\bar\sigma(B_{0, 2r})\subseteq \Hom_{\CB^f}(\ob {2r}, \ob 0)$. Suppose  $d\in\bar{\mathbb {NT}}_{2r, 0}/\sim$.
Thanks to  \eqref{form} and Proposition~\ref{cspan3},  $d\in  d_r\End_{\CB^f}(\ob {2r})$, where $$d_r=\lcap^{\otimes r}=\bar\sigma(\vec{\prod}_{k=r}^1 \gamma_{2k-1 ,2k-1, 2k-1}^{0})\in\bar\sigma(B_{0, 2r}).$$ So, $d\in \bar\sigma(B_{0, 2r})  \End_{\CB^f}(\ob {2r})$.

 In~\cite[Lemmas~2.6,2.7]{Yu}, Yu  proved that $B_{2r,2r}$ is a left $W_{a, 2r}$-module, where $B_{2r,2r}$ is given in \eqref{brr}.
Mimicking  arguments there, one can verify   $\End_{\mathcal C}(\ob {2r})B_{0, 2r}\subseteq B_{0, 2r}$
without any difficulty.
Note that $\bar \sigma^2= Id$ and $\bar\sigma(\End_{\mathcal C}(\ob {2r}))=\End_{\mathcal C}(\ob {2r})$. So,   $\bar\sigma(B_{0, 2r})  \End_{\mathcal C}(\ob {2r})\subseteq \bar\sigma(B_{0, 2r}) $, forcing $d\in \bar\sigma(B_{0, 2r}) $. By  Proposition~\ref{span1}, we have  $\bar\sigma(B_{0, 2r})\supseteq \Hom_{\CB^f}(\ob {2r}, \ob 0)$.
 \end{proof}

 \begin{Prop}\label{cycba}As $\mathbb K$-module,   $\Hom_{\CB^f}(\ob {2r}, \ob 0)$ is spanned by $\bar{\mathbb {NT}}_{2r, 0}^a/\sim$.
\end{Prop}
\begin{proof}  Suppose  $\vec{\prod}_{k=1}^r  {\gamma}_{i_k,j_k,{2k-1}}^{s_k }
 \in B_{0, 2r}$ (see \eqref{c9}). We claim that there is a $d \in \bar{\mathbb {NT}}_{2r, 0}^a$ such that $d=\vec{\prod}_{k=1}^r  \bar\sigma({\gamma}_{i_k,j_k,{2k-1}}^{s_k })$ as morphisms in $\CB^f$. If so, by~Lemma~\ref{add},
 we immediately have the result.

In order to simplify notation,  let $E_{i, i}=1$ and $E_{i,j}= E_{i-1}E_{i-2}\cdots E_{j}$ if $i>j$. Thanks to \eqref{bbb} and the functor in  Lemma~\ref{ktau}(2), $A_k=A_kE_{k-1}T_{k}T_{k-1}$  for all  $k>1$ in $\AB$ and hence in $\CB^f$. So, $A_kT_{k-1, k+1}^{inv} =A_kE_{k-1}$.
 Using it together with braid relations  yields $$ A_lE_{l, j}T^{inv}_{j, i}
=A_lT^{inv}_{l,i}T^{inv}_{l+1,j+1},  \text{    for all possible  $i\leq j\leq l$.}$$
Obviously, $i_1=j_1=1$.  So,
  $$\begin{aligned}& \vec{\prod}_{\ell=1}^r   \bar\sigma({\gamma}_{i_\ell,j_\ell,{2\ell-1}}^{s_\ell })
=\vec{\prod}_{\ell=1}^r  A_{2\ell-1}E_{2\ell-1, j_\ell} T^{inv}_{j_\ell, i_\ell} X_{i_\ell}^{s_\ell}
 = \vec{\prod}_{\ell=1}^r  A_{2\ell-1}T^{inv}_{2\ell-1,i_\ell} T^{inv}_{2\ell,j_\ell+1}\prod_{k=1}^r X_{i_k}^{s_k}\\
&=A_1A_3\cdots A_{2r-1}\vec{\prod}_{\ell=1}^r  T^{inv}_{2\ell-1,i_\ell}T^{inv}_{2\ell,j_\ell+1}\prod_{k=1}^rX_{i_k}^{s_k}
=  \lcap^{\otimes r}   \vec{\prod}_{\ell=1}^r  T^{inv}_{2\ell-1,i_\ell}T^{inv}_{2\ell,j_\ell+1}
\prod_{k=1}^r X_{i_k}^{s_k}.\end{aligned}$$
Define
$B=\{ \lcap^{\otimes r} T^{inv}_w\mid w\in\mathcal{D}_{r,2r}\}$, where $\mathcal{D}_{r,2r}$ is given in \eqref{c13}. By Proposition~\ref{cspan3}, as sets of morphisms in $\AB$,  $B\supseteq \bar{\mathbb {NT}}^1_{2r, 0}/\sim$.
Thanks to Theorem~\ref{affbasis},  the number of morphisms in    $\bar{\mathbb {NT}}^1_{2r, 0}/\sim$   is  $\sharp \mathcal{D}_{r,2r}=(2r-1)!!$. Since $\sharp B\le  \sharp \mathcal{D}_{r,2r}$,  $B=  \bar{\mathbb {NT}}^1_{2r, 0}/\sim $ in $\AB$. So, Proposition~\ref{cspan3} gives an explicit bijection between $ \bar{\mathbb {NT}}^1_{2r, 0}/\sim $ and $B$.
Moreover, there is a $d_1\in \bar {\mathbb {NT}}^1_{2r, 0}$ such that $d_1=\lcap^\otimes r\vec{\prod}_{\ell=1}^rT^{inv}_{2l-1,i_{l}}T^{inv}_{2l,j_{l}+1}$ and  any $i_\ell$ is a left endpoint of a cap in $d_1$. Then the claim follows if we assume
 $d=d_1\circ \prod_{k=1}^r X_{i_k}^{s_k}$.
\end{proof}

\begin{Assumption}\label{assum1}
Until  the end of Proposition ~\ref{keytheorem},  we assume $2r\le \min\{q_1, q_2, \ldots, q_k\}$. We also keep the setting for parabolic quantum groups in section~3. Fix $I\in\{I_1, I_2\}$ in \eqref{defofpiee} such that $\mathfrak g\neq \mathfrak {so}_{2n+1}$ when $I=I_1$. Moreover, $a$ is always  $\text{deg} f_I(t)$ and $k=\lfloor (a-1)/2\rfloor +1$, where $f_{I}(t)$ is given in Definition~\ref{fi}.
\end{Assumption}

 Suppose $d\in \bar{\mathbb{NT}}_{2r,0}^a/\sim$ such that $conn(\hat d)=\{(i_l, j_l)\mid 1\le l\le r\}$.
   For any $c\in \mathbb Z$ and $0\le c\le  a-1$, define
  $$
   a_{i_l,c}=\begin{cases} p_c+l, &\text{ if $0\leq c\leq k-1$,}\\
    p_{2k-c}-l+1, & \text{ if $I=I_1$ and $k\leq c\leq a-1$,}\\
 p_{2k-c-1}-l+1,  &\text{ if $I=I_2$ and $k\leq c\leq a-1$,}\\
 \end{cases}$$
 where $p_0=0$ and $ p_j=\sum_{i=1}^j q_i$.
Since we are assuming $q_j\ge 2r$ for all admissible $j$, we have the following Lemma, immediately.
\begin{Lemma}\label{add9} Suppose  $b, c\in \{0,1, \ldots. a-1\}$ and $1\leq l\leq r$. We have \begin{itemize}
\item [(1)]$p_{c}<{a_{i_l,c}}\leq p_{c}+r\leq p_{c+1}-r$ if $0\leq c\leq k-1$,
\item [(2)]$p_{2k-c-1}+r\leq p_{2k-c}-r<{a_{i_l,c}}\leq p_{2k-c}$ if $I=I_1$ and $k\leq c\leq a-1$,
\item [(3)]$p_{2k-c-2}+r\leq p_{2k-c-1}-r<{a_{i_l,c}}\leq p_{2k-c-1}$ if $I=I_2$ and $k\leq c\leq a-1$,
\item [(4)]$  a_{i_j,b}=a_{i_l, c}  \text{ if and only if } (j,b)=(l,c)$,
 \item [(5)]$a_{i_j,b}=l  \text{ if and only if } (j,b)=(l,0)$.
 \end{itemize}
 \end{Lemma}

\begin{Lemma}
 \label{aaa} Suppose $1\leq l\leq r$. Then  $\beta_{l,c}\in\mathcal{R}^+\setminus\mathcal{R}_{I}^+$,  and $\beta_{l,c}<\beta_{l,c+1}$   if $ 1\leq c< k$  and   $\beta_{l,c+1}<\beta_{l,c}$ if $ k\leq c< a-1$, where
  \begin{equation}\label{rootnew}
\beta_{l,c}=\begin{cases}
\epsilon_{a_{i_l,c-1}}-\epsilon_{a_{i_l,c}}, & \text{if $1\leq c\leq k-1$,}\\
\epsilon_{a_{i_l,k}}+\epsilon_{a_{i_l,k-1}},& \text{if $c=k$ ,}\\
\epsilon_{a_{i_l,c}}-\epsilon_{a_{i_l,c-1}},& \text{if $k<c\leq a-1$.}\\
\end{cases}
\end{equation}
\end{Lemma}
\begin{proof} The result follows from  Lemma~\ref{add9} and \eqref{ord}. \end{proof}

\begin{example}\label{c14}Suppose $(k, q_1, q_2)=(2, 4, 9)$ in \eqref{defofpiee}. Let ($r, d)=(2, \gamma_{\ob 2})$, where $\gamma_{\ob 2}$ is given in \eqref{varep}. Then $conn(d)=\{(1,4), (2,3)\}$.   If $(I, a)=(I_1, 4)$, then \begin{itemize}\item $a_{1,0}=1$, $a_{1,1}=5$, $a_{1,2}=13$, $a_{1,3}=4$,
  \item $a_{2,0}=2$, $a_{2,1}=6$, $a_{2,2}=12$, $a_{2,3}=3$,\item $\beta_{1,1}=\epsilon_1-\epsilon_5$, $\beta_{1,2}=\epsilon_5+\epsilon_{13}$, ${\beta_{1,3}}=\epsilon_4-\epsilon_{13}$, and $\beta_{1,1}<\beta_{1,3}<\beta_{1,2}$, where $<$ is the convex order in \eqref{ord}.
  \item $\beta_{2,1}=\epsilon_2-\epsilon_6$, ${\beta_{2,2}}=\epsilon_6+\epsilon_{12}$, $\beta_{2,3}=\epsilon_3-\epsilon_{12}$,  and $\beta_{2,1}<\beta_{2,3}<\beta_{2,2}$.
  \end{itemize}
 \end{example}

\begin{Defn}\label{cocy}
   Suppose $d\in \bar{\mathbb{NT}}_{2r,0}^a/\sim$ and $conn(\hat d)=\{(i_l, j_l)\mid 1\le l\le r\}$. For $1\leq l\leq r$ and  $1\leq c\leq a-1$, define
 $x^{\pm}(d)_{i_l,c}= \vec{\prod}_{j=1}^c x^\pm_{\beta_{l,j}}$ and   $x^{\pm}(d)_{i_l,0}=1$, where $\beta_{l, j}$ is given in \eqref{rootnew}.
\end{Defn}

Given two sequences of positive roots $K,  H$, we write $x_K^-\sim x_H^-$ and $x_K^+\sim x_H^+$
if $K$ can be obtained from $H$ by place permutation. Recall $V$ is the natural $\U_v(\mathfrak g)$-module with basis $\{v_1,v_2,...,v_\mathrm N \}$ and $j'=\mathrm N +1-j$ for all $1\leq j\leq \mathrm N $. For any positive root $\beta$, we have  already  described  the action of $x_\beta^+$ on $V$ in Lemmas~\ref{natural0}-\ref{natural3}. We will freely use those results  in the proof of Lemma~\ref{add7}.
\begin{Lemma}\label{add7}Keep the notations in Definition \ref{cocy}. Suppose $x_H^+\sim x^{+}(d)_{i_l,c}$. Assume  $h=a_{i_l,c}$ if $0\leq c\leq k-1$ and $h=a_{i_l,c}'$ otherwise. Then
  $x^{+}(d)_{i_l,c}v_h=yv_l$ for some $y\in \mathcal{A}_0\setminus\mathcal{A}_1$, and $x_H^+v_h\neq 0$ if and only if $x_H^+= x^{+}(d)_{i_l,c}$.
\end{Lemma}
\begin{proof}If $c=0$ there is nothing to prove. Via Lemma~\ref{add9} and \eqref{rootnew}, $x^{+}_{\beta_{l,j}}v_h\neq 0$ if and only if $j=c$.  Further, $$x^{+}_{\beta_{l,c}}v_h=\begin{cases}y_1v_{a_{i_l,c-1}}, & \text{if $1\leq c\leq k-1$,}\\
y_2v_{a_{i_l,k-1}},& \text{if $c=k$ ,}\\
y_3v_{a_{i_l,c-1}'},& \text{if $k<c\leq a-1$.}\\
\end{cases}$$
for some  $y_1, y_2, y_3\in \mathcal{A}_0\setminus\mathcal{A}_1$.
Note that  $a_{i_l,0}=l$.  The result follows from  induction on $c$.
\end{proof}

 Recall $\mathcal S_{I,l}$ in \eqref{sir} for any $l\in \mathbb N$ and   $\mathcal R^+=\{\beta_j|1\leq j\leq \ell(w_0)\}$ such that $\beta_i<\beta_j$ in the sense of \eqref{ord} if $i<j$. To simplify the notation, we write  $${\textsc{x}}^-_{\mathbf i}=
 \vec{\prod}_{j= \ell(w_{0, I}^{-1} w_0)}^1 (x^-_{\beta_{t_j}})^{i_j}$$  for any  $\mathbf i\in \mathbb N^{ \ell(w_{0, I}^{-1} w_0)}$, then $\mathcal S_{I,l}=\left\{{\textsc{x}}^-_{\mathbf i} m_I\otimes v_{\mathbf j}\mid  \mathbf i\in \mathbb N^\ell(w_{0, I}^{-1} w_0),  \mathbf j\in \underline {\mathrm N} ^l\right\}$.
 We say ${\textsc{x}}_{\mathbf i}^-$ is of degree $|\mathbf i|=\sum_{j=1}^{ \ell(w_{0, I}^{-1} w_0)} i_j$. In this case,
  we also say    ${\textsc{x}}^-_{\mathbf i}  m_{I}\otimes v_{\mathbf j}$ is of degree $|\mathbf i|$.

\begin{Defn}\label{mij} For any $l, j\in \mathbb N$, let $M^j_{I, l}$ be the free  $\mathcal A_0$-module with basis given by  $\{z_v^{\text{deg} (y)+j} y\mid y\in \mathcal{S}_{I,l}\}$, where $\mathcal{A}_0$ is given in Definition~\ref{codeg}.
\end{Defn}
 Obviously, $ M^1_{I, l}\subset M^0_{I, l}$ and $z_v M^0_{I, l}=M^1_{I, l}$.

 \begin{Lemma}\label{cmonoial} Suppose  $\mathbf  i\in \mathbb N^{\ell(w_{0, I}^{-1} w_0)}$.

  \begin{itemize}\item[(1)]If $\beta\in \mathcal R^+_{I}$, then $z_v^{|\mathbf i| +1}x_\beta^- {\textsc{x}}^-_{\mathbf i}  m_{I}\equiv 0   \pmod {M^1_{I, 0}}$.
  \item[(2)] If $\beta\in \mathcal{R}^+\setminus\mathcal{R}_{I}^{+}$ (see \eqref{pararoot}),  then $z_v^{|\mathbf i|+1} x_\beta^-  {\textsc{x}}^-_{\mathbf i} m_{I}\equiv z_v^{|\mathbf i|+1}v^c {\textsc{x}}^-_{\mathbf s }m_{I} \pmod  { M^1_{I, 0}} $ for some $c\in\mathbb{Z}$ and $\mathbf s\in \mathbb N^{\ell(w_{0, I}^{-1} w_0)}$  such that  ${\textsc{x}}^-_{\mathbf s}\sim x_\beta^- {\textsc{x}}^-_{\mathbf i}$.
  \end{itemize}
  \end{Lemma}
\begin{proof}  Obviously, (1)-(2) hold if ${\textsc{x}}^-_{\mathbf i}=1$. We assume ${\textsc{x}}^-_{\mathbf i}\neq 1$ and hence $|\mathbf i|\ge 1$.
Suppose $x_{\mathbf c}^-\sim x_{\beta}^- {\textsc{x}}^-_{\mathbf i}$ and  $\mathcal{R}_{I}^+=\{\beta_{j_l}\mid 1\le l\le \ell(w_{0, I})\}$ such that ${j_a}<{j_b}$ for all admissible $a<b$.
Thanks to Lemma~\ref{commut} and Proposition~\ref{commut1}, for any  $x_H^-$ such that $x_H^-\sim x_{\beta}^- {\textsc{x}}^-_{\mathbf i}$, we have \begin{equation}\label{tilde}z_v^{|\mathbf i|+1}(x_H^--v^b x_{\mathbf c}^-)=\sum_{{\mathbf r\in \mathbb{N}^{\ell(w_0)}}}a_{\mathbf r}z_v^{\vert \mathbf  r\vert+1}x_{\mathbf r}^- \end{equation} for some $b\in\mathbb{Z}$ and some $a_{\mathbf r}\in \mathcal A_0$.
 First of all, we assume $x_\beta^- {\textsc{x}}^-_{\mathbf i}=x_{\mathbf c}^-$.
 By Lemma~\ref{commut},
$ a_{\mathbf r}=0$    if $r_t\neq 0$ for some $t$ such that $\beta_t\geq\beta$.

 If $\beta\not\in \mathcal{R}_{I}^+$, then $x_{\mathbf c}^-= {\textsc{x}}^-_{\tilde{\mathbf c}}$, where
 $\tilde{\mathbf c}\in \mathbb N^{\ell(w_{0, I}^{-1} w_0)}$ such that $\tilde c_j=c_{t_j}$ for all admissible $j$. In this case,  (2) automatically  holds. Otherwise, $\beta\in \mathcal{R}_{I}^+$ and $\beta=\beta_{j_l}$ for some $j_l>t_{\ell(w_{0, I}^{-1}w_0)}$.  Suppose $x_H^-={\textsc{x}}^-_{\mathbf i}x_{\beta_{j_l}}^-$ in
   \eqref{tilde}. Note that    ${\textsc{x}}^-_{\mathbf i}x^-_{\beta_{j_l}}m_{I}=0$.
 We claim the RHS of \eqref{tilde} acts on $m_{I}$ is in $M_{I, 0}^1$. If so, we have (1), immediately.

   We prove our claim  by induction on $l$. If $l=1$ and $\beta=\beta_{j_1}$,
any  term in the RHS of \eqref{tilde} acts on $m_{I}$ is in  $M_{I, 0}^1$, and
  the claim follows. Suppose  $l>1$. Any monomial $x_{\mathbf s}^-$ in  RHS of \eqref{tilde} satisfies  $s_j=0$ unless  $\beta_j<\beta_{j_l}$. If $\sum_{l=1}^{\ell(w_{0, I})}s_{j_l}=0$, then $z_v^{\vert \mathbf s \vert+1} x^-_{\mathbf s}m_{I}\in M_{I, 0}^1$, and there is nothing to  prove.
 Otherwise, let $t$ be the minimal number such that $s_{j_t}\neq 0$, then $j_t<j_l$ (i.e. $t<l$).  By  induction assumption on $t$ and Lemma \ref{commut},
 $$z_v^{\sum_{\ell=1}^{j_t-1}s_\ell+1}
 x^-_{\beta_{j_t}}\vec{\prod}_{i=j_t-1}^{1}(x^-_{\beta_i})^{s_i}m_{I}
=\left(\sum_{\mathbf r\in\mathbb{N}^{\ell(w_{0, I}^{-1} w_0)}} b_{\mathbf r}z_v^{\vert \mathbf r\vert +1}\textsc{x}_{\mathbf r}\right)m_{I}$$ where $b_{\mathbf r}\in \mathcal{A}_0$ and  $b_{\mathbf r}= 0$ unless $r_l=0$ for all $1\leq l\leq \ell(w_{0, I}^{-1} w_0)$ such that $t_l>{j_t}$.
So,  $$z_v^{| \mathbf s|}x^-_{\mathbf s}m_{I}=z_v^{\sum_{\ell={j_t}}^{\ell(w_0)}s_\ell-1}\vec{\prod}_{i={\ell(w_0)}}^{j_t+1}x^-_{\beta_i}
 (x^-_{\beta_{j_t}})^{s_{j_t}-1}\left(\sum_{\mathbf r\in\mathbb{N}^{\ell(w_{0, I}^{-1} w_0)}} b_{\mathbf r}z_v^{\vert \mathbf r\vert +1}\textsc{x}_{\mathbf r}\right)m_{I}.$$
By induction assumption on
 $\sum_{l=1}^{\ell(w_{0, I})}s_{j_l}-1$, we have $z_v^{| \mathbf s|+1}x^-_{\mathbf s}m_{I}\in M_{I, 0}^1$,
  proving the claim.

  We have proved (1)-(2) when $x_\beta^- {\textsc{x}}^-_{\mathbf i}=x_{\mathbf c}^-$. In general, we assume $x_H^-=x_\beta^-{\textsc{x}}^-_{\mathbf i}$. Using \eqref{tilde}, we see that the general case follows from  those  when  $x_\beta^-{\textsc{x}}^-_{\mathbf i}=x_{\mathbf c}^-$.
 \end{proof}

\begin{Cor}\label{to}Suppose
$H, K$ are sequences of positive roots and $j\in\{0,1\}$. Then \begin{itemize}\item[(1)]$z_v^{\ell (K)}x^-_K\otimes x^-_H$ stabilizes $M^j_{I, 1}$,
\item[(2)]$z_v^{\ell (K)}x^-_K\otimes x^+_H$ stabilizes $M^j_{I, 1}$.
\end{itemize}\end{Cor}
\begin{proof} For any basis  element $v_l\in V$,
by Lemmas~\ref{natural0}--\ref{natural3}, $x^+_Hv_l=g_1v_{k_1}$ (resp., $x^-_Hv_l=g_2v_{k_2}$) for some $g_1, g_2\in \mathcal A_0\setminus \mathcal A_1$ and some $v_{k_1}, v_{k_2}\in V$ if   $x^+_Hv_l\neq 0$ (resp., $x^-_Hv_l\neq 0$).    Then the (1)-(2) follow from Lemma~\ref{cmonoial}.
\end{proof}

\begin{Lemma}\label{c1}Suppose $\Phi\in\{\Theta, \bar\Theta\}$. For any  $\mathbf s\in \mathbb N^{\ell(w_{0, I}^{-1} w_0) }$ and any basis element  $v_l$ of $V$, we have
\begin{itemize}\item[(1)] $z_v^{|\mathbf s|}P\Phi P(\textsc{x}^-_{\mathbf s} m_{I}\otimes v_l)\equiv z_v^{|\mathbf s|} \textsc{x}^-_{\mathbf s} m_{I}\otimes v_l  \pmod{M^1_{I, 1}}$,
      \item[(2)]  $P\Phi P(M^j_{I, 1})\subseteq M^j_{I, 1}, j=0, 1$.
\end{itemize} \end{Lemma}
\begin{proof}
We prove (1) by induction on $\sum_{j=1}^{\ell(w_{0, I}^{-1} w_0)} s_j\beta_{t_j}$. Since $x^+_\beta m_{I}=0$ for any positive root $\beta$, by \eqref{theta123},  $$P\Phi P(m_{I}\otimes v_l)=m_{I}\otimes v_l.$$
  This proves (1)  when $\sum_{j=1}^{\ell(w_{0, I}^{-1} w_0)}s_j\beta_{t_j}=0$. Otherwise, we pick   $\ell$, the maximal number such that $s_{\ell}\neq 0$.  Let $y_\Phi=z_v^{|\mathbf s|}P\Phi P(   \textsc{x}^-_{\mathbf s}    m_{I}\otimes v_l)$ and $\mathbf c\in \mathbb{N}^{\ell(w_{0, I}^{-1} w_0)}$ such that $c_k=s_k$ for any $k\neq \ell$ and $c_\ell=s_\ell-1$.
Then,
 $$\begin{aligned} & y_\Theta- z_v^{|\mathbf s|} P\Theta(\bar\Delta(x_{\beta_{t_\ell}}^-)(v_l\otimes \textsc{x}^-_{\mathbf c}m_{I}))\\ = & -z_v^{|\mathbf s|} P\Theta(\sum_{K,H}g_{K,H}z_v^{\ell(H)}x_{K}^-v_l\otimes k_{wt(K)}x_{H}^-\textsc{x}^-_{\mathbf c}m_{I}),~\text{by Proposition~\ref{roo}(2)}
   \\=&-z_vP\Theta P(\sum_{K,H}g_{K,H}z_v^{\ell(H)+|\mathbf c|}k_{wt(K)}x_{H}^-\otimes x_{K}^-)(\textsc{x}^-_{\mathbf c}m_{I}\otimes v_l)
   \\ \in& M_{I, 1}^1,\text{by Corollary~\ref{to}(1) and induction assumption on $\sum_{j=1}^{   \ell(w_{0, I}^{-1} w_0)    }s_j\beta_{t_j}-wt (K)$},
 \end{aligned}$$
where   $K, H$ are sequences of positive roots such that  $wt (K)+wt (H)=\beta_{t_\ell}$,  $K\neq \emptyset$  and $g_{K,H}\in \mathcal A_0$. In particular, $g_{K,H}=1$ when $(K, H)=(\beta_{t_\ell},  \emptyset)$.
 We  have
  $$\begin{aligned} &z_v^{|\mathbf s|} P\Theta(\bar\Delta(x_{\beta_{t_\ell}}^-)(v_l\otimes \textsc{x}^-_{\mathbf c}m_{I}))= z_v^{|\mathbf s|}P\Delta(x_{\beta_{t_\ell}}^-)\Theta(v_l\otimes \textsc{x}^-_{\mathbf c}m_{I}),~\text{by \eqref{qrm}} \\= &z_v P\Delta(x_{\beta_{t_\ell}}^-)PP\Theta P(z_v^{|\mathbf c|} \textsc{x}^-_{\mathbf c}m_{I}\otimes v_l)
\\  = & z_vP\Delta(x_{\beta_{t_\ell}}^-)P(z_v^{|\mathbf c|}\textsc{x}^-_{\mathbf c}m_{I}\otimes v_l+x ), ~\text{by induction assumption on   $\sum_{j=1}^{ \ell(w_{0, I}^{-1} w_0) }c_j\beta_{t_j}$,}\end{aligned}$$
where $x\in {M}_{I, 1}^1$. By Proposition~\ref{roo}(1), as operators in $\End(M^{\mathfrak p_{I}}(\lambda_{I, \mathbf c})\otimes V)$,  we have\begin{equation} \label{rst} z_vP \Delta(x_{\beta_{t_\ell}}^-)P-z_v x_{\beta_{t_\ell}}^{-}\otimes 1 =z_vk_{-\beta_{t_\ell}}\otimes x_{\beta_{t_\ell}}^{-}+ \sum_{K_1,H_1}z_v^{\ell(H_1)+1} h_{K_1,H_1}k_{-wt (K_1)}x_{H_1}^{-}\otimes x_{K_1}^{-}, \end{equation}
 where $K_1, H_1$ are non-empty sequences of positive roots
such that  $wt(K_1)+wt(H_1)=\beta_{t_\ell}$ and $h_{K_1,H_1}\in \mathcal A_0$.
Thanks to  Corollary~\ref{to}(1), the RHS of \eqref{rst} sends
 ${M}_{I, 1}^0$ to ${M}_{I, 1}^1$ and  $z_v x_{\beta_{t_\ell}}^{-}\otimes 1$ fixes   ${M}_{I, 1}^1$.
 Therefore, $$y_\Theta\equiv (z_v  x_{\beta_{t_\ell}}^{-}\otimes 1)(z_v^{|\mathbf c|} \textsc{x}^-_{\mathbf c}m_{I}\otimes v_l)\pmod {M^1_{I, 1}}\equiv z_v^{|\mathbf s|} \textsc{x}^-_{\mathbf s}   m_{I}\otimes v_l\pmod {M^1_{I, 1}}.$$
This proves (1) for $\Theta$. Using \eqref{qrm11}   instead of \eqref{qrm}, one can check the result on $y_{\bar\Theta}$ similarly. Finally, the last assertion follows from previous result on $y_{\Phi}$.  \end{proof}

\begin{Lemma}\label{cactingd}Keep the setting above.  \begin{itemize}\item [(a)]
$z_v^{-(|\mathbf s| +c)}c_{(\mathbf s, h), (\mathbf r, l)}(X^t)\in \mathcal A_1$ if  one of conditions holds: (1) $0\le -t<c<k$,  (2) $k\leq c \leq a-1$, and $-k<t\leq c-k$,
\item [(b)] $z_v^{-(|\mathbf s| +c)}c_{(\mathbf s, h), (\mathbf r, l)}(X^t)\in \mathcal A_0\setminus \mathcal A_1$ if and only if one of conditions holds:  (1) $ h=a_{i_l,c}$ provided that $t=-c$ and $0< c<k$,
(2) $h=a_{i_l,c}'$ provided that either $t=c-k+1$ and  $k\leq c\leq a-1$ or $a$ is even and
    $(t, c)=(-k, a-1)$.
    \end{itemize}
\end{Lemma}
\begin{proof} First, we discuss the actions of $\Theta^t$ and $\bar \Theta^t$ on  $M^{\mathfrak p_{I} }(\lambda_{I, \mathbf c})\otimes V$.
By Lemmas~\ref{natural0}-\ref{natural3}, $(x_\beta^{\pm})^3=0$  in $\End(V)$ for any $\beta\in\mathcal{R}^+$.  Thanks to \eqref{theta123}, for any positive $j\ge 1$, we have  \begin{equation}\label{expan}\bar\Theta^j=\sum_H g_{H,j} z_v^{\ell(H)} x_H^-\otimes x_H^+,\quad \Theta^j=\sum_{H'} h_{H',j} z_v^{\ell(H')} x_{H'}^-\otimes x_{H'}^+,\end{equation} where  $H$(resp., $H'$) ranges over all sequences of positive roots such that $\ell(H)<\infty$ (resp., $\ell(H')<\infty$ ) and $g_{H,j} , h_{H',j} \in \mathcal A_0$.
 In order to divide such $H, H'$ into different classes, we define $\psi(H)=\#\{s\mid \beta_{\ell_s}>\beta_{\ell_{s+1}}, 1\leq s\leq b-1\}$ and  $\phi(H)=\#\{s\mid \beta_{\ell_s}<\beta_{\ell_{s+1}}, 1\leq s\leq b-1\}$,  where $H=(\beta_{\ell_1},...,\beta_{\ell_b})$. Thanks to \eqref{theta123},
 $\psi(H)\leq j-1$ and $\phi(H')\leq j-1$ if $H$ and $H'$ are those in \eqref{expan} and $g_{H,j} \in \mathcal A_0\setminus \mathcal A_1$ (resp., $h_{H',j}   \in \mathcal A_0\setminus \mathcal A_1$) if $\psi(H)= j-1$ (resp., $\phi(H')= j-1$).

Thanks to Corollary~\ref{to}(2) and  \eqref{pif},
  \begin{equation}\label{fff}  \Phi (M^j_{I, 1})\subseteq M^j_{I, 1}, \text{ if $j\in \{0, 1\}$  and $~\Phi\in\{\Theta, \bar \Theta, \pi , \pi ^{-1} \}$}.\end{equation} It is easy to see $P\pi =\pi P$. Now, we use
  Lemma~\ref{c1} to obtain:   \begin{equation}\label{expan1}
( P\pi ^{-1}\bar\Theta P\pi ^{-1}\bar\Theta)^{\pm b}(z_v^{\vert \mathbf s\vert} \textsc{x}^-_{\mathbf s} m_{I}\otimes v_h)\equiv z_v^{\vert \mathbf s\vert}(\pi ^{-2}\bar\Theta)^{\pm b}(\textsc{x}^-_{\mathbf s}m_{I}\otimes v_h) \pmod{ M^1_{I, 1}}\end{equation} for $b\geq 0$ where $\textsc{x}^-_{\mathbf s}$ and  $v_h$ are given before this lemma. If $c_{(\mathbf s, h), (\mathbf r, l)}(X^t)\neq 0$, then $wt(\textsc{x}^-_{\mathbf s}m_{I}\otimes v_h)=wt(\textsc{x}^-_{\mathbf r} m_{I}\otimes v_l)$, and hence
\begin{equation}\label{wt1}wt(v_{h})=\begin{cases}
      \epsilon_{a_{{i_l},c}}, & \text{if $0\leq c\leq k-1$,}\\ -\epsilon_{a_{{i_l},c}}, & \text{if $ k\leq c\leq a-1$}.\end{cases}\end{equation}
   In other words,
      \begin{equation} \label {wt1h} {h}=\begin{cases}
      a_{{i_l},c}, & \text{if $0\leq c\leq k-1$,}\\ a_{{i_l},c}', & \text{if $ k\leq c\leq a-1$.}\\ \end{cases}
\end{equation}
      In the following, we always  assume that \eqref{wt1} holds. Otherwise, $c_{(\mathbf s, h), (\mathbf r, l)}(X^t)= 0$, and there is nothing to prove.

For any sequence $J$ of  positive
roots with $\ell(J)<\infty$,   let $c_{(\mathbf s, h), (\mathbf r, l)}(J)$ be the coefficient of $\textsc{x}^-_{\mathbf r} m_{I}\otimes v_l$ in the expression of $z_v^{|\mathbf  s |+\ell(J)}x^-_J\otimes x^+_J(\textsc{x}^-_{\mathbf s}m_{I}\otimes v_h)$. Let  $J_0$ be the  sequence of positive roots such that $x^{\pm}_{J_0}=x^\pm(d)_{i_l,c}$.  Then $\ell(J_0)=c$.
         Thanks to Lemma~\ref{cmonoial}, \begin{equation}\label{coeff} z_v^{-(c+|\mathbf s|)}c_{(\mathbf s, h), (\mathbf r, l)}(J)\in \mathcal A_1, \text{ if $x^-_J\nsim x^{-}_{J_0}$.}\end{equation}
By Lemma~\ref{add7}, we have
\begin{equation}\label{coeff1}x_{J_1}^+v_{h}=0, \text{  for any $J_1$ such that  $x_{J_1}^+\sim x^+_{J_0}$ and $J_1\neq J_0$,}  \end{equation}  and $x_{J_0}^+v_{h}= f v_l$  for some $f\in\mathcal A_0\setminus \mathcal A_1$.
 By Lemma~\ref{cmonoial}, $$z_v^{(c+|\mathbf s|)}x_{J_0}^-\textsc{x}^-_{\mathbf s}m_{I}\equiv z_v^{(c+|\mathbf s|)}v^p \textsc{x}^-_{\mathbf r} m_{I} \pmod{M^1_{I, 0}}$$ for some $p\in\mathbb Z$.  Therefore,
\begin{equation}\label{coeff2} z_v^{-(c+|\mathbf s|)}c_{(\mathbf s, h), (\mathbf r, l)}(J_0)\in \mathcal A_0\setminus \mathcal A_1.\end{equation}
For any $t\in\{\lfloor \frac{a-1}{2}\rfloor, \lfloor\frac{a-1}{2}\rfloor-1,\ldots,-\lfloor \frac{a}{2}\rfloor\}$, let $c_{(\mathbf s, h), (\mathbf r, l)}(\bar\Theta^{t})$ be the coefficient of $\textsc{x}^-_{\mathbf r} m_{I}\otimes v_l$ in the expression of  $z_v^{\vert \mathbf s\vert}\bar\Theta^{t}( \textsc{x}^-_{\mathbf s } m_{I}\otimes v_h)$.
Thanks to Lemma~\ref{aaa}, \begin{equation}\label{psi1}\psi(J_0)=\begin{cases}0, &\text{if $1\leq c\leq k$}\\  c-k, &\text {if $k<c\leq a-1$}\end{cases} \text{ and }\ \   \phi(J_0)=\begin{cases}c-1, &\text{if $1\leq c\leq k$}\\ k-1, &\text{if $k<c\leq a-1$}\end{cases}.\end{equation}
So,
\begin{itemize}\item $\psi(J_0)\geq t$ if $k\leq c\leq a-1$, $0\leq t\leq c-k$,
\item $\phi(J_0)\geq-t$ if $0<c<k$, $-c<t\leq 0$ or $k\leq c\leq a-1$, $-c<t\leq0$,
\item $\psi(J_0)= t-1$ if $k\leq c\leq a-1$, $t=c-k+1$,
\item $\phi(J_0)=-t-1$ if $0<c<k$, $t=-c$ or $2|a$ and $(c, t)=(a-1, -k)$.
\end{itemize}Suppose $t$ is given in  (a). If  $t\leq0$, then  $\phi(J_0)\geq-t$. Thanks to \eqref{expan}, $x_{J_0}^-\otimes x_{J_0}^+$ can not appear in the expression of $\Theta^{-t}$. By  \eqref{coeff}-\eqref{coeff1},
 $z_v^{-(|\mathbf  s| +c)}  c_{(\mathbf s, h), (\mathbf r, l)}(\Theta^{-t})\in  \mathcal A_1$. If $t>0$. Then  $\psi(J_0)\geq t$. Thanks to \eqref{expan}, $x_{J_0}^-\otimes x_{J_0}^+$ can not appear in the expression of $\bar\Theta^{t}$. By  \eqref{coeff}-\eqref{coeff1},
 $z_v^{-(|\mathbf  s| +c)}  c_{(\mathbf s, h), (\mathbf r, l)}(\bar\Theta^{t})\in  \mathcal A_1$. In any  case, \begin{equation} \label{ggg123} z_v^{-(|\mathbf  s| +c)}  c_{(\mathbf s, h), (\mathbf r, l)}(\bar\Theta^{t})\in  \mathcal A_1.\end{equation}

 If $t$ is given in (b)(1) or $2|a$ and $(c, t)=(a-1, -k)$ in (b)(2), then  $t<0$ and  $\phi(J_0)=-t-1$.  Thanks to \eqref{expan}, $x_{J_0}^-\otimes x_{J_0}^+$ do appear in $\Theta^{-t}$ with coefficient $z_v^ch_{J_{0},t}  $, where $h_{J_{0},t}  \in \mathcal A_0\setminus \mathcal A_1$. By  \eqref{coeff}-\eqref{coeff2},
 $z_v^{-(|\mathbf  s| +c)}  c_{(\mathbf s, h), (\mathbf r, l)}(\Theta^{-t})\in  \mathcal A_0\setminus \mathcal A_1$. If $t$ is given in (b)(2) and $t>0$, then $\psi(J_0)=t-1$.  Thanks to \eqref{expan}, $x_{J_0}^-\otimes x_{J_0}^+$ appears in $\bar\Theta^t$ with coefficient $z_v^cg_{J_{0},t}  $, where $g_{J_{0},t} \in \mathcal A_0\setminus \mathcal A_1$. By  \eqref{coeff}-\eqref{coeff2},
 $z_v^{-(|\mathbf  s| +c)}  c_{(\mathbf s, h), (\mathbf r, l)}(\bar\Theta^{t})\in  \mathcal A_0\setminus \mathcal A_1$. In any  case, \begin{equation} \label{ggg1234} z_v^{-(|\mathbf  s| +c)}  c_{(\mathbf s, h), (\mathbf r, l)}(\bar\Theta^{t})\in  \mathcal A_0\setminus \mathcal A_1.\end{equation}

 We have  $\Psi_{M^{\mathfrak p_{I}}(\lambda_{I, \mathbf c})}(X^t)=( \delta P\pi ^{-1}\bar\Theta P\pi ^{-1}\bar\Theta)^t$, where $\delta$ is given in \eqref{lamb}.
  So both (a) and (b) follow from  \eqref{expan1} and  \eqref{pif}, immediately.
\end{proof}

 Assume  $\mathbb N_a=\{0, 1, 2,...,a-1\}$.  Let  $\phi:   \{\lfloor \frac{a-1}{2}\rfloor, \lfloor\frac{a-1}{2}\rfloor-1,\ldots,-\lfloor \frac{a}{2}\rfloor\}\rightarrow \mathbb N_a$ be the map such that
 \begin{equation}\label {biidjijw}\phi(j)=\begin{cases} -j, & \text{if $-k<j\leq 0$,}\\  k+j-1, & \text{if $0<j\leq k-1$,}\\
  \end{cases}
  \end{equation}
and $\phi(-k)=a-1$, which is available only if $a=2k$.

\begin{Cor}\label{cactingd2} Suppose   $Y\in\{X, \bar\Theta\}$ and $c\in \mathbb N_a$.
\begin{itemize}\item [(1)]  If $\phi(t)<c$, then  $z_v^{-(|\mathbf  s| +c)}  c_{(\mathbf s, h), (\mathbf r, l)}(Y^{t})\in \mathcal A_1$.
\item [(2)] Suppose  $\phi(t)=c$. Then $z_v^{-(|\mathbf  s| +c)}  c_{(\mathbf s, h), (\mathbf r, l)}(Y^{t})\in  \mathcal A_0\setminus \mathcal A_1$ if and only if $h$ is given in \eqref{wt1h}.
\end{itemize}
\end{Cor}
\begin{proof} Note that   conditions in (1) (resp., (2)) is equivalent to the  conditions in
 Lemma~\ref{cactingd}(a) (resp., (b)). So, the current results follow from Lemma~\ref{cactingd} when $Y=X$.
 If $Y=\bar\Theta$, the  results follows from \eqref{ggg123} and \eqref{ggg1234}.
\end{proof}

 For any  $1\le i\le 2r-1$, define $$ \tilde{R}_i=Id_{M^{\mathfrak p_{I}}(\lambda_{I, \mathbf c})}\otimes Id_{V}^{\otimes {i-1}}\otimes \tilde{R}\otimes Id_{V}^{\otimes {2r-i-1}},$$  where $\tilde{{R}}\in\End(V^{\otimes 2})$ such that $\tilde{{R}}(v_j\otimes v_l)=v^{(wt(v_j)\mid wt(v_l))}v_j\otimes v_l$ for  any $v_j, v_l\in  V$.
Define \begin{equation}\label{t5}\tilde{X}_j^{\pm 1}=\tilde{R}_{j-1}\tilde {X}_{j-1}^{\pm 1}\tilde{R}_{j-1},~ 2\leq j\leq 2r,\end{equation} where
$\tilde{X}_1^{\pm 1}=(\delta \pi ^{-2}\bar\Theta)^{\pm 1}\otimes Id_{V}^{\otimes {2r-1}}$, and
$\delta$ is given  in \eqref{lamb}.
For any $\psi, \psi_1\in \End(M^0_{I,2r})$, we write $$\psi\approx\psi_1 \text{ if $\psi(y)\equiv\psi_1(y)\pmod{M^1_{I,2r}}$ for any $y\in M^0_{I,2r}$.}$$  As mentioned before, we use $d$ to denote $\Psi_{M^{\mathfrak p_{I}}(\lambda_{I, \mathbf c})}(d)$ for any admissible $d$. The following results follow immediately from
 Lemma~\ref{act123} and \eqref{expan1}.

 \begin{Lemma}\label{kkkey}
 As morphisms in $\End(M^0_{I,2r})$, we have
 \begin{itemize}\item[(1)] $T_j^{\pm1}\approx \tilde{R}_j$ for all $1\le j\le 2r-1$,  \item [(2)] $X_1^{\pm1}\approx \tilde X_1^\pm $. \end{itemize} \end{Lemma}

\begin{Lemma} \label{approx} If $d\in \{T_1, \ldots, T_{2r-1}\}\cup\{X_1, \ldots, X_{2r}\}$, and
 $j\in \{0, 1\}$, then  $d^{\pm 1}{M}^j_{I,2r}\subseteq {M}^j_{I,2r}$.
\end{Lemma}

\begin{proof} Since  $T_\ell^{\pm1}$ acts on $M_{I, 2r}$ via $ Id_{M^{\mathfrak p_{I}}(\lambda_{I, \mathbf c})}\otimes Id_{V}^{\otimes {\ell-1}}\otimes \text{{R}}_{V,V}^{\pm1}\otimes Id_{V}^{\otimes {2r-\ell-1}}$,
 by   Lemma~\ref{act123}, we have the result when $d\in \{T_1, \ldots, T_{2r-1}\}$. Let $\Phi\in \{\Theta, \bar\Theta\}$.
  Thanks to Lemma~\ref{c1}, ${M}^j_{I,2r}$ is fixed by  $P\Phi P$.
 By \eqref{fff}, ${M}^j_{I,2r}$ is fixed by  $\Theta, \bar\Theta$ and $\pi ^{\pm1}$.
  Note that $\Psi_{M^{\mathfrak p_{I}}(\lambda_{I, \mathbf c})}(X^{\pm1})=(\delta P\pi ^{-1}\bar\Theta P\pi ^{-1}\bar\Theta)^{\pm1}$. So we have the result when $d=X_1$. The general case follows from the equation $X_j=T_{j-1} X_{j-1}T_{j-1}$.
  \end{proof}

  \begin{Defn}\label{add8}
   Suppose $d\in \bar{\mathbb{NT}}_{2r,0}^a/\sim$ and $conn(\hat d)=\{(i_l, j_l)| 1\leq i\leq r\}$.  For any $\mathbf{e}\in \mathbb N_a^r$, define
\begin{itemize}\item[(1)]  $\tilde{x}^-(d)_{\mathbf {e}}=\textsc{x}_{\mathbf s}^-$ such that $\textsc{x}_{\mathbf s}^-\sim \vec{\Pi}_{l=1}^{r}x^-(d)_{i_l,e_l}$, where $x^-(d)_{i_l,e_l}$ is given in Definition~\ref{cocy},
     \item[(2)]$ {v}(d)_{\mathbf e}=v_{b_1}\otimes\ldots\otimes v_{b_{2r}}\in V^{\otimes {2r}}$ such that  $b_{j_l}=l'$ and $b_{i_l}=a_{{i_l},e_{l}}$ (resp.,  $a_{{i_l},e_{l}}'$)
      if $0\leq e_l\leq k-1$ (resp., $ k\leq e_l\leq a-1$).\end{itemize}
\end{Defn}
\begin{example}\label{c15}Keep the notations in Example~\ref{c14}. We have  \begin{itemize}\item
${\beta_{1,1}}<{\beta_{2,1}}<{\beta_{2,3}}<{\beta_{1,3}}<{\beta_{1,2}}<{\beta_{2,2}}$
  \item ${v}(d)_{\mathbf e}=v_1\otimes v_2\otimes v_{2'}\otimes v_{1'}$ and $\tilde{x}^-(d)_{\mathbf {e}}=1$ if  $\mathbf {e}=(0,0)$,
  \item  $ {v}(d)_{\mathbf e}=v_{4'}\otimes v_{3'}\otimes v_{2'}\otimes v_{1'}$ and  $\tilde{x}^-(d)_{\mathbf {e}}=x^-_{\beta_{2,2}}x^-_{\beta_{1,2}}x^-_{\beta_{1,3}}x^-_{\beta_{2,3}}x^-_{\beta_{2,1}}x^-_{\beta_{1,1}}$, if $\mathbf e=(3,3)$.
  \end{itemize}
 \end{example}

 Suppose $d,d'\in \bar{\mathbb{NT}}_{2r,0}^a/\sim$ and $\mathbf e\in \mathbb N_a^r$. Let  $c_{d, d'}$  be the coefficient of $\tilde{x}^-(d')_{\mathbf{e}}m_{I}$ in $d (m_{I}\otimes {v}(d')_{\mathbf{e}})$. It is reasonable since
$d(m_{I}\otimes {v}(d')_{\mathbf{e}})$ can be expressed as  a linear combination of  elements in $\mathcal S_{I, 0}$ and $\tilde{x}^-(d')_{\mathbf{e}}m_{I}\in \mathcal S_{I, 0}$
 (see  Lemma \ref{tsmodule1}).
 Thanks to  \eqref{form},    $d=\hat{d}\prod_{l=1}^{2r} \circ X_{l}^{b_{d, l}}$. Suppose   $conn (\hat d)=\{(i_l, j_l)\mid 1\leq l\leq r\}$. Then, $b_{d, j_l}=0$, $1\le l\le r$.  Define
 $\mathbf{e}(d)= (\phi(b_{d, i_1}), \ldots, \phi(b_{d, i_r}))$, where $\phi$ is given in \eqref{biidjijw}.
 For any $\mathbf{t}, \mathbf{s}\in \mathbb{N}_a^r$, we say $\mathbf{t}\succcurlyeq \mathbf{s}$ if $t_l\geq s_l$, $1\leq l\leq r$.

\begin{Lemma}\label{cactingd1} Suppose $\mathbf e\in \mathbb N_a^r$ and  $d,d'\in \bar{\mathbb{NT}}_{2r,0}^a/\sim$ such that $d=\hat{d}\prod_{l=1}^{2r}  X_{l}^{b_{d, l}}$. We have
\begin{itemize}\item[(1)]$z_v^{-\vert \mathbf {e}\vert}c_{d,d'}\in \mathcal A_1$ if either $\mathbf e(d)\not\succcurlyeq \mathbf e$ or  $\mathbf e(d)=\mathbf e$ and  $\eta(d)\neq\eta(d')$, where $|\mathbf e|=\sum_{i=1}^r e_i$ and $\eta(d)$ is given in Definition \ref{eta}.
 \item[(2)]$z_v^{-\vert \mathbf {e}\vert}c_{d,d'}\in \mathcal A_{0}\setminus\mathcal A_1$ if  $\mathbf e(d)=\mathbf e$ and  $\eta(d)=\eta(d')$.
\end{itemize}
\end{Lemma}

\begin{proof}  For all $0\le b<c\le 2r$,  let $\phi_{b,c}:\U_v(\mfg)^{\otimes2}\to \U_v(\mfg)^{\otimes(2r+1)}$ be the linear map  such that $$\phi_{b,c}(x\otimes y)=1^{\otimes b}\otimes \overset{b+1\text{th}} {x}\otimes 1^{\otimes (c-b-1)}\otimes\overset{c+1\text{th}} y\otimes 1^{\otimes (2r-c)}, $$ for all $ x\otimes y\in \U(\mfg)^{\otimes2}$.
Given an $m_{I}\otimes {v}(d')_{\mathbf{e}}\in \mathcal S_{I, 2r}$, where ${v}(d')_{\mathbf{e}}$ is defined in Definition \ref{add8}, we claim
  $$ (d-v^s \hat d \vec{\prod}_{l=1}^{2r}  \phi_{0, l}(\bar\Theta)^{b_{d, l}})(m_{I}\otimes {v}(d')_{\mathbf{e}})\in M_{I, 0}^1$$
for some  $s\in \mathbb Z$
 depending on $m_{I}\otimes {v}(d')_{\mathbf{e}}$.

In fact,  by Lemmas~\ref{kkkey}-\ref{approx}, we have  $\prod_{l=1}^{2r} X_{l}^{b_{d, l}}\approx \vec{\prod}_{l=1}^{2r} \tilde X_{l}^{b_{d, l}}$. Thanks to \eqref{pif} and \eqref{t5}
   \begin{equation}\label{key123} \vec{\prod}_{l=1}^{2r} \tilde X_{l}^{b_{d, j}}(m_{I}\otimes {v}(d')_{\mathbf{e}})=v^{t}\vec{\prod}_{l=1}^{2r} \phi_{0, l}(\bar\Theta)^{b_{d, l}} (m_{I}\otimes {v}(d')_{\mathbf{e}})\end{equation}   for some $t\in \mathbb Z$
 depending on $m_{I}\otimes {v}(d')_{\mathbf{e}}$.   By Proposition~\ref{cspan3},  $\hat d=d_rT^{inv}_{w}$ for some $w\in \mathcal D_{r, 2r}$,  where $d_r=\lcap^{\otimes r}\in \bar{\mathbb{NT}}_{2r, 0}$.
 By Lemmas~\ref{AU},~\ref{approx},  $d_r{M}^j_{I,2r}\subseteq{M}^j_{I,0}$ and   $\hat d {M}^j_{I,2r}\subseteq{M}^j_{I,0}$ if $j\in\{0, 1\}$. Combining these results yields our claim.

We write the RHS of \eqref{key123} as a linear combination of $\mathcal{S}_{I, 2r}$ in Lemma \ref{tsmodule1}.
For any $v_{\mathbf{i}}=v_{i_1}\otimes v_{i_2}\otimes...\otimes v_{i_{2r}}\in V^{\otimes 2r}$, let $a_{d, d', \mathbf{i}}$ be the coefficient of $\tilde{x}^-(d')_{\mathbf{e}}m_{I}\otimes v_{\mathbf{i}}$ in this expression.
Thanks to Corollary~\ref{cactingd2}, $z_v^{-\vert\mathbf{e}\vert}a_{d, d', \mathbf{i}}\in {\mathcal A}_1$ if $\mathbf e(d)\not\succcurlyeq \mathbf e$.  Since $\hat d
{M}^1_{I,2r}\subseteq{M}^1_{I,0}$, we immediately have   $z_v^{-\vert \mathbf {e}\vert}c_{d,d'}\in \mathcal A_1$ in this case. This is the first assertion in (1).

Now, we assume $\mathbf e= \mathbf e(d)$. By  Corollary~\ref{cactingd2} again,
 $z_v^{-\vert\mathbf{e}\vert}a_{d, d', \mathbf{i}}\in {\mathcal A}_0\setminus{\mathcal A}_1$    if and  only if $\mathbf i=\eta(d') $,
  where $\eta(d')$ is given in Definition~\ref{eta}.

Let $b_{d,d',\mathbf i}$ be the coefficient of $\tilde{x}^-(d')_{\mathbf{e}}m_{I}$ in $ \hat d(z_v^{\vert\mathbf{e}\vert}\tilde{x}^-(d')_{\mathbf{e}}m_{I}\otimes v_{\mathbf{i}})$. We have  $z_v^{-\vert\mathbf{e}\vert}b_{d, d', \mathbf{i}}\in {\mathcal A}_0$ since
   $z_v^{\vert\mathbf{e}\vert}\tilde{x}^-(d')_{\mathbf{e}}m_{I}\otimes v_{\mathbf{i}}\in M^0_{I,2r}$ and $\hat d {M}^0_{I,2r}\subseteq{M}^0_{I,0}$.

Thanks to \eqref{sss123},   we have  $z_v^{-\vert\mathbf{e}\vert}b_{d, d', \mathbf{i}}\in {\mathcal A}_1$   if $\mathbf i\neq \eta(d)$ and
$z_v^{-\vert\mathbf{e}\vert}b_{d, d', \mathbf{i}}\in {\mathcal A}_0\setminus{\mathcal A}_1 $ if
 $\mathbf i= \eta(d)$.
 So, $z_v^{-\vert \mathbf {e}\vert}c_{d,d'}\in {\mathcal A}_1$  if $\eta(d)\neq\eta(d')$ and $z_v^{-\vert \mathbf {e}\vert}c_{d,d'}\in {\mathcal A}_0\setminus{\mathcal A}_1$  if $\eta(d)=\eta(d')$. This proves (2) and the second assertion in (1).
\end{proof}

By \eqref{vv},  $v=q$ if $\mathfrak g\in \{\mathfrak{sp}_{2n}, \mathfrak{so}_{2n}\}$ and $v=q^{1/2}$ if $\mathfrak g=\mathfrak {so}_{2n+1}$. In Proposition~\ref{keytheorem},  we assume (1) $\mathbb K=\mathbb C(q^{1/2})$, (2) $z=z_q$, (3) $\delta$ is given in \eqref{lamb}, (4) $\omega_0$ is determined by  \eqref{para1}, (5) $\omega$ is $\mathbf u$-admissible and (6) $f_I(t)$ is given in Definition~\ref{fi} where $I\in \{I_1, I_2\}$.

\begin{Prop}\label{keytheorem}
  $\Hom_{\CB^{f_I}}(\ob {2r},\ob {0})$ has  $\mathbb K$-basis given by $\bar{\mathbb {NT}}_{2r, 0}^a/\sim$.
\end{Prop}

\begin{proof}By  Proposition \ref{cycba}, it is enough to prove $f_d(q^{1/2})=0$  for all $d$ if $$\sum_{d\in B}f_d(q^{1/2})d=0,$$
where $B$ is a finite subset of $\bar{\mathbb {NT}}_{2r, 0}^a/\sim$ and $f_d(q^{1/2})\in\mathbb K$.

If it were false, we assume all previous $f_d(q^{1/2})$'s are non-zero. Without loss of any generality, we can assume that
$f_d(q^{1/2})\in \mathcal A_0$ for all  previous $d$ and moreover, there is at least one $d'$ such that
 $f_{d'}(q^{1/2})\in \mathcal A_0\setminus\mathcal A_1$. Let $B_1=\{d\in B\mid f_{d}(q^{1/2}) \not\in \mathcal A_1\}$ and $B_2=\{d\in B_1\mid \vert{\mathbf{e}}(d)\vert \text{ is  maximal }\}$, where ${\mathbf{e}}(d)$ is given before Lemma~\ref{cactingd1}. Obviously $B_2\neq \varnothing$.

Keep the notation in Definition \ref{add8}.
      Thanks to Lemma~\ref{cactingd1},
  the coefficient  of $\tilde{x}^-(d_1)_{\mathbf e(d_1)}m_{I}$ in $\sum_{d\in B}f_d(q^{1/2})d(m_{I}\otimes v_{\mathbf e(d_1)})$ is
  $ f_{d_1}(q^{1/2})z_v^{\vert{\mathbf{e}}(d_1)\vert}(g(q^{1/2})+h(q^{1/2}))$ for any $d_1\in B_2$, where $g(q^{1/2})\in \mathcal A_{0}\setminus\mathcal A_{1}$ and $h(q^{1/2})\in \mathcal A_1$. Then $f_{d_1}(q^{1/2})=0\in \mathcal A_1$, a contradiction since   $d_1\in B_1$. \end{proof}

In Theorem~\ref{calz}, we keep the following assumptions:
 \begin{itemize} \item $\mathcal Z=\mathbb Z[\hat{q}^{\pm 1}, (\hat{q}-\hat{q}^{-1})^{-1}, \hat u_1^{\pm1}, \ldots, \hat u_a^{\pm1}]$, where $\hat q$ and $\hat q-\hat q^{-1}$, $\hat u_1, \ldots, \hat u_a$ are indeterminates.
 \item  $\hat f(t)=\prod_{j=1}^a(t-\hat u_j)\in \mathcal Z[t]$,
  \item $z=z_{\hat {q}}$, $\delta=\alpha \prod_{i=1}^a \hat u_i$ where $\alpha\in \{1, -1\}$ if $a$ is odd and $\alpha\in \{-\hat{q}, \hat{q}^{-1}\}$ if $a$ is even,
       \item $\delta, \omega_0, z$ satisfy \eqref{para1},
 \item $\{\tilde{\omega}_j\mid j\in \mathbb Z \}$ is $\hat{\mathbf u}$-admissible in the sense of Definition~\ref{uad} and $\tilde{\omega}_0=\omega_0$.\end{itemize}

\begin{Theorem}\label{calz}
The $\mathcal Z$-module  $\Hom_{\CB^{\hat f}}(\ob {2r},\ob {0})$ has basis given by  $\bar{\mathbb {NT}}_{2r, 0}^a/\sim$.
\end{Theorem}

\begin{proof}
We assume $\mfg=\mathfrak{so}_{2n}$ (resp.,  $\mathfrak{sp}_{2n}$)  if $\alpha \in \{\hat q^{-1}, 1\}$ (resp.,  $\alpha \in \{-\hat q, -1\}$), where $n=\sum_{j=1}^k q_j$ such that $2r\le \min\{q_1, q_2, \ldots, q_k\}$ and $k=\lfloor\frac{a-1}{2}\rfloor+1$. We also assume $I=I_1$ if $a=2k$ and $I=I_2$ if $a=2k-1$, where $I_1, I_2$ is given in \eqref{defofpiee}. This enables us to use freely  previous results in section~3 and those in this section.
In any case, $v=q$, where $v$ is the defining parameter for $\U_v(\mathfrak g)$ over $\mathbb F$.
For any $s\geq 1$ and any sequence $\mathbf i=(i_1, i_2, \ldots, i_s)\in \mathbb Z^s$, let  $$q^{\mathbf i}=( q^{i_1},  q^{i_2},\ldots,  q^{i_s}).
$$
For any $a\in \{2k, 2k-1\}$, we specialize  $(\hat{q}, \hat u_1, \ldots, \hat u_a)$  at $(q, \epsilon_\mfg q^{\mathbf b_a})$ where $\mathbf b_{2k}=(b_1,\ldots, b_{2k})$ and $\mathbf b_{2k-1}$ is obtained from $\mathbf b_{2k}$ by removing $b_{k+1}$.
In any case,  $b_j$'s are given in \eqref{polf12}  and $\hat f(t)$ is specialized to $f_{I}(t)$ in Definition~\ref{fi}, respectively.
 Since $\{\tilde{\omega}_j\mid j\in \mathbb Z\}$ is $\hat {\mathbf u}$-admissible, such $\tilde{\omega}_j$'s are uniquely determined by $\hat u_j$'s. So,  $\tilde{\omega}$ is specialized to $\omega$ in Lemma~\ref{ghom123} for $\mfg\in\{\mathfrak{sp}_{2n}, \mathfrak{so}_{2n}\}$ with respect to $\alpha$.

 Since  $\mathbb{C}(q^{1/2})$ is a $\mathcal{Z}$-module on which  $\hat{q}, \hat u_1, \ldots, \hat u_a$ act via $q$ and $\epsilon_\mfg q^{\mathbf{b}_a}$  in an obvious way,
  there is a $\mathcal{Z}$-linear monoidal functor $\mathcal{F}: \AB_{\mathcal{Z}}\rightarrow \AB_{\mathbb{C}(q^{1/2})}$  sending  generators to generators with the same names. Let $J^I$ be the right tensor ideal of $\AB_{\mathbb{C}(q^{1/2})}$ generated by $f_{I}(X)$, $\Delta_j-\omega_j1_{\ob 0}$, $j\in \mathbb{Z}$. Since
  $\mathcal{F}(\Delta_j-\tilde\omega_j1_{\ob 0})=\Delta_j-\omega_j1_{\ob 0}$ for all  $j\in \mathbb{Z}$, and
  $\mathcal{F}(\hat f(X))=f_{I}(X)$,
   $\mathcal{F}$ induces a $\mathcal{Z}$-linear functor $\tilde{\mathcal{F}}: \CB_{\mathcal{Z}}^{\hat f}\rightarrow \CB_{\mathbb{C}(q^{1/2})}^{f_{I}}$, where $\CB_{\mathbb{C}(q^{1/2})}^{f_{I}}=\AB_{\mathbb{C}(q^{1/2})}/J^I$.

Now, we prove  $\sum_{d\in \bar{\mathbb {NT}}_{2r, 0}^a/\sim}  g_{d}d=0$ only if $g_{d}=0$ for all $d$, where $g_{d}\in\mathcal{Z}$.
In fact, we have  $\sum_{d\in \bar{\mathbb {NT}}_{2r, 0}^a/\sim}   g_{d}(q, \epsilon_\mfg q^{\mathbf{b}_a}) d=0$ where  $g_{d}(q, \epsilon_\mfg q^{\mathbf{b}_a})£º=\tilde{\mathcal{F}}( g_{d})$.
By Proposition~\ref{keytheorem}, $g_d(q, \epsilon_\mfg q^{\mathbf{b}_a})=0$ for all $d\in \bar{\mathbb {NT}}_{2r, 0}^a/\sim$.
  Thanks to \eqref{polf12},
  \begin{equation} \label{uij3} b_j=\begin{cases} 2(c_j+\sum_{j\leq l\leq k}q_l)-\epsilon_\mfg, &\text{if $1\le j\le k$}\\
   2(-c_{2k-j+1}-\sum_{k+2\leq l\leq j}q_{2k-l+2})+\epsilon_\mfg ,  & \text{if $k+1\leq j\leq 2k$.}\\ \end{cases}
   \end{equation}
 So,  we can view $ b_j$'s as polynomials of $q_j$'s and  $c_l$'s, where $1\leq j\leq k$ and  $ 1\leq l\leq k$ (resp., $1\leq l\leq k-1 $) if $a=2k$ (resp., $a=2k-1$). Let
 \begin{equation} \label{ua} \mathbf v_a=\begin{cases} (b_1,\ldots,b_k,b_{2k}, \ldots, b_{a-k+1}), & \text{if $a=2k$,}\\
   (b_1,\ldots,b_k,b_{2k}, \ldots, b_{a-k+3}), & \text{if $a=2k-1$.}\\
   \end{cases}
 \end{equation}
 Let $\phi_a: \mathbb C^a\rightarrow \mathbb C^a$ be the morphism such that
 $\phi_a(q_1, \ldots,q_k, c_1,\ldots, c_{a-k})= \mathbf v_a$. Thanks to \eqref{uij3}, the  Jacobi matrix  $J_{\phi_a}$ of $\phi_a$ satisfies
$$J_{\phi_{2k}}=\begin{pmatrix}
    H_{k} & 2E \\
    -H_k+2 E & -2 E \\
  \end{pmatrix}
,$$
where $E$ is the  $k\times k$  identity matrix  and $H_k$ is the upper-triangular matrix such that the $(l,j)$th entry is $2$ for all admissible  $j\ge l$.
The  Jacobi matrix $J_{\phi_{2k-1}}$ is obtained from $J_{\phi_{2k}}$ by deleting its  $2k$th  row and $2k$th column.
It is routine to check that   $$\text{det}J_{\phi_{2k}}=(-1)^{k}4^{k} \text{ and $\text{det}J_{\phi_{2k-1}}=(-1)^{k-1}2^{2k-1}$.}$$ So,  $\phi_a$ is always  dominant.
Define
   $$O_a=\{(q_1, \ldots,q_k, c_1,\ldots, c_{a-k})\mid c_t, q_j\in \mathbb Z \text{ and }q_j\geq 2r, 1\leq j\leq k, 1\leq t\leq a-k\}. $$
 Then  $O_a$   is Zariski dense in $\mathbb C^a$.
 Since $\phi_a$ is  dominant, $\phi_a(O_a)$ is Zariski dense in $\mathbb C^a$.
  For any $a\in \{2k, 2k-1\}$ define $\phi(O_a)_q=\{\epsilon_\mfg q^{\mathbf v_a}  \mid \mathbf v_a\in \phi_a(O_a)\}$.
 Suppose $0\neq q^*\in \mathbb C$ such that  $q^*$ is not  a  root of $1$.
Specializing   $q$ at $q^*$  yields  a bijection between $\phi_a(O_a)_{q*}$ and  $\phi_a(O_a)$. This shows that  $\phi_a(O_a)_{q*}$ is Zariski dense in $\mathbb C^a$.  This observation together with $g_d(q, \epsilon_\mfg q^{\mathbf{b}_a})=0$ yield $g_d=0$. So,    $\bar{\mathbb {NT}}_{2r, 0}^a/\sim$ is $\mathcal  Z$-linear independent.  Now,
the result follows immediately from  Proposition~\ref{cycba}.
\end{proof}

\begin{proof}[\textbf{Proof of Theorem~\ref{cycb}}]If $m+s$ is odd , then $\Hom_{\CB^f}(\ob m, \ob s)=0$. So, we assume   $m+s=2r$ for some   $r\in\mathbb N$.

$``{\Longleftarrow}"$: By Theorem~\ref{calz}, $\Hom_{\CB^{\hat f}}(\ob 2r, \ob 0)$ has $\mathcal Z$-basis given by   $\bar{\mathbb {NT}}_{2r, 0}$. By  arguments on base change (see, e.g in \cite{RS3}), we see  that $\bar{\mathbb {NT}}_{2r, 0}^a/\sim$ is linear independent over $\mathbb K$.
Thanks to Proposition~~\ref{cycba},  $\Hom_{\CB^{\hat f}}(\ob 2r, \ob 0)$
has $\mathbb K$-basis given by  $\bar{\mathbb {NT}}_{2r, 0}^a/\sim$. In general, the result follows from
 Lemma~\ref{etannxs}(3).

$``\Longrightarrow":$ At the beginning of this section,  we have explained that there is an algebra epimorphism  $ \bar \gamma: W_{a, r}\rightarrow \End_{\CB^f}(\ob r)$, where $W_{a, r}$ is the  cyclotomic Birman-Wenzl-Murakami algebra.
Goodman~\cite{G09} has proved that $W_{a, r}$ is always free over $\mathbb K$  with rank $c^r ((2r-1)!!-r!)+a^rr!$, where
$c$ is the minimal positive integer such that $ e_1f(x_1)=0$ and $c=\text{deg} f(x_1)\le a$. Since  $ \bar \gamma$ is an epimorphism and  $\text{rank}( \End_{\CB^f}(\ob r))=a^r (2r-1)!!$, we have $c=a$. This shows that  $e_1, e_1x_1, \ldots, e_1x_1^{a-1}$ is $\mathbb K$-linear independent. Otherwise, we can find
a $f(x)$ such that $c=\text{deg} f(x_1)< a$, and $\text{rank}(W_{a, r})=c^r ((2r-1)!!-r!)+a^rr!<a^r (2r-1)!!$,  a contradiction.  So, $W_{a, r}$ is admissible in the sense of  \cite[Corollary~4.5]{WYu}.  Goodman~\cite{G09} proved  that $W_{a, r}$ is admissible if and only if $\omega$ is $\mathbf u$-admissible in our sense.
  This completes the proof.
  \end{proof}

The following result follows immediately from arguments above.
\begin{Cor} Keep the Assumption~\ref{asump}. Suppose $\omega$ is $\mathbf u$-admissible in the sense of Definition~\ref{uad}.
 Then    $ W_{a, r}
\cong \End_{\CB^f}(\ob r)$ as $\mathbb K$-algebras.
\end{Cor}
\begin{appendix}
\section{Proof of Proposition~\ref{commut1}}

The aim of this section is to prove Proposition~\ref{commut1}.
  More explicitly, we need to prove  $c_{\mathbf r}^\pm \in \mathcal A_{|\mathbf r|-1}$, where $c_{\mathbf r}^{\pm }$ are  given in Lemma~ \ref{commut}.
  Via   the $\mathbb C$-linear anti-automorphism $\bar{\tau}$ of $\U_v(\mathfrak g)$ and \eqref{switch}, it is enough for us to deal with  $c_{\mathbf r}^{-}$.  So, we
  compute $[x_\nu^-, x_\alpha^-]_v=x_\nu^-x_{\alpha}^--v^{(\nu\mid \alpha)} x_{\alpha}^-x_{\nu}^-$ for any positive roots $\alpha, \nu\in \mathcal R^+$ such that $\nu<\alpha$. Our arguments depend on minimal pairs in Corollary~\ref{equa} with respect to
 the convex order in \eqref{ord}.
 For such a pair $\alpha$ and $\nu$, there  are four cases we have to consider:
\begin{itemize}\item[(1)] $\alpha+\nu\in \mathcal R^+$ and $(\alpha,\nu)$ is a minimal pair.
\item[(2)] $\alpha+\nu\in \mathcal R^+$ and $\{\alpha, \nu\}$ is not a minimal pair.
\item [(3)]  $\alpha+\nu\notin \mathcal R^+$ and there are $\beta, \gamma\in \mathcal R^+$ such that $\alpha>\beta\geq \gamma>\nu$ and $\alpha+\nu=\beta+\gamma$.
\item [(4)]  $\alpha+\nu\notin \mathcal R^+$ and there are no  $\beta, \gamma\in \mathcal R^+$ such that $\alpha>\beta\ge \gamma>\nu$ and $\alpha+\nu=\beta+\gamma$.
    \end{itemize}

In case (1),  $(\alpha,\nu)$ is one of minimal pairs  in Corollary~\ref{equa}.
In Lemma~A.1-Lemma~A.3, We give  detailed information on  $\{\alpha, \nu\}$ which appears in (2)-(4).

\begin{Lemma}\label{equal}
Suppose $\alpha, \beta, \gamma, \nu\in \mathcal R^+$  such that $\alpha>\beta>\gamma>\nu$ and
 $\alpha+\nu=\beta+\gamma\in\mathcal{R}^+$. Then $\alpha+\nu\in\{\epsilon_i+\epsilon_j, 2\epsilon_k\mid i<j, k< n-1\}$. Further,
\begin{itemize}\item[(1)]If $\alpha+\nu=\epsilon_i+\epsilon_j$, then $\alpha\in\{\epsilon_j-\epsilon_k, \epsilon_j+\epsilon_l, \epsilon_j, 2\epsilon_j\mid k>j+1, l> j\}$. In this case,
\begin{enumerate}\item $\beta= \epsilon_j-\epsilon_t$, $j<t<k$, if $\alpha=\epsilon_j-\epsilon_k$,\item
  $\beta=\epsilon_j-\epsilon_f$, $j< f\leq n$,  if $\alpha=\epsilon_j, 2\epsilon_j$,
  \item   $\beta\in\{\epsilon_j+\epsilon_s, \epsilon_j-\epsilon_t, \epsilon_j\mid s>l, t>j\}$, if $\alpha=\epsilon_j+\epsilon_l$.\end{enumerate}
\item[(2)]If $\alpha+\nu=2\epsilon_i$,  then $\alpha\in\{\epsilon_i+\epsilon_j \mid i<j<n\}$. In this case,  $\beta=\epsilon_i+\epsilon_k$, where $j<k\leq n$.
\end{itemize}
\end{Lemma}
\begin{proof} We decompose a positive root $\mu$  into a summation of two positive roots as follows. If  $\mu\in \{\epsilon_i-\epsilon_j,  \epsilon_i+\epsilon_j\}$ and  $i<j$, we have
\begin{itemize}\item$\epsilon_i-\epsilon_j=(\epsilon_m-
\epsilon_j)+(\epsilon_i-\epsilon_m), i<m<j$,
\item$\epsilon_i+\epsilon_j=(\epsilon_j-\epsilon_k)+(\epsilon_i+\epsilon_k)=(\epsilon_j+\epsilon_l)+(\epsilon_i-\epsilon_l)=(\epsilon_i)+(\epsilon_j)=(2\epsilon_j)+(\epsilon_i-\epsilon_j), j<k$, $i<l\neq j$.
\end{itemize}
Suppose $\mu\in \{\epsilon_i, 2\epsilon_i\}$. Then $\epsilon_i=(\epsilon_k)+
(\epsilon_i-\epsilon_k)$ and
$2\epsilon_i=(\epsilon_i+\epsilon_k)+(\epsilon_i-\epsilon_k)$. Thanks to \eqref{ord}, $\alpha+\nu\notin\{ \epsilon_i-\epsilon_j, \epsilon_k\}$ for all admissible $i, j, k$.
Now,  (1)-(2) follow from \eqref{ord}.  \end{proof}

\begin{Lemma}\label{equal2}Suppose  $\alpha, \beta, \gamma, \nu\in \mathcal R^+$  such that $\alpha>\beta\geq \gamma>\nu$
and $\alpha+\nu=\beta+\gamma\not\in\mathcal{R}^+$. Then one of (a)-(j) holds:
\begin{itemize}\item[(a)]$(\alpha, \nu, \beta, \gamma)=(\epsilon_k-\epsilon_l, \epsilon_i-\epsilon_j, \epsilon_k-\epsilon_j, \epsilon_i-\epsilon_l)$, $i<k<j<l$,
\item[(b)]$(\alpha, \nu, \beta, \gamma)=(\epsilon_k+\epsilon_j, \epsilon_i+\epsilon_l, \epsilon_k+\epsilon_l, \epsilon_i+\epsilon_j)$, $i<k<j<l$,
\item[(c)]$(\alpha, \nu, \beta, \gamma)=(\epsilon_k+\epsilon_j, \epsilon_i-\epsilon_l, \epsilon_k-\epsilon_l, \epsilon_i+\epsilon_j)$, $i<k<j$, $i<k<l$  and $j\neq l$,
\item[(d)]$(\alpha, \nu, \beta, \gamma)=(\epsilon_k, \epsilon_i-\epsilon_j, \epsilon_k-\epsilon_j, \epsilon_i)$, $i<k<j$,
\item[(e)]$(\alpha, \nu, \beta, \gamma)=(\epsilon_k+\epsilon_j, \epsilon_i, \epsilon_k, \epsilon_i+\epsilon_j)$, $i<k<j$,
\item[(f)]$(\alpha, \nu, \beta, \gamma)=(2\epsilon_k, \epsilon_i-\epsilon_j, \epsilon_k-\epsilon_j, \epsilon_i+\epsilon_k)$, $i<k<j$,
\item[(g)]$(\alpha, \nu, \beta, \gamma)=(\epsilon_k+\epsilon_j, 2\epsilon_i, \epsilon_i+\epsilon_k, \epsilon_i+\epsilon_j)$, $i<k<j$,
\item[(h)]$(\alpha, \nu, \beta, \gamma)=(2\epsilon_k, 2\epsilon_i, \epsilon_i+\epsilon_k, \epsilon_i+\epsilon_k)$, $i<k$,
\item[(i)] $(\alpha, \nu, \beta, \gamma)=(\epsilon_i+\epsilon_j, \epsilon_i-\epsilon_j, \epsilon_i+\epsilon_k, \epsilon_i-\epsilon_k)$, $i<j<k$. In this case, $\mfg\in \{\mathfrak{so}_{2n}, \mathfrak{so}_{2n+1}\}$.
\item[(j)]$(\alpha, \nu, \beta, \gamma)=(\epsilon_i+\epsilon_j, \epsilon_i-\epsilon_j, \epsilon_i, \epsilon_i)$. In this case, $\mfg=\mathfrak{so}_{2n+1}$.
\end{itemize}
\end{Lemma}
\begin{proof}
 We  write $\alpha+\nu=\sum_{l=1}^nb_l\epsilon_l$ since we are assuming   $\alpha, \nu\in \mathcal R^+$.   Note that $\alpha+\nu\not\in\mathcal R^+$. By \eqref{typeb}
 $\alpha+\nu\in I_1\cup I_2$, where $$I_1=\{3\epsilon_i\pm\epsilon_k, 2\epsilon_i-\epsilon_k-\epsilon_j, \epsilon_i+\epsilon_k-2\epsilon_j,  2\epsilon_i\pm\epsilon_l\}$$ and $$\begin{aligned} I_2= \{& \epsilon_k+\epsilon_i-\epsilon_l-\epsilon_j, \epsilon_k+\epsilon_i+\epsilon_l-\epsilon_j,\epsilon_k+\epsilon_i+\epsilon_l+\epsilon_j,\\ &  2\epsilon_i+\epsilon_k+\epsilon_j, 2\epsilon_i+\epsilon_k-\epsilon_j,   2\epsilon_k+2\epsilon_i,  \epsilon_k+\epsilon_i+\epsilon_j, \epsilon_k+\epsilon_i-\epsilon_l, 2\epsilon_i\}.\end{aligned}$$
 Since any element in $I_1$ can be uniquely decomposed into the summation of two positive roots, $\alpha+\nu\notin I_1$. When $\alpha+\nu\in I_2$, one can check (a)-(j) via \eqref{ord} directly. For example, if $\alpha+\nu=\epsilon_k+\epsilon_i-\epsilon_l-\epsilon_j$, we can assume $\alpha=\epsilon_k-\epsilon_l$ and $\nu=\epsilon_i-\epsilon_j$  without loss of any generality.
 In this case, if $\alpha+\nu=\beta+\gamma$ satisfying  $\alpha>\beta\geq \gamma>\nu$, then $\beta=\epsilon_k-\epsilon_j$, $\gamma=\epsilon_i-\epsilon_j$. and $i<k<j<l$. This verifies (a). Since  (b)-(j) can be checked in a similar way, we omit details.
\end{proof}

\begin{Lemma}\label{equal1}Suppose  $\alpha, \nu\in \mathcal R^+$ and $\alpha>\nu$. If $\alpha+\nu=\sum_{i=1}^j\gamma_i$, for some $j\ge 3$ and  $  \gamma_i\in \mathcal R^+$,   such that   $\alpha>\gamma_i>\nu$,     then $\{\alpha,\nu\}$ is one of  pairs in Lemma~\ref{equal2}(f)-(h). In this case, $\mathfrak g=\mathfrak{sp}_{2n}$.
\end{Lemma}
\begin{proof}

At moment, let $\mathcal{R}^{+}_{\mfg}$ be the set of  positive roots with respect to   $\mfg$. Thanks to \eqref{typeb} and  \eqref{ord},  $\mathcal{R}^{+}_{\mathfrak{so}_{2n}}\subset \mathcal{R}^{+}_{\mathfrak{sp}_{2n}}$ and  $\alpha>\nu$ in $\mathcal{R}^{+}_{\mathfrak{sp}_{2n}}$ if $\alpha>\nu$ in $\mathcal{R}^{+}_{\mathfrak{so}_{2n}}$. So, it is enough to consider  $\mfg\in\{\mathfrak{so}_{2n+1}, \mathfrak{sp}_{2n}\}$.
Suppose  \begin{equation} \label{kkk123} \alpha+\nu=\sum_{i=1}^j\gamma_i=\sum_{l=1}^n b_l\epsilon_l\end{equation}  where $\gamma_i$'s are required positive roots. By \eqref{typeb}, $\sum_{l=1}^n\vert b_l\vert\leq 4$. Let $A=\{\gamma_1, \gamma_2, \ldots, \gamma_j\}$.

 Suppose $\mfg=\mathfrak{sp}_{2n}$.
 Since $j\geq 3$ and  $\sum_{i=1}^j\gamma_i=\sum_{l=1}^n b_l\epsilon_l$, there is a $t$ such that  $\{\epsilon_k-\epsilon_t, x+\epsilon_t\}\subset A$ for some $k<t$ and some
 $x\in \{\pm\epsilon_s\mid 1\leq s\leq n, x+\epsilon_t\in \mathcal{R}_{\mathfrak{sp}_{2n}}^{+}\}$. Otherwise, $\sum_{l=1}^n\vert b_l\vert\geq 6$, a contradiction.

 Without loss of any generality, we  assume $\{\gamma_1, \gamma_2\}=\{\epsilon_k-\epsilon_t, x+\epsilon_t\}\subset \mathcal{R}^{+}_{\mathfrak{sp}_{2n}}$ as sets  and $\gamma_1<\gamma_2$. So, $x\neq -\epsilon_k$ and hence
 $\gamma'_1=\gamma_1+\gamma_2=\epsilon_k+x\in \mathcal{R}_{\mathfrak{sp}_{2n}}$.   However,  $x+\epsilon_k$ is a summation of two positive roots, it can not be  a negative root. So, $\gamma'_1\in \mathcal{R}^+_{\mathfrak{sp}_{2n}}$   and
 $\gamma_1<\gamma'_1<\gamma_2$. So, $\alpha+\nu$ is a summation  of $j-1$'s positive roots and all of them
 are between $\alpha $ and $\nu$.
 Using the above arguments repeatedly, we have
 \begin{equation}\label{sum} \lambda_1+\lambda_2+\lambda_3=\alpha+\nu\end{equation} where $\alpha>\lambda_i>\nu$ and $\lambda_1+\lambda_2=\beta\in\mathcal{R}_{\mathfrak{sp}_{2n}}^+$. Without loss of any generality, we can assume  $\lambda_1<\beta<\lambda_2$ and
$(\lambda_2, \lambda_1)$ is a minimal pair of $\beta$.
 Since $\beta+\lambda_3=\alpha+\nu$ and $\alpha>\lambda_3>\nu$ and $\alpha>\beta>\nu$, $\{\alpha,\nu\}$ has to be one of  pairs in Lemmas~\ref{equal}-\ref{equal2}.
 Since
     $$\begin{aligned} (2\epsilon_k)+(\epsilon_i-\epsilon_j)&=(\epsilon_k-\epsilon_n)+(\epsilon_k-\epsilon_j)+(\epsilon_i+\epsilon_n),\\
(\epsilon_k+\epsilon_j)+2\epsilon_i&=(\epsilon_i-\epsilon_j)+(\epsilon_k+\epsilon_j)+(\epsilon_i+\epsilon_j),\\
(2\epsilon_k)+(2\epsilon_i)&=(\epsilon_k-\epsilon_j)+(\epsilon_i+\epsilon_j)+(\epsilon_i+\epsilon_k),\\
\end{aligned}$$ $\{\alpha, \nu\}$ can be one of pairs   in Lemmas~\ref{equal2}(f)-(h). Otherwise, we know the exact information on $\beta$. We use Corollary~\ref{equa} to conclude that there is no required minimal pair $(\lambda_2, \lambda_1)$ of $\beta$ such that $\nu<\lambda_i<\alpha$. This  contradicts to  \eqref{sum}. We give an example as follows. One can   check other cases in a similar way.

If $\{\alpha,\nu\}$ is the pair in Lemma~\ref{equal2}(a), then $\{\beta, \lambda_3\}=\{\epsilon_k-\epsilon_j, \epsilon_i-\epsilon_l\}$, $i<k<j<l$. In this case, $\epsilon_i-\epsilon_j<\epsilon_k-\epsilon_m<\epsilon_k-\epsilon_l<\epsilon_m-\epsilon_j$ for each minimal pair $(\epsilon_m-\epsilon_j,\epsilon_k-\epsilon_m)$ of $\epsilon_k-\epsilon_j$ in Corollary~\ref{equa} and either $\epsilon_i-\epsilon_m\leq\epsilon_i-\epsilon_j$ or $\epsilon_m-\epsilon_l>\epsilon_k-\epsilon_l$ for each minimal pair $(\epsilon_m-\epsilon_l,\epsilon_i-\epsilon_m)$ of $\epsilon_i-\epsilon_l$ in Corollary~\ref{equa}. We can not find the required minimal pair $(\lambda_2, \lambda_1)$ of $\beta$. This is a contradiction.

\medskip

Suppose $\mfg=\mathfrak{so}_{2n+1}$. Thanks to \eqref{typeb} and \eqref{kkk123},  $b_l\leq 2$ for all $1\leq l\leq n$.
We will get a contradiction in any case. First, we assume
$b_s=2$ for some $1\leq s\leq n$. By \eqref{typeb}, $\alpha, \nu\in\{\epsilon_s, \epsilon_s\pm\epsilon_j, \epsilon_l+\epsilon_s\mid l<s<j\}$.

 Suppose  $\nu\in\{\epsilon_s, \epsilon_s\pm\epsilon_j\mid s< j\}$.  Since
 $\alpha>\nu$, we have $\alpha\neq \epsilon_l+\epsilon_s$ if $l<s$. Thanks to \eqref{ord}, the coefficient of $\epsilon_s$ in the expression of $\gamma_i$ is $1$ for $\alpha>\gamma_i>\nu$. Since $j\geq 3$ and  $\sum_{i=1}^j\gamma_i=\sum_{l=1}^n b_l\epsilon_l$, $b_s\geq 3$, a contradiction.

  Suppose $\nu=\epsilon_l+\epsilon_s$ for some $l<s$.  Then $\alpha\in \{ \epsilon_k+\epsilon_s, \epsilon_s, \epsilon_s-\epsilon_j \}$ and  either  $k> s$  or $k< s$. Firstly, if $\alpha=\epsilon_s$, then $\alpha+\nu=\sum_{i=1}^j\gamma_i=\epsilon_l+2\epsilon_s$ and either $\epsilon_s+\epsilon_t\in A$ or $\epsilon_s-\epsilon_f\in A$ for some $t$ and $f$. Otherwise,  $b_s\neq 2$.
\begin{itemize}\item[(a)]
Suppose   $\epsilon_s+\epsilon_t\in A$. Since $\nu<\epsilon_s+\epsilon_t<\alpha$, by \eqref{ord},
 $l<t<s$.   Further, $\epsilon_{k_{0}}-\epsilon_t\in A$ for some ${k_{0}}<t$. Otherwise,  $b_t\neq 0$. If ${k_{0}}\leq l$,  by \eqref{ord} $\epsilon_{k_{0}}-\epsilon_t<\epsilon_l+\epsilon_s$, a contradiction. Suppose  $t>{k_{0}}>l$. Since $b_{k_{0}}=0$, there is a $k_1$ such that $\epsilon_{k_1}-\epsilon_{k_0}\in A$. If $k_1<l$, we get a contradiction via \eqref{ord}. Otherwise, we repeat above arguments so that we can assume $k_1<l$. This leads to a contradiction, too.

   \item [(b)]
   Suppose $\epsilon_s-\epsilon_{f}\in A $ for some $f>s$. Since $b_{f}=0$ and $\alpha>\gamma_i>\nu$,   $\epsilon_{j}+\epsilon_{f}\in A $ for some $l<j<s$.  Further,
    $b_j=0$ and
  $\epsilon_{j_0}-\epsilon_{j}\in A$ for some $j_0<j$.
  If $j_0\le l$, by \eqref{ord}, $\epsilon_{j_0}-\epsilon_j<\nu$, a contradiction.
If $j_0>l$, this  also leads to  a contradiction by using  arguments in (a).\end{itemize}

Secondly,  if    $\alpha=\epsilon_s-\epsilon_j$, then $\alpha+\nu=\epsilon_l+2\epsilon_s-\epsilon_j$,  $l<s<j$. So, $b_l=1$.   Note that  $\gamma_i>\nu=\epsilon_l+\epsilon_s$. We have
$\epsilon_l+\epsilon_t\in A$ for some $l<t<s$. Consequently, $b_t=0$ and  $\epsilon_f-\epsilon_t\in A $ for some $f<t$.  Now, comparing $f$ and $l$ and  using  arguments in (a) repeatedly leads to a contradiction.

Finally, if  $\alpha=\epsilon_k+\epsilon_s$, then  $\alpha+\nu=\epsilon_k+2\epsilon_s+\epsilon_l$ with $l<s$.
  Suppose $ k<s$. Since $\alpha>\nu$, we have   $l<k<s$ . Note that  $b_s=2$ and $\nu<\gamma_i<\alpha$. By \eqref{ord},   $  \epsilon_t+\epsilon_s\in A$ for some $l<t<k$.
 So, $b_t=0$ and $ \epsilon_j-\epsilon_t\in A$ for some $j<t$.
  Now, comparing $f$ and $l$ and  using  arguments in (a) repeatedly leads to a contradiction.
 When $k>s$, $b_k=1$. One can get a contradiction by arguments on the case $k<s$.

 We have proved that  $b_s\neq 2$ for any admissible  $s$. It remains to deal with the case $b_s<2$
  for all $1\leq s\leq n$.
We claim  that one of the following cases has to  happen:
 \begin{itemize} \item [(1)] $\epsilon_f, \epsilon_g\in A$ for some $f\neq g$,
  \item [(2)] $\epsilon_k-\epsilon_t, x+\epsilon_t\in A$ for some $x\in\{\pm\epsilon_s, 0\mid 1\leq s\leq n, x+\epsilon_t\in \mathcal{R}_{\mathfrak{so}_{2n+1}}^{+}\}$ and $1\leq k<t\leq n$. \end{itemize}
      It is enough to prove that (1) happens if we assume (2) does not happen. In fact, since $j\geq 3$,  there is a pair $\{\gamma_i,\gamma_t\}$ such that $\gamma_i=\epsilon_f$ and $\gamma_t=\epsilon_g$, for some $f$ and $g$. Otherwise,
 $\sum_{i=1}^n|b_i|\geq 5$, a contradiction. Suppose   $f= g$. Since we are assuming
    $b_f<2$,  $ \epsilon_k-\epsilon_f\in A$ and $\{\epsilon_k-\epsilon_f, \epsilon_f\}\subset  A $.
    This shows that (2) happens, a contradiction.  So, $f\neq g$ and (1) happens.

   In the first case,  we have $\epsilon_f+\epsilon_g\in  \mathcal{R}_{\mathfrak{so}_{2n+1}}^{+}$ since $f\neq g$. In the second case, if $x\neq \epsilon_k$, we have $\epsilon_k-\epsilon_t+x+\epsilon_t=x+\epsilon_k\in \mathcal{R}_{\mathfrak{so}_{2n+1}}^{+}$. If  $x=\epsilon_k$. Since $b_k<2$, $\epsilon_l-\epsilon_k\in A$ for some $l$ and $\epsilon_l-\epsilon_k+\epsilon_k-\epsilon_t=\epsilon_l-\epsilon_t\in \mathcal{R}_{\mathfrak{so}_{2n+1}}^{+}$. In any of these cases, we can assume that $$\gamma_1'=\gamma_1+\gamma_2\in \mathcal{R}_{\mathfrak{so}_{2n+1}}^{+}, \text{ and  $\gamma_1'+\sum_{i=3}^{j}\gamma_i=\alpha+\nu$,}$$ $\nu<\gamma_1<\gamma_1'<\gamma_2<\alpha$. By induction on $j$, we need to deal with the case  $\lambda_1+\lambda_2+\lambda_3=\alpha+\nu$ and $\lambda_1+\lambda_2\in \mathcal{R}_{\mathfrak{so}_{2n+1}}^{+}$.
 Via Corollary~\ref{equa}, one can check that this will never happen by arguments similar to those for the case $\mfg=\mathfrak{sp}_{2n}$.
\end{proof}

\begin{Cor}\label{monoial1}Suppose $\alpha, \nu\in \mathcal{R}^{+}$  such that
 $\alpha>\nu$ and    $\alpha+\nu\notin \mathcal{R}^{+}$. If  there are no $\beta, \gamma\in \mathcal R^+$ such that $\alpha>\beta\geq \gamma>\nu$ and $\alpha+\nu=\beta+\gamma$,   then  $x_{\nu}^{-}x_{\alpha}^{-}=v^{(\alpha\mid\nu)}x_{\alpha}^{-}x_{\nu}^{-}$.
\end{Cor}
\begin{proof} By assumption  and  Lemma~\ref{equal1}, $\alpha+\nu\neq \sum_{i=1}^j \gamma_i$, where $\alpha>\gamma_i>\nu$, $\gamma_i\in \mathcal R^+$ and $j\geq 2$.  Since  $\alpha+\nu\notin \mathcal{R}^{+}$, by Lemma~\ref{commut}, we have $x_{\nu}^{-}x_{\alpha}^{-}=v^{(\alpha\mid\nu)}x_{\alpha}^{-}x_{\nu}^{-}$.
\end{proof}

 Thanks to \eqref{vv},  $q=v^2$ if $\mfg=\mathfrak{so}_{2n+1}$ and $q=v$ if $\mfg\in\{\mathfrak{so}_{2n}, \mathfrak{sp}_{2n}\}$. Recall that  $[k]=\frac{v^k-v^{-k}}{v-v^{-1}}$ and  $z_q=q-q^{-1}$.  We are going to deal with pairs $\{\alpha, \nu\}$ in (2)-(3). First, we list some equalities which follow from Lemma~\ref{min}.
Recall that $ \Upsilon_\beta$ is the set of minimal pairs of $\beta\in \mathcal R^+$.

\begin{Cor}\label{25a}
Suppose $\mfg\in\{\mathfrak{so}_{2n+1}, \mathfrak{so}_{2n}, \mathfrak{sp}_{2n}\}$. \begin{itemize}
\item[(1)] $[x_{\epsilon_i-\epsilon_{k}}^-, x_{\epsilon_{k}-\epsilon_{j}}^-]_v=-q^{-1}x_{\epsilon_i-\epsilon_j}^-$, if $1\leq i<k<j\leq n$,
\item[(2)] $[x_{\epsilon_i+\epsilon_{j+1}}^-, x_{\epsilon_{j}-\epsilon_{j+1}}^-]_v=-q^{-1}x_{\epsilon_i+\epsilon_{j}}^-$, if $1\leq i<j< n$,
\item[(3)] $[x_{\epsilon_i}^-, x_{\epsilon_{n}}^-]_v=-[2]x_{\epsilon_i+\epsilon_{n}}^-$, if $1\leq i< n$ and $\mfg=\mathfrak{so}_{2n+1}$,
\item[(4)] $[x_{\epsilon_i-\epsilon_{k}}^-, x_{\epsilon_{k}}^-]_v=-q^{-1}x_{\epsilon_i}^-$, if $1\leq i<k\leq n$ and $\mfg=\mathfrak{so}_{2n+1}$,
\item[(5)] $[x_{\epsilon_i-\epsilon_{k}}^-, x_{\epsilon_{k}+\epsilon_{j}}^-]_v=-q^{-1}x_{\epsilon_i+\epsilon_j}^-$, if $1\leq i<k<j\leq n$,
\item[(6)] $[x_{\epsilon_i-\epsilon_{n}}^-, x_{2\epsilon_{n}}^-]_v=-q^{-2}x_{\epsilon_i+\epsilon_n}^-$, if $1\leq i< n$ and $\mfg=\mathfrak{sp}_{2n}$,
\item[(7)] $[x_{\epsilon_i-\epsilon_{n}}^-, x_{\epsilon_i+\epsilon_{n}}^-]_v=-[2]x_{2\epsilon_i}^-$, if $1\leq i< n$ and $\mfg=\mathfrak{sp}_{2n}$.
\end{itemize}
\end{Cor}

\begin{proof} In (1)-(7), we compute $[x_\nu^-, x_\alpha^-]_\nu$ for some admissible $\alpha, \nu\in \mathcal R^+$. Thanks to
Corollary ~\ref{equa}, such $(\alpha, \nu)$'s  are minimal pairs. So, (1)-(7) follow from  Lemma~\ref{min}, immediately.
\end{proof}

\begin{Lemma}\label{commute1}For all admissible $i<j<k$, let  $a_k=[x_{\epsilon_i+\epsilon_k}^-, x_{\epsilon_j-\epsilon_k}^-]_v$. Then
$$a_k
=(-q)^{j-k}x_{\epsilon_i+\epsilon_j}^--z_q\sum_{t=j+1}^{k-1}(-q)^{t-k}x_{\epsilon_j-\epsilon_t}^-x_{\epsilon_i+\epsilon_t}^-.$$
\end{Lemma}
\begin{proof}  The required formula for $a_{j+1}$ follows from Corollary~\ref{25a}(2). Suppose $k>j+1$. Then
$$\begin{aligned}
&a_j=-qx_{\epsilon_i+\epsilon_k}^-
[x_{\epsilon_j-\epsilon_{k-1}}^-, x_{\epsilon_{k-1}-\epsilon_{k}}^-]_v-q^{-1}x_{\epsilon_{j}-\epsilon_{k}}^-x_{\epsilon_i+\epsilon_k}^-, \text{ by Corollary~\ref{25a}(1)}
\\&=\frac{1}{q}(z_qx_{\epsilon_j-\epsilon_{k-1}}^-x_{\epsilon_i+\epsilon_{k-1}}^--[x_{\epsilon_i+\epsilon_{k-1}}^-
, x_{\epsilon_j-\epsilon_{k-1}}^-]_v-x_{\epsilon_{j}-\epsilon_{k}}^-x_{\epsilon_i+\epsilon_k}^-)
-[x_{\epsilon_j-\epsilon_{k-1}}^-, x_{\epsilon_{k-1}-\epsilon_{k}}^-]_vx_{\epsilon_i+\epsilon_{k}}^-\\&=-q^{-1}a_{k-1}+q^{-1}z_qx_{\epsilon_j-\epsilon_{k-1}}^-x_{\epsilon_i+\epsilon_{k-1}}^-, \text{ by Corollary~\ref{25a}(1)}
\\& =(-q)^{j+1-k}a_{j+1}-
\sum_{t=j+1}^{k-1}(-q)^{t-k}z_qx_{\epsilon_j-\epsilon_t}^-x_{\epsilon_i+\epsilon_t}^-,~\text{by induction assumption on $k$},
\\ & = (-q)^{j-k}x_{\epsilon_i+\epsilon_j}^--z_q\sum_{t=j+1}^{k-1}(-q)^{t-k}
x_{\epsilon_j-\epsilon_t}^-x_{\epsilon_i+\epsilon_t}^-, ~\text{by  Corollary~\ref{25a}(2), }
\end{aligned}$$
where the second and third equalities follow from $[x_{\epsilon_i+\epsilon_k}^-, x_{\epsilon_j-\epsilon_{k-1}}^-]_v=0$ (see  Corollary~\ref{monoial1}) and Corollary~\ref{25a}(2).
  So, we have the required formula on $a_k$ in general.
  \end{proof}

\begin{Lemma}\label{commute2}For any $j$, $i<j\leq n$, let $a_j=[x_{\epsilon_i-\epsilon_j}^-, x_{\epsilon_i+\epsilon_j}^-]_v$. Then
$$
a_j -\sum_{t=j+1}^{n}(-q)^{1+j-t}z_qx_{\epsilon_i+\epsilon_t}^-x_{\epsilon_i-\epsilon_t}^-
=\begin{cases}0, & \text{if $\mfg=\mathfrak{so}_{2n} $,} \\ -(-q)^{j-n} [2] x_{2\epsilon_i}^-, & \text{if $\mfg=\mathfrak{sp}_{2n} $,}\\ (-q)^{j-n} {[2]}^{-1} (1-q^{-1})(x_{\epsilon_i}^-)^2, & \text{if $\mfg=\mathfrak{so}_{2n+1} $.}\\ \end{cases}$$
\end{Lemma}
\begin{proof}
Suppose $j<n$.
$$\begin{aligned}
&a_j=-qx_{\epsilon_i-\epsilon_j}^-
[x_{\epsilon_i+\epsilon_{j+1}}^-, x_{\epsilon_{j}-\epsilon_{j+1}}^-]_v-x_{\epsilon_{i}+\epsilon_{j}}^-x_{\epsilon_i-\epsilon_j}^-, \text{ by Corollary~\ref{25a}(2)}\\& =\frac{1}{q}(x_{\epsilon_{j}-\epsilon_{j+1}}^-x_{\epsilon_i-\epsilon_j}^--x_{\epsilon_i-\epsilon_{j+1}}^-)
x_{\epsilon_i+\epsilon_{j+1}}^--qx_{\epsilon_i+\epsilon_{j+1}}^-(x_{\epsilon_{j}-\epsilon_{j+1}}^-x_{\epsilon_i-\epsilon_j}^--x_{\epsilon_i-\epsilon_{j+1}}^-)
-x_{\epsilon_{i}+\epsilon_{j}}^-x_{\epsilon_i-\epsilon_j}^-\\&=-q[x_{\epsilon_i+\epsilon_{j+1}}^-, x_{\epsilon_{j}-\epsilon_{j+1}}^-]_vx_{\epsilon_i-\epsilon_j}^--q^{-1}[x_{\epsilon_i-\epsilon_{j+1}}^-
, x_{\epsilon_i+\epsilon_{j+1}}^-]_v+z_qx_{\epsilon_i+\epsilon_{j+1}}^-x_{\epsilon_i-\epsilon_{j+1}}^-
-x_{\epsilon_{i}+\epsilon_{j}}^-x_{\epsilon_i-\epsilon_j}^-\\&=-q^{-1}a_{j+1}+z_qx_{\epsilon_i+\epsilon_{j+1}}^-x_{\epsilon_i-\epsilon_{j+1}}^-
,\text{ by Corollary~\ref{25a}(2)}\\&=(-q)^{j-n}a_{n}+\sum_{t
=j+1}^{n}(-q)^{1+j-t}z_qx_{\epsilon_i+\epsilon_t}^-x_{\epsilon_i-\epsilon_t}^-, \text { by induction on $n-j$},\end{aligned}$$
where the second and third equalities follow from  $[x_{\epsilon_i-\epsilon_j}^-, x_{\epsilon_i+\epsilon_{j+1}}^-]_v=0$ (see  Corollary~\ref{monoial1}) and Corollary~\ref{25a}(1).
Therefore,  the result follows from the corresponding result on   $a_n$.

 When $\mfg\in \{ \mathfrak{so}_{2n} , \mathfrak{sp}_{2n}\}$,  the required formulae on $a_n$  follow from Corollary~\ref{monoial1} and  Corollary~\ref{25a} (7), respectively.
If $\mfg=\mathfrak{so}_{2n+1}$, then
$$\begin{aligned}
a_n&=-[2]^{-1}x_{\epsilon_i-\epsilon_n}^-
[x_{\epsilon_i}^-, x_{\epsilon_n}^-]_v-x_{\epsilon_i+\epsilon_n}^-x_{\epsilon_{i}-\epsilon_{n}}^-,~\text{by Corollary~\ref{25a}(3)}\\&=[2]^{-1}(q^{-1}(x_{\epsilon_n}^-x_{\epsilon_i-\epsilon_n}^--
x_{\epsilon_i}^-)x_{\epsilon_i}^--x_{\epsilon_i}^-(x_{\epsilon_n}^-x_{\epsilon_i-\epsilon_n}^--
x_{\epsilon_i}^-))-x_{\epsilon_i+\epsilon_n}^-x_{\epsilon_{i}-\epsilon_{n}}^-\\&= [2]^{-1}(-[x_{\epsilon_i}^-, x_{\epsilon_n}^-]_v
x_{\epsilon_{i}-\epsilon_{n}}^--q^{-1}(x_{\epsilon_{i}}^-)^2
+(x_{\epsilon_{i}}^-)^2)-x_{\epsilon_i+\epsilon_n}^-x_{\epsilon_{i}-\epsilon_{n}}^-\\&=[2]^{-1}(1-q^{-1})(x_{\epsilon_{i}}^-)^2, ~\text{by Corollary~\ref{25a}(3)},
\end{aligned}$$
 where the second and third equalities follow from $[x_{\epsilon_i-\epsilon_n}^-, x_{\epsilon_i}^-]_v=0$ (see  Corollary~\ref{monoial1}) and Corollary~\ref{25a}(4).
\end{proof}

\begin{Lemma}\label{commute3}Suppose $\mathfrak g=\mathfrak{so}_{2n+1}$. For all admissible $i<j<n$, we have  $$[x_{\epsilon_i}^-, x_{\epsilon_j}^-]_v=-[2](-q)^{j-n}x_{\epsilon_i+\epsilon_j}^-+[2]\sum_{t=j+1}^{n}(-q)^{t-n}z_qx_{\epsilon_j-
\epsilon_t}^-x_{\epsilon_i+\epsilon_t}^-.$$
\end{Lemma}
\begin{proof}We have
 $$\begin{aligned}
{[}x_{\epsilon_i}^-, x_{\epsilon_{j}}^-{]}_v&=-qx_{\epsilon_i}^-
[x_{\epsilon_j-\epsilon_{n}}^-, x_{\epsilon_{n}}^-]_v- x_{\epsilon_{j}}^-x_{\epsilon_i}^-,~\text{by Corollary~\ref{25a}(4)}\\&=(x_{\epsilon_{n}}^-x_{\epsilon_i}^--[2]x_{\epsilon_i+\epsilon_{n}}^-)
x_{\epsilon_j-\epsilon_{n}}^--qx_{\epsilon_j-\epsilon_{n}}^-(x_{\epsilon_{n}}^-x_{\epsilon_i}^--[2]x_{\epsilon_i+\epsilon_{n}}^-)
- x_{\epsilon_{j}}^-x_{\epsilon_i}^-\\&=-q[x_{\epsilon_j-\epsilon_{n}}^-, x_{\epsilon_{n}}^-]_vx_{\epsilon_i}^--[2][x_{\epsilon_i+\epsilon_n}^-, x_{\epsilon_j-\epsilon_n}^-]_v+[2]z_qx_{\epsilon_j-\epsilon_n}^-x_{\epsilon_i+\epsilon_n}^-
- x_{\epsilon_{j}}^-x_{\epsilon_i}^-\\&=-[2][x_{\epsilon_i+\epsilon_n}^-, x_{\epsilon_j-\epsilon_n}^-]_v+[2]z_qx_{\epsilon_j-\epsilon_n}^-x_{\epsilon_i+\epsilon_n}^-,~\text{by Corollary~\ref{25a}(4)},\\
&=-[2](-q)^{j-n}x_{\epsilon_i+\epsilon_j}^-+[2]\sum_{t=j+1}^{n}(-q)^{t-n}
z_qx_{\epsilon_j-\epsilon_t}^-x_{\epsilon_i+\epsilon_t}^-, \text{by  Lemma~\ref{commute1},}\\
\end{aligned}$$
 where the second and third equalities follow from  $[x_{\epsilon_i}^-, x_{\epsilon_j-\epsilon_{n}}^-]_v=0$ (see  Corollary~\ref{monoial1}) and Corollary~\ref{25a}(3).  \end{proof}

\begin{Lemma}\label{commute4}If $(\alpha, \nu, \beta, \gamma)$ is one of pairs in Lemma~\ref{equal2}(a)-(e). Then $[x_{\nu}^-, x_{\alpha}^-]_v=z_qx_{\beta}^-x_{\gamma}^-$.
\end{Lemma}
\begin{proof}Suppose  $(\alpha, \nu, \beta, \gamma)$ is given in Lemma~\ref{equal2}(a). Then
$$\begin{aligned}{[}x_{\epsilon_i-\epsilon_{j}}^-&,  x_{\epsilon_{k}-\epsilon_{l}}^-]_v
 =-qx_{\epsilon_i-\epsilon_j}^-
[x_{\epsilon_k-\epsilon_{j}}^-, x_{\epsilon_{j}-\epsilon_{l}}^-]_v-x_{\epsilon_k-\epsilon_{l}}^-x_{\epsilon_i-\epsilon_{j}}^-,~\text{by Corollary~\ref{25a}(1)}\\= & q^{-1}(x_{\epsilon_{j}-\epsilon_{l}}^-x_{\epsilon_i-\epsilon_j}^--x_{\epsilon_i-\epsilon_{l}}^-)
x_{\epsilon_k-\epsilon_{j}}^--qx_{\epsilon_k-\epsilon_{j}}^-(x_{\epsilon_{j}-\epsilon_{l}}^-x_{\epsilon_i-\epsilon_j}^--x_{\epsilon_i-\epsilon_{l}}^-)
-x_{\epsilon_k-\epsilon_{l}}^-x_{\epsilon_i-\epsilon_{j}}^-\\ =&-q[x_{\epsilon_k-\epsilon_{j}}^-, x_{\epsilon_{j}-\epsilon_{l}}^-]_vx_{\epsilon_i-\epsilon_{j}}^-
+qx_{\epsilon_k-\epsilon_{j}}^-x_{\epsilon_{i}-\epsilon_{l}}^--q^{-1}x_{\epsilon_{i}-\epsilon_{l}}^-x_{\epsilon_k-\epsilon_{j}}^--x_{\epsilon_k-\epsilon_{l}}^-x_{\epsilon_i-\epsilon_{j}}^-
\\= &
z_qx_{\epsilon_k-\epsilon_{j}}^-x_{\epsilon_{i}-\epsilon_{l}}^-,~\text{by Corollary~\ref{25a}(1)},\end{aligned}$$  where the second and third equalities follow from $[x_{\epsilon_i-\epsilon_j}^-, x_{\epsilon_k-\epsilon_{j}}^-]_v=0$, $[x_{\epsilon_i-\epsilon_l}^-, x_{\epsilon_k-\epsilon_{j}}^-]_v=0$ (see Corollary~\ref{monoial1}) and Corollary~\ref{25a}(1).

Suppose  $(\alpha, \nu, \beta, \gamma)$ is given in Lemma~\ref{equal2}(b). Then
\begin{equation}\label{ccc3}\begin{aligned}&{[}x_{\epsilon_i+\epsilon_l}^-, x_{\epsilon_k+\epsilon_j}^-{]_v}
=-q
[x_{\epsilon_i-\epsilon_{k}}^-, x_{\epsilon_{k}+\epsilon_{l}}^-]_vx_{\epsilon_{k}+\epsilon_{j}}^--x_{\epsilon_k+\epsilon_{j}}^-x_{\epsilon_i+\epsilon_{l}}^-,~\text{by Corollary~\ref{25a}(5)}\\=& -q(x_{\epsilon_{k}+\epsilon_{j}}^-x_{\epsilon_i-\epsilon_k}^--x_{\epsilon_i+\epsilon_{j}}^-)
x_{\epsilon_k+\epsilon_{l}}^-+q^{-1}x_{\epsilon_k+\epsilon_{l}}^-(x_{\epsilon_{k}+\epsilon_{j}}^-x_{\epsilon_i-\epsilon_k}^--x_{\epsilon_i+\epsilon_{j}}^-)-x_{\epsilon_k+\epsilon_{j}}^-x_{\epsilon_i+\epsilon_{l}}^-
\\=&-qx_{\epsilon_k+\epsilon_{j}}^-[x_{\epsilon_i-\epsilon_{k}}^-, x_{\epsilon_{k}+\epsilon_{l}}^-]_v
+qx_{\epsilon_i+\epsilon_{j}}^-x_{\epsilon_k+\epsilon_{l}}^--q^{-1}x_{\epsilon_k+\epsilon_{l}}^-x_{\epsilon_i+\epsilon_{j}}^--x_{\epsilon_k+\epsilon_{j}}^-x_{\epsilon_i+\epsilon_{l}}^-
\\=&
z_qx_{\epsilon_k+\epsilon_l}^-x_{\epsilon_i+\epsilon_j}^-,~\text{by Corollary~\ref{25a}(5)},\end{aligned}\end{equation}
where the second and third equalities follow from $[x_{\epsilon_k+\epsilon_l}^-, x_{\epsilon_h+\epsilon_{j}}^-]_v=0$, if $h\in\{i, k\}$ (see  Corollary~\ref{monoial1}) and Corollary~\ref{25a}(5).

 Suppose $(\alpha, \nu, \beta, \gamma)$  is given in Lemma~\ref{equal2}(c).
 So $(\alpha, \nu, \beta, \gamma)=(\epsilon_k+\epsilon_j, \epsilon_i-\epsilon_l, \epsilon_k-\epsilon_l, \epsilon_i+\epsilon_j)$, $i<k<j$, $i<k<l$  and $j\neq l$.
 If  $j>l$, then \begin{equation} \label{ccc2} \begin{aligned}{[}&x_{\epsilon_i-\epsilon_l}^-}  , {x_{\epsilon_k+\epsilon_j}^-]_v
=-qx_{\epsilon_i-\epsilon_l}^-
[x_{\epsilon_k-\epsilon_{l}}^-, x_{\epsilon_{l}+\epsilon_{j}}^-]_v-x_{\epsilon_{k}+\epsilon_{j}}^-x_{\epsilon_i-\epsilon_l}^-,~\text{by Corollary~\ref{25a}(5)}\\
&=q^{-1}(x_{\epsilon_{l}+\epsilon_{j}}^-x_{\epsilon_i-\epsilon_l}^--x_{\epsilon_i+\epsilon_{j}}^-)
x_{\epsilon_k-\epsilon_{l}}^--qx_{\epsilon_k-\epsilon_{l}}^-(x_{\epsilon_{l}+\epsilon_{j}}^-x_{\epsilon_i-\epsilon_l}^--x_{\epsilon_i+\epsilon_{j}}^-)
-x_{\epsilon_{k}+\epsilon_{j}}^-x_{\epsilon_i-\epsilon_l}^-\\&=-q
[x_{\epsilon_k-\epsilon_{l}}^-, x_{\epsilon_{l}+\epsilon_{j}}^-]_vx_{\epsilon_i-\epsilon_l}^-
+qx_{\epsilon_k-\epsilon_l}^-x_{\epsilon_i+\epsilon_j}^--q^{-1}x_{\epsilon_i+\epsilon_j}^-x_{\epsilon_k-\epsilon_l}^--x_{\epsilon_{k}+\epsilon_{j}}^-x_{\epsilon_i-\epsilon_l}^-
\\&=
z_qx_{\epsilon_k-\epsilon_l}^-x_{\epsilon_i+\epsilon_j}^-,~\text{by Corollary~\ref{25a}(5)}, \end{aligned}\end{equation}where the second and third equalities follow from $[x_{\epsilon_i-\epsilon_l}^-, x_{\epsilon_k-\epsilon_{l}}^-]_v=0$, $ [x_{\epsilon_i+\epsilon_j}^-, x_{\epsilon_k-\epsilon_{l}}^-]_v=0$ (see  Corollary~\ref{monoial1}) and Corollary~\ref{25a}(5).
If $j<l$, then
   \begin{equation}\label{ccc4}\begin{aligned}{[}&x_{\epsilon_i-\epsilon_l}^-, x_{\epsilon_{k}+\epsilon_{j}}^-{]_v}
=-q
[x_{\epsilon_i-\epsilon_k}^-, x_{\epsilon_k-\epsilon_{l}}^-]_vx_{\epsilon_k+\epsilon_j}^--x_{\epsilon_{k}+\epsilon_{j}}^-x_{\epsilon_i-\epsilon_l}^-
,~\text{by Corollary~\ref{25a}(1)}\\&=q^{-1}x_{\epsilon_k-\epsilon_{l}}^-(x_{\epsilon_{k}+\epsilon_{j}}^-x_{\epsilon_i-\epsilon_k}^--x_{\epsilon_i+\epsilon_{j}}^-)
-q(x_{\epsilon_{k}+\epsilon_{j}}^-x_{\epsilon_i-\epsilon_k}^--x_{\epsilon_i+\epsilon_{j}}^-)x_{\epsilon_k-\epsilon_{l}}^--x_{\epsilon_{k}+\epsilon_{j}}^-x_{\epsilon_i-\epsilon_l}^-
\\&=-qx_{\epsilon_k+\epsilon_j}^-[x_{\epsilon_i-\epsilon_k}^-, x_{\epsilon_k-\epsilon_{l}}^-]_v
-q^{-1}x_{\epsilon_k-\epsilon_l}^-x_{\epsilon_{i}+\epsilon_{j}}^-+qx_{\epsilon_{i}+\epsilon_{j}}^-x_{\epsilon_k-\epsilon_l}^--x_{\epsilon_{k}+\epsilon_{j}}^-x_{\epsilon_i-\epsilon_l}^-
\\&=
z_qx_{\epsilon_k-\epsilon_l}^-x_{\epsilon_{i}+\epsilon_{j}}^-,~\text{by Corollary~\ref{25a}(1)},\end{aligned}\end{equation}
where the second and third equalities follow from $[x_{\epsilon_k-\epsilon_l}^-, x_{\epsilon_k+\epsilon_{j}}^-]_v=0$, $[x_{\epsilon_i+\epsilon_j}^-, x_{\epsilon_k-\epsilon_{l}}^-]_v=0$ (see  Corollary~\ref{monoial1}) and Corollary~\ref{25a}(5).

 Suppose $(\alpha, \nu, \beta, \gamma)$  is given in Lemma~\ref{equal2}(d). Then
 \begin{equation}\label{ccc1}  \begin{aligned}
{ [x_{\epsilon_i-\epsilon_j}^-, x_{\epsilon_k}^-]_v}
&=-q[x_{\epsilon_{i}-\epsilon_{k}}^-, x_{\epsilon_k-\epsilon_{j}}^-]_vx_{\epsilon_k}^--x_{\epsilon_k}^-x_{\epsilon_i-\epsilon_j}^-,~\text{by Corollary~\ref{25a}(1)}\\&=-q(x_{\epsilon_k}^-x_{\epsilon_i-\epsilon_k}^--x_{\epsilon_i}^-)
x_{\epsilon_k-\epsilon_{j}}^-+q^{-1}x_{\epsilon_k-\epsilon_{j}}^-(x_{\epsilon_k}^-x_{\epsilon_i-\epsilon_k}^--x_{\epsilon_i}^-)-x_{\epsilon_k}^-x_{\epsilon_i-\epsilon_j}^-
\\&=-qx_{\epsilon_k}^-[x_{\epsilon_{i}-\epsilon_{k}}^-, x_{\epsilon_k-\epsilon_{j}}^-]_v
+qx_{\epsilon_i}^-x_{\epsilon_k-\epsilon_j}^--q^{-1}x_{\epsilon_k-\epsilon_j}^-x_{\epsilon_i}^--x_{\epsilon_k}^-x_{\epsilon_i-\epsilon_j}^-
\\&=
z_qx_{\epsilon_k-\epsilon_j}^-x_{\epsilon_i}^-,~\text{by Corollary~\ref{25a}(1)},\end{aligned}\end{equation}
where the second and third equalities follow from $[x_{\epsilon_k-\epsilon_j}^-, x_{\epsilon_k}^-]_v=0$, $ [x_{\epsilon_i}^-, x_{\epsilon_k-\epsilon_{j}}^-]_v=0$ (see  Corollary~\ref{monoial1}) and Corollary~\ref{25a}(4).

Finally, suppose $(\alpha, \nu, \beta, \gamma)$  is given in Lemma~\ref{equal2}(e). We have
$$\begin{aligned}{[x_{\epsilon_i}^-, x_{\epsilon_k+\epsilon_j}^-]_v}&=-q[x_{\epsilon_{i}-\epsilon_{k}}^-, x_{\epsilon_k}^-]_vx_{\epsilon_k+\epsilon_j}^--x_{\epsilon_k+\epsilon_j}^-x_{\epsilon_i}^-,~\text{by Corollary~\ref{25a}(4)}\\&
=-q(x_{\epsilon_k+\epsilon_j}^-x_{\epsilon_i-\epsilon_k}^--x_{\epsilon_i+\epsilon_j}^-)
x_{\epsilon_k}^-+q^{-1}x_{\epsilon_k}^-(x_{\epsilon_k+\epsilon_j}^-x_{\epsilon_i-\epsilon_k}^--x_{\epsilon_i+\epsilon_j}^-)-x_{\epsilon_k+\epsilon_j}^-x_{\epsilon_i}^-
\\&=-qx_{\epsilon_k+\epsilon_j}^-[x_{\epsilon_{i}-\epsilon_{k}}^-, x_{\epsilon_k}^-]_v
+qx_{\epsilon_i+\epsilon_j}^-x_{\epsilon_k}^--q^{-1}x_{\epsilon_k}^-x_{\epsilon_i+\epsilon_j}^--x_{\epsilon_k+\epsilon_j}^-x_{\epsilon_i}^-
\\&=
z_qx_{\epsilon_k}^-x_{\epsilon_i+\epsilon_j}^-,~\text{by Corollary~\ref{25a}(4)},\end{aligned}$$
where the second and third equalities follow from $[x_{\epsilon_k}^-, x_{\epsilon_k+\epsilon_j}^-]_v=0$, $[x_{\epsilon_i+\epsilon_{j}}^-, x_{\epsilon_k}^-]_v=0$ (see  Corollary~\ref{monoial1}) and Corollary~\ref{25a}(5).
\end{proof}

\begin{Lemma}\label{commute5}Suppose $1\leq i<j<n$. For any $k,\  j<k\leq n$,  let $a_k=[x_{\epsilon_i-\epsilon_k}^-, x_{\epsilon_j+\epsilon_k}^-]_v$. Then $$ \begin{aligned} & a_k+\sum_{t=k+1}^n(-q)^{k-t}z_qx_{\epsilon_j+\epsilon_t}^-
x_{\epsilon_i-\epsilon_t}^--z_q
x_{\epsilon_j-\epsilon_{k}}^-x_{\epsilon_i+\epsilon_{k}}^- +z_q
(-q)^{k-n}x_{\epsilon_j-\epsilon_{n}}^-x_{\epsilon_i+\epsilon_{n}}^-\\ & =\begin{cases} (-q)^{k-n} \{ z_qx_{\epsilon_j-\epsilon_n}^-x_{\epsilon_i+\epsilon_n}^-+[2]^{-1}q^{-1}(z_qx_{\epsilon_j}^-x_{\epsilon_i}^--[x_{\epsilon_i}^-, x_{\epsilon_j}^-]_v)\}, &\text{if $\mathfrak g=\mathfrak{so}_{2n+1}$,}\\ (-q)^{k-n} [x_{\epsilon_i+\epsilon_n}^-, x_{\epsilon_j-\epsilon_n}^-]_v ,  & \text{if $\mathfrak g=\mathfrak{so}_{2n}$,}\\(-q)^{k-n}\{(1+q^{-2})z_qx_{\epsilon_j-\epsilon_{n}}^-x_{\epsilon_i+\epsilon_{n}}^--q^{-2}[x_{\epsilon_i+\epsilon_n}^-, x_{\epsilon_j-\epsilon_n}^-]_v \},  & \text{if $\mathfrak g=\mathfrak{sp}_{2n}$.}\\
 \end{cases}\end{aligned} $$
  Moreover,  $[x_{\epsilon_i+\epsilon_n}^-, x_{\epsilon_j-\epsilon_n}^-]_v$ and  $[x_{\epsilon_i}^-, x_{\epsilon_j}^-]_v$ have been computed in  Lemma~\ref{commute1} and  Lemma~\ref{commute3}, respectively.
\end{Lemma}
\begin{proof} Suppose $k<n$. Then
$$\begin{aligned}&a_k=-qx_{\epsilon_i-\epsilon_k}^-[x_{\epsilon_j+\epsilon_{k+1}}^-, x_{\epsilon_{k}-\epsilon_{k+1}}^-]_v,~\text{by Corollary~\ref{25a}(2)}
\\&=q^{-1}(x_{\epsilon_k-\epsilon_{k+1}}^-x_{\epsilon_i-\epsilon_{k}}^-
-x_{\epsilon_i-\epsilon_{k+1}}^-)x_{\epsilon_j+\epsilon_{k+1}}^--q^{-1}x_{\epsilon_{j}+\epsilon_{k}}^-x_{\epsilon_i-\epsilon_k}^-\\&\qquad-q(x_{\epsilon_j+\epsilon_{k+1}}^-x_{\epsilon_i-\epsilon_k}^-+z_qx_{\epsilon_j-\epsilon_k}^-
x_{\epsilon_i+\epsilon_{k+1}}^-)x_{\epsilon_k-\epsilon_{k+1}}^-\\&=q^{-1}x_{\epsilon_k-\epsilon_{k+1}}^-(x_{\epsilon_j+\epsilon_{k+1}}^-x_{\epsilon_i-\epsilon_k}^-+z_qx_{\epsilon_j-\epsilon_k}^-
x_{\epsilon_i+\epsilon_{k+1}}^-)-q^{-1}x_{\epsilon_i-\epsilon_{k+1}}^-x_{\epsilon_j+\epsilon_{k+1}}^-\\ &\qquad-x_{\epsilon_j+\epsilon_{k+1}}^-(x_{\epsilon_k-\epsilon_{k+1}}^-x_{\epsilon_i-\epsilon_k}^-
-x_{\epsilon_i-\epsilon_{k+1}}^-)-qz_qx_{\epsilon_j-\epsilon_k}^-x_{\epsilon_i+\epsilon_{k+1}}^-x_{\epsilon_k-\epsilon_{k+1}}^-
-q^{-1}x_{\epsilon_{j}+\epsilon_{k}}^-x_{\epsilon_i-\epsilon_k}^-\\&=-[x_{\epsilon_j+\epsilon_{k+1}}^-, x_{\epsilon_{k}-\epsilon_{k+1}}^-]_vx_{\epsilon_i-\epsilon_{k}}^-+q^{-1}z_qx_{\epsilon_k-\epsilon_{k+1}}^-x_{\epsilon_j-\epsilon_{k}}^-x_{\epsilon_i+\epsilon_{k+1}}^-
-q^{-1}[x_{\epsilon_i-\epsilon_{k+1}}^-, x_{\epsilon_j+\epsilon_{k+1}}^-]_v
\\ &\qquad +(1-q^{-2})x_{\epsilon_j+\epsilon_{k+1}}^-x_{\epsilon_i-\epsilon_{k+1}}^--qz_qx_{\epsilon_j-\epsilon_{k}}^-x_{\epsilon_i+\epsilon_{k+1}}^-x_{\epsilon_k-\epsilon_{k+1}}^-
-q^{-1}x_{\epsilon_{j}+\epsilon_{k}}^-x_{\epsilon_i-\epsilon_k}^-
\\&=\frac{z_q}{q}(x_{\epsilon_j+\epsilon_{k+1}}^-x_{\epsilon_i-\epsilon_{k+1}}^-
+x_{\epsilon_j-\epsilon_{k+1}}^-x_{\epsilon_i+\epsilon_{k+1}}^-+qx_{\epsilon_j-\epsilon_{k}}^-x_{\epsilon_i+\epsilon_{k}}^-)-\frac{1}{q}a_{k+1},\text{by Corollary~\ref{25a}(1)(2)}
\\&=(-q)^{k-n}a_n-z_q\sum_{t=k+1}^n(-q)^{k-t}x_{\epsilon_j+\epsilon_t}^-
x_{\epsilon_i-\epsilon_t}^-+z_q
x_{\epsilon_j-\epsilon_{k}}^-x_{\epsilon_i+\epsilon_{k}}^--z_q
(-q)^{k-n}x_{\epsilon_j-\epsilon_{n}}^-x_{\epsilon_i+\epsilon_{n}}^-,\end{aligned}$$ where the second and third equalities follow from \eqref{ccc2} and Corollary~\ref{25a}(1) and, the last equality follows from induction assumption on $n-k$.
In order to complete the proof, it remain to show that $a_n$ has the required formula.
If  $\mathfrak g=\mathfrak{so}_{2n}$, then   $$\begin{aligned}
a_n&=-q[x_{\epsilon_i-\epsilon_j}^-, x_{\epsilon_j-\epsilon_{n}}^-]_vx_{\epsilon_{j}+\epsilon_{n}}^--q^{-1}x_{\epsilon_{j}+\epsilon_{n}}^-x_{\epsilon_i-\epsilon_n}^-,~\text{by Corollary~\ref{25a}(1)}\\&=q^{-1}x_{\epsilon_j-\epsilon_{n}}^-(x_{\epsilon_{j}+\epsilon_{n}}^-x_{\epsilon_i-\epsilon_j}^--x_{\epsilon_i+\epsilon_{n}}^-)
-(x_{\epsilon_{j}+\epsilon_{n}}^-x_{\epsilon_i-\epsilon_j}^--x_{\epsilon_i+\epsilon_{n}}^-)x_{\epsilon_j-\epsilon_{n}}^--q^{-1}x_{\epsilon_{j}+\epsilon_{n}}^-x_{\epsilon_i-\epsilon_n}^-
\\&=-x_{\epsilon_j+\epsilon_n}^-[x_{\epsilon_i-\epsilon_j}^-, x_{\epsilon_j-\epsilon_{n}}^-]_v
+x_{\epsilon_i+\epsilon_n}^-x_{\epsilon_j-\epsilon_n}^--q^{-1}x_{\epsilon_j-\epsilon_n}^-x_{\epsilon_i+\epsilon_n}^-
-q^{-1}x_{\epsilon_{j}+\epsilon_{n}}^-x_{\epsilon_i-\epsilon_n}^-\\&=[x_{\epsilon_i+\epsilon_n}^-, x_{\epsilon_j-\epsilon_n}^-]_v,~\text{by Corollary~\ref{25a}(1)},\end{aligned}$$ where the second and third equalities follow from  $[x_{\epsilon_j-\epsilon_n}^-, x_{\epsilon_j+\epsilon_{n}}^-]_v=0$ (see  Corollary~\ref{monoial1}) and Corollary~\ref{25a}(5). If  $\mathfrak g=\mathfrak{so}_{2n+1}$, then
$$\begin{aligned}&a_n
=\frac{1}{[2]}(\frac{1}{q}(x_{\epsilon_n}^-x_{\epsilon_i-\epsilon_n}^-
-x_{\epsilon_i}^-)x_{\epsilon_j}^--(x_{\epsilon_j}^-x_{\epsilon_i-\epsilon_n}^-+z_qx_{\epsilon_j-\epsilon_n}^-
x_{\epsilon_i}^-)x_{\epsilon_n}^-)-\frac{1}{q}x_{\epsilon_{j}+\epsilon_{n}}^-x_{\epsilon_i-\epsilon_n}^-, \text{ by \eqref{ccc1}}\\&=[2]^{-1}(q^{-1}x_{\epsilon_n}^-(x_{\epsilon_j}^-x_{\epsilon_i-\epsilon_n}^-+z_qx_{\epsilon_j-\epsilon_n}^-
x_{\epsilon_i}^-)-q^{-1}x_{\epsilon_i}^-x_{\epsilon_j}^--q^{-1}x_{\epsilon_j}^-(x_{\epsilon_n}^-x_{\epsilon_i-\epsilon_n}^-
-x_{\epsilon_i}^-))\\&\qquad-z_q[2]^{-1}x_{\epsilon_j-\epsilon_n}^-x_{\epsilon_i}^-x_{\epsilon_n}^--q^{-1}x_{\epsilon_{j}+\epsilon_{n}}^-x_{\epsilon_i-\epsilon_n}^-, \text{ by Corollary~\ref{25a}(4),}
\\&=\frac{1}{q[2]}(z_qx_{\epsilon_{n}}^-x_{\epsilon_j-\epsilon_n}^-
x_{\epsilon_i}^--[x_{\epsilon_{j}}^-, x_{\epsilon_{n}}^-]_vx_{\epsilon_i-\epsilon_n}^--[x_{\epsilon_{i}}^-, x_{\epsilon_{j}}^-]_v-z_qqx_{\epsilon_j-\epsilon_n}^-x_{\epsilon_i}^-x_{\epsilon_n}^-)
-\frac{1}{q}x_{\epsilon_{j}+\epsilon_{n}}^-x_{\epsilon_i-\epsilon_n}^-
\\&=-\frac{1}{q[2]}([x_{\epsilon_i}^-, x_{\epsilon_j}^-]_v-z_qx_{\epsilon_{n}}^-x_{\epsilon_j-\epsilon_n}^-
x_{\epsilon_i}^-)-z_q\frac{1}{q[2]}(x_{\epsilon_n}^-x_{\epsilon_j-\epsilon_n}^--x_{\epsilon_j}^-)x_{\epsilon_i}^-+z_qx_{\epsilon_j-\epsilon_n}^-
x_{\epsilon_i+\epsilon_n}^-\\&=
-q^{-1}[2]^{-1}[x_{\epsilon_i}^-, x_{\epsilon_j}^-]_v+z_qx_{\epsilon_j-\epsilon_n}^-
x_{\epsilon_i+\epsilon_n}^-+q^{-1}[2]^{-1}z_qx_{\epsilon_j}^-x_{\epsilon_i}^-,\end{aligned}$$
where  the fourth equality follows from  Corollary~\ref{25a}(3)-(4). If $\mathfrak g=\mathfrak{sp}_{2n}$, then
$$\begin{aligned}a_n&=-q^{2}x_{\epsilon_i-\epsilon_n}^-[x_{\epsilon_j-\epsilon_{n}}^-, x_{2\epsilon_{n}}^-]_v-q^{-1}x_{\epsilon_{j}+\epsilon_{n}}^-x_{\epsilon_i-\epsilon_n}^-,~\text{by Corollary~\ref{25a}(6)}\\&=
q^{-2}(x_{2\epsilon_{n}}^-x_{\epsilon_i-\epsilon_n}^--x_{\epsilon_i+\epsilon_n}^-)x_{\epsilon_j-\epsilon_n}^-
-qx_{\epsilon_j-\epsilon_n}^-(x_{2\epsilon_{n}}^-x_{\epsilon_i-\epsilon_n}^--x_{\epsilon_i+\epsilon_n}^-)-q^{-1}x_{\epsilon_{j}+\epsilon_{n}}^-x_{\epsilon_i-\epsilon_n}^-
\\&=-q[x_{\epsilon_j-\epsilon_{n}}^-, x_{2\epsilon_{n}}^-]_vx_{\epsilon_i-\epsilon_{n}}^-
-q^{-2}[x_{\epsilon_i+\epsilon_{n}}^-, x_{\epsilon_j-\epsilon_{n}}^-]_v+q^{-1}([2]z_qx_{\epsilon_j-\epsilon_{n}}^-x_{\epsilon_i+\epsilon_{n}}^-
-x_{\epsilon_{j}+\epsilon_{n}}^-x_{\epsilon_i-\epsilon_n}^-)\\&=
-q^{-2}[x_{\epsilon_i+\epsilon_{n}}^-, x_{\epsilon_j-\epsilon_{n}}^-]_v+q^{-1}[2]z_qx_{\epsilon_j-\epsilon_{n}}^-x_{\epsilon_i+\epsilon_{n}}^-,~\text{by Corollary~\ref{25a}(6)},\end{aligned}$$
where the second and third equalities follow from $[x_{\epsilon_i-\epsilon_n}^-, x_{\epsilon_j-\epsilon_{n}}^-]_v=0$ (see  Corollary~\ref{monoial1}) and Corollary~\ref{25a}(6).
\end{proof}
\begin{Lemma}\label{commute6}Suppose $\mathfrak g=\mathfrak{sp}_{2n}$ and $i<j<n$. We have
$$[x_{\epsilon_i-\epsilon_j}^-, x_{2\epsilon_j}^-]_v=-(-q)^{j-n-2}x_{\epsilon_i+\epsilon_j}^-+z_q\sum_{t=j+1}^n(-q)^{t-n-2}
x_{\epsilon_j-\epsilon_t}^-x_{\epsilon_i+\epsilon_t}^-.$$
\end{Lemma}
\begin{proof}We have
$$\begin{aligned}&[x_{\epsilon_i-\epsilon_{j}}^-, x_{2\epsilon_{j}}^-]_v
 =[2]^{-1}(-x_{\epsilon_i-\epsilon_{j}}^-[x_{\epsilon_j-\epsilon_{n}}^-, x_{\epsilon_j+\epsilon_{n}}^-]_v)-q^{-2}x_{2\epsilon_{j}}^-x_{\epsilon_i-\epsilon_{j}}^-,~\text{by Corollary~\ref{25a}(7)}\\= &
[2]^{-1}q^{-1}((x_{\epsilon_j+\epsilon_{n}}^-x_{\epsilon_i-\epsilon_{j}}^--x_{\epsilon_i+\epsilon_{n}}^-)
x_{\epsilon_j-\epsilon_{n}}^--(x_{\epsilon_j-\epsilon_{n}}^-x_{\epsilon_i-\epsilon_{j}}^--x_{\epsilon_i-\epsilon_{n}}^-)
x_{\epsilon_j+\epsilon_{n}}^-)-q^{-2}x_{2\epsilon_{j}}^-x_{\epsilon_i-\epsilon_{j}}^-\\= &
[2]^{-1}(q^{-2}x_{\epsilon_j+\epsilon_{n}}^-(x_{\epsilon_j-\epsilon_{n}}^-x_{\epsilon_i-\epsilon_{j}}^--x_{\epsilon_i-\epsilon_{n}}^-)
-q^{-1}x_{\epsilon_i+\epsilon_{n}}^-x_{\epsilon_j-\epsilon_{n}}^- )  - q^{-2}[2]^{-1}x_{\epsilon_j-\epsilon_{n}}^- \times \\ & (x_{\epsilon_j+\epsilon_{n}}^-x_{\epsilon_i-\epsilon_{j}}^--x_{\epsilon_i+\epsilon_{n}}^-)
+[2]^{-1}q^{-1}x_{\epsilon_i-\epsilon_{n}}^-x_{\epsilon_j+\epsilon_{n}}^--q^{-2}x_{2\epsilon_{j}}^-x_{\epsilon_i-\epsilon_{j}}^-,~\text{by Corollary~\ref{25a}(1),(5)}\\
= & [2]^{-1}q^{-1}(
[x_{\epsilon_i-\epsilon_{n}}^-, x_{\epsilon_j+\epsilon_{n}}^-]_v-q^{-1}[x_{\epsilon_j-\epsilon_{n}}^-, x_{\epsilon_j+\epsilon_{n}}^-]_vx_{\epsilon_i-\epsilon_{j}}^--[x_{\epsilon_i+\epsilon_{n}}^-, x_{\epsilon_j-\epsilon_{n}}^-]_v)-q^{-2}x_{2\epsilon_{j}}^-x_{\epsilon_i-\epsilon_{j}}^-\\= &
[2]^{-1}(q^{-1}[x_{\epsilon_i-\epsilon_{n}}^-, x_{\epsilon_j+\epsilon_{n}}^-]_v-q^{-1}[x_{\epsilon_i+\epsilon_{n}}^-, x_{\epsilon_j-\epsilon_{n}}^-]_v),~\text{by Corollary~\ref{25a}(7)}\\
= & [2]^{-1}(q^{-1}(-q^{-2}[x_{\epsilon_i+\epsilon_n}^-, x_{\epsilon_j-\epsilon_n}^-]_v+(1+q^{-2})z_q
x_{\epsilon_j-\epsilon_{n}}^-x_{\epsilon_i+\epsilon_{n}}^-))
 -  [2]^{-1}q^{-1}[x_{\epsilon_i+\epsilon_{n}}^-, x_{\epsilon_j-\epsilon_{n}}^-]_v,\\
=& -q^{-2}[x_{\epsilon_i+\epsilon_{n}}^-, x_{\epsilon_j-\epsilon_{n}}^-]_v
+q^{-2}z_qx_{\epsilon_j-\epsilon_{n}}^-x_{\epsilon_i+\epsilon_{n}}^-,\\
\end{aligned}$$
where the second equality follows from Corollary~\ref{25a}(1),(5) and the sixth equality follows from the  formula on $[x_{\epsilon_i-\epsilon_{n}}^-, x_{\epsilon_j+\epsilon_{n}}^-]_v$ in Lemma~\ref{commute5}.
\end{proof}

\begin{Lemma}\label{commute7}For all admissible $i<l<j$, we have $$\begin{aligned}{[x_{\epsilon_i-\epsilon_j}^-, x_{2\epsilon_l}^-]_v}=-&(-q)^{l-n-1}z_qx_{\epsilon_l-\epsilon_j}^-x_{\epsilon_i+\epsilon_l}^--z_q^2\sum_{t=j+1}^{n}(-q)^{t-n}
x_{\epsilon_l-\epsilon_{t}}^-x_{\epsilon_l-\epsilon_{j}}^-x_{\epsilon_i+\epsilon_{t}}^-\\&
+z_q^2(-q)^{j-n-1}(x_{\epsilon_l-\epsilon_{j}}^-)^2x_{\epsilon_i+\epsilon_{j}}^-+z_q^2\sum_{t=l+1}^{j-1}(-q)^{t-n-1}x_{\epsilon_l-\epsilon_{j}}^-
x_{\epsilon_l-\epsilon_{t}}^-x_{\epsilon_i+\epsilon_{t}}^-.\end{aligned}$$
\end{Lemma}
\begin{proof}We have
$$\begin{aligned}&{[x_{\epsilon_i-\epsilon_j}^-, x_{2\epsilon_l}^-]_v}=-q[x_{\epsilon_{i}-\epsilon_{l}}^-, x_{\epsilon_l-\epsilon_{j}}^-]_vx_{2\epsilon_l}^--x_{2\epsilon_l}^-x_{\epsilon_i-\epsilon_j}^-,
~\text{by Corollary~\ref{25a}(1)}\\=&x_{\epsilon_l-\epsilon_{j}}^-(q^{-2}x_{2\epsilon_l}^-x_{\epsilon_i-\epsilon_l}^-+[x_{\epsilon_i-\epsilon_l}^-,
 x_{2\epsilon_l}^-]_v)-(qx_{2\epsilon_l}^-x_{\epsilon_i-\epsilon_l}^-+q^3[x_{\epsilon_i-\epsilon_l}^-,
 x_{2\epsilon_l}^-]_v)x_{\epsilon_l-\epsilon_{j}}^--x_{2\epsilon_l}^-x_{\epsilon_i-\epsilon_j}^-
 \\ =&-qx_{2\epsilon_l}^-[x_{\epsilon_{i}-\epsilon_{l}}^-, x_{\epsilon_l-\epsilon_{j}}^-]_v+x_{\epsilon_l-\epsilon_{j}}^-[x_{\epsilon_i-\epsilon_l}^-,
 x_{2\epsilon_l}^-]_v-q^3[x_{\epsilon_i-\epsilon_l}^-,
 x_{2\epsilon_l}^-]_vx_{\epsilon_l-\epsilon_{j}}^--x_{2\epsilon_l}^-x_{\epsilon_i-\epsilon_j}^-\\
 = & x_{\epsilon_l-\epsilon_{j}}^-[x_{\epsilon_i-\epsilon_l}^-,
 x_{2\epsilon_l}^-]_v-q^3[x_{\epsilon_i-\epsilon_l}^-,
 x_{2\epsilon_l}^-]_vx_{\epsilon_l-\epsilon_{j}}^-,
~\text{by Corollary~\ref{25a}(1)}\\ =&(-q)^{l-n-1}(x_{\epsilon_l-\epsilon_{j}}^-x_{\epsilon_i+\epsilon_l}^--q^{3}x_{\epsilon_i+\epsilon_l}^-x_{\epsilon_l-\epsilon_{j}}^-)
+
 z_q\sum_{l+1\leq t\leq n, t\neq j}(-q)^{t-n+1}x_{\epsilon_l-\epsilon_t}^-x_{\epsilon_i+\epsilon_t}^-x_{\epsilon_l-\epsilon_{j}}^-     \\
& +
 z_q\sum_{t=l+1}^n(-q)^{t-n-2}x_{\epsilon_l-\epsilon_{j}}^-x_{\epsilon_l-\epsilon_t}^-x_{\epsilon_i+\epsilon_t}^-
 +
 z_q(-q)^{j-n+1}x_{\epsilon_l-\epsilon_j}^-([x_{\epsilon_i+\epsilon_j}^-, x_{\epsilon_l-\epsilon_{j}}^-]_v \\ &  +q^{-1}x_{\epsilon_l-\epsilon_{j}}^-x_{\epsilon_i+\epsilon_j}^-), \text{by Lemma~\ref{commute6}}
 \\= & (-q)^{l-n-2}(q^4-1)x_{\epsilon_l-\epsilon_{j}}^-x_{\epsilon_i+\epsilon_{l}}^--z_q^2\sum_{t=j+1}^n(-q)^{t-n}x_{\epsilon_l-\epsilon_{t}}^-x_{\epsilon_l-\epsilon_{j}}^-
x_{\epsilon_i+\epsilon_{t}}^-
\\ &
+z_q\sum_{t=l+1}^{j-1}(-q)^{t-n-2}(1-q^4)x_{\epsilon_l-\epsilon_{j}}^-x_{\epsilon_l-\epsilon_{t}}^-
x_{\epsilon_i+\epsilon_{t}}^-+z_q(-q)^{j-n-2}(x_{\epsilon_l-\epsilon_{j}}^-)^2x_{\epsilon_i+\epsilon_{j}}^-
\\ &
-z_q(-q)^{j-n}x_{\epsilon_l-\epsilon_{j}}^-(x_{\epsilon_l-\epsilon_{j}}^-x_{\epsilon_i+\epsilon_{j}}^--(-q)^{l+1-j}x_{\epsilon_i+\epsilon_l}^-
-qz_q\sum_{t=l+1}^{j-1}(-q)^{t-j}x_{\epsilon_l-\epsilon_t}^-x_{\epsilon_i+\epsilon_t}^-)
\\=&-(-q)^{l-n-1}z_qx_{\epsilon_l-\epsilon_j}^-x_{\epsilon_i+\epsilon_l}^--z_q^2\sum_{t=j+1}^{n}(-q)^{t-n}
x_{\epsilon_l-\epsilon_{t}}^-x_{\epsilon_l-\epsilon_{j}}^-x_{\epsilon_i+\epsilon_{t}}^-\\ &
+z_q^2(-q)^{j-n-1}(x_{\epsilon_l-\epsilon_{j}}^-)^2x_{\epsilon_i+\epsilon_{j}}^-+z_q^2\sum_{t=l+1}^{j-1}(-q)^{t-n-1}x_{\epsilon_l-\epsilon_{j}}^-
x_{\epsilon_l-\epsilon_{t}}^-x_{\epsilon_i+\epsilon_{t}}^-.\end{aligned}$$
The second and the third equalities follow from $[x_{\epsilon_l-\epsilon_j}^-, x_{2\epsilon_l}^-]_v=0$ (see  Corollary~\ref{monoial1}).
  The sixth equality follows from Lemma \ref{commute1} and $[x_{\epsilon_i+\epsilon_{t_1}}^-, x_{\epsilon_l-\epsilon_{j}}^-]_v=0$, $[x_{\epsilon_l-\epsilon_{t_2}}^-, x_{\epsilon_l-\epsilon_{t_3}}^-]_v=0$ where $l\leq t_1\leq n$, $t_1\neq j$ and $l< t_2<t_3\leq n$.
Such formulae follows from  Corollary~\ref{monoial1}.
\end{proof}
\begin{Lemma}\label{commute8}Suppose $\mathfrak g=\mathfrak{sp}_{2n}$. For all admissible $i<k<j$,  let $a_j=[x_{2\epsilon_i}^-, x_{\epsilon_k+\epsilon_{j}}^-]_v$. Then
 $$\begin{aligned}a_j =-&(-q)^{k-n-1}z_qx_{\epsilon_i+\epsilon_{k}}^-x_{\epsilon_i+\epsilon_{j}}^-
+z_q^2\sum_{t=k+1}^j(-q)^{t-n-1}x_{\epsilon_k-\epsilon_{t}}^-x_{\epsilon_i+\epsilon_{t}}^-x_{\epsilon_i+\epsilon_{j}}^-
\\&-z_q^2\sum_{t=j+1}^n(-q)^{t-n}x_{\epsilon_k-\epsilon_{t}}^-x_{\epsilon_i+\epsilon_{j}}^-x_{\epsilon_i+\epsilon_{t}}^-.\end{aligned}$$
\end{Lemma}
\begin{proof} We have
$$\begin{aligned}a_n=&-[2]^{-1}[x_{\epsilon_i-\epsilon_{n}}^-, x_{\epsilon_i+\epsilon_{n}}^-]_vx_{\epsilon_k+\epsilon_{n}}^--x_{\epsilon_k+\epsilon_{n}}^-x_{2\epsilon_i}^-,
~\text{by Corollary~\ref{25a}(7)}
\\=&[2]^{-1}x_{\epsilon_i+\epsilon_{n}}^-([x_{\epsilon_i-\epsilon_{n}}^-, x_{\epsilon_k+\epsilon_{n}}^-]_v+q^{-1}
x_{\epsilon_k+\epsilon_{n}}^-x_{\epsilon_i-\epsilon_{n}}^-)-x_{\epsilon_k+\epsilon_{n}}^-x_{2\epsilon_i}^-\\& -[2]^{-1}q([x_{\epsilon_i-\epsilon_{n}}^-, x_{\epsilon_k+\epsilon_{n}}^-]_v+q^{-1}
x_{\epsilon_k+\epsilon_{n}}^-x_{\epsilon_i-\epsilon_{n}}^-)x_{\epsilon_i+\epsilon_{n}}^-
\\=&[2]^{-1}(x_{\epsilon_i+\epsilon_{n}}^-[x_{\epsilon_i-\epsilon_{n}}^-, x_{\epsilon_k+\epsilon_{n}}^-]_v-q[x_{\epsilon_i-\epsilon_{n}}^-, x_{\epsilon_k+\epsilon_{n}}^-]_vx_{\epsilon_i+\epsilon_{n}}^-)
\\ =&
q^{-1}z_qx_{\epsilon_i+\epsilon_{n}}^-x_{\epsilon_k-\epsilon_{n}}^-x_{\epsilon_i+\epsilon_{n}}^--
z_qx_{\epsilon_k-\epsilon_{n}}^-(x_{\epsilon_i+\epsilon_{n}}^-)^2
\\ =& q^{-1}z_q[x_{\epsilon_i+\epsilon_{n}}^-, x_{\epsilon_k-\epsilon_{n}}^-]_vx_{\epsilon_i+\epsilon_{n}}^--q^{-1}z_q^2x_{\epsilon_k-\epsilon_{n}}^-(x_{\epsilon_i+\epsilon_{n}}^-)^2
.\end{aligned}$$
The third equality follows from  $[x_{\epsilon_i+\epsilon_n}^-, x_{\epsilon_k+\epsilon_{n}}^-]_v=0$ (see  Corollary~\ref{monoial1}) and Corollary~\ref{25a}(7). The  fourth equality follow from Lemma~\ref{commute5} and  $[x_{\epsilon_i+\epsilon_n}^-, x_{\epsilon_i+\epsilon_{h}}^-]_v=0$, $[x_{\epsilon_i+\epsilon_n}^-, x_{\epsilon_k-\epsilon_{t}}^-]_v=0$, $k<t<n$, $k\leq h<n$. Such formulae follow from  Corollary~\ref{monoial1}. Thanks to the formula of $[x_{\epsilon_i+\epsilon_n}^-, x_{\epsilon_k-\epsilon_n}^-]_v$ in  Lemma~\ref{commute1}, $a_n$ has the required formula. Suppose $j<n$. Then
 $$\begin{aligned}a_j=&-[2]^{-1}[x_{\epsilon_i-\epsilon_{n}}^-, x_{\epsilon_i+\epsilon_{n}}^-]_vx_{\epsilon_k+\epsilon_{j}}^--x_{\epsilon_k+\epsilon_{j}}^-x_{2\epsilon_i}^-,
~\text{by Corollary~\ref{25a}(7)}
\\ =& [2]^{-1}x_{\epsilon_i+\epsilon_{n}}^-(x_{\epsilon_{k}+\epsilon_{j}}^-x_{\epsilon_{i}-\epsilon_{n}}^-
+z_qx_{\epsilon_{k}-\epsilon_{n}}^-x_{\epsilon_{i}+\epsilon_{j}}^-)-x_{\epsilon_k+\epsilon_{j}}^-x_{2\epsilon_i}^-\\-&[2]^{-1}
x_{\epsilon_i-\epsilon_{n}}^-(x_{\epsilon_{k}+\epsilon_{j}}^-x_{\epsilon_{i}+\epsilon_{n}}^-
+z_qx_{\epsilon_{k}+\epsilon_{n}}^-x_{\epsilon_{i}+\epsilon_{j}}^-),~\text{by \eqref{ccc3}, \eqref{ccc4}}
\\= & [2]^{-1}((x_{\epsilon_{k}+\epsilon_{j}}^-x_{\epsilon_{i}+\epsilon_{n}}^-
+z_qx_{\epsilon_{k}+\epsilon_{n}}^-x_{\epsilon_{i}+\epsilon_{j}}^-)x_{\epsilon_i-\epsilon_{n}}^-+
z_qx_{\epsilon_i+\epsilon_{n}}^-x_{\epsilon_{k}-\epsilon_{n}}^-x_{\epsilon_{i}+\epsilon_{j}}^-)-x_{\epsilon_k+\epsilon_{j}}^-x_{2\epsilon_i}^-\\
-&[2]^{-1}((x_{\epsilon_{k}+\epsilon_{j}}^-x_{\epsilon_{i}-\epsilon_{n}}^-
+z_qx_{\epsilon_{k}-\epsilon_{n}}^-x_{\epsilon_{i}+\epsilon_{j}}^-)x_{\epsilon_i+\epsilon_{n}}^-+
z_qx_{\epsilon_i-\epsilon_{n}}^-x_{\epsilon_{k}+\epsilon_{n}}^-x_{\epsilon_{i}+\epsilon_{j}}^-),~\text{by \eqref{ccc3}, \eqref{ccc4}}\\ = & \frac{1}{[2]}(z_q([x_{\epsilon_i+\epsilon_{n}}^-, x_{\epsilon_k-\epsilon_{n}}^-]_v-[x_{\epsilon_i-\epsilon_{n}}^-, x_{\epsilon_k+\epsilon_{n}}^-]_v)
x_{\epsilon_i+\epsilon_{j}}^--x_{\epsilon_k+\epsilon_{j}}^-[x_{\epsilon_i-\epsilon_{n}}^-, x_{\epsilon_i+\epsilon_{n}}^-]_v)-x_{\epsilon_k+\epsilon_{j}}^-x_{2\epsilon_i}^-
\\ =& [2]^{-1}z_q([x_{\epsilon_i+\epsilon_{n}}^-, x_{\epsilon_k-\epsilon_{n}}^-]_v-[x_{\epsilon_i-\epsilon_{n}}^-, x_{\epsilon_k+\epsilon_{n}}^-]_v)
x_{\epsilon_i+\epsilon_{j}}^-,
~\text{by Corollary~\ref{25a}(7)}
\\ = & \sum_{t=k+1}^{n-1}(-q)^{t-n-1}z_q^2x_{\epsilon_k-\epsilon_t}^-x_{\epsilon_i+\epsilon_t}^-
x_{\epsilon_i+\epsilon_{j}}^--z_q^2x_{\epsilon_k-\epsilon_{n}}^-
x_{\epsilon_i+\epsilon_{j}}^-x_{\epsilon_i+\epsilon_{n}}^--z_q(-q)^{k-n-1}x_{\epsilon_i+\epsilon_k}^-x_{\epsilon_i+\epsilon_j}^-\\
= &
z_q^2\sum_{t=k+1}^j(-q)^{t-n-1}x_{\epsilon_k-\epsilon_{t}}^-x_{\epsilon_i+\epsilon_{t}}^-x_{\epsilon_i+\epsilon_{j}}^-
-(-q)^{k-n-1}z_qx_{\epsilon_i+\epsilon_{k}}^-x_{\epsilon_i+\epsilon_{j}}^-\\&-z_q^2\sum_{t=j+1}^n(-q)^{t-n}x_{\epsilon_k-\epsilon_{t}}^-x_{\epsilon_i+\epsilon_{j}}^-x_{\epsilon_i+\epsilon_{t}}^-
.\end{aligned}$$ The fourth equality follows from  $[x_{\epsilon_i\pm\epsilon_n}^-, x_{\epsilon_i+\epsilon_{j}}^-]_v=0$ (see  Corollary~\ref{monoial1}). The sixth equality follows from Lemmas \ref{commute1}, \ref{commute5} and the last equality follows from  $[x_{\epsilon_i\pm\epsilon_t}^-, x_{\epsilon_i+\epsilon_{j}}^-]_v=0$ if $j<t<n$ (see  Corollary~\ref{monoial1}).
\end{proof}
\begin{Lemma}\label{commute9}Suppose $\mathfrak g=\mathfrak{sp}_{2n}$. For all admissible $i<j$, let $a_j=[x_{2\epsilon_i}^-, x_{2\epsilon_j}^-]_v$.
 Then $$\begin{aligned}& a_j=-[2]^{-1}(-q)^{2j-2n-1}z_q(x_{\epsilon_i+\epsilon_{j}}^-)^2
-z_q^3\sum_{t=j+1}^{n-1}(-q)^{t-n-1}x_{\epsilon_j-\epsilon_{n}}^-
x_{\epsilon_j-\epsilon_{t}}^-x_{\epsilon_i+\epsilon_{t}}^-x_{\epsilon_i+\epsilon_{n}}^-\\&\qquad
+z_q^2\sum_{t=j+1}^{n}(-q)^{j+t-2n-1}x_{\epsilon_j-\epsilon_{t}}^-
x_{\epsilon_i+\epsilon_{j}}^-x_{\epsilon_i+\epsilon_{t}}^-
+[2]^{-1}z_q^3\sum_{t=j+1}^{n}(-q)^{2t-2n-2}(x_{\epsilon_j-\epsilon_{t}}^-)^2
(x_{\epsilon_i+\epsilon_{t}}^-)^2\\&\qquad-z_q^3\sum_{t=j+1}^{n-1}\sum_{s=t+1}^{n-1}(-q)^{s+t-2n-1}x_{\epsilon_j-\epsilon_{s}}^-
x_{\epsilon_j-\epsilon_{t}}^-x_{\epsilon_i+\epsilon_{t}}^-x_{\epsilon_i+\epsilon_{s}}^-.\end{aligned}$$
\end{Lemma}
\begin{proof} Suppose $j=n$. Then
$$\begin{aligned}a_n &=-[2]^{-1}[x_{\epsilon_i-\epsilon_{n}}^-, x_{\epsilon_i+\epsilon_{n}}^-]_vx_{2\epsilon_n}^- -x_{2\epsilon_n}^-x_{2\epsilon_i}^-,
~\text{by Corollary~\ref{25a}(7)}
\\&=[2]^{-1}(q^{-2}x_{\epsilon_i+\epsilon_{n}}^-(x_{2\epsilon_n}^-x_{\epsilon_i-\epsilon_{n}}^--x_{\epsilon_i+\epsilon_{n}}^-)
-(x_{2\epsilon_n}^-x_{\epsilon_i-\epsilon_{n}}^--x_{\epsilon_i+\epsilon_{n}}^-)x_{\epsilon_i+\epsilon_{n}}^-)-x_{2\epsilon_n}^-x_{2\epsilon_i}^-
 \\&=[2]^{-1}(-x_{2\epsilon_n}^-[x_{\epsilon_i-\epsilon_{n}}^-, x_{\epsilon_i+\epsilon_{n}}^-]_v+(1-q^{-2})(x_{\epsilon_i+\epsilon_{n}}^-)^2)-x_{2\epsilon_n}^-x_{2\epsilon_i}^-,  \text{ by Corollary~\ref{25a}(6)}
\\&=q^{-1}[2]^{-1}z_q(x_{\epsilon_i+\epsilon_{n}}^-)^2,
~\text{by Corollary~\ref{25a}(7)},  \end{aligned}$$ where the second equality follows from  $[x_{\epsilon_i+\epsilon_n}^-, x_{2\epsilon_n}^-]_v=0$ (see  Corollary~\ref{monoial1}).
So, we have the required formula for $j=n$ follows.

 If  $j<n$, then        \begin{equation}\begin{aligned}\label{046}a_j  &=-[2]^{-1}x_{2\epsilon_i}^-[x_{\epsilon_j-\epsilon_{n}}^-, x_{\epsilon_j+\epsilon_{n}}^-]_v-x_{2\epsilon_j}^-x_{2\epsilon_i}^-,
~\text{by Corollary~\ref{25a}(7)}\\&=[2]^{-1}([x_{2\epsilon_i}^-, x_{\epsilon_j+\epsilon_{n}}^-]_vx_{\epsilon_j-\epsilon_{n}}^-
-x_{\epsilon_j-\epsilon_{n}}^-[x_{2\epsilon_i}^-, x_{\epsilon_j+\epsilon_{n}}^-]_v),
\end{aligned}\end{equation} where the last equation follows from $[x_{2\epsilon_i}^-, x_{\epsilon_j-\epsilon_{n}}^-]_v=0$ (see Corollary~\ref{monoial1})
and Corollary~\ref{25a}(7).
We use  Lemma~\ref{commute8} to  rewrite ${[}x_{2\epsilon_i}^-, x_{\epsilon_j+\epsilon_{n}}^-{]}_vx_{\epsilon_j-\epsilon_{n}}^-$ as follows: $$\begin{aligned}
& {[}x_{2\epsilon_i}^-, x_{\epsilon_j+\epsilon_{n}}^-{]}_vx_{\epsilon_j-\epsilon_{n}}^-=-(-q)^{j-n-1}z_qx_{\epsilon_i+\epsilon_{j}}^-(q^{-1}x_{\epsilon_j-\epsilon_{n}}^-x_{\epsilon_i+\epsilon_{n}}^-+[x_{\epsilon_i+\epsilon_{n}}^-, x_{\epsilon_j-\epsilon_{n}}^-]_v)\\&\quad+z_q^2\sum_{t=j+1}^n(-q)^{t-n-1}
x_{\epsilon_j-\epsilon_{t}}^-x_{\epsilon_i+\epsilon_{t}}^-(q^{-1}x_{\epsilon_j-\epsilon_{n}}^-x_{\epsilon_i+\epsilon_{n}}^-+[x_{\epsilon_i+\epsilon_{n}}^-, x_{\epsilon_j-\epsilon_{n}}^-]_v)\\&=
-(-q)^{j-n-1}z_q
(x_{\epsilon_j-\epsilon_{n}}^-x_{\epsilon_i+\epsilon_{j}}^-x_{\epsilon_i+\epsilon_{n}}^-+x_{\epsilon_i+\epsilon_{j}}^-
[x_{\epsilon_i+\epsilon_{n}}^-, x_{\epsilon_j-\epsilon_{n}}^-]_v)\\&\quad+z_q^2\sum_{t=j+1}^{n-1}(-q)^{t-n-1}(x_{\epsilon_j-\epsilon_{n}}^-x_{\epsilon_j-\epsilon_{t}}^-
x_{\epsilon_i+\epsilon_{t}}^-x_{\epsilon_i+\epsilon_{n}}^-+x_{\epsilon_j-\epsilon_{t}}^-x_{\epsilon_i+\epsilon_{t}}^-
[x_{\epsilon_i+\epsilon_{n}}^-, x_{\epsilon_j-\epsilon_{n}}^-]_v)\\&
\quad-q^{-1}z_q^2x_{\epsilon_j-\epsilon_{n}}^-\left(x_{\epsilon_i+\epsilon_{n}}^-[x_{\epsilon_i+\epsilon_{n}}^-, x_{\epsilon_j-\epsilon_{n}}^-]_v +(q^{-2}x_{\epsilon_j-\epsilon_{n}}^-x_{\epsilon_i+\epsilon_{n}}^-+q^{-1}[x_{\epsilon_i+\epsilon_{n}}^-. x_{\epsilon_j-\epsilon_{n}}^-]_v)x_{\epsilon_i+\epsilon_{n}}^-\right).
\end{aligned}$$
The last equality follows from  $[x_{\epsilon_i+\epsilon_{j}}^-, x_{\epsilon_j-\epsilon_{n}}^-]_v=0$,  $[x_{\epsilon_j-\epsilon_{k}}^-, x_{\epsilon_j-\epsilon_{n}}^-]_v=0$, and $[x_{\epsilon_i+\epsilon_{k}}^-, x_{\epsilon_j-\epsilon_{n}}^-]=0$ where $j< k<n$.
Such equalities follows from Corollary~\ref{monoial1}. We also use  Lemma~\ref{commute8} to rewrite $x_{\epsilon_j-\epsilon_{n}}^-[x_{2\epsilon_i}^-, x_{\epsilon_j+\epsilon_{n}}^-]_v$ in \eqref{046}.  So,
$$\begin{aligned}{[x_{2\epsilon_i}^-, x_{2\epsilon_j}^-]_v}&=[2]^{-1}(-(-q)^{j-n-1}z_qx_{\epsilon_i+\epsilon_{j}}^-+z_q^2\sum_{t=j+1}^{n}(-q)^{t-n-1}
x_{\epsilon_j-\epsilon_{t}}^-x_{\epsilon_i+\epsilon_{t}}^-)[x_{\epsilon_i+\epsilon_{n}}^-, x_{\epsilon_j-\epsilon_{n}}^-]_v\\&\qquad+[2]^{-1}
q^{-2}z_q^3(x_{\epsilon_j-\epsilon_{n}}^-)^2(x_{\epsilon_i+\epsilon_{n}}^-)^2-[2]^{-1}q^{-2}z_q^2x_{\epsilon_j-\epsilon_{n}}^-[x_{\epsilon_i+\epsilon_{n}}^-, x_{\epsilon_j-\epsilon_{n}}^-]_vx_{\epsilon_i+\epsilon_{n}}^-.
\end{aligned}$$
In the following, we compute  $x_{\epsilon_i+\epsilon_{j}}^-[x_{\epsilon_i+\epsilon_{n}}^-, x_{\epsilon_j-\epsilon_{n}}^-]_v$, and $x_{\epsilon_j-\epsilon_{t}}^-x_{\epsilon_i+\epsilon_{t}}^-[x_{\epsilon_i+\epsilon_{n}}^-, x_{\epsilon_j-\epsilon_{n}}^-]_v$, $j<t\leq n$, and  $x_{\epsilon_j-\epsilon_{n}}^-[x_{\epsilon_i+\epsilon_{n}}^-, x_{\epsilon_j-\epsilon_{n}}^-]_vx_{\epsilon_i+\epsilon_{n}}^-$.
  In any case, we rewrite   $[x_{\epsilon_i+\epsilon_{n}}^-, x_{\epsilon_j-\epsilon_{n}}^-]$ via  Lemma~\ref{commute1}.
 So, we have the formula on  $ x_{\epsilon_j-\epsilon_{n}}^-[x_{\epsilon_i+\epsilon_{n}}^-, x_{\epsilon_j-\epsilon_{n}}^-]_vx_{\epsilon_i+\epsilon_{n}}^-$
 directly. In other cases,  we
need  extra commutative relations as follows: \begin{itemize}\item[(a)]$[x_{\epsilon_i+\epsilon_{j}}^-, x_{\epsilon_j-\epsilon_{h}}^-]_v=0$, if $j< h\leq n$,
\item[(b)]$[x_{\epsilon_i+\epsilon_{n}}^-, x_{\epsilon_i+\epsilon_{s}}^-]_v=0$, $ [x_{\epsilon_i+\epsilon_{n}}^-, x_{\epsilon_j-\epsilon_{h}}^-]_v=0$, if $i< s<n$, $j<h<n$,
\item[(c)]$[x_{\epsilon_i+\epsilon_{h_3}}^-, x_{\epsilon_j-\epsilon_{h_4}}^-]_v=0$, $ [x_{\epsilon_i+\epsilon_{h_2}}^-, x_{\epsilon_i+\epsilon_{h_1}}^-]_v=0$, $[x_{\epsilon_j-\epsilon_{h_1}}^-, x_{\epsilon_j-\epsilon_{h_2}}^-]_v=0$, if $h_1<h_2$ and $h_i>j$ for all $1\leq i\leq4$.
\end{itemize} All  equalities  in (a)-(c)  follow from Corollary~\ref{monoial1}. So,
$$\begin{aligned}&x_{\epsilon_i+\epsilon_{j}}^-[x_{\epsilon_i+\epsilon_{n}}^-, x_{\epsilon_j-\epsilon_{n}}^-]_v
 \overset{(a)}=(-q)^{j-n}(x_{\epsilon_i+\epsilon_j}^-)^2+z_q\sum_{t=j+1}^{n-1}(-q)^{t-n+1}x_{\epsilon_j-\epsilon_t}^-x_{\epsilon_i+\epsilon_j}^-x_{\epsilon_i+\epsilon_t}^-,
 \\&x_{\epsilon_j-\epsilon_{n}}^- x_{\epsilon_i+\epsilon_{n}}^-[x_{\epsilon_i+\epsilon_{n}}^-, x_{\epsilon_j-\epsilon_{n}}^-]_v\overset{(b)}=(-q)^{j-n}x_{\epsilon_j-\epsilon_{n}}^-(x_{\epsilon_i+\epsilon_j}^-+z_q\sum_{t=j+1}^{n-1}(-q)^{t-j+1}
x_{\epsilon_j-\epsilon_t}^-x_{\epsilon_i+\epsilon_t}^-)x_{\epsilon_i+\epsilon_{n}}^-.\end{aligned}$$
Finally, we simplify $x_{\epsilon_j-\epsilon_{t}}^-x_{\epsilon_i+\epsilon_{t}}^-[x_{\epsilon_i+\epsilon_{n}}^-, x_{\epsilon_j-\epsilon_{n}}^-]_v$ for any $j<t<n$.
$$\begin{aligned}&x_{\epsilon_j-\epsilon_{t}}^-x_{\epsilon_i+\epsilon_{t}}^-[x_{\epsilon_i+\epsilon_{n}}^-, x_{\epsilon_j-\epsilon_{n}}^-]_v
\\= & -(-q)^{j-n+1}x_{\epsilon_j-\epsilon_{t}}^-x_{\epsilon_i+\epsilon_j}^-x_{\epsilon_i+\epsilon_{t}}^-+z_q\sum_{s=j+1}^{t-1}(-q)^{s-n+1}x_{\epsilon_j-\epsilon_{t}}^-
x_{\epsilon_j-\epsilon_s}^-x_{\epsilon_i+\epsilon_s}^-x_{\epsilon_i+\epsilon_{t}}^-\\&
+z_q\sum_{s=t+1}^{n-1}(-q)^{s-n+1}x_{\epsilon_j-\epsilon_{s}}^-
x_{\epsilon_j-\epsilon_t}^-x_{\epsilon_i+\epsilon_t}^-x_{\epsilon_i+\epsilon_{s}}^-
-(-q)^{t-n}z_qx_{\epsilon_j-\epsilon_{t}}^-
x_{\epsilon_i+\epsilon_t}^-x_{\epsilon_j-\epsilon_{t}}^-
x_{\epsilon_i+\epsilon_t}^-, \text{ by (c)}\\
=& -(-q)^{j-n+1}x_{\epsilon_j-\epsilon_{t}}^-x_{\epsilon_i+\epsilon_j}^-x_{\epsilon_i+\epsilon_{t}}^-+z_q\sum_{s=j+1}^{t-1}(-q)^{s-n+1}x_{\epsilon_j-\epsilon_{t}}^-
x_{\epsilon_j-\epsilon_s}^-x_{\epsilon_i+\epsilon_s}^-x_{\epsilon_i+\epsilon_{t}}^-\\&
+z_q\sum_{s=t+1}^{n-1}(-q)^{s-n+1}x_{\epsilon_j-\epsilon_{s}}^-
x_{\epsilon_j-\epsilon_t}^-x_{\epsilon_i+\epsilon_t}^-x_{\epsilon_i+\epsilon_{s}}^-
+(-q)^{t-n-1}z_q
(x_{\epsilon_j-\epsilon_{t}}^-)^2(x_{\epsilon_i+\epsilon_{t}}^-)^2\\&\qquad
-z_q(-q)^{j-n}x_{\epsilon_j-\epsilon_{t}}^-x_{\epsilon_i+\epsilon_{j}}^-x_{\epsilon_i+\epsilon_{t}}^-+z_q^2\sum_{s=j+1}^{t-1}(-q)^{s-n}x_{\epsilon_j-\epsilon_{t}}^-
x_{\epsilon_j-\epsilon_{s}}^-x_{\epsilon_i+\epsilon_{s}}^-x_{\epsilon_i+\epsilon_{t}}^-.
\end{aligned}$$ We remark that the last  equality  can be checked directly by rewriting $x_{\epsilon_i+\epsilon_t}^-x_{\epsilon_j-\epsilon_{t}}^-$ via Lemma \ref{commute1} and (c).
Now, we
 rewrite $[x_{2\epsilon_i}^-, x_{2\epsilon_j}^-]_v$ via \eqref{046}. We have
 $$\begin{aligned}a_j&=-(-q)^{2j-2n-1}[2]^{-1}z_q(x_{\epsilon_i+\epsilon_{j}}^-)^2
-z_q^3\sum_{t=j+1}^{n-1}(-q)^{t-n-1}x_{\epsilon_j-\epsilon_{n}}^-
x_{\epsilon_j-\epsilon_{t}}^-x_{\epsilon_i+\epsilon_{t}}^-x_{\epsilon_i+\epsilon_{n}}^-\\&
+z_q^2\sum_{t=j+1}^{n}(-q)^{j+t-2n-1}x_{\epsilon_j-\epsilon_{t}}^-
x_{\epsilon_i+\epsilon_{j}}^-x_{\epsilon_i+\epsilon_{t}}^-
+[2]^{-1}z_q^3\sum_{t=j+1}^{n}(-q)^{2t-2n-2}(x_{\epsilon_j-\epsilon_{t}}^-)^2
(x_{\epsilon_i+\epsilon_{t}}^-)^2\\&+[2]^{-1}z_q^3\sum_{t=j+1}^{n-1}\sum_{s=t+1}^{n-1}(-q)^{s+t-2n}x_{\epsilon_j-\epsilon_{s}}^-
x_{\epsilon_j-\epsilon_{t}}^-x_{\epsilon_i+\epsilon_{t}}^-x_{\epsilon_i+\epsilon_{s}}^-
\\&+[2]^{-1}z_q^3\sum_{t=j+1}^{n-1}\sum_{s=j+1}^{t-1}(-q)^{s+t-2n-2}x_{\epsilon_j-\epsilon_{t}}^-
x_{\epsilon_j-\epsilon_{s}}^-x_{\epsilon_i+\epsilon_{s}}^-x_{\epsilon_i+\epsilon_{t}}^-
\\&=-[2]^{-1}(-q)^{2j-2n-1}z_q(x_{\epsilon_i+\epsilon_{j}}^-)^2
-z_q^3\sum_{t=j+1}^{n-1}(-q)^{t-n-1}x_{\epsilon_j-\epsilon_{n}}^-
x_{\epsilon_j-\epsilon_{t}}^-x_{\epsilon_i+\epsilon_{t}}^-x_{\epsilon_i+\epsilon_{n}}^-\\&
+z_q^2\sum_{t=j+1}^{n}(-q)^{j+t-2n-1}x_{\epsilon_j-\epsilon_{t}}^-
x_{\epsilon_i+\epsilon_{j}}^-x_{\epsilon_i+\epsilon_{t}}^-
+[2]^{-1}z_q^3\sum_{t=j+1}^{n}(-q)^{2t-2n-2}(x_{\epsilon_j-\epsilon_{t}}^-)^2
(x_{\epsilon_i+\epsilon_{t}}^-)^2\\&-z_q^3\sum_{t=j+1}^{n-1}\sum_{s=t+1}^{n-1}(-q)^{s+t-2n-1}x_{\epsilon_j-\epsilon_{s}}^-
x_{\epsilon_j-\epsilon_{t}}^-x_{\epsilon_i+\epsilon_{t}}^-x_{\epsilon_i+\epsilon_{s}}^-,\end{aligned}$$ proving   the required formula about $a_j$.\end{proof}

``\textbf{Proof of Proposition~\ref{commut1}}" We have described explicitly possible pairs $\{\alpha, \nu\}$ in (1)-(4).
Thanks to  Lemma~\ref{min}, $c_{\mathbf r}^-\in \mathcal A_{|\mathbf r|-1}$ if it appears in
Lemma~\ref{commut}(1) for the expression of  $x_\nu^-x_{\alpha}^--v^{(\nu\mid \alpha)} x_{\alpha}^-x_{\nu}^-$ such that $\{\alpha, \nu\}$ appears in (1).
If $\{\alpha, \nu\}$ appears in (4), the corresponding result follows from   Corollary~\ref{monoial1}.
If  $\{\alpha, \nu\}$ appears in (2),  the corresponding result follows from Lemmas~\ref{commute1}, \ref{commute3}, \ref{commute5}--\ref{commute6} and
those for  $\mfg=\mathfrak{sp}_{2n}$ in Lemma~\ref{commute2}. Finally, if  $\{\alpha, \nu\}$ appears in (3),  the corresponding result follows from Lemmas~\ref{commute4}, \ref{commute7}--\ref{commute9} and those for  $\mfg\in\{\mathfrak{so}_{2n},\mathfrak{so}_{2n+1}\}$ in Lemma~\ref{commute2}. This completes the proof of Proposition~\ref{commut1} for  $c_{\mathbf r}^-$. Applying $\bar\tau$ yields the results on $c_{\mathbf r}^+$.

\section{Proof of Proposition~\ref{roo}}
 Later on,  we will freely use $k_\lambda x_{\nu}^-=v^{-(\lambda\mid \nu)}x_{\nu}^-k_{\lambda}, \forall (\nu, \lambda)\in\mathcal{R}^+\times  \mathcal{P}$.

\begin{Lemma}\label{copro1}Suppose that $\mfg \in\{ \mathfrak{so}_{2n}, \mathfrak{sp}_{2n}, \mathfrak{so}_{2n+1}\}$ and $1\leq i<j\leq n $. Then
\begin{itemize}\item[(1)]$\Delta(x_{\epsilon_i-\epsilon_j}^{-})=x_{\epsilon_i-\epsilon_j}^{-}\otimes k_{-(\epsilon_i-\epsilon_j)}+1\otimes x_{\epsilon_i-\epsilon_j}^{-}-qz_q\sum_{i<t<j}x_{\epsilon_t-\epsilon_j}^{-}\otimes k_{-(\epsilon_t-\epsilon_j)}x_{\epsilon_i-\epsilon_t}^{-}$,  \item[(2)]$\overline{\Delta}(x_{\epsilon_i-\epsilon_j}^{-})=x_{\epsilon_i-\epsilon_j}^{-}\otimes k_{\epsilon_i-\epsilon_j}+1\otimes x_{\epsilon_i-\epsilon_j}^{-}-z_q\sum_{i<t<j}x_{\epsilon_i-\epsilon_t}^{-}\otimes k_{\epsilon_i-\epsilon_t}x_{\epsilon_t-\epsilon_j}^{-}$.
\end{itemize}
\end{Lemma}
\begin{proof}
If $j-i=1$, then $x_{\epsilon_i-\epsilon_j}^{-}=x_i^-$. By \eqref{rell}-\eqref{rell1}, we immediately have  (1) and (2). In general,
 by Corollary~\ref{25a}(1),
$$\Psi(x_{\epsilon_i-\epsilon_j}^{-})
=\Psi(x_{\epsilon_{j-1}-\epsilon_j}^-)\Psi(x_{\epsilon_i-\epsilon_{j-1}}^-)-
q\Psi(x_{\epsilon_i-\epsilon_{j-1}}^-)\Psi(x_{\epsilon_{j-1}-\epsilon_j}^-),$$ $\Psi\in \{\Delta, \bar\Delta\}.$
Rewriting $\Psi(x_{\epsilon_i-\epsilon_{j-1}}^-)$ by induction assumption on  $j-1-i$ and using     $x_{\epsilon_i-\epsilon_t}^-x_{\epsilon_{j-1}-\epsilon_j}^-=x_{\epsilon_{j-1}-\epsilon_j}^-x_{\epsilon_i-\epsilon_t}^-$ if $i<t<j-1$ (see Corollary~\ref{monoial1}) yield the following two equations:
 $$\begin{aligned}
\Delta(x_{\epsilon_i-\epsilon_j}^{-})=&x_{\epsilon_i-\epsilon_j}^-\otimes k_{-(\epsilon_i-\epsilon_j)}+1\otimes x_{\epsilon_i-\epsilon_j}^-
-qz_qx_{\epsilon_{j-1}-\epsilon_j}^-\otimes k_{-(\epsilon_{j-1}-\epsilon_j)}x_{\epsilon_i-\epsilon_{j-1}}^-
\\& +q^2z_q\sum_{i<t<j-1}[x_{\epsilon_t-\epsilon_{j-1}}^{-}, x_{\epsilon_{j-1}-\epsilon_j}^-]_v\otimes k_{-(\epsilon_t-\epsilon_j)}x_{\epsilon_i-\epsilon_t}^{-},
\end{aligned}$$
$$\begin{aligned}\bar\Delta(x_{\epsilon_i-\epsilon_j}^{-})=& x_{\epsilon_i-\epsilon_j}^-\otimes k_{\epsilon_i-\epsilon_j}+1\otimes x_{\epsilon_i-\epsilon_j}^-
-z_q\sum_{i<t<j-1}x_{\epsilon_i-\epsilon_t}^{-}\otimes k_{\epsilon_i-\epsilon_t}x_{\epsilon_{j-1}-\epsilon_j}^-x_{\epsilon_t-\epsilon_{j-1}}^{-}
\\ &
+qz_q\sum_{i<t<j-1}x_{\epsilon_i-\epsilon_t}^{-}\otimes k_{\epsilon_i-\epsilon_t}[x_{\epsilon_t-\epsilon_{j-1}}^{-}, x_{\epsilon_{j-1}-\epsilon_j}^-]_v.
\end{aligned}$$
 Using  Corollary~\ref{25a}(1) to simplify $[x_{\epsilon_t-\epsilon_{j-1}}^{-}, x_{\epsilon_{j-1}-\epsilon_j}^-]_v$ above, we immediately have  (1)-(2).
\end{proof}
\begin{Lemma}\label{copro2}Suppose that $\mfg = \mathfrak{so}_{2n+1}$ and $1\leq j\leq n$. Then
\begin{itemize}\item[(1)]$\Delta(x_{\epsilon_j}^{-})=x_{\epsilon_j}^{-}\otimes k_{-\epsilon_j}+1\otimes x_{\epsilon_j}^{-}-qz_q\sum_{j<t\leq n}x_{\epsilon_t}^{-}\otimes k_{-\epsilon_t}x_{\epsilon_j-\epsilon_t}^{-}$,  \item[(2)]$\overline{\Delta}(x_{\epsilon_j}^{-})=x_{\epsilon_j}^{-}\otimes k_{\epsilon_j}+1\otimes x_{\epsilon_j}^{-}-z_q\sum_{j<t\leq n}x_{\epsilon_j-\epsilon_t}^{-}\otimes k_{\epsilon_j-\epsilon_t}x_{\epsilon_t}^{-}$.
\end{itemize}
\end{Lemma}
\begin{proof}
If $n-j=0$, $x_{\epsilon_j}^{-}=x_{n}^{-}$. By \eqref{rell} and \eqref{rell1}, we have  (1)-(2).   Suppose $n-j>0$. By Corollary~\ref{25a}(4),
 $$\Psi(x_{\epsilon_j}^{-})
=\Psi(x_{\epsilon_n}^-)\Psi(x_{\epsilon_j-\epsilon_n}^-)-
q\Psi(x_{\epsilon_j-\epsilon_n}^-)\Psi(x_{\epsilon_n}^-)
,$$  $\Psi\in \{\Delta, \bar\Delta\}$. Rewriting $\Psi(x_{\epsilon_j-\epsilon_{n}}^-)$ via Lemma~\ref{copro1} and using     $x_{\epsilon_j-\epsilon_t}^-x_{\epsilon_n}^-=x_{\epsilon_n}^-x_{\epsilon_j-\epsilon_t}^-$ if $j<t<n$ (see Corollary~\ref{monoial1}) yield the following two equations:
  $$\Delta(x_{\epsilon_j}^{-})
=x_{\epsilon_j}^-\otimes k_{-{\epsilon_j}}+1\otimes x_{\epsilon_j}^--qz_qx_{\epsilon_n}^-\otimes k_{-\epsilon_n}x_{\epsilon_j-\epsilon_{n}}^-+q^2z_q\sum_{j<t<n}[x_{\epsilon_t-\epsilon_{n}}^{-}, x_{\epsilon_n}^-]_v\otimes k_{-\epsilon_t}x_{\epsilon_j-\epsilon_t}^{-},
$$
$$\bar\Delta(x_{\epsilon_j}^{-})=x_{\epsilon_j}^-\otimes k_{{\epsilon_j}}+1\otimes x_{\epsilon_j}^--z_qx_{\epsilon_j-\epsilon_{n}}^-\otimes k_{\epsilon_j-\epsilon_{n}}x_{\epsilon_{n}}^-+qz_q\sum_{j<t<n}x_{\epsilon_j-\epsilon_{t}}^{-}\otimes k_{\epsilon_j-\epsilon_t}[x_{\epsilon_t-\epsilon_n}^{-}, x_{\epsilon_n}^-]_v.
$$
Finally, we have   (1) and (2)  after we use  Corollary~\ref{25a}(4) to simplify  $[x_{\epsilon_t-\epsilon_n}^{-}, x_{\epsilon_n}^-]_v$ above.
\end{proof}

Suppose $1\leq h<l\leq n $. Define
 $$D_{\epsilon_h+\epsilon_l}=\begin{cases}\{\epsilon_h-\epsilon_j, \epsilon_h, \epsilon_h+\epsilon_t \mid h<j\leq n, l< t\leq n\}, &\text{if $\mfg=\mathfrak{so}_{2n+1}$,}\\ \{\epsilon_h-\epsilon_j, \epsilon_h+\epsilon_t \mid h<j\leq n, l<t\leq n\}, &\text{if $\mfg\in\{\mathfrak{so}_{2n}, \mathfrak{sp}_{2n}\}$.}\\
 \end{cases} $$

\begin{Lemma}\label{copro3}Suppose that $\mfg \in\{ \mathfrak{so}_{2n}, \mathfrak{sp}_{2n}, \mathfrak{so}_{2n+1}\}$ and $1\leq h<l\leq n $.
 Then
\begin{itemize}\item[(1)]$\Delta(x_{\epsilon_h+\epsilon_l}^-)=x_{\epsilon_h+\epsilon_l}^{-}\otimes k_{-(\epsilon_h+\epsilon_l)}+1\otimes x_{\epsilon_h+\epsilon_l}^{-} + \sum_{I,\gamma}h_{I,\gamma}x_I^{-}\otimes k_{-wt(I)}x_\gamma^{-}$,  \item[(2)]$\overline{\Delta}(x_{\epsilon_h+\epsilon_l}^{-})=x_{\epsilon_h+\epsilon_l}^{-}\otimes k_{\epsilon_h+\epsilon_l}+1\otimes x_{\epsilon_h+\epsilon_l}^{-} + \sum_{I,\gamma}g_{I,\gamma}x_{\gamma}^{-}\otimes k_{\gamma}x_{I}^{-}$,
\end{itemize}
where  $I$ ranges over  non-empty sequence of positive roots and $\gamma\in D_{\epsilon_h+\epsilon_l}$ such that $wt(I)+\gamma=\epsilon_h+\epsilon_l$
and, both $h_{I,\gamma}$ and $g_{I,\gamma}$ are in $\mathcal A_{\ell(I)}$.\end{Lemma}
\begin{proof} Let  $l=n$. First, we assume  $\mfg=\mathfrak{sp}_{2n}$.   By Corollary~\ref{25a}(6),
$$\Psi(x_{\epsilon_h+\epsilon_n}^{-})
=\Psi(x_{2\epsilon_{n}}^-)\Psi(x_{\epsilon_h-\epsilon_{n}}^-)-
q^2\Psi(x_{\epsilon_h-\epsilon_{n}}^-)\Psi(x_{2\epsilon_{n}}^-),  \Psi\in \{\Delta, \bar\Delta\}.$$
By Corollary~\ref{monoial1}, we have
\begin{equation}\label{aaaa1} x_{\epsilon_h-\epsilon_t}^-x_{2\epsilon_n}^-=x_{2\epsilon_n}^-x_{\epsilon_h-\epsilon_t}^-,  \text{ if $h<t<n$.}
\end{equation}
Rewriting  $\Psi(x_{\epsilon_h-\epsilon_{n}}^-)$ and $\Psi(x_{2\epsilon_{n}}^-)$  via Lemma~\ref{copro1}, \eqref{rell}-\eqref{rell1} and using \eqref{aaaa1}   yield the following two equations:
$$ \begin{aligned}\Delta(x_{\epsilon_h+\epsilon_n}^{-})
&=x_{\epsilon_h+\epsilon_n}^-\otimes k_{-({\epsilon_h+\epsilon_n})}+1\otimes x_{\epsilon_h+\epsilon_n}^-+(1-q^4)x_{2\epsilon_n}^-\otimes k_{-2\epsilon_n}x_{\epsilon_h-\epsilon_n}^-\\&+q^3z_q\sum_{h<t<n}
[x_{\epsilon_t-\epsilon_{n}}^{-}, x_{2\epsilon_n}^-]_v\otimes k_{-(\epsilon_t+\epsilon_n)}x_{\epsilon_h-\epsilon_t}^{-},\end{aligned}$$
$$\begin{aligned}\bar\Delta(x_{\epsilon_h+\epsilon_n}^{-})
&=x_{\epsilon_h+\epsilon_n}^-\otimes k_{{\epsilon_h+\epsilon_n}}+1\otimes x_{\epsilon_h+\epsilon_n}^-+(q^{-2}-q^2)x_{\epsilon_h-\epsilon_n}^-\otimes k_{\epsilon_h-\epsilon_n}x_{2\epsilon_n}^-\\&+q^2z_q\sum_{h<t<n}x_{\epsilon_h-\epsilon_{t}}^{-}\otimes k_{\epsilon_h-\epsilon_t}[x_{\epsilon_t-\epsilon_n}^{-}, x_{2\epsilon_n}^-]_v\end{aligned}$$
Using Corollary~\ref{25a}(6) to simplify $[x_{\epsilon_t-\epsilon_n}^{-}, x_{2\epsilon_n}^-]_v$ above, we have
\begin{equation}\label{ddd4}\begin{aligned} \Delta(x_{\epsilon_h+\epsilon_n}^{-})=&x_{\epsilon_h+\epsilon_n}^-\otimes k_{-({\epsilon_h+\epsilon_n})}+1\otimes x_{\epsilon_h+\epsilon_n}^- -z_qq^2[2]x_{2\epsilon_n}^-\otimes k_{-2\epsilon_n}x_{\epsilon_h-\epsilon_n}^-\\&\quad-qz_q\sum_{h<t<n}
x_{\epsilon_t+\epsilon_{n}}^{-}\otimes k_{-(\epsilon_t+\epsilon_n)}x_{\epsilon_h-\epsilon_t}^{-},
\end{aligned}\end{equation}
\begin{equation}\label{ddd5}\begin{aligned}\bar\Delta(x_{\epsilon_h+\epsilon_n}^{-})
&
=x_{\epsilon_h+\epsilon_n}^-\otimes k_{{\epsilon_h+\epsilon_n}}+1\otimes x_{\epsilon_h+\epsilon_n}^--z_q[2]x_{\epsilon_h-\epsilon_n}^-\otimes k_{\epsilon_h-\epsilon_n}x_{2\epsilon_n}^-\\&\quad-z_q\sum_{h<t< n}x_{\epsilon_h-\epsilon_{t}}^{-}\otimes k_{\epsilon_h-\epsilon_t}x_{\epsilon_t+\epsilon_n}^{-}.
\end{aligned}\end{equation}

Secondly, we assume  $\mfg=\mathfrak{so
}_{2n}$.
 If $h=n-1$, then  $x_{\epsilon_h+\epsilon_l}^-=x_n^-$. In this case,   (1) and (2) follow from \eqref{rell}--\eqref{rell1}. In general, by Corollary~\ref{25a}(5),  $$\Psi(x_{\epsilon_h+\epsilon_n}^{-})
=\Psi(x_{\epsilon_{n-1}+\epsilon_{n}}^-)\Psi(x_{\epsilon_h-\epsilon_{n-1}}^-)-
q\Psi(x_{\epsilon_h-\epsilon_{n-1}}^-)\Psi(x_{\epsilon_{n-1}+\epsilon_{n}}^-), \text{ $\Psi\in \{\Delta, \bar\Delta\}$.}$$
By Corollary~\ref{monoial1}), we have
\begin{equation}\label{aaaa2}  x_{\epsilon_h-\epsilon_t}^-x_{\epsilon_{n-1}+\epsilon_n}^-=x_{\epsilon_{n-1}+\epsilon_n}^-x_{\epsilon_h-\epsilon_t}^-, \text{ if $h<t<n-1$.}\end{equation}
Rewriting $\Psi(x_{\epsilon_h-\epsilon_{n}}^-)$ and $\Psi(x_{\epsilon_{n-1}+\epsilon_{n}}^-)$  via Lemma~\ref{copro1} and \eqref{rell}-\eqref{rell1} and using \eqref{aaaa2}, we have   the following two equations:
 $$ \begin{aligned}\Delta (x_{\epsilon_h+\epsilon_n}^{-})
&=x_{\epsilon_h+\epsilon_n}^-\otimes k_{-{\epsilon_h-\epsilon_n}}+1\otimes x_{\epsilon_h+\epsilon_n}^--qz_qx_{\epsilon_{n-1}+\epsilon_n}^-\otimes k_{-\epsilon_{n-1}-\epsilon_n}x_{\epsilon_h-\epsilon_{n-1}}^-\\ &\quad+q^2z_q
\sum_{h<t<n-1}[x_{\epsilon_t-\epsilon_{n-1}}^{-}, x_{\epsilon_{n-1}+\epsilon_n}^-]_v\otimes k_{-\epsilon_t-\epsilon_{n}}x_{\epsilon_h-\epsilon_t}^{-},\end{aligned}$$
$$\begin{aligned}\bar\Delta(x_{\epsilon_h+\epsilon_n}^{-})
&=x_{\epsilon_h+\epsilon_n}^-\otimes k_{{\epsilon_h+\epsilon_n}}+1\otimes x_{\epsilon_h+\epsilon_n}^--z_qx_{\epsilon_h-\epsilon_{n-1}}^-\otimes k_{\epsilon_h-\epsilon_{n-1}}x_{\epsilon_{n-1}+\epsilon_n}^- \\&\quad+qz_q\sum_{h<t<n-1}x_{\epsilon_h-\epsilon_{t}}^{-}\otimes k_{\epsilon_h-\epsilon_t}[x_{\epsilon_t-\epsilon_{n-1}}^{-}, x_{\epsilon_{n-1}+\epsilon_n}^-]_v.
\end{aligned}$$
Finally, we have $$\begin{aligned}\Delta(x_{\epsilon_h+\epsilon_n}^{-})
=& x_{\epsilon_h+\epsilon_n}^-\otimes k_{-{\epsilon_h-\epsilon_n}}+1\otimes x_{\epsilon_h+\epsilon_n}^--qz_q\sum_{h<t\leq n-1}x_{\epsilon_t+\epsilon_{n}}^{-}\otimes k_{-\epsilon_t-\epsilon_n}x_{\epsilon_h-\epsilon_t}^{-},
\end{aligned}$$ $$\begin{aligned}\bar\Delta(x_{\epsilon_h+\epsilon_n}^{-})
=x_{\epsilon_h+\epsilon_n}^-\otimes k_{{\epsilon_h+\epsilon_n}}+1\otimes x_{\epsilon_h+\epsilon_n}^--z_q\sum_{h<t\leq n-1}x_{\epsilon_h-\epsilon_{t}}^{-}\otimes k_{\epsilon_h-\epsilon_t}x_{\epsilon_{t}+\epsilon_n}^-.
\end{aligned}$$
after we use Corollary~\ref{25a}(5) to simplify $[x_{\epsilon_t-\epsilon_{n-1}}^{-}, x_{\epsilon_{n-1}+\epsilon_n}^-]_v$ above.

 Suppose $\mfg=\mathfrak{so}_{2n+1}$. Thanks to Corollary~\ref{25a}(3), $$\Psi(x_{\epsilon_h+\epsilon_n}^{-})={[}2{]}^{-1}\Psi(x_{\epsilon_{n}}^-)\Psi(x_{\epsilon_h}^-)-{[}2{]}^{-1}
\Psi(x_{\epsilon_h}^-)\Psi(x_{\epsilon_{n}}^-),  ~\Psi\in\{\Delta, \bar \Delta\}.$$
By Corollary~\ref{monoial1}, we have
\begin{equation}\label{aaaa3} x_{\epsilon_h-\epsilon_t}^-x_{\epsilon_n}^-=x_{\epsilon_n}^-x_{\epsilon_h-\epsilon_t}^-, \text{ if $h<t<n$.}\end{equation}
Rewriting $\Psi(x_{\epsilon_h}^-)$ and $\Psi(x_{\epsilon_n}^-)$ via Lemma~\ref{copro2} and \eqref{rell}-\eqref{rell1} and using   \eqref{aaaa3} yield the following two equations:
  $$\begin{aligned}&\Delta(x_{\epsilon_h+\epsilon_n}^{-})
=x_{\epsilon_h+\epsilon_n}^-\otimes k_{-({\epsilon_h+\epsilon_n})}+1\otimes x_{\epsilon_h+\epsilon_n}^-+[2]^{-1}qz_qx_{\epsilon_n}^{-}\otimes k_{-\epsilon_n}[x_{\epsilon_h-\epsilon_n}^{-}, x_{\epsilon_n}^{-}]_v\\&+[2]^{-1}qz_q\sum_{h<t< n}[x_{\epsilon_t}^{-}, x_{\epsilon_n}^{-}]_v\otimes k_{-\epsilon_t}x_{\epsilon_h-\epsilon_t}^{-}k_{-\epsilon_n}-[2]^{-1}qz_q(1-q)(x_{\epsilon_n}^{-})^2\otimes k_{-2\epsilon_n}x_{\epsilon_h-\epsilon_n}^-,
\end{aligned}$$$$\begin{aligned}&\bar\Delta(x_{\epsilon_h+\epsilon_n}^{-})
=x_{\epsilon_h+\epsilon_n}^-\otimes k_{{\epsilon_h+\epsilon_n}}+1\otimes x_{\epsilon_h+\epsilon_n}^-+[2]^{-1}qz_q[x_{\epsilon_h-\epsilon_n}^{-}, x_{\epsilon_n}^{-}]_v\otimes k_{\epsilon_h}x_{\epsilon_n}^{-}\\& +[2]^{-1}z_q(1-q^{-1})x_{\epsilon_h-\epsilon_n}^-\otimes k_{\epsilon_h-\epsilon_n}(x_{\epsilon_n}^{-})^2+[2]^{-1}z_q\sum_{h<t< n}x_{\epsilon_h-\epsilon_t}^{-}\otimes k_{\epsilon_h-\epsilon_t}[x_{\epsilon_t}^{-}, x_{\epsilon_n}^{-}]_v.
\end{aligned}$$
After we use Corollary~\ref{25a}(3)-(4) to simplify  both  $[x_{\epsilon_t}^{-}, x_{\epsilon_n}^{-}]_v$ and $[x_{\epsilon_h-\epsilon_n}^{-}, x_{\epsilon_n}^{-}]_v$  above, we have
$$\begin{aligned}\Delta(x_{\epsilon_h+\epsilon_n}^{-})
=& x_{\epsilon_h+\epsilon_n}^-\otimes k_{-({\epsilon_h+\epsilon_n})}+
1\otimes x_{\epsilon_h+\epsilon_n}^--[2]^{-1}z_qx_{\epsilon_n}^{-}\otimes k_{-\epsilon_n} x_{\epsilon_h}^{-}\\ & -[2]^{-1}qz_q(1-q)(x_{\epsilon_n}^{-})^2\otimes k_{-2\epsilon_n}x_{\epsilon_h-\epsilon_n}^- -qz_q\sum_{h<t<n}x_{\epsilon_t+\epsilon_{n}}^{-}\otimes k_{-(\epsilon_t+\epsilon_n)}x_{\epsilon_h-\epsilon_t}^{-},
\end{aligned}$$
 $$\begin{aligned}\bar\Delta(x_{\epsilon_h+\epsilon_n}^{-})
=& x_{\epsilon_h+\epsilon_n}^-\otimes k_{{\epsilon_h+\epsilon_n}}+1\otimes x_{\epsilon_h+\epsilon_n}^-+[2]^{-1}z_q(1-q^{-1})x_{\epsilon_h-\epsilon_n}^-\otimes k_{\epsilon_h-\epsilon_n}(x_{\epsilon_n}^{-})^2\\& -[2]^{-1}z_qx_{\epsilon_h}^{-}\otimes k_{\epsilon_h} x_{\epsilon_n}^{-}-z_q\sum_{h<t<n}x_{\epsilon_h-\epsilon_t}^{-}\otimes k_{\epsilon_h-\epsilon_t}x_{\epsilon_t+\epsilon_n}^{-}.
\end{aligned}$$

So far, we have verified   (1)--(2)  when $l=n$.  Suppose $n-l>0$.
{By Corollary~\ref{25a}(2)}, for $\Psi\in\{\Delta, \bar\Delta\}$, $$\Psi(x_{\epsilon_h+\epsilon_l}^{-})
=\Psi(x_{\epsilon_l-\epsilon_{l+1}}^-)\Psi(x_{\epsilon_h+\epsilon_{l+1}}^-)-
q\Psi(x_{\epsilon_h+\epsilon_{l+1}}^-)\Psi(x_{\epsilon_l-\epsilon_{l+1}}^-)
.$$  Note that  $\Psi(x_{\epsilon_l-\epsilon_{l+1}}^-)$ can be computed by  via \eqref{rell}-\eqref{rell1}, and   $\Psi(x_{\epsilon_h+\epsilon_{l+1}}^-)$ can be computed  by induction assumption on $n-(l+1)$.  We have
  $$\begin{aligned}&\Delta(x_{\epsilon_h+\epsilon_{l+1}}^-)=x_{\epsilon_h+\epsilon_{l+1}}^{-}\otimes k_{-(\epsilon_h+\epsilon_{l+1})}+1\otimes x_{\epsilon_h+\epsilon_{l+1}}^{-} + \sum_{I_1,\gamma_1}h_{I_1,\gamma_1}x_{I_1}^{-}\otimes k_{-wt (I_1)}x_{\gamma_1}^{-},\\&\bar\Delta(x_{\epsilon_h+\epsilon_{l+1}}^-)=x_{\epsilon_h+\epsilon_{l+1}}^{-}\otimes k_{\epsilon_h+\epsilon_{l+1}}+1\otimes x_{\epsilon_h+\epsilon_{l+1}}^{-}+ \sum_{I_1,\gamma_1}g_{I_1,\gamma_1}x_{\gamma_1}^{-}\otimes k_{\gamma_1}x_{I_1}^{-},\end{aligned}$$ where $I_1$ ranges over non-empty sequence of positive roots and $\gamma_1\in D_{\epsilon_h+\epsilon_{l+1}}$  such that $wt(I_1)+\gamma_1=\epsilon_h+\epsilon_{l+1}$ and $ h_{I_1,\gamma_1}, g_{I_1,\gamma_1}\in\mathcal A_{\ell(I_1)}$.
 Suppose  $l+1<t\leq n$, $h<j<l$ or $l+1<j\leq n$.  By Corollary~\ref{monoial1}, we have \begin{equation} \label{aaaa4} \begin{aligned}  & x_{\epsilon_h+\epsilon_t}^-x_{\epsilon_l-\epsilon_{l+1}}^-=x_{\epsilon_l-\epsilon_{l+1}}^-x_{\epsilon_h+\epsilon_t}^-,  \ \ x_{\epsilon_h-\epsilon_{j}}^-x_{\epsilon_l-\epsilon_{l+1}}^-=x_{\epsilon_l-\epsilon_{l+1}}^-x_{\epsilon_h-\epsilon_{j}}^-,\\ & x_{\epsilon_h}^-x_{\epsilon_l-\epsilon_{l+1}}^-=x_{\epsilon_l-\epsilon_{l+1}}^-x_{\epsilon_h}^-, \ \  \quad\quad\ \
   x_{\epsilon_h-\epsilon_{l+1}}^-x_{\epsilon_l-\epsilon_{l+1}}^-=qx_{\epsilon_l-\epsilon_{l+1}}^-x_{\epsilon_h-\epsilon_{l+1}}^-. \end{aligned}\end{equation}
 Using \eqref{aaaa4} to simplify two equations above, we have
 \begin{equation}\begin{aligned}\label{eee1}\Delta&(x_{\epsilon_h+\epsilon_l}^{-})
-x_{\epsilon_h+\epsilon_l}^-\otimes k_{-({\epsilon_h+\epsilon_l})}-1\otimes x_{\epsilon_h+\epsilon_l}^-+qz_q
x_{\epsilon_l-\epsilon_{l+1}}^-\otimes k_{-({\epsilon_l-\epsilon_{l+1}})}x_{\epsilon_h+\epsilon_{l+1}}^-
\\& +q\sum_{I_1} h_{I_1,\epsilon_h-\epsilon_l}x_{I_1}^{-}
\otimes k_{-wt(I_1)}[x_{\epsilon_h-\epsilon_l}^-, x_{\epsilon_l-\epsilon_{l+1}}^-]_v
 \\ & =\sum_{I_1,\gamma_1} h_{I_1,\gamma_1}(x_{\epsilon_l-\epsilon_{l+1}}^-x_{I_1}^{-}- q^{1-(\gamma_1, \epsilon_l-\epsilon_{l+1})}x_{I_1}^{-}x_{\epsilon_l-\epsilon_{l+1}}^- )\otimes k_{-wt(I_1)-({\epsilon_l-\epsilon_{l+1}})}x_{\gamma_1}^{-}
\end{aligned}\end{equation} and \begin{equation}\begin{aligned}\label{eee2}\bar\Delta&(x_{\epsilon_h+\epsilon_l}^{-})
- x_{\epsilon_h+\epsilon_l}^-\otimes k_{{\epsilon_h+\epsilon_l}}-1\otimes x_{\epsilon_h+\epsilon_l}^-+z_q
x_{\epsilon_h+\epsilon_{l+1}}^-\otimes k_{\epsilon_h+\epsilon_{l+1}}x_{\epsilon_l-\epsilon_{l+1}}^-
\\&+q\sum_{I_1}g_{I_1,\epsilon_h-\epsilon_l}[x_{\epsilon_h-\epsilon_l}^{-},
x_{\epsilon_l-\epsilon_{l+1}}^{-}]_v\otimes k_{\epsilon_h-\epsilon_{l+1}}x_{I_1}^{-}
\\&=q^{(\gamma_1, \epsilon_l-\epsilon_{l+1})}\sum_{I_1,\gamma_1}g_{I_1,\gamma_1}x_{\gamma_1}^{-}\otimes k_{\gamma_1}(x_{\epsilon_l-\epsilon_{l+1}}^-x_{I_1}^{-}-q^{1-(\gamma_1, \epsilon_l-\epsilon_{l+1})}x_{I_1}^{-}x_{\epsilon_l-\epsilon_{l+1}}^-).
\end{aligned}\end{equation}
Thanks to Corollary~\ref{25a}(1), we  replace $[x_{\epsilon_h-\epsilon_l}^{-},
x_{\epsilon_l-\epsilon_{l+1}}^{-}]_v$ by $-q^{-1}x_{\epsilon_h-\epsilon_{l+1}}^{-}$ in those two equations.
So, we only need to deal with the (RHS) of \eqref{eee1}--\eqref{eee2}. Since $ h_{I_1,\gamma_1}, g_{I_1,\gamma_1}\in A_{\ell(I_1)}$,
by Proposition~\ref{commut1}(1), for each pair $\{I_1, \gamma_1\}$ there exists $c\in\mathbb{Z}$ such that: $$D(x_{\epsilon_l-\epsilon_{l+1}}^-x_{I_1}^{-}-q^{1-(\gamma_1, \epsilon_l-\epsilon_{l+1})}x_{I_1}^{-}x_{\epsilon_l-\epsilon_{l+1}}^-)= D(1-q^c)x_{\epsilon_l-\epsilon_{l+1}}^-x_{I_1}^{-}+\sum_{J}c_{J}x_{J}^{-}, D\in\{ h_{I_1,\gamma_1}, g_{I_1,\gamma_1}\},$$where $J$ is non-empty sequence of positive
roots such that $wt(J)=wt(I_1)$ and $c_{J}\in\mathcal A_{\ell(J)}$. We rewrite  the (RHS) of \eqref{eee1}--\eqref{eee2} via the above equation.
Now, (1)--(2) follow,
since $(1-q^c)D\in\mathcal A_{\ell(I_1)+1}$, no mater whether $c=0$ or not and, $x_{I}^-=x_{\epsilon_l-\epsilon_{l+1}}^-x_{I_1}^{-}$  where $I=(I_1, \epsilon_l-\epsilon_{l+1})$, $\ell(I)=\ell(I_1)+1$.\end{proof}

\begin{Lemma}\label{copro4}Suppose that $\mfg=\mathfrak{sp}_{2n}$ and $1\leq h\leq n$. Then \begin{itemize}\item[(1)]$\Delta(x_{2\epsilon_h}^{-})=x_{2\epsilon_h}^{-}\otimes k_{-2\epsilon_h}+1\otimes x_{2\epsilon_h}^{-} + \sum_{I,J} h_{I,J}x_I^{-}\otimes k_{-wt(I)}x_J^{-}$,  \item[(2)]$\overline{\Delta}(x_{2\epsilon_h}^{-})=x_{2\epsilon_h}^{-}\otimes k_{2\epsilon_h}+1\otimes x_{2\epsilon_h}^{-} + \sum_{I,J}g_{I,J}x_I^{-}\otimes k_{wt(I)}x_J^{-}$,
\end{itemize}
where $I$, $J$ are non-empty sequences of positive roots such that $wt(I)+wt(J)=2\epsilon_h$ and, both $ h_{I,J}$ and $g_{I,J}$ are in $\mathcal A_{max\{\ell(I), \ell(J)\}}$.
\end{Lemma}

\begin{proof}
If $h=n$, then $x_{2\epsilon_h}^{-}=x_{2n}^-$. In this case,     (1)-(2) follow from \eqref{rell}--\eqref{rell1}. Suppose $h<n$. By Corollary~\ref{25a}(7),
$$\Psi(x_{2\epsilon_h}^{-})
=[2]^{-1}\Psi(x_{\epsilon_h+\epsilon_{n}}^-)\Psi(x_{\epsilon_h-\epsilon_{n}}^-)-
[2]^{-1}\Psi(x_{\epsilon_h-\epsilon_{n}}^-)\Psi(x_{\epsilon_h+\epsilon_{n}}^-),$$
$\Psi\in\{\Delta, \bar\Delta\}$. By Corollary~\ref{monoial1},
\begin{equation}\label{aaaa5} x_{\epsilon_h-\epsilon_t}^-x_{\epsilon_{h}\pm\epsilon_n}^-=qx_{\epsilon_{h}\pm\epsilon_n}^-x_{\epsilon_h-\epsilon_t}^-, \text{ if $h<t<n$.}\end{equation}
 Rewriting $\Psi(x_{\epsilon_h-\epsilon_{n}}^-)$ and $\Psi(x_{\epsilon_h+\epsilon_{n}}^-)$  via Lemma~\ref{copro1} and \eqref{ddd4}-\eqref{ddd5} and using \eqref{aaaa5}, we have
  $$\begin{aligned}&\Delta(x_{2\epsilon_h}^{-})
=x_{2\epsilon_h}^-\otimes k_{2\epsilon_h}^{-1}+1\otimes x_{2\epsilon_h}^-+q^3z_q^2x_{2\epsilon_n}^-\otimes k_{-2\epsilon_n}(x_{\epsilon_h-\epsilon_{n}}^-)^2-z_qx_{\epsilon_{h}+\epsilon_n}^-\otimes k_{-(\epsilon_{h}+\epsilon_n)}x_{\epsilon_h-\epsilon_{n}}^-\\&~+q^2z_q^2\sum_{h<t<n}(x_{2\epsilon_n}^-x_{\epsilon_{t}-\epsilon_n}^-\otimes k_{-(\epsilon_{t}+\epsilon_n)}x_{\epsilon_h-\epsilon_{n}}^-x_{\epsilon_h-\epsilon_{t}}^--qx_{\epsilon_{t}-\epsilon_n}^-x_{2\epsilon_{n}}^-\otimes k_{-(\epsilon_{t}+\epsilon_n)}x_{\epsilon_h-\epsilon_{t}}^-x_{\epsilon_h-\epsilon_{n}}^-)\\&~
+[2]^{-1}q^2z_q^2\sum_{h<s< n, h<t<n}q^{(\epsilon_s, \epsilon_t)}(x_{\epsilon_{s}+\epsilon_n}^-x_{\epsilon_{t}-\epsilon_n}^--x_{\epsilon_{s}-\epsilon_n}^-x_{\epsilon_{t}+\epsilon_n}^-)\otimes k_{-(\epsilon_{s}+\epsilon_t)}x_{\epsilon_h-\epsilon_{s}}^-x_{\epsilon_h-\epsilon_{t}}^-
\\&~-q[2]^{-1}z_q\sum_{h<s< n}([x_{\epsilon_{h}-\epsilon_n}^-, x_{\epsilon_{s}+\epsilon_n}^-]_v-[x_{\epsilon_{h}+\epsilon_n}^-, x_{\epsilon_{s}-\epsilon_n}^-]_v)\otimes k_{-(\epsilon_{h}+\epsilon_s)}x_{\epsilon_h-\epsilon_{s}}^-
,
\end{aligned}$$
and
 $$\begin{aligned}\bar\Delta&(x_{2\epsilon_h}^{-})
=x_{2\epsilon_h}^-\otimes k_{2\epsilon_h}+1\otimes x_{2\epsilon_h}^-+\frac{1}{q}z_q^2  (x_{\epsilon_h-\epsilon_{n}}^-)^2\otimes k_{2\epsilon_h-2\epsilon_n}x_{2\epsilon_n}^--z_qx_{\epsilon_h-\epsilon_{n}}^-\otimes k_{\epsilon_h-\epsilon_{n}}x_{\epsilon_{h}+\epsilon_n}^-\\&+z_q^2\sum_{h<t<n}  (x_{\epsilon_h-\epsilon_{n}}^-x_{\epsilon_h-\epsilon_{t}}^- \otimes k_{2\epsilon_h-\epsilon_{n}-\epsilon_t}x_{2\epsilon_n}^-x_{\epsilon_{t}-\epsilon_n}^--q  x_{\epsilon_h-\epsilon_{t}}^-x_{\epsilon_h-\epsilon_{n}}^- \otimes k_{2\epsilon_h-\epsilon_{n}-\epsilon_t}x_{\epsilon_{t}-\epsilon_n}^-x_{2\epsilon_{n}}^-)
\\&
+[2]^{-1}z_q^2\sum_{h<s< n, h<t<n}q^{-(\epsilon_s, \epsilon_t)}  x_{\epsilon_h-\epsilon_{s}}^-x_{\epsilon_h-\epsilon_{t}}^- \otimes k_{2\epsilon_h-\epsilon_{s}-\epsilon_t}(x_{\epsilon_{s}+\epsilon_n}^-x_{\epsilon_{t}-\epsilon_n}^--x_{\epsilon_{s}-\epsilon_n}^-x_{\epsilon_{t}+\epsilon_n}^-)
\\&-[2]^{-1}qz_q\sum_{h<t<n}x_{\epsilon_h-\epsilon_{t}}^-\otimes k_{\epsilon_h-\epsilon_{t}}([x_{\epsilon_{h}-\epsilon_n}^-, x_{\epsilon_{s}+\epsilon_n}^-]_v-[x_{\epsilon_{h}+\epsilon_n}^-, x_{\epsilon_{t}-\epsilon_n}^-]_v).
\end{aligned}$$ Then (1)-(2) can be verified by rewriting $[x_{\epsilon_{h}-\epsilon_n}^-, x_{\epsilon_{s}+\epsilon_n}^-]_v$ and $[x_{\epsilon_{h}+\epsilon_n}^-, x_{\epsilon_{t}-\epsilon_n}^-]_v$ via Lemmas \ref{commute1}, \ref{commute5}, respectively.
 \end{proof}
``\textbf{Proof of Proposition~\ref{roo}}"  Thanks to  \eqref{typeb}, any positive root $\beta$ is one of those in Lemmas~\ref{copro1}--\ref{copro4}.  So, Proposition~\ref{roo} follows from Lemmas~\ref{copro1}--\ref{copro4}. \qed
\end{appendix}

\end{document}